\RequirePackage{fix-cm}
\documentclass[smallextended]{svjour3}       
\smartqed  
\usepackage{graphicx} 
\usepackage{caption}
\usepackage{subcaption}
\usepackage{xspace}                       
\usepackage{cite}
\usepackage{amsmath}                      
\usepackage{amsfonts}                     
\usepackage{mathtools}                    

\usepackage[usenames,dvipsnames]{xcolor}  
\usepackage{threeparttable} 
\usepackage{placeins} 

\usepackage{hyperref} 
\hypersetup{
    colorlinks,
    linkcolor={black},
    citecolor={black},
    urlcolor={blue!80!black}
}

\newcommand{\eqnref}[1]{(\ref{#1})}

\newcommand{\figref}[1]{Figure~\ref{#1}}

\newcommand{\thmref}[1]{Theorem~\ref{#1}}

\newcommand{\etal}[0]{{\em et~al.\@}\xspace}
\newcommand{\eg}[0]{{e.g.\@}\xspace}
\newcommand{\ie}[0]{{i.e.\@}\xspace}

\newcommand{\norm}[1]{\left\lVert#1\right\rVert}
\newcommand{\abs}[1]{\left\lvert#1\right\rvert}
\newcommand{\T}[0]{\mathrm{T}} 
\DeclareMathOperator{\diag}{diag}

\newcommand{\bsym}[1]{\ensuremath{\boldsymbol{#1}}}

\newcommand{\fnc}[1]{\ensuremath{\mathcal{#1}}}
\newcommand{\vecfnc}[1]{\ensuremath{\bsym{\mathcal{#1}}}} 

\renewcommand{\vec}[1]{\ensuremath{\bsym{#1}}}

\newcommand{\mat}[1]{\ensuremath{\mathsf{#1}}}

\newcommand{\U}[0]{\ensuremath{\fnc{U}}}

\newcommand{\zeros}[0]{\ensuremath{\vec{0}}}
\newcommand{\ones}[0]{\ensuremath{\vec{1}}}
\newcommand{\bu}[0]{\ensuremath{\vec{u}}}

\newcommand{\Hnrm}[0]{\mat{H}}
\newcommand{\HI}[0]{\mat{H}^{-1}}

\newcommand{\pder}[2][]{\dfrac{\partial #1}{\partial #2}} 
\newcommand{\der}[2][]{\dfrac{\mathrm{d} #1}{\mathrm{d} #2}} 
\spnewtheorem{assumption}{Assumption}{\bfseries}{\itshape}

\makeatletter
\renewcommand{\@seccntformat}[1]{%
  \ifcsname format@#1\endcsname
    \csname format@#1\endcsname
  \else
    \csname the#1\endcsname\quad 
  \fi
}
\g@addto@macro\appendix{%
  \def\format@section{Appendix \thesection: }%
}
\makeatother

\begin{document}
\begingroup
\renewcommand*{\arraystretch}{1.3}

\title{Stable Volume Dissipation for High-Order Finite-Difference and Spectral-Element Methods with the Summation-by-Parts Property\thanks{Portions of this paper previously appeared in the second author's PhD thesis~\cite{CraigPenner2022_thesis}.}
}

\titlerunning{Stable Volume Dissipation for FD and SD Methods with the SBP Property} 

\author{Alex Bercik            \and
        David A.\ Craig Penner \and 
        David W.\ Zingg 
}


\institute{A.\ Bercik \at
              University of Toronto Institute for Aerospace Studies, Toronto, Canada \\
              \email{alex.bercik@mail.utoronto.ca}
           \and
           D.\ A.\ Craig Penner \at
              University of Toronto Institute for Aerospace Studies, Toronto, Canada \\
              \email{david.a.craigpenner@nasa.gov}  \\
              \emph{Present address:} Analytical Mechanics Associates, NASA Ames Research Center, Moffett Field, United States
           \and
           D.\ W.\ Zingg \at
              University of Toronto Institute for Aerospace Studies, Toronto, Canada \\
              \email{david.zingg@utoronto.ca}
}

\date{Received: date / Accepted: date}

\maketitle

\begin{abstract}
The construction of stable, conservative, and accurate volume dissipation is extended to discretizations that possess a generalized summation-by-parts (SBP) property within a tensor-product framework. The dissipation operators can be applied to any finite-difference or spectral-element scheme that uses the SBP framework, including high-order entropy-stable schemes. Additionally, we clarify the incorporation of a variable coefficient within the operator structure and analyze the impact of a boundary correction matrix on operator structure and accuracy. Following the theoretical development and construction of novel dissipation operators, we relate the presented volume dissipation to the use of upwind SBP operators. When applied to spectral-element methods, the presented approach yields unique dissipation operators that can also be derived through alternative approaches involving orthogonal polynomials. Numerical examples featuring the linear convection, Burgers, and Euler equations verify the properties of the constructed dissipation operators and assess their performance compared to existing upwind SBP schemes, including linear stability behaviour. When applied to entropy-stable schemes, the presented approach results in accurate and robust methods that can solve a broader range of problems where comparable existing methods fail.

\end{abstract}


\section{Introduction} \label{sec:introduction}

State-of-the-art simulation methods in science and engineering are often constrained by scale and computational cost. High-order discretizations offer the potential to increase computational efficiency; however, their widespread adoption for nonlinear partial differential equations (PDEs) is limited in part due to stability concerns, especially in the presence of under-resolved scales or discontinuities that can lead to numerical instabilities~\cite{wang_high-order_2013, gassner_accuracy_2013}. Stabilization techniques such as artificial dissipation are commonly employed to ensure robustness~\cite{pulliam2014}, but the resulting schemes often lack rigorous stability proofs and require problem-specific tuning.

Summation-by-parts (SBP) operators, often combined with simultaneous approximation terms (SATs), provide a framework for constructing provably stable high-order discretizations. Further benefits of the SBP-SAT framework include discrete conservation, $C^0$ mesh continuity requirements between blocks, compatibility with adaptive mesh refinement, and suitability for parallelization~\cite{DelReyFernandez2014_review, Svard_review}. Initially developed for finite-difference (FD) schemes~\cite{Kreiss_1974, Strand_1994}, the SBP framework can also be applied to finite-volume~\cite{nordstrom_finite_2001, nordstrom_finite_2003, Chandrashekar_finite_2016} and spectral-element (SE) methods, including continuous Galerkin~\cite{Hicken2020, hicken_multidimensional_2016, Abgrall_cg1, Abgrall_cg2}, discontinuous Galerkin~\cite{Carpenter_entropy_2014, Chan_discretely_2018, gassner_accuracy_2013}, and flux reconstruction~\cite{Ranocha_fr, Montoya_fr}. Classical SBP (CSBP) operators are FD operators characterized by uniform nodal distributions that include boundary nodes. More broadly, however, FD SBP operators allow for a fixed number of non-equispaced nodes near boundaries, including hybrid Gauss-trapezoidal-Lobatto (HGTL) operators~\cite{ddrf_operators} and the optimal diagonal-norm `accurate' operators of Mattsson~\etal~\cite{Mattsson2014,Mattsson2018} (henceforth referred to as Mattsson operators). FD SBP operators are also not required to include boundary nodes, such as hybrid Gauss-trapezoidal (HGT) operators~\cite{ddrf_operators}. SE SBP operators are characterized by a fixed number of degrees of freedom per element, with mesh refinement performed by increasing the number of elements. Examples include Legendre-Gauss-Lobatto (LGL) and Legendre-Gauss (LG) operators~\cite{DelReyFernandez2014_generalized}. A recent comparison in~\cite{Boom2024} showed that Mattsson operators are often the most efficient in minimizing solution error when applied using either FD or SE refinement strategies.

Provably entropy-stable schemes were first proposed in~\cite{Tadmor_1987, Tadmor_2003} to alleviate the stability issues associated with high-order methods for nonlinear problems. Subsequently, the SBP framework enabled the construction of high-order entropy-stable schemes on finite domains~\cite{Fisher_2013}, which have since shown promise across various physical problems~\cite{chen_review}. Although volume dissipation is not required for provable stability within the context of these schemes, the damping of unresolved high-frequency modes it provides can nevertheless improve accuracy, accelerate convergence to steady state, and help preserve positivity of thermodynamic variables \cite{pulliam2014}. Moreover, central `energy-stable' schemes---such as in~\cite{hicken2008}---that lack provable stability for nonlinear problems remain widely used, and for these schemes, volume dissipation can be essential for stabilization when these schemes are applied to nonlinear problems.

Concerns have also been raised regarding the local linear stability of entropy-stable schemes, wherein unphysical perturbation growth can lead to positivity violations of thermodynamic variables in the absence of positivity-preserving mechanisms~\cite{Gassner2022}, potentially impacting the convergence of such schemes~\cite{Strang1964}. Nonlinear stability refers to an integral bound on the energy or entropy of the system, such as those found in entropy-stability proofs. Local linear stability, however, refers to whether the linearized discretization shares the same stability properties as the linearized continuous PDE~\cite{Gassner2022}. For the Burgers equation with a solution that does not change sign, local linear stability can be characterized by eigenvalues of the semi-discrete Jacobian having non-positive real parts. For linear problems and discretizations, these two stability concepts coincide. For nonlinear problems, however, they are distinct, and many nonlinearly stable schemes demonstrate weaker local linear stability properties than nonlinearly unstable central schemes \cite{Gassner2022}. It is not yet understood to what degree adding volume dissipation can mitigate this issue.

Although many methods exist for removing under-resolved high-frequency modes, implementing a volume dissipation model has proven to be a simple and effective strategy for SBP discretizations~\cite{Mattsson2004, Hicken_2014, Kord2023}. Similarly, upwind operators introduce volume dissipation through flux splitting~\cite{Steger_flux, vanLeer_flux-vector}, and have been applied to SBP discretizations~\cite{Mattsson2017, Ranocha2024, Glaubitz2024}. Recently in~\cite{duru2024dual}, volume dissipation inspired by the upwind operators of~\cite{Mattsson2017} was applied to the skew-symmetric form of the Euler equations of~\cite{NordstromSkewForm}, resulting in a scheme less likely to crash due to positivity issues than typical flux splitting or entropy-conserving methods. An entropy-stable volume dissipation model was also proposed in~\cite{Ranocha_dissipation}, where the dissipation operator is diagonalizable with respect to a Legendre polynomial basis, enabling its use as a filter in SE schemes. 

This paper has two objectives. The first objective is to develop high-order volume dissipation that is dimensionally consistent, can maintain design order, and preserve nonlinear stability when combined with existing entropy-stable SBP schemes. This work extends the approach of Mattsson~\etal~\cite{Mattsson2004} for constructing volume dissipation for CSBP operators to generalized FD diagonal-norm SBP operators, including nonuniform nodal distributions, no repeating interior operator, and potential absence of boundary nodes. We adapt the dissipation structure from~\cite{Mattsson2004} to allow for the inclusion of a variable coefficient and to ensure dimensional consistency, as well as introduce an optional modification for use with entropy-stable schemes. The approach presented applies to both FD and SE SBP methods. The second objective is to reassess the role of volume dissipation in the context of provably stable schemes, motivated by the possibility that dissipation may remain beneficial even when not required for stabilization.

The framework is first introduced in \S \ref{sec:art_diss_scal}, then extended to systems of conservation laws in \S \ref{sec:systems}. Novel artificial dissipation operators are constructed in \S \ref{sec:construction}, where all coefficients are provided in the supplementary material. We end the section with proofs of stability, conservation, and accuracy in \S \ref{sec:proofs}. In \S \ref{sec:implement_vol_diss_for_sbp}, practical implementation details are addressed. First, the inclusion of a boundary correction matrix is discussed in \S\ref{sec:effect_B}, followed by a discussion on defining the variable coefficient at half-nodes in \S \ref{sec:variable_coeffient}. The extension to multiple dimensions is clarified in \S\ref{sec:multiple_dim}. In \S \ref{sec:upwind_and_spectral}, we address recent approaches for incorporating volume dissipation through the use of upwind operators. We first examine connections to upwind schemes in \S \ref{sec:relation_to_upwind}, then in \S \ref{sec:gsbp_uniqueness} show that SE dissipation operators are unique. Finally, in \S \ref{sec:results}, numerical experiments are used to verify the stability, conservation, and accuracy properties of the constructed dissipation operators.

\section{Volume Dissipation for Summation-by-Parts Schemes}\label{sec:vol_diss_for_sbp}

\subsection{Scalar Conservation Laws} \label{sec:art_diss_scal}

To motivate the structure of the volume dissipation operators used in this work, consider the linear convection equation in one spatial dimension. Taking inspiration from~\cite{Mattsson2004}, suppose we include an additional continuous operator on the right-hand side as follows,
\begin{equation}\label{eq:1d_conv}
    \frac{\partial \U}{\partial t} + a(x) \frac{\partial \U}{\partial x} = \varepsilon (-1)^{s+1} \Delta x^{2s-1} \frac{\partial^s}{\partial x^s} \left( \abs{a(x)} \frac{\partial^s \fnc{U}}{\partial x^s} \right), 
\end{equation}
where $\U(x,t)$ is the solution, $a(x)$ is a (possibly spatially varying) wave speed, and appropriate initial and boundary conditions are assumed. In the right-hand-side operator, $s \in \mathbb{Z}_{> 0}$ and $\varepsilon \in \mathbb{R}_{\geq 0}$ are parameters, while the coefficient~$a(x)$ and characteristic spacing $\Delta x$ ensure dimensional consistency.

Assuming stability of the unmodified equation (\ie $\varepsilon=0$), by linearity it is sufficient to examine only the right-hand-side operator. Using the energy method, multiplying \eqnref{eq:1d_conv} by $\fnc{U}$ and integrating over the domain $\Omega$ gives
\begin{equation}\label{eq:cont_stab}
    \int_\Omega \fnc{U}\frac{\partial \fnc{U}}{\partial t} \, \mathrm{d} x = \varepsilon (-1)^{s+1}\Delta x^{2s-1}\int_\Omega \fnc{U}\frac{\partial^s}{\partial x^s}\left( \abs{a(x)} \frac{\partial^s\fnc{U}}{\partial x^s} \right) \mathrm{d} x.
\end{equation}
Using Leibniz's rule, the definition of an $L^2$ norm, and expanding the right-hand side of \eqnref{eq:cont_stab} using integration-by-parts gives
\begin{align}
\frac{1}{2} \frac{\mathrm{d}\norm{\fnc{U}}^2}{\mathrm{d}t} ={}& \underbrace{\varepsilon \Delta x^{2s-1} \left[ \sum_{i=1}^s (-1)^{s+i} \left( \frac{\partial^{i-1} \fnc{U}}{\partial x^{i-1} } \right) \left(\frac{\partial^{s-i}}{\partial x^{s-i}} \left( \abs{a(x)} \frac{\partial^s \mathcal{U}}{\partial x^s} \right) \right) \right] \Bigg\rvert_{\partial \Omega}}_{\text{Term A}} \notag \\
&\underbrace{{} - \varepsilon \Delta x^{2s-1} \int_{\Omega} \abs{a(x)} \left( \frac{\partial^s \fnc{U}}{\partial x^s} \right)^2 \mathrm{d} x}_{\text{Term B}}, \label{eq:cont_stab_final}
\end{align}
where $\partial\Omega$ denotes the surface of $\Omega$. While the sign of Term~A in \eqnref{eq:cont_stab_final} is not known in general, Term~B is always negative. Term~B is \textit{dissipative} because it acts to decrease the time derivative of the energy norm. Therefore, we choose to design discrete artificial volume dissipation operators that mimic Term~B when evaluated using the energy method.

Suppose now that \eqnref{eq:1d_conv} is discretized using a typical SBP scheme (\eg, \cite{DelReyFernandez2014_review}). For simplicity, here and throughout the remainder of this paper, we assume homogeneous boundary conditions, since stability of the homogeneous problem implies stability of the inhomogeneous problem~\cite{Svard2019}. Consider
\begin{equation}\label{eq:scalar}
    \frac{\mathrm{d}\bu}{\mathrm{d}t} = \mat{P}\bu + \mat{A}_\mat{D}\bu,
\end{equation}
where $\mat{P}\bu$ denotes an energy-stable and conservative SBP semi-discretization of the unmodified equation (\ie $\varepsilon=0$) that implicitly enforces the homogeneous boundary conditions, and $\mat{A}_\mat{D}$ is an artificial volume dissipation operator. We refer to~\cite{CraigPenner2022_thesis} for a general discussion on the inclusion of artificial dissipation through both interface and volume dissipation. Here, we focus on the volume dissipation operator $\mat{A}_\mat{D}$. A semi-discrete approximation of Term~B in \eqnref{eq:cont_stab_final} that preserves the order of accuracy of the scheme can be constructed as
\begin{equation} \label{eq:ap_diss_sca}
- \varepsilon \Delta x^{2s-1} (\mat{D}^s \bu)^{\T} \Hnrm \mat{A} (\mat{D}^s \bu) \approx - \varepsilon \Delta x^{2s-1} \int_\Omega \abs{a(x)} \left( \frac{\partial^s \fnc{U}}{\partial x^s} \right)^2 \, \mathrm{d}x,
\end{equation}
where $\Delta x$ is the spatial mesh size, $\mat{A} = \diag(\abs{a(x_1)},\abs{a(x_2)},\dots,\abs{a(x_n)})$ is the projection of the variable coefficient onto the diagonal, and
\begin{equation}\label{eq:Ds}
\mat{D}^s = \underbrace{\mat{D}\mat{D}\dots\mathsf{D}}_{s \text{ terms}},
\end{equation}
where $\mat{D}$ is an SBP operator approximating the first derivative. The dissipation operator that results in \eqnref{eq:ap_diss_sca} following a semi-discrete energy analysis is
\begin{equation}\label{eq:basic_diss_form}
    \mat{A}_\mat{D} = - \varepsilon \Delta x^{2s-1} \HI (\mat{D}^s )^{\T} \Hnrm \mat{A} \mat{D}^s.
\end{equation}
Using the SBP property, $\mat{D}^\T = \left( \mat{E} - \mat{Q} \right) \mat{H}^{-\T}$, and the symmetry of $\Hnrm$, we find
\begin{equation}
\begin{split}
    \mat{A}_\mat{D} &= \varepsilon \left( -1 \right)^{s+1} \Delta x^{2s-1} \mat{D}^s \mat{A} \mat{D}^s \\
    &- \sum_{i=1}^s \varepsilon \left( -1 \right)^{s+i} \Delta x^{2s-1} \HI \left( \mat{D}^{i-1} \right)^\T \mat{E} \mat{D}^{s-i} \mat{A} \mat{D}^s,
\end{split}
\end{equation}
which demonstrates that the artificial dissipation is constructed to mimic \eqnref{eq:1d_conv} while removing the unwanted surface contributions of Term~A in \eqnref{eq:cont_stab_final}.

More generally, a volume dissipation operator can be expressed as
\begin{equation} \label{eq:diss_main_form}
    \mat{A}_\mat{D} = - \varepsilon \HI \tilde{\mat{D}}_s^\T \mat{C} \tilde{\mat{D}}_s \ , \quad \tilde{\mat{D}}_s \approx \Delta x^s \der[^s]{x^s}
\end{equation}
where $\mat{C}$ is a symmetric positive semi-definite matrix, and $\tilde{\cdot}$ indicates an \emph{undivided} operator, allowing us to remove the explicit dependence on $\Delta x$ in~\eqnref{eq:basic_diss_form}.  Choosing $\tilde{\mat{D}}_s = \Delta x^s \mat{D}^s$ and $\mat{C}=\tilde{\Hnrm}\mat{A}$ recovers~\eqnref{eq:basic_diss_form}, where $\tilde{\Hnrm} = \Hnrm / \Delta x$ is the undivided norm. However, this leads to dissipation operators whose interior stencils are wider than necessary for a given order of interior stencil accuracy, which unnecessarily increases computational cost through additional floating-point operations~\cite{CraigPenner2018}. In this work, we set $\tilde{\mat{D}}_s$ to be a difference operator of minimum width. The $\tilde{\mat{D}}_s$ matrices are constructed in \S \ref{sec:construction} for various nodal distributions, and different possible choices of $\mat{C}$ are discussed in \S \ref{sec:variable_coeffient} and \S \ref{sec:effect_B}.

\subsection{Extension to Nonlinear Systems of Conservation Laws} \label{sec:systems}

Consider a system of $n$ equations in one spatial dimension,
\begin{equation} \label{eq:cons_law}
    \pder[\vecfnc{U}(x,t)]{t} +  \pder[\vecfnc{F}(x,t)]{x} = 0 \ , \quad \forall \ (x,t) \in \Omega \times \mathbb{R}_{\geq 0}
\end{equation}
where $\vecfnc{U} , \vecfnc{F} \in \mathbb{R}^n$. Many systems of conservation laws are also endowed with a convex entropy $\fnc{S} \in \mathbb{R}$ such that transformation to the entropy variables $\vecfnc{W} = \partial \fnc{S} / \partial \vecfnc{U}$ symmetrizes the conservation law \eqnref{eq:cons_law} \cite{Friedrichs1971, Harten1983, Tadmor_2003}. For the 1D Euler equations,
\begin{equation} \label{eq:1D_euler}
    \vecfnc{U} = \begin{bmatrix}
        \rho \\
        \rho v \\
        e
    \end{bmatrix} , \quad 
    \vecfnc{F} = \begin{bmatrix}
        \rho v \\
        \rho v^2 + p \\
        \rho (e + p)
    \end{bmatrix} , \quad
    \fnc{S} = - \frac{\rho s}{\gamma - 1} , \quad
    \vecfnc{W} = \begin{bmatrix}
        \frac{\gamma-s}{\gamma-1} - \frac{1}{2} \frac{\rho v^2}{p} \\
        \frac{\rho v}{p} \\
        - \frac{\rho}{p}
    \end{bmatrix},
\end{equation}
where $s=\ln\left(\frac{p}{\rho^\gamma}\right)$ is the physical entropy \cite{Hughes1986}, and the equation of state $p = (\gamma-1)\left( e - \frac{1}{2} \rho v^2 \right)$ fully specifies the system. Only two modifications need to be made to extend \eqnref{eq:diss_main_form} to nonlinear systems. The first is to replace $\tilde{\mat{D}}_s$ and $\Hnrm$ with $\tilde{\mat{D}}_s \otimes \mat{I}_n$ and $\Hnrm \otimes \mat{I}_n$, respectively, where $\mat{I}_n$ is the identity matrix of size $n$, and $\otimes$ is a Kronecker product. The second modification is to replace the variable coefficient matrix $\mat{A}$ in \eqnref{eq:ap_diss_sca} with a block-diagonal matrix that is dimensionally consistent with the
flux Jacobian and constructed in the spirit of characteristic-based
dissipation~\cite{harten_upwinding_1983}. Depending on whether the dissipation operator is applied to conservative or entropy variables, and whether a scalar or matrix dissipation scheme is desired, the definition of $\mat{A}$ will differ.

In a direct generalization of \eqnref{eq:ap_diss_sca}, dissipation applied to the conservative variables can take the following form after premultiplication by $\bu^\T \Hnrm$:
\begin{equation} \label{eq:ap_diss_cons}
- (\tilde{\mat{D}}^s \bu)^{\T} \tilde{\Hnrm} \mat{A} (\tilde{\mat{D}} \bu) \approx - \Delta x^{2s-1} \int_\Omega \left( \frac{\partial^s \fnc{U}}{\partial x^s} \right)^\T \abs{\pder[\vecfnc{F}]{\vecfnc{U}}} \left( \frac{\partial^s \fnc{U}}{\partial x^s} \right) \, \mathrm{d}x,
\end{equation}
where $\mat{A} = \diag \left(\mat{A}_1, \dots, \mat{A}_N \right) $ is a block diagonal matrix containing blocks $\mat{A}_j \in \mathbb{R}^{n \times n}$ that are some approximation to the flux Jacobian $\partial \vecfnc{F} / \partial \vecfnc{U}$ at node~$j$ with eigenvalues replaced by their absolute values. We require $\mat{A}_j$ to be symmetric positive semi-definite to ensure that \eqnref{eq:ap_diss_cons} remains negative, preserving numerical stability. A straightforward way to achieve this is via a scalar dissipation scheme,
\begin{equation} \label{eq:scalar_diss}
    \mat{A}_j = \abs{\lambda_\text{max}} \cdot \mat{I}_{n}
\end{equation}
where $\lambda_\text{max}$ is the maximum eigenvalue of the flux Jacobian at node $j$. A matrix dissipation scheme can also be considered,
\begin{equation} \label{eq:matrix_diss}
    \mat{A}_j = \mat{X}_j \abs{\mat{\Lambda}_j} \mat{X}_j^{-1}
\end{equation}
where $\mat{X}_j$ and $\mat{\Lambda}_j$ are the eigenvector and eigenvalue matrices of the flux Jacobian at node $j$. However, unless the conservation law is symmetric, the above matrix dissipation will not be provably stable. An alternative therefore is to apply dissipation to the entropy variables $\vec{w}$. Premultiplying by $\vec{w}^\T \Hnrm$ gives 
\begin{equation} \label{eq:ap_diss_ent}
- (\tilde{\mat{D}}^s \vec{w})^{\T} \tilde{\Hnrm} \mat{A} (\tilde{\mat{D}}^s \vec{w}) \approx - \Delta x^{2s-1} \int_\Omega \left( \frac{\partial^s \fnc{W}}{\partial x^s} \right)^\T \abs{\pder[\vecfnc{F}]{\vecfnc{U}}} \pder[\vecfnc{U}]{\vecfnc{W}} \left( \frac{\partial^s \fnc{W}}{\partial x^s} \right) \, \mathrm{d}x,
\end{equation}
where in order to ensure consistent dimensions, $\mat{A}_j$ must now approximate the absolute flux Jacobian multiplied by the symmetric and positive-definite change of variable matrix $\partial \vecfnc{U} / \partial \vecfnc{W}$ that symmetrizes the system. For example, a scalar-matrix dissipation scheme can be considered as
\begin{equation} \label{eq:scalarscalar_diss}
    \mat{A}_j = \abs{\lambda_\text{max}} \pder[\vec{u}_j]{\vec{w}_j}
\end{equation}
and a matrix-matrix dissipation scheme can be considered as
\begin{equation} \label{eq:scalarmatrix_diss}
    \mat{A}_j = \mat{X}_j \abs{\mat{\Lambda}_j} \mat{X}_j^{\T} , \quad \mat{X}_j \mat{X}_j^\T = \pder[\vec{u}_j]{\vec{w}_j}
\end{equation}
where the eigenvectors $\mat{X}_j$ are computed analytically using the scaling theorem of \cite{Barth1999} to ensure the relation on the right of \eqnref{eq:scalarmatrix_diss}. Scalar-scalar dissipation can be applied using the spectral radius of the change of variable matrix, but we have found this to be overly dissipative in numerical experiments. Therefore, scalar-scalar entropy dissipation is not considered is this work.

\subsection{Construction Procedure for \texorpdfstring{$\tilde{\mat{D}}_s$}{Ds} Operators}\label{sec:construction}

In \cite{Mattsson2004}, $\tilde{\mat{D}}_s$ operators were constructed for use with classical SBP operators. We now extend this approach to nonuniform nodal distributions and those without boundary nodes, allowing the volume dissipation~\eqnref{eq:diss_main_form} developed in the previous sections to be used with generalized FD and SE SBP operators. This is accomplished by introducing unknowns near the boundaries of the $\tilde{\mat{D}}_s$ matrices and solving the following accuracy equations, which ensure that $\tilde{\mat{D}}_s$ approximates an undivided derivative of degree $s$:
\begin{align} \label{eq:accuracy_conditions}
\begin{split}
    \tilde{\mat{D}}_s \tilde{\vec{x}}^i &= \zeros , \quad i \in [0,s-1] \\
    \tilde{\mat{D}}_s \tilde{\vec{x}}^s &= \ones \times s!,
\end{split}
\end{align}
where $\tilde{\vec{x}}=[\tilde{x}_1, \tilde{x}_2, \dots, \tilde{x}_N]^\T$ is the nodal distribution for an undivided reference domain $[0, N-1]$ that obeys the natural ordering $0 \leq \tilde{x}_1 < \tilde{x}_2 < \cdots < \tilde{x}_n \leq N-1$. The general form for each operator is provided in the supplementary material, where we adopt the minimum number of unknowns necessary at the boundaries to solve the accuracy conditions \eqnref{eq:accuracy_conditions}, leading to unique solutions. Operators constructed for nodal distributions associated with HGTL, HGT, Mattsson, LGL, and LG operators are also provided in the supplementary material. The resulting FD dissipation operators could be further optimized by adding more unknowns at the boundaries (\eg as in \cite{Diener2007}), but we defer this to future work. Since SE operators have a fixed number of $N=p+1$ nodes, the dissipation order is restricted to $s=p$, yielding unique operators and precluding further optimization, as discussed in \S \ref{sec:gsbp_uniqueness}.

This approach yields volume dissipation operators $\mat{A}_\mat{D}$ that are naturally of minimum-width. Alternatively, in \cite{CraigPenner2018} minimum-width dissipation operators were constructed by adding correction terms to the wide-stencil dissipation \eqnref{eq:basic_diss_form}. In the supplementary material we show that the two approaches are largely equivalent.

\subsection{Stability, Conservation, and Accuracy Proofs of Volume Dissipation} \label{sec:proofs}

In this section we discuss the stability, conservation, and accuracy properties of SBP schemes augmented with the volume dissipation \eqnref{eq:diss_main_form}. We assume that the underlying semi-discrete system of equations is stable and conservative so that it is sufficient to examine only the properties of the dissipation model~\cite{DelReyFernandez2014_review}. Such proofs for the underlying system can be found for example in \cite{Crean2018, DelReyFernandez2014_review, chen_review}.

\begin{theorem}
    Augmenting a conservative semi-discretization \eqnref{eq:scalar} with the artificial dissipation operator \eqnref{eq:diss_main_form} results in a conservative semi-discrete scheme.
\end{theorem}
\begin{proof}
    Substituting \eqnref{eq:diss_main_form} into \eqnref{eq:scalar}, ignoring $\mat{P} \bu$, and multiplying both sides by $\vec{1}^\T \Hnrm$ gives
    \begin{equation*}
        \vec{1}^\T \Hnrm \der[\bu]{t} = - \varepsilon  \left( \tilde{\mat{D}}_s \vec{1} \right)^\T \mat{C} \left( \tilde{\mat{D}}_s \bu \right) = \vec{0} \quad \text{since} \quad \tilde{\mat{D}}_s \vec{1} = \vec{0} .
        \tag*{\qed}
    \end{equation*}
\end{proof}

The following theorems consider two distinct notions of stability: energy and entropy stability. If the underlying scheme (with $\varepsilon = 0$) is energy-stable, then applying volume dissipation to the conservative variables preserves energy stability. Conversely, if the scheme is entropy-stable, then applying the volume dissipation to the entropy variables preserves entropy stability.

\begin{theorem} \label{thm:energy_stability}
    Augmenting an energy-stable semi-discretization \eqnref{eq:scalar} with the artificial dissipation operator \eqnref{eq:diss_main_form} acting on conservative variables $\bu$ results in an energy-stable semi-discrete scheme for a positive semi-definite $\mat{C}$.
\end{theorem}
\begin{proof}
    Substituting \eqnref{eq:diss_main_form} into \eqnref{eq:scalar}, ignoring $\mat{P} \bu$, and multiplying both sides by $\bu^\T \Hnrm$ gives
    \begin{equation*}
        \frac{1}{2} \der[\norm{\bu}_\Hnrm^2]{t} = 
        \bu^\T \Hnrm \der[\bu]{t} = - \varepsilon \left( \tilde{\mat{D}}_s \bu \right)^\T \mat{C} \left( \tilde{\mat{D}}_s \bu \right) \leq 0 
    \end{equation*}
    since $\varepsilon \geq 0$ and $\mat{C}$ is positive semi-definite. \qed
\end{proof}

\begin{theorem}
\label{thm:entropy_stability}
    Augmenting an entropy-stable semi-discretization,
    \begin{equation}
    \frac{\mathrm{d}\bu}{\mathrm{d}t} = \mat{P}\bu + \mat{A}_\mat{D}\vec{w},
    \end{equation}
    with the artificial dissipation operator \eqnref{eq:diss_main_form} acting on entropy variables $\vec{w}$ results in an entropy-stable semi-discrete scheme for a positive semi-definite $\mat{C}$.
\end{theorem}
\begin{proof}
    Substituting \eqnref{eq:diss_main_form} into \eqnref{eq:scalar}, where we now have $\mat{A}_\mat{D} \vec{w}$ instead of $\mat{A}_\mat{D} \vec{u}$, ignoring $\mat{P} \bu$, and multiplying both sides by $\vec{w}^\T \Hnrm$ gives
    \begin{equation*}
        \vec{1}^\T \Hnrm \der[\vec{s}]{t} = 
        \vec{w}^\T \Hnrm \der[\bu]{t} = - \varepsilon \left( \tilde{\mat{D}}_s \vec{w} \right)^\T \mat{C} \left( \tilde{\mat{D}}_s \vec{w} \right) \leq 0 .
        \tag*{\qed}
    \end{equation*}
\end{proof}

We now seek to formalize the convergence rates of FD SBP schemes augmented with volume dissipation by using the theory developed in~\cite{Svard2019}, for which it necessary to introduce the following concepts. Discussion regarding SE schemes is deferred to \S \ref{sec:relation_to_upwind}. Consider a semi-discrete operator~$\mat{P}$ approximating a continuous differential operator~$\fnc{P}$. A \emph{nullmode}~$\vec{w}$ of~$\mat{P}$ is an eigenvector corresponding to the nullspace of~$\mat{P}$, \ie $\mat{P} \vec{w} = \vec{0}$. We say~$\mat{P}$ is \emph{nullspace consistent} if it has the same nullspace as~$\fnc{P}$, \ie~$\vec{w}$ are restrictions of the nullmodes~$\fnc{W}$ of~$\fnc{P}$ (including boundary conditions). An SBP scheme with quadrature~$\Hnrm$ is \emph{nullspace invariant} if at any given time, the nullmodes~$\vec{w}$ and the numerical solution~$\vec{u}$ satisfy~$\vec{w}^\T \Hnrm \mat{P} \vec{u} = 0$. We refer to \cite{Svard2019} for details.

\begin{theorem}[4.6 of \cite{Svard2019}] \label{thm:svard}
    Consider a linear, constant-coefficient, scalar PDE with highest spatial derivative of order $m$. Let $\mat{P}$ be the linear semi-discrete operator approximating the initial-boundary value problem (IBVP). If $r$ is the order of the boundary truncation error, $q$ the order of the internal truncation error with $q>r$, and
    \begin{itemize}
        \item the IBVP is energy well-posed, with smooth and compatible data such that the exact solution is smooth.
        \item all nullmodes of the IBVP are polynomials.
        \item the scheme is nullspace consisent and nullspace invariant.
        \item the discrete inner product is a $q$th-order quadrature rule.
    \end{itemize}
    then the convergence rate of the solution error is $r+m$ if the nullspace of $\mat{P}$ is trivial, and $\min{\left(q, r+m\right)}$ if the nullspace of $\mat{P}$ is non-trivial.
\end{theorem}

We refer the reader to~\cite{Svard2019} for definitions of terms used in the above theorem, as well as detailed discussions of the necessary assumptions. As explained in~\cite{Svard2019}, most SBP schemes satisfy the above conditions. It will be necessary, however, to add the following assumption to those included in \thmref{thm:svard}.
\begin{assumption} \label{assumption1}
Consider a linear SBP semi-discretization operator $\mat{P}$ with only polynomial nullmodes, augmented as in \eqnref{eq:scalar} with a dissipation model $\mat{A}_\mat{D}$. The nullspace of $\mat{P} + \mat{A}_\mat{D}$ is the same as the nullspace of $\mat{P}$.
\end{assumption}
It is straightforward to prove that ${\text{null}\left(\mat{P}\right)\subseteq\text{null}\left(\mat{P}+\mat{A}_\mat{D}\right)}$, since by construction, any polynomial nullmodes of $\mat{P}$ will also be in the nullspace of $\mat{A}_\mat{D}$ (see \S \ref{sec:construction}). However, showing that $\mat{A}_\mat{D}$ does not introduce any new nullmodes to the discretization, \ie, ${\text{null}\left(\mat{P}+\mat{A}_\mat{D}\right)\subseteq\text{null}\left(\mat{P}\right)}$, is difficult to prove in general. Nevertheless, we have manually checked that Assumption \ref{assumption1} holds for all linear SBP discretizations presented in this work. It is now possible to state the following theorem, which has the practical implication that choosing dissipation operators with $s=p+1$ results in solution error convergence rates of $p+1$. Error estimates for upwind SBP schemes in~\cite{Jiang2024} suggest that a sharper convergence rate of $p+1.5$ may also be provable.

\begin{theorem}
    Consider a linear, constant-coefficient, scalar SBP semi-\linebreak discretization~\eqnref{eq:scalar}. Assume the SBP semi-discretization $\mat{P}$ satisfies the assumptions of \thmref{thm:svard} with $r=p$ and $q=2p$, Assumption \ref{assumption1}, and the dissipation model $\mat{A}_\mat{D}$ is constructed using \eqnref{eq:diss_main_form} with $s\geq p+1$. Then the convergence rate of the solution error is $p+m$ if the nullspace of $\mat{P}$ is trivial, and $\min{\left(2p, p+m\right)}$ if the nullspace of $\mat{P}$ is non-trivial.
\end{theorem}
\begin{proof}
     Nullspace consistency follows immediately from Assumption \ref{assumption1} and the assumption that $\mat{P}$ is nullspace consistent. Nullspace invariance is also ensured by the symmetry of the dissipation operator with respect to the discrete norm~$\Hnrm$. Since the scheme is linear, nullspace invariance of the dissipation operator~$\mat{A}_\mat{D}$ is sufficient to ensure nullspace invariance of the overall discretization.

    It remains to be shown that the boundary and internal truncation errors of the resulting scheme are also $r=p$ and $q=2p$, respectively. Once again by linearity, it suffices to show this is true for $\mat{A}_\mat{D}$ alone. Since $\tilde{\mat{D}}_s \approx \Delta x^s \mathrm{d}^s/\mathrm{d}x^s$ contains a repeating interior stencil, which has the general property 
    $$\left[ \tilde{\mat{D}}^\T_s \right]_{i,i+j} = \left[ \tilde{\mat{D}}_s \right]_{i+j,i} = \left[ \tilde{\mat{D}}_s \right]_{i,i-j} ,$$ 
    the interior of $\tilde{\mat{D}}^\T$ approximates $(-1)^{s+1} \Delta x^s \mathrm{d}^s/\mathrm{d}x^s$. Considering the contribution of $\HI$, the interior of $\mat{A}_\mat{D}$ therefore approximates $\Delta x^{2s-1} \mathrm{d}^{2s}/\mathrm{d}x^{2s}$ multiplied by some constant. A Taylor expansion of $\mat{A}_\mat{D}$ will therefore result in a leading-order term of order $2s-1 = 2p+1 > q$. Considering the boundaries, since $\mat{A}_\mat{D}$ annihilates polynomials up to degree $s-1$ by construction, a Taylor expansion of $\mat{A}_\mat{D}$ must contain only derivatives of degree $\geq s$. Dimensional consistency of \eqnref{eq:diss_main_form} again ensures that Taylor expansions at the boundary of $\mat{A}_\mat{D}$ contain leading-order terms of order at least $s - 1 = p = r$. Therefore, to leading order, the Taylor expansions of \eqnref{eq:scalar} remain unaffected. The expected convergence rates then follow immediately from \thmref{thm:svard}. \qed
\end{proof}

\section{Implementation Details} \label{sec:implement_vol_diss_for_sbp}

In the previous section, we introduced volume dissipation in the form \eqnref{eq:diss_main_form} with an unspecified positive semi-definite matrix $\mat{C}$. A natural choice, motivated by \S\ref{sec:art_diss_scal}, is $\mat{C} = \tilde{\Hnrm} \mat{A}$, where $\tilde{\Hnrm}$ is an undivided norm, and $\mat{A}$ represents the variable coefficient. In the following sections, we will instead justify the choice $\mat{C} = \mat{B} \mat{A}$, where $\mat{B}$ is a diagonal boundary correction matrix consisting of a fixed number of zeros at the boundaries, and ones in the interior. 
For scalar problems, $\mat{C}$ is positive semi-definite if both $\mat{B}$ and $\mat{A}$ are diagonal and non-negative. For systems of equations, $\mat{A}$ is assumed to be a symmetric positive semi-definite block-diagonal matrix. Using the same $\mat{B}$ (and $\tilde{\Hnrm}$, if it is included) for each equation ensures that $\mat{C}$ is also positive semi-definite. 

As will be explained in \S\ref{sec:gsbp_uniqueness}, all possible choices of $\mat{C}$ are equivalent up to a multiplicative constant for SE operators. Therefore, we exclusively consider FD operators in the following sections \S\ref{sec:effect_B} and \S\ref{sec:variable_coeffient}.

\subsection{Effect of the Boundary Correction Matrix \texorpdfstring{$\mat{B}$}{B}} \label{sec:effect_B}

To understand the effect of including a boundary correction matrix $\mat{B}$, we will illustrate via example for the linear convection equation with $a = 1$. Consider the third-order accurate dissipation operator from \cite{Mattsson2004}, given by
\begin{equation}
    \mat{A}_{\mat{D},2} = -\HI\tilde{\mat{D}}_2^\T\mat{B} \tilde{\mat{D}}_2,
\end{equation}
where
\begin{equation*}
    \tilde{\mat{D}}_2= \scalebox{0.8}{$\begin{bmatrix}
    -1 & 2 & -1 & & & & \\
    -1 & 2 & -1 & & & & \\
    & -1 & 2 & -1 & & & \\
    & & \ddots & \ddots & \ddots
    \end{bmatrix}$} \quad \text{and} \quad \mat{B} = \diag(0,1,\dots,1,0).
\end{equation*}
 Note that in \cite{Mattsson2004}, this operator is referred to as fourth-order accurate since an undivided $\tilde{\mat{H}}^{-1}$ is used, whereas in this work a dimensional $\HI$ is used to ensure that the dissipation operator has dimensions that are consistent with the overall discretization.
Applying the dissipation operator $\mat{A}_{\mat{D},2}$ to the solution vector $\vec{u}$ leads to the following Taylor series expansion at each node:
\begin{equation}\label{eq:diss_2_b1}
    \mat{A}_{\mat{D},2}\vec{u} =\frac{1}{\Delta x} 
    \scalebox{0.8}{$
    \begin{bmatrix}
-2 & 4 & -2 & & &  \\
2 & -5 & 4 & -1 & &  \\
-1 & 4 & -6 & 4 & -1 &  \\
& \ddots & \ddots & \ddots & \ddots & \ddots \end{bmatrix}$} \vec{u} = \renewcommand\arraystretch{1.8}\scalebox{0.9}{$\begin{bmatrix}
    -2 \Delta x \left(\frac{\partial^2 u}{\partial x^2}\right)_1 - 2\Delta x^2 \left(\frac{\partial^3 u}{\partial x^3}\right)_1 + \fnc{O}(\Delta x^3) \\
    \Delta x \left(\frac{\partial^2 u}{\partial x^2}\right)_2 - \Delta x^2 \left(\frac{\partial^3 u}{\partial x^3}\right)_2 + \fnc{O}(\Delta x^3) \\
    -\Delta x^3 \left(\frac{\partial^4 u}{\partial x^4}\right)_3 - \frac{\Delta x^5}{6} \left(\frac{\partial^6 u}{\partial x^6}\right)_3 + \fnc{O}(\Delta x^7) \\
    \vdots
\end{bmatrix}$}.
\end{equation}
The first two rows have first-order leading terms, while the interior rows have third-order leading terms. Instead choosing $\mat{B}$ as the identity matrix leads to the following expansion:
\begin{equation}\label{eq:diss_2_b2}
    \mat{A}_{\mat{D},2}\vec{u} =\frac{1}{\Delta x} \scalebox{0.8}{$\begin{bmatrix}
-4 & 8 & -4 & & & & \\
4 & -9 & 6 & -1 & & & \\
-2 & 6 & -7 & 4 & -1 & & \\
& -1 & 4 & -6 & 4 & -1 & \\
& & \ddots & \ddots & \ddots & \ddots & \ddots \end{bmatrix}$}\vec{u} = \renewcommand\arraystretch{1.8}\scalebox{0.85}{$\begin{bmatrix}
    -4 \Delta x \left(\frac{\partial^2 u}{\partial x^2}\right)_1 - 4\Delta x^2 \left(\frac{\partial^3 u}{\partial x^3}\right)_1 + \fnc{O}(\Delta x^3) \\
    3\Delta x \left(\frac{\partial^2 u}{\partial x^2}\right)_2 - \Delta x^2 \left(\frac{\partial^3 u}{\partial x^3}\right)_2 + \fnc{O}(\Delta x^3) \\
    -\Delta x \left(\frac{\partial^2 u}{\partial x^2}\right)_3 + \Delta x^2 \left(\frac{\partial^3 u}{\partial x^3}\right)_3 + \fnc{O}(\Delta x^3) \\
    -\Delta x^3 \left(\frac{\partial^4 u}{\partial x^4}\right)_4 - \frac{\Delta x^5}{6} \left(\frac{\partial^6 u}{\partial x^6}\right)_4 + \fnc{O}(\Delta x^7) \\
    \vdots
\end{bmatrix}$}.
\end{equation}
In contrast to \eqnref{eq:diss_2_b1}, the first three rows now have first-order leading terms. Moreover, the coefficients are larger near the boundary. Therefore, introducing zeros at the boundaries of $\mat{B}$ reduces the amount of dissipation applied near the block boundaries by reducing both the number of low-order terms, as well as their coefficients. In \S \ref{sec:LCE} this is shown to have a beneficial effect on the spectral radii of the semi-discretization.

It is also worthwhile to mention the impact of including an undivided norm via $\mat{C} = \tilde{\Hnrm} \mat{B} \mat{A}$. For FD operators, the interior of $\mat{A}_\mat{D}$ is unaffected by the inclusion of $\tilde{\Hnrm}$ since $\tilde{\Hnrm}_{ii} = 1$ there. At the boundaries, however, the above discussion indicates that the dissipation may be either proportionally increased or decreased, depending on the value of $\tilde{\Hnrm}$ (assuming $\mat{B} \neq 0$ there). To prevent the dissipation from being increased at boundary nodes, we argue that it is likely preferable to omit $\tilde{\Hnrm}$ via $\mat{C} = \mat{B} \mat{A}$. This statement will be supported by numerical results in \S \ref{sec:LCE}.

\begin{remark}
Although \S\ref{sec:gsbp_uniqueness} demonstrates that including $\mat{B}$ has no effect on the SE operators used in this work (with $s=p$), similar boundary corrections were included in the alternative SE dissipation operators of~\cite{Ranocha_dissipation, wei_jump_2025}. In a continuous framework, these corrections ensure that terms arising from integration by parts vanish at the boundaries so that no viscous coupling terms (which are often included in artificial-viscosity approaches~\cite{persson_viscosity_2006, barter_shock_2010, zingan_entropy_2013, yu_viscosity_2020}) are required~\cite{Ranocha_dissipation}. This may provide additional intuition for the role of~$\mat{B}$.
\end{remark}

\subsection{Options for Volume Dissipation with a Variable Coefficient} \label{sec:variable_coeffient}

In the following, we assume a scalar variable coefficient $a(x)$. However, the same concepts can be directly extended to systems of conservation laws by generalizing $a(x)$ to flux Jacobian matrices as discussed in \S \ref{sec:systems}.

Consider the volume dissipation model \eqnref{eq:diss_main_form} with $\mat{C} = \mat{B} \mat{A}$. The simplest choice for the variable coefficient matrix $\mat{A}$ is 
\begin{equation} \label{eq:A_simple}
    \mat{A} = \diag \left( a_1,\dots,a_N  \right),
\end{equation}
where $a_j = \abs{a(x_j)}$, $j=1,\dots,N$ are the absolute-valued variable coefficients at the nodes, ensuring only positive contributions to the dissipation. For odd~$s$, since the minimum-width FD stencils for odd derivatives are inherently asymmetric, it may be preferable to define the variable coefficient at the half-nodes $a_{j-1/2}$ between $a_{j-1}$ and $a_j$,
\begin{equation}  \label{eq:A_halfnodes}
    \mat{A} = \diag \left( 0 , \frac{a_1 + a_2}{2} , \dots , \frac{a_{N-1} + a_N}{2} \right).
\end{equation}
This ensures that the nodal variable coefficients appear consistently when reflected about the anti-diagonal of the dissipation operator, such that dissipation remains directionally unbiased. The initial value of $0$ arises from the inability to define a half-node at $a_{1/2}$. Note that we have adopted the convention of using backwards-biased stencils, although an equivalent problem appears using forward-biased stencils.

To illustrate this concept via example, consider the constant-coefficient first-order accurate dissipation operator from \cite{Mattsson2004}, where
\begin{equation}
    \tilde{\mat{D}}_1= \scalebox{0.8}{$\begin{bmatrix}
    -1 & 1 & & \\
    -1 & 1 & & \\
    & \ddots & \ddots \\
    & & -1 & 1  \\
    & & & -1 & 1  \\
    \end{bmatrix}$} \quad \text{and} \quad \mat{B} = \diag(0,1,\dots,1).
\end{equation}
Defining the variable coefficient directly at the nodes as in \eqnref{eq:A_simple} results in the following directionally-biased dissipation operator
\begin{align}\label{eq:vol_diss_1st_biased}
    \mat{A}_{\mat{D},1} &= -\HI\tilde{\mat{D}}_1^\T\mat{B} \mat{A} \tilde{\mat{D}}_1 
    = \frac{1}{\Delta x}\scalebox{0.9}{$\begin{bmatrix}
- 2 a_2 & 2 a_2 & & & \\
a_2 &  - \left( a_2 + a_3 \right) & a_3 & & \\
& \ddots & \ddots & \ddots & \\
& & a_{N-1} & - \left(a_{N-1} + a_N \right) & a_N \\
& & & 2 a_{N} & - 2 a_{N} \\
\end{bmatrix}$}.
\end{align}
Note the absence of the coefficient $a_1$. Alternatively, defining the variable coefficient at the half-nodes as in \eqnref{eq:A_halfnodes} results in the following unbiased operator,
\begin{align}\label{eq:vol_diss_1st}
\begin{split}
    \mat{A}_{\mat{D},1} &= -\HI\tilde{\mat{D}}_1^\T\mat{B} \mat{A} \tilde{\mat{D}}_1 \\
    &= \frac{1}{\Delta x}\scalebox{0.9}{$\begin{bmatrix}
- \left( a_1 + a_2 \right) &  a_1 + a_2 & & & \\
\frac{a_1}{2} + \frac{a_2}{2} & - \left( \frac{a_1}{2} + a_2 + \frac{a_3}{2} \right) & \frac{a_2}{2} + \frac{a_3}{2} & & \\
\qquad \ddots & \qquad \ddots & \qquad \ddots & \\
& \frac{a_{N-2}}{2} + \frac{a_{N-1}}{2} & -\left( \frac{a_{N-2}}{2} + a_{N-1} + \frac{a_N}{2} \right) & \frac{a_{N-1}}{2} + \frac{a_N}{2} \\
& &  a_{N-1} + a_{N}  & - \left( a_{N-1} + a_{N} \right) \\
\end{bmatrix}$}.
\end{split}
\end{align}
Therefore, for the remainder of this work, whenever $s$ is odd we define the variable coefficients at half-nodes as in \eqnref{eq:A_halfnodes}. If including the undivided norm $\tilde{\mat{H}}$ in the dissipation operator, \ie, $\mat{C} = \tilde{\mat{H}} \mat{B} \mat{A}$, it may also be desirable to define the entries of  $\tilde{\mat{H}}$ at half-nodes for odd $s$ to ensure symmetry about the anti-diagonal. For simplicity, we defer investigation of the impact of defining $\tilde{\mat{H}}$ at half-nodes to further work, and instead choose to exclude the norm matrix through the choice $\mat{C} = \mat{B} \mat{A}$.

Another option to construct unbiased dissipation operators while avoiding the definition of the variable coefficient at half-nodes is to adopt the correction procedure of \cite{CraigPenner2018}. However, this perspective is largely equivalent to the current approach, and is therefore only discussed in the supplementary material.

\subsection{Extension to Multiple Dimensions} \label{sec:multiple_dim}

The dissipation operators presented thus far can be extended to multiple dimensions using a tensor-product formulation, as presented in \cite{DelReyFernandez2019_extension}. Consider a system of $n$ equations in two spatial dimensions with computational coordinates $(\xi,\eta)$ and a curvilinear transformation to physical coordinates $(x,y)$, characterized by metric terms and a metric Jacobian $\fnc{J}$. By defining
\begin{align*}
    \tilde{\mat{D}}_{s,\xi}  &= \tilde{\mat{D}}_{s,\xi}^\text{(1D)} \otimes \mat{I}_\eta \otimes \mat{I}_n,
    & \mat{B}_{\xi}     &= \mat{B}_\xi^\text{(1D)} \otimes \mat{I}_\eta \otimes \mat{I}_n, \\
    \mat{H}_{\eta}   &= \mat{I}_\xi \otimes \mat{H}_\eta^\text{(1D)} \otimes \mat{I}_n,
    & \mat{H}        &= \mat{H}_\xi^\text{(1D)} \otimes \mat{H}_\eta^\text{(1D)} \otimes \mat{I}_n,
\end{align*}
dissipation in the $\xi$ direction can be applied using
\begin{align} \label{eq:2d_diss}
     \mat{A}_{\mat{D}, \xi} &= -\diag\left(\fnc{J}^{-1}\right)\mat{H}^{-1} \tilde{\mat{D}}_{s,\xi}^\T \mat{H}_\eta \mat{B}_{\xi} \mat{A}_\xi \tilde{\mat{D}}_{s,\xi},
\end{align}
where $\mat{A}_\xi$ is the flux Jacobian in the $\xi$ direction that implicitly includes metric terms, ensuring proper scaling of $\Delta x$ and $\Delta y$. The above form generalizes the choice that excludes $\tilde{\Hnrm}$, \ie, $\mat{C} = \mat{B} \mat{A}$. Regardless, note the presence of the divided $\mat{H}_\eta$, which is necessary to ensure proper scaling in terms of $\Delta \eta$ (unless the computational space is defined such that $\Delta \eta = 1$, in which case $\mat{H}_\eta = \tilde{\mat{H}}_\eta$). As discussed in \S \ref{sec:variable_coeffient}, for odd $s$ we choose to define $\mat{A}_\xi$ at the half-nodes.  The proofs of stability and conservation in \S \ref{sec:proofs} follow immediately by recognizing that integration is performed by left-multiplying by $\bu^\T \Hnrm \fnc{J}$ and $\vec{1}^\T \Hnrm \fnc{J}$ as opposed to $\bu^\T \Hnrm$ and $\vec{1}^\T \Hnrm$. 

Free-stream preservation (alternatively referred to as satisfaction of metric identities) is an important property of a numerical method when generalizing to curvilinear coordinates  \cite{pulliam2014}. This can be verified by considering a constant solution $\vecfnc{U} = \vec{1}$ and flux $\vecfnc{F} = \vec{1}$. By construction, $\tilde{\mat{D}}_s^\text{(1D)} \vec{1} = \vec{0}$, and therefore $\mat{A}_{\mat{D}} \vec{1} = \vec{0}$ always. As a consequence, as long as the underlying scheme is free-stream preserving, the scheme augmented with volume dissipation in the form given above will also be free-stream preserving.

\section{Connection to Upwinding and Existing Spectral-Element Dissipation}\label{sec:upwind_and_spectral}

\subsection{Relationship between Volume Dissipation and Upwind SBP Operators} \label{sec:relation_to_upwind}

Rather than adding a volume dissipation operator, it is possible to include volume dissipation by employing difference operators with upwind and downwind stencils. This was first applied to SBP schemes in~\cite{Mattsson2004}, followed by the construction of novel upwind finite-difference (UFD) SBP operators in~\cite{Mattsson2017} with minimized spectral radii. Upwind spectral-element (USE) SBP operators were subsequently constructed in~\cite{Glaubitz2024}. Following~\cite{Ranocha2024} for general nonlinear hyperbolic conservation laws, a flux-vector splitting of the flux $\vec{f} = \vecfnc{F}(\bu)$ with
\begin{equation*}
    \vec{f} =\vec{f}^+ + \vec{f}^-
\end{equation*}
can be performed so that the upwind semi-discretization
\begin{equation} \label{eq:upwind}
    \frac{\mathrm{d}\bu}{\mathrm{d}t} + \mat{D}_+ \vec{f}^- + \mat{D}_- \vec{f}^+ = \textbf{SAT}
\end{equation}
introduces the desired volume dissipation, where $\vec{f}^\pm$ and $\mat{D}_{\pm}$ are as defined in~\cite{Ranocha2024} and the Kronecker product $\mat{D}_{\pm} = \mat{D}_{\pm}^{(1\text{D})} \otimes \mat{I}_{n}$ is implicit for systems of $n$ equations. To make the connection to the current volume dissipation clear, flux-vector splittings can be rewritten in the form
\begin{equation*}
    \vec{f}^+ = \frac{1}{2} \left( \vec{f} + \vec{f}_\mat{D}  \right), \quad 
    \vec{f}^- = \frac{1}{2} \left( \vec{f} - \vec{f}_\mat{D} \right), \quad  
    \vec{f}_\mat{D} = \vec{f}^+ - \vec{f}^-,
\end{equation*}
where $\vec{f}_\mat{D}$ is the dissipative portion resulting from the flux-vector splitting. Using properties described in~\cite{Mattsson2017}, we can also rewrite upwind operators as
\begin{equation*}
    \mat{D} = \frac{1}{2} \left( \mat{D}_+ + \mat{D}_- \right), \quad 
    \varepsilon \mat{H}^{-1} \mat{S} = - \frac{1}{2} \left( \mat{D}_+ - \mat{D}_- \right),
\end{equation*}
where $\mat{D}$ is a central-difference SBP operator with norm matrix $\mat{H}$, and $\mat{S}$ is symmetric and positive semi-definite. Although it is typically implicit in $\mat{S}$, we include a constant $\varepsilon \in \mathbb{R}_{\geq 0}$ in the above to emphasize that the dissipative portion can be modified by a constant depending on the construction of $\mat{D}_{\pm}$, such as in~\cite{Glaubitz2024}. Notably, $\mat{D}$ will not coincide with the more typically-used narrow stencil SBP operators (\eg those found in~\cite{DelReyFernandez2014_review}). Instead, $\mat{D}$ will often have wider, more accurate stencils in the interior~\cite{Mattsson2017}. A comparison between the operators can be found in Table~\ref{tab:comparison}, which relates the orders of accuracy and expected error convergence rate (typically one greater than the order of accuracy at the boundary \cite{gustafsson}) to the upwind degree $p_u$ as defined in~\cite{Mattsson2017}.

\begin{table}[t]
\centering
\caption{A comparison of the central $\mat{D}$ and dissipation $\mat{A}_\mat{D}$ operators resulting from the use of upwind schemes in the form \eqnref{eq:upwind_equiv} and central SBP schemes \eqnref{eq:scalar} augmented with dissipation \eqnref{eq:diss_main_form} described in this work. Orders of accuracy (\eg resulting from a Taylor expansion) are reported both in the interior and at block boundaries along with the expected error convergence rate. The limiting orders of accuracy for the discretizations are highlighted in red. $p_\text{u}$ refers to the degree of the upwind operator as defined in \cite{Mattsson2017}.}
\label{tab:comparison}
\begin{threeparttable}
\renewcommand{\arraystretch}{1.2} 
\begin{tabular}{lccccc}
\hline
\textbf{Operator} 
& \begin{tabular}{@{}c@{}}$\mat{D}$ \textbf{Int.} \end{tabular} 
& \begin{tabular}{@{}c@{}}$\mat{D}$ \textbf{Bdy.} \end{tabular}
& \begin{tabular}{@{}c@{}}$\mat{A}_\mat{D}$ \textbf{Int.} \end{tabular}
& \begin{tabular}{@{}c@{}}$\mat{A}_\mat{D}$ \textbf{Bdy.} \end{tabular} 
& \begin{tabular}{@{}c@{}} \textbf{Convergence} \end{tabular}\\ \hline
UFD even $p_\text{u} = 2p$ 
& $2p $ & ${\color{BrickRed}p}$ & $2p + 1$ & ${\color{BrickRed}p}$   & $\geq p+1.5$ \cite{Jiang2024} \\
UFD odd $p_\text{u} = 2p+1$ 
& $2p + 2 $ & ${\color{BrickRed}p}$ & $2 p + 1$ & ${\color{BrickRed}p}$ & $\geq p+1.5$ \cite{Jiang2024}   \\
FD $s = p$ 
& $2p$ & $p$ & $2p - 1$ & ${\color{BrickRed}p-1}$ & $\geq p$  \\
FD $s = p+1$  
& $2 p $ & ${\color{BrickRed}p}$ & $2p + 1$ & ${\color{BrickRed}p}$  & $\geq p+1$ \\
\tnote{*} SE $s=p$
& $p $ & $p$ & ${\color{BrickRed}p-1}$ & ${\color{BrickRed}p-1}$  & $\geq p$ \\ \hline
\end{tabular}
\begin{tablenotes}\footnotesize
\item[*] Alternatively could be labeled as `USE $p_u = p-1$' by the equivalence shown in \S\ref{sec:gsbp_uniqueness}.
\end{tablenotes}
\end{threeparttable}
\end{table}

We can now rewrite the upwind flux-vector splitting scheme \eqnref{eq:upwind} as
\begin{equation} \label{eq:upwind_equiv}
    \frac{\mathrm{d}\bu}{\mathrm{d}t} + \mat{D} \vec{f} = \textbf{SAT} - \varepsilon \mat{H}^{-1} \mat{S} \vec{f}_\mat{D}, 
\end{equation}
which is in the form of a central discretization plus a dissipation model. For linear symmetric constant-coefficient problems, where we can write $\vec{f}_\mat{D} = \mat{A} \bu$ for some block-diagonal matrix $\mat{A}$ with constant and symmetric positive semi-definite blocks $\mat{A}_n \geq 0$, the equivalence with \eqnref{eq:scalar} is immediately clear by defining $\mat{A}_\mat{D} = - \varepsilon \mat{H}^{-1} \mat{S} \mat{A}$. Stability also follows immediately since 
\begin{equation} \label{eq:upwind_stability}
    \mat{S} \mat{A} = \left( \mat{S} \mat{A} \right)^\T =  \mat{S}^{(1 \text{D})} \otimes \mat{A}_{n} \geq 0 .
\end{equation}
As noted in~\cite{Mattsson2017}, for constant-coefficient linear problems on Cartesian grids, the upwind discretization \eqnref{eq:upwind} or \eqnref{eq:upwind_equiv} is equivalent to the volume dissipation introduced in \S \ref{sec:art_diss_scal}. For more general cases, however, we will now describe the differences between \eqnref{eq:upwind} and the dissipation presented in this work.

One notable disadvantage of the flux-vector splitting approach is application to curvilinear coordinates. Because the variable coefficient is effectively applied \emph{outside} the dissipation operator by acting on the contravariant flux $\vec{f}_\mat{D}$, the curvilinear grid mapping must be restricted to a sufficiently low degree polynomial in order to satisfy free-stream preservation~\cite{Ranocha2024}. For practical simulations, this can limit the performance of a high-order method~\cite{DelReyFernandez2019_extension,CraigPenner2022}. By contrast, as explained in \S \ref{sec:multiple_dim}, the dissipation operators presented in this work exactly preserve free-stream by construction, allowing for more straightforward application to high-order polynomial or even nonpolynomial grids.

In \cite{Ranocha2024} and \cite{Glaubitz2024}, it was argued that upwind SBP discretizations possess desirable local linear stability properties. Periodic upwind SBP semi-discretizations of scalar conservation laws were shown to have Jacobians with non-positive real parts when the baseflow has everywhere positive (or negative) wave speeds, \ie, when $ \fnc{F}'(\bu) > 0$ \cite{Ranocha2024}. To the knowledge of the authors, such a proof is not possible for central SBP semi-discretizations augmented with dissipation~\eqnref{eq:diss_main_form}. Although nonlinear stability follows immediately from the positive semi-definiteness of $\mat{A}_\mat{D}$, linearizations of $\mat{A}_\mat{D}$ no longer have this property due to the presence of the nonlinear variable coefficient inside the operator. Regardless, numerical experiments in \S \ref{sec:burgers} show that the volume dissipation presented in this work is locally linearly stabilizing. Furthermore, results in \S \ref{sec:1deuler} suggest that entropy-stable schemes with volume dissipation are in fact more locally linearly stable than upwind SBP discretizations when considering more general problems. A further advantage of applying the dissipation \eqnref{eq:diss_main_form}, is that one can retain nonlinear stability properties of an underlying scheme while adjusting the value of the coefficient $\varepsilon$ to achieve local linear stability, if this is desired. This is demonstrated through numerical examples in \S \ref{sec:1deuler}.

Perhaps most importantly, upwind semi-discretizations \eqnref{eq:upwind} are not provably stable for nonlinear problems. The presence of a nonlinear $\vec{f}_\mat{D}$ outside of $\mat{S}$ makes it difficult to reproduce the proofs in \S \ref{sec:proofs} in general. For example, consider the Burgers equation with the splitting $f^\pm = \frac{1}{2} \left( \frac{u^2}{2} \pm \abs{u} u \right)$ such that $\fnc{F}(u) = \frac{u^2}{2}$. A lack of commutativity between $\mat{S}$ and $\mat{A} = \text{diag} \left( \abs{\bu} \right)$ means that \eqnref{eq:upwind_stability} no longer holds; $\mat{S} \mat{A}$ is not symmetric positive semi-definite. $\mat{S} \mat{A} + \left( \mat{S} \mat{A} \right)^\T$ can have negative eigenvalues, even though $\mat{S} \mat{A}$ will not. To make this concrete, consider a single block $x \in [0,1]$ with 5 equispaced nodes. Using the $p_\text{u}=2$ UFD operators\footnote{Using $p_\text{u}=3$ operators will similarly result in antidissipation of $\bu^T \mat{H} \mat{A}_\mat{D} \bu \approx 0.06 > 0.$ Also, note that $\left[ \mat{D}_- \right]_{i,j} = - \left[ \mat{D}_+ \right]_{N+1-j,N+1-i}$.} from \cite{Mattsson2017} (where $p_\text{u}$ is as defined in ~\cite{Mattsson2017}), 
\begin{equation*}
    \mat{D}_+ = \scalebox{0.8}{$\begin{bmatrix}
        -12 & 20 & -8 \\
        -0.8 & -4 & 6.4 & -1.6 \\
        & & -6 & 8 & -2 \\
        & & & -4 & 4 \\
        & & & -4 & 4
    \end{bmatrix}$} , \quad
    \mat{H} = \frac{1}{16} \scalebox{0.8}{$\begin{bmatrix}
        1 \\
        & 5 \\
        & & 4 \\
        & & & 5 \\
        & & & & 1
    \end{bmatrix}$} ,
\end{equation*}
the same splitting as above, and a solution $\vec{u} = 2 + 6 \vec{x} - \vec{x}^2$, then
\begin{equation*}
    \bu^T \mat{H} \mat{A}_\mat{D} \bu = - \bu^T \mat{S} \mat{U} \bu \approx 0.185 > 0 \quad , \quad \mat{U} = \text{diag}(\abs{\bu}) .
\end{equation*}
In other words, the dissipation introduced by the upwind vector splitting is in fact \emph{antidissipative} in the $\mat{H}$ energy norm despite a smooth, nearly linear,
and positive solution $u \geq 2$ on $x \in [0,1]$. It is clear therefore that such upwind formulations do not in general lead to provably stable nonlinear dissipation.

\subsection{Uniqueness of Volume Dissipation for Spectral-Element Operators} \label{sec:gsbp_uniqueness}

While the FD dissipation operators presented in this work differ from those recovered from the UFD operators of~\cite{Mattsson2017}, applying the construction procedure of \S \ref{sec:construction} to SE operators results in the same 1D dissipation operators as those in~\cite{Hicken2020} and recovered from the USE operators of~\cite{Ranocha2024}, up to a multiplicative constant. Enforcing design order accuracy through $\mat{A}_\mat{D} \vec{x}^k = \vec{0} \ \forall k \in \lbrace 0 , 1 \dots p \rbrace$ results in the trivial operator $\mat{A}_\mat{D} = 0$~\cite{Hicken2020, Ranocha2024}. Alternatively, we can reduce the accuracy by one order to obtain viable dissipation matrices, based on the following theorem, which is similar to Theorem 1 of~\cite{Hicken2020}.
\begin{theorem} \label{thm:element_equiv}
    Consider a symmetric volume dissipation operator $\mat{A}_\mat{D} \in \mathbb{R}^{N \times N}$ defined on a 1D reference element with $N=p+1 \geq 2$ unique nodes $\lbrace \vec{x} \rbrace \in [-1,1]$. If $\mat{A}_\mat{D} \vec{x}^k = \vec{0} \ \forall k \in \lbrace 0 , 1 \dots p-1 \rbrace$, then $\mat{A}_\mat{D}$ is uniquely defined up to a multiplicative constant.
\end{theorem}
\begin{proof}
    The case where $\mat{A}_\mat{D} \vec{x}^p = \vec{0}$ results in the trivial operator $\mat{A}_\mat{D} = 0$ \cite{Hicken2020, Ranocha2024}. Consider instead the more general case, $\mat{A}_\mat{D} \vec{x}^p \neq \vec{0}$.
    The assumption that $\mat{A}_\mat{D}$ annihilates polynomials up to degree $p-1$ is equivalent to saying that the polynomial subspace $\fnc{P}_{p-1} = \text{span}\left( \lbrace \vec{x}^k \rbrace_{k=0}^{p-1} \right) \subset \mathbb{R}^N$ lies in the null-space of $\mat{A}_\mat{D}$. Since $\fnc{P}_{p-1}$ is of dimension $p=N-1$, $\text{dim}\left( \text{Nul} \left( \mat{A}_\mat{D} \right) \right) \geq N-1 $.
    Furthermore, since $\vec{x}^p$ does not lie in the null-space of $\mat{A}_\mat{D}$, $\text{dim}\left( \text{Col} \left( \mat{A}_\mat{D} \right) \right) \geq 1 $. Therefore, $\mat{A}_\mat{D}$ is rank 1, and since it is also symmetric, it is diagonalizable and has exactly one non-zero eigenvalue $\lambda$ with corresponding eigenvector $\vec{v}$ where $\norm{\vec{v}} = 1$.
	As a consequence, $\mat{A}_\mat{D}$ can be expressed as $\mat{A}_\mat{D} = \lambda \vec{v} \vec{v}^\T$. Taking an orthogonal basis $\lbrace \vec{\phi}_k \rbrace_{k=1}^{N-1}$ for $\fnc{P}_{p-1}$, by considering the inner products $\lambda \vec{v} \vec{v}^\T \vec{\phi}_k = 0$, we observe that $\vec{v}$ is orthogonal to $\fnc{P}_{p-1}$, and since the orthogonal complement to this subspace is unique and of dimension 1, we conclude that $\vec{v}$ is unique. Therefore, the decomposition $\mat{A}_\mat{D} = \lambda \vec{v} \vec{v}^\T$ is uniquely determined by the polynomial annihilation condition, up to the value of $\lambda$. \qed
\end{proof}

Examination of the SE operators presented in this work and those of \cite{Hicken2020} and \cite{Ranocha2024} reveals that they satisfy the conditions of \thmref{thm:element_equiv}. Intuitively, this uniqueness can be explained by the fact that SE operators with minimal degrees of freedom are unique and nondissipative. If we relax the discretization's polynomial exactness by one degree to introduce dissipation, $\mat{A}_{\mat{D}}$ must then be uniquely specified up to a constant. Consequently, for SE operators, the boundary correction matrix $\mat{B}$ and the variable coefficient averaging at half-nodes for odd~$s$ are irrelevant, since the resulting dissipation operators will only be modified by a multiplicative constant. However, the inclusion of variable coefficients either inside or outside the dissipation operator, as discussed in \S \ref{sec:relation_to_upwind}, as well as extensions to systems of conservation laws discussed in \S \ref{sec:systems} remain important, and will result in different stability properties of the numerical schemes. It is possible to obtain more general SE dissipation operators by reducing polynomial exactness even further (\eg~\cite{Ranocha_dissipation}), however alternative strategies to maintain design-order accuracy and dimensional consistency are then required, such as scaling $\varepsilon$ by some power of $h$ or a nonlinear design-order functional, as is commonly done in artificial-viscosity approaches \cite{bassi_accurate_1995,zingan_entropy_2013,yu_viscosity_2020,wei_jump_2025}.

\section{Numerical Results and Discussion} \label{sec:results}

All numerical results are reproducible via the open-source repository \url{https://github.com/alexbercik/ESSBP/tree/main/Paper-StableVolumeDissipation}. Unless otherwise specified, time-marching is performed using the eighth-order adaptive Dormand–Prince algorithm \cite{Hairer1993} with relative and absolute tolerances taken below $10^{-12}$ such that temporal errors are negligible.

\subsection{Linear Convection Equation} \label{sec:LCE}

We first examine the spectra of central SBP semi-discretizations of the one-dimensional linear convection equation on a periodic domain $x \in [0,1]$, described for example in \cite{DelReyFernandez2014_review}, augmented with different variants of the FD volume dissipation \eqnref{eq:diss_main_form}. To ensure design order, we use $s=p+1$. We compare these to the spectra of UFD SBP semi-discretizations described in \cite{Mattsson2017}. The normalized spectra for CSBP $p=3$ and $p=4$ operators are plotted in \figref{fig:LCEeigs_initial} using 80 nodes and upwind SATs (Lax-Friedrichs numerical fluxes). Experimentally, we find that a value of $\varepsilon = 3.125 \times 5^{-s}$ applies a comparable amount of dissipation to UFD schemes without impacting the spectral radius or time step restriction. More generally, however, this coefficient may be problem-dependent, and lower values may be preferable. Therefore, for the remainder of this work, we also use a smaller value of $\varepsilon = 0.625 \times 5^{-s}$. Further discussion on selecting $\varepsilon$ can be found in the supplementary material. Following the discussions of \S\ref{sec:effect_B} and \S\ref{sec:variable_coeffient}, we compare four variants of the dissipation operators by optionally including the boundary correction $\mat{B}$ and undivided norm $\tilde{\mat{H}}$ (\emph{not} applied at half-nodes for odd $s$) in the dissipation operators. While including $\tilde{\mat{H}}$ has relatively little impact on the spectra, the inclusion of $\mat{B}$ is observed to have a beneficial effect on the spectral radius, particularly for higher degrees $p$ and values of $\varepsilon$. Eigenvalues with large negative real parts correspond to eigenvectors with large oscillations near the block boundaries, consistent with the observation in~\S \ref{sec:effect_B} that $\mat{B}$ reduces the dissipation applied at the boundaries. Additional spectra for varying degrees $p$ of CSBP, HGTL, HGT, and Mattsson operators exhibit similar behaviour, and are shown in Appendix~\ref{sec:additional_figures}. For SE operators, we use the formulation $\mat{A}_\mat{D} = -\varepsilon \HI \tilde{\mat{D}}_s^\T \tilde{\mat{D}}_s$ (since it was shown in \S\ref{sec:gsbp_uniqueness} that the choice of $\mat{C}$ is irrelevant) with $s=p$, and set dissipation coefficients $\varepsilon \approx 0.1 \times 2.25^{-p}$ such that the addition of volume dissipation increases the spectral radius of the semi-discretizations by approximately $20 \%$. Spectra for LGL and LG operators are included in Appendix~\ref{sec:additional_figures}.

\begin{figure}[t] 
    \centering
    \begin{subfigure}[t]{0.32\textwidth}
        \centering
        \includegraphics[width=\textwidth, trim={6 10 6 6}, clip]{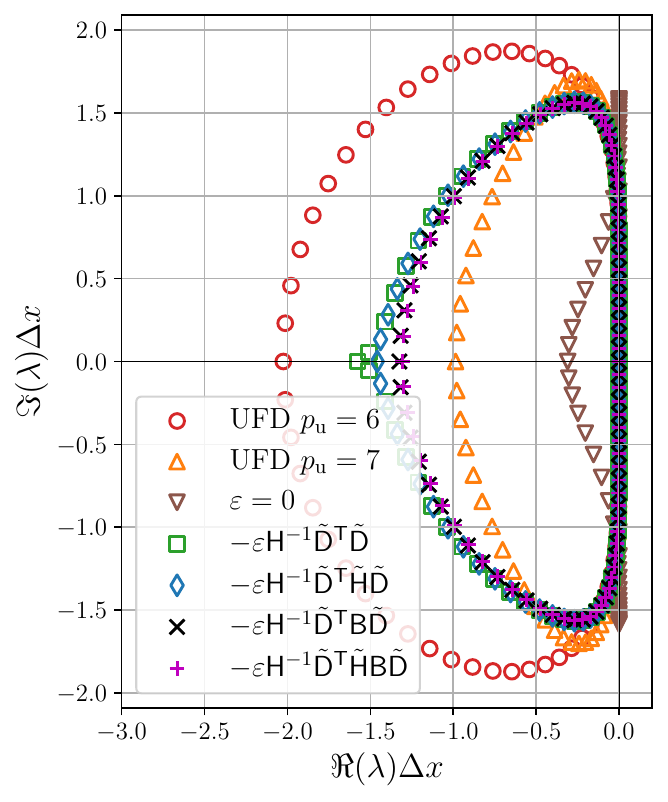}
        \caption{CSBP $p=3$, $\varepsilon = 0.005$}
    \end{subfigure}
    \hfill
    \begin{subfigure}[t]{0.32\textwidth}
        \centering
        \includegraphics[width=\textwidth, trim={6 10 6 6}, clip]{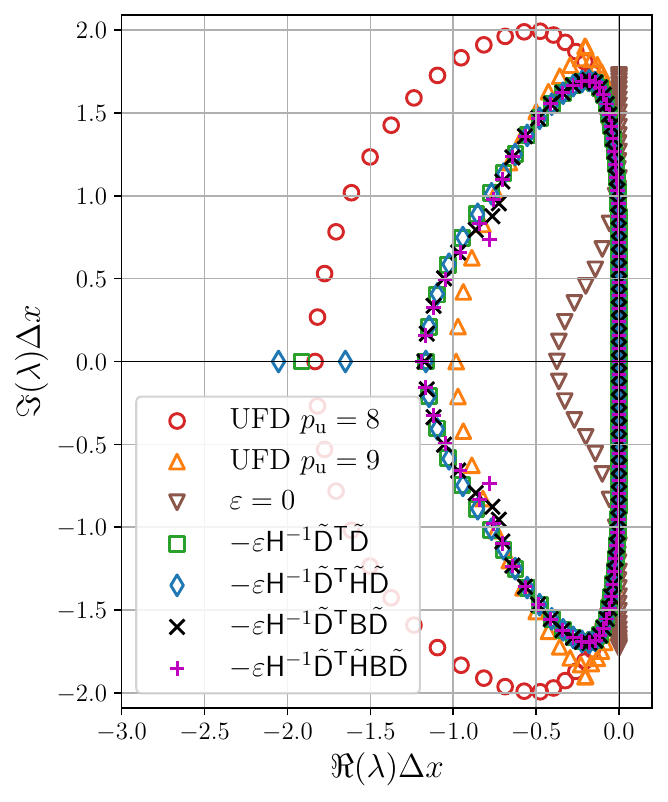}
        \caption{CSBP $p=4$, $\varepsilon = 0.001$}
    \end{subfigure}
    \hfill
    \begin{subfigure}[t]{0.32\textwidth}
        \centering
        \includegraphics[width=\textwidth, trim={6 10 6 6}, clip]{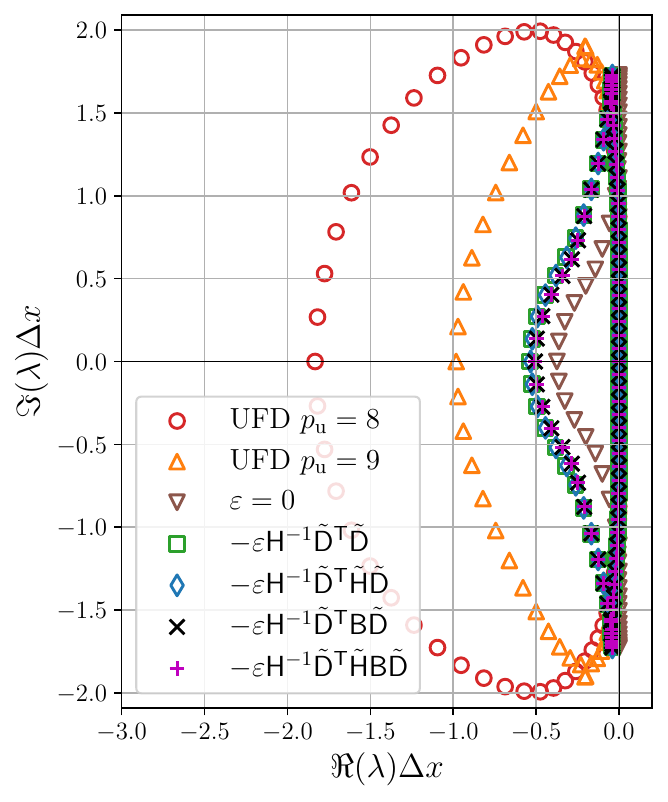}
        \caption{CSBP $p=4$, $\varepsilon = 0.0002$}
    \end{subfigure}
    \caption{Eigenspectra for the linear convection equation of the UFD SBP semi-discretizations of \cite{Mattsson2017} and central CSBP semi-discretizations with 80 nodes, upwind (Lax-Friedrichs) SATs, augmented with dissipation operators $s=p+1$.}
    \label{fig:LCEeigs_initial}
\end{figure}

\begin{figure}[!t] 
    \centering
    \begin{subfigure}[t]{0.32\textwidth}
        \centering
        \includegraphics[width=\textwidth, trim={5 10 5 5}, clip]{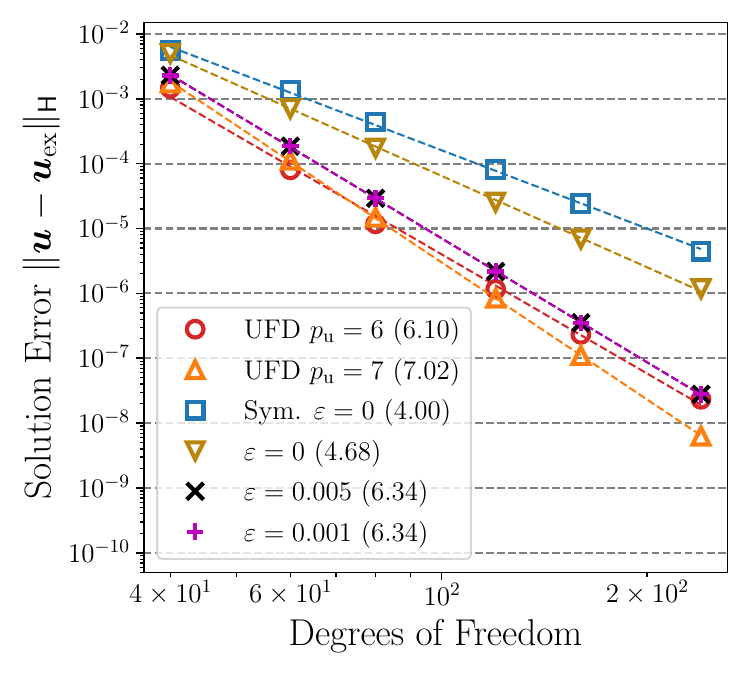}
        \caption{CSBP $p=3$, $t=1$}
    \end{subfigure}
    \hfill
    \begin{subfigure}[t]{0.32\textwidth}
        \centering
        \includegraphics[width=\textwidth, trim={5 10 5 5}, clip]{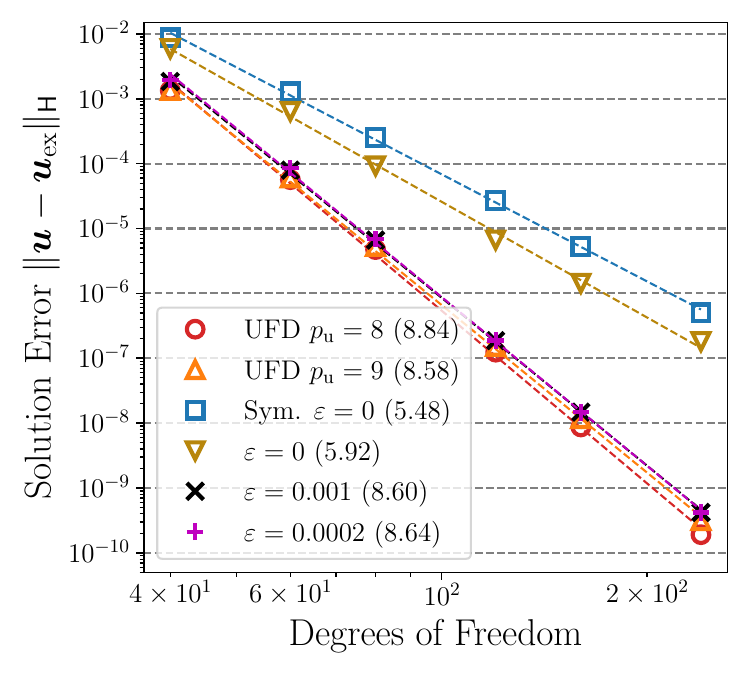}
        \caption{CSBP $p=4$, $t=1$}
    \end{subfigure}
    \hfill
    \begin{subfigure}[t]{0.32\textwidth}
        \centering
        \includegraphics[width=\textwidth, trim={5 10 5 5}, clip]{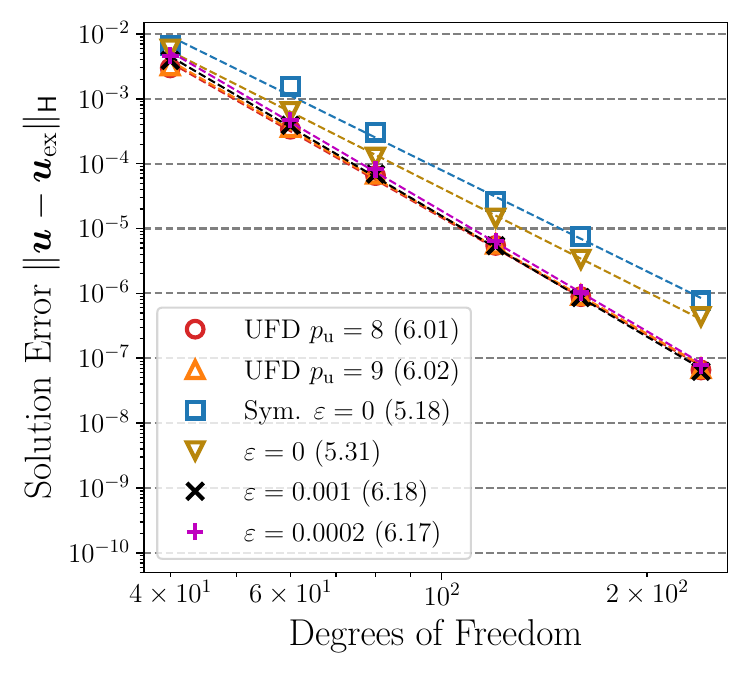}
        \caption{CSBP $p=4$, $t=1.5$}
    \end{subfigure}
    \caption{Solution error convergence after 1 or 1.5 periods of the linear convection equation for UFD and central SBP schemes with either symmetric (\ie, no dissipation) or upwind SATs, and volume dissipation $\mat{A}_\mat{D} = -\varepsilon \HI \tilde{\mat{D}}_s^\T \mat{B} \tilde{\mat{D}}_s$ with $s=p+1$. Convergence rates are given in the legends.}
    \label{fig:LCEconv}

\bigskip

    \begin{subfigure}[t]{0.32\textwidth}
        \centering
        \includegraphics[width=\textwidth, trim={5 10 5 5}, clip]{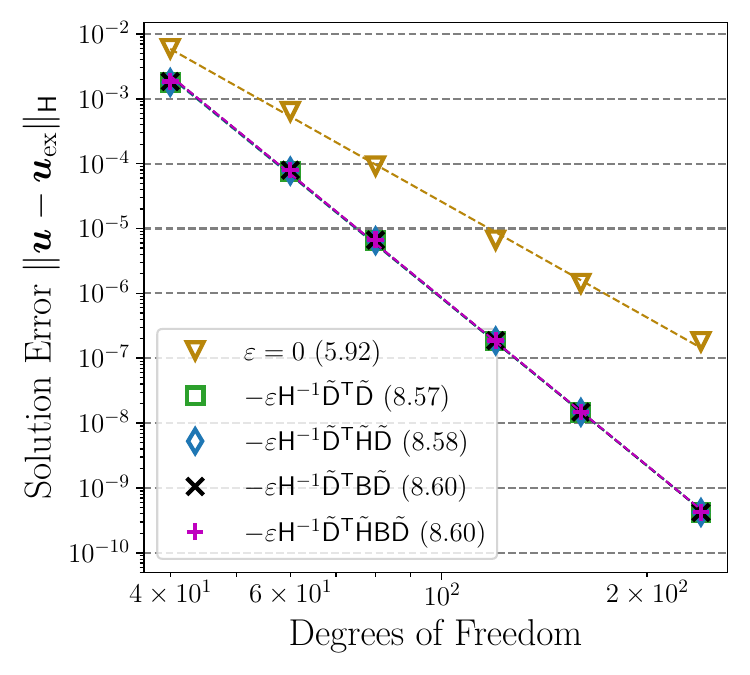}
        \caption{CSBP $p=4$, $t=1$}
        \label{fig:LCEconv2_a}
    \end{subfigure}
    \hfill
    \begin{subfigure}[t]{0.32\textwidth}
        \centering
        \includegraphics[width=\textwidth, trim={5 10 5 5}, clip]{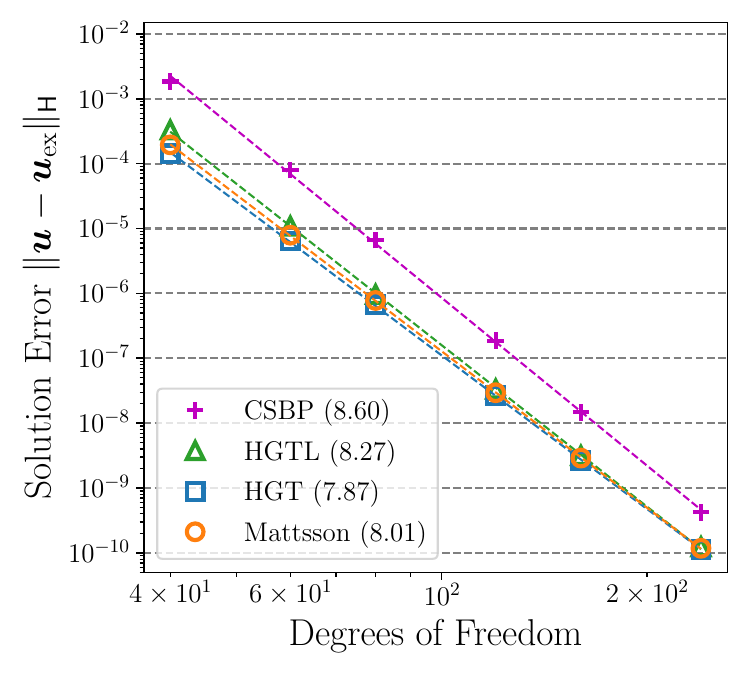}
        \caption{$p=4$, $t=1$}
        \label{fig:LCEconv2_b}
    \end{subfigure}
    \hfill
    \begin{subfigure}[t]{0.32\textwidth}
        \centering
        \includegraphics[width=\textwidth, trim={5 10 5 5}, clip]{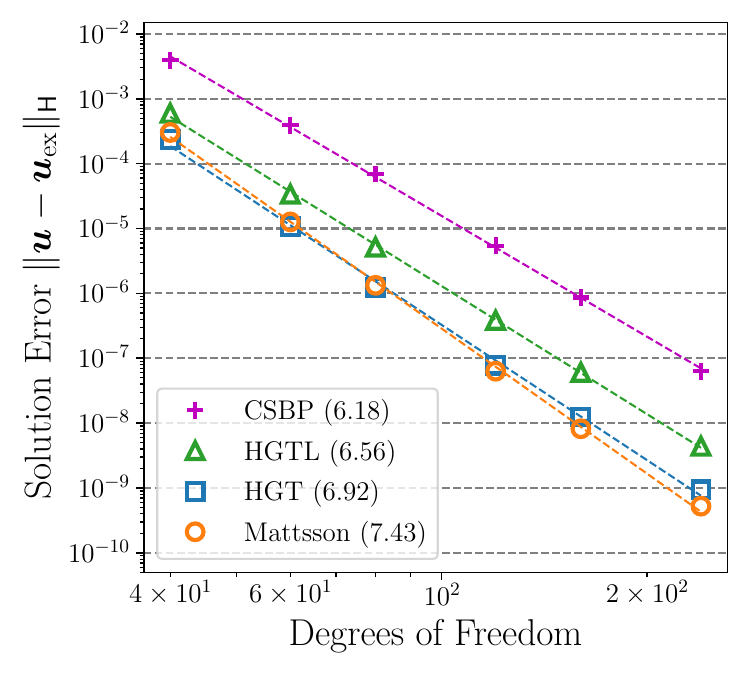}
        \caption{$p=4$, $t=1.5$}
        \label{fig:LCEconv2_c}
    \end{subfigure}
    \caption{Solution error convergence after 1 or 1.5 periods of the linear convection equation for central SBP schemes with upwind SATs and volume dissipation with $\varepsilon=0.0002$. Subfigure (a) compares different formulations of $\mat{A}_\mat{D}$, while (b) and (c) compare different operators, each using $\mat{A}_\mat{D} = -\varepsilon \HI \tilde{\mat{D}}_s^\T \mat{B} \tilde{\mat{D}}_s$.}
    \label{fig:LCEconv2}
\end{figure}

To verify the order of accuracy of the FD dissipation operators, we perform a grid-refinement. We use the initial condition
\begin{equation} \label{eq:gaussian}
    \fnc{U}\left(x,0\right) = e^{0.5 \left( \frac{x- 0.5}{0.08} \right)^2}
\end{equation}
and evaluate the solution error $\norm{\vec{u}-\fnc{U}}_\Hnrm = \left(\vec{u}-\fnc{U}\right)^\T \Hnrm \left(\vec{u}-\fnc{U}\right)$ after either one or one and a half periods on grids with between 40 and 240 nodes. Motivated by the spectra of \figref{fig:LCEeigs_initial} and arguments in \S \ref{sec:implement_vol_diss_for_sbp}, we use the formulation ${\mat{A}_\mat{D} = -\varepsilon \HI \tilde{\mat{D}}_s^\T \mat{B} \tilde{\mat{D}}_s}$ with $s=p+1$. Results for CSBP $p=3$ and $p=4$ are plotted in \figref{fig:LCEconv}, where we again include UFD discretizations of \cite{Mattsson2017} for comparison, as well as discretizations with no volume dissipation (but dissipation from upwind SATs), and no dissipation at all (using symmetric SATs). Error convergence rates are computed using a linear best-fit through the data. All discretizations achieve the expected convergence rates of~$\geq p+1$, and are in fact consistent with the sharper $\geq p+1.5$ error estimates of~\cite{Jiang2024}. Furthermore, when the final time is chosen such that the Gaussian is in the interior of the domain, discretizations with volume dissipation converge at a much greater than expected rate of~$\geq 2p$, matching the order of the interior stencil. For example, for CSBP $p=4$ operators, at $t=1.5$, when the Gaussian is at a block interface, discretizations with volume dissipation gain one additional order of convergence and one-order-of-magnitude reduction in error on the finest grid level. At $t=1$, however, when the Gaussian is in the block interior, discretizations with volume dissipation gain two and a half orders of convergence and a three-order-of-magnitude reduction in error on the finest grid level. In general, the benefit to solution accuracy of volume dissipation increases with $p$. For larger $p$, all discretizations with volume dissipation perform nearly identically, but for smaller $p$, the odd $p_\text{u}$ upwind formulation appears preferable, likely due to its wider and more accurate internal stencil. \figref{fig:LCEconv2_a} demonstrates that the exclusion of $\Hnrm$ has little impact on solution accuracy, further justifying the choice $\mat{C}=\mat{B}\mat{A}$. Additional plots in Appendix~\ref{sec:additional_figures} demonstrate similar results for HGTL, HGT, and Mattsson operators; however, the addition of volume dissipation does not appear to improve the solution accuracy of Mattsson operators. Nevertheless, volume dissipation may still be beneficial for Mattsson operators in other ways, such as in accelerating convergence to steady-state and maintaining positivity of thermodynamic variables (demonstrated in \S\ref{sec:KelvinHelmholtz}). A comparison in \figref{fig:LCEconv2_b} demonstrates that using HGTL, HGT, or Mattsson operators can result in a roughly two-order-of-magnitude reduction in error as compared to equispaced CSBP operators. Since SE SBP schemes augmented with volume dissipation \eqnref{eq:diss_main_form} are equivalent to the USE discretizations of \cite{Glaubitz2024} in the context of the linear convection equation, we defer the verification of their accuracy to \cite{Glaubitz2024} and Appendix \ref{sec:appendix_spectral}.

\begin{figure}[t] 
    \centering
    \begin{subfigure}[t]{0.48\textwidth}
        \centering
        \includegraphics[width=\textwidth, trim={10 13 15 15}, clip]{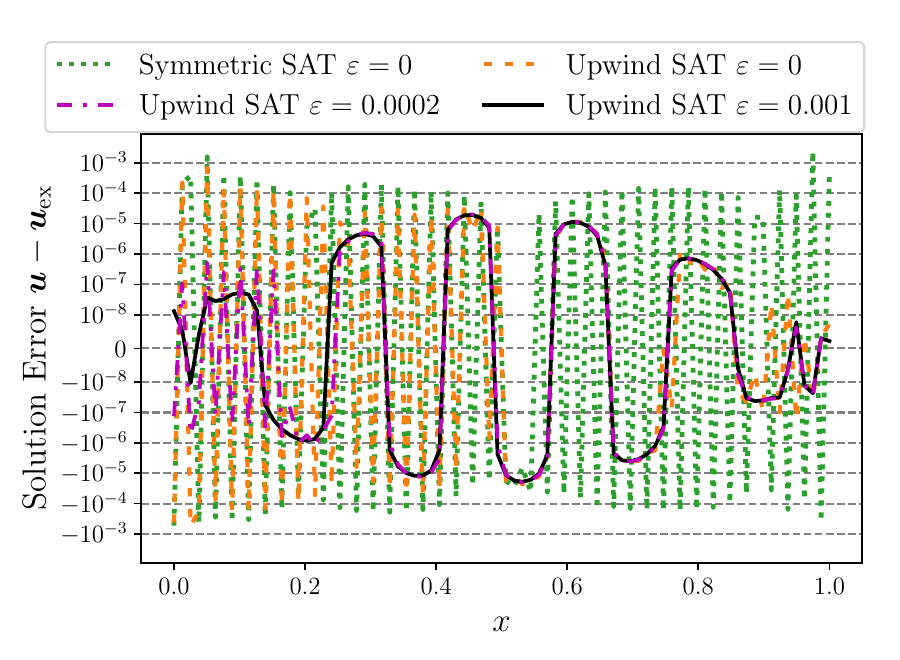}
        \caption{Gaussian, 1 block, 80 nodes, $t=1$}
    \end{subfigure}
    \hfill
    \begin{subfigure}[t]{0.48\textwidth}
        \centering
        \includegraphics[width=\textwidth, trim={10 13 15 15}, clip]{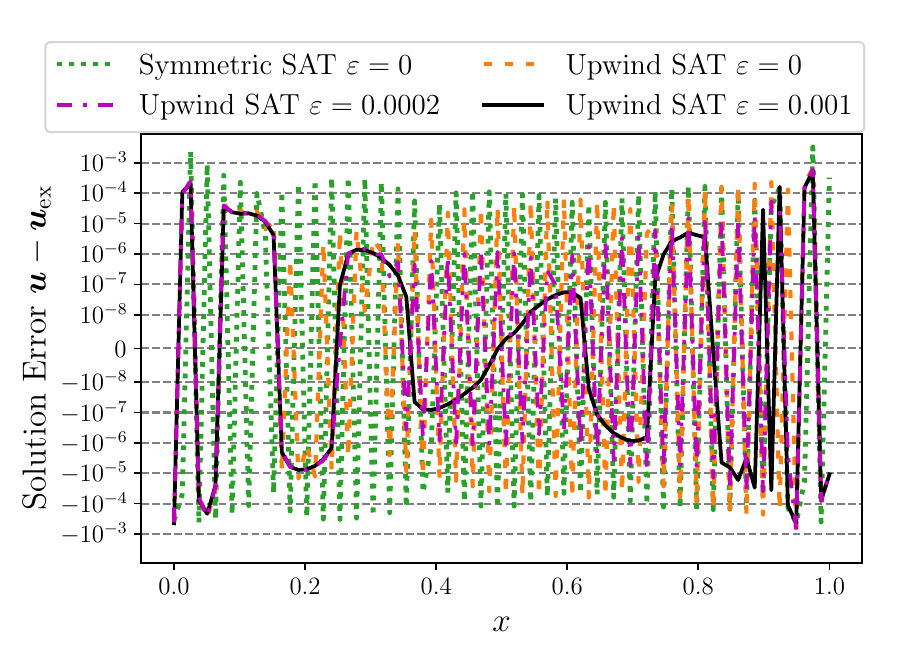}
        \caption{Gaussian, 1 block, 80 nodes, $t=1.5$}
    \end{subfigure}
    \vskip\baselineskip 
    \begin{subfigure}[t]{0.48\textwidth}
        \centering
        \includegraphics[width=\textwidth, trim={10 13 15 15}, clip]{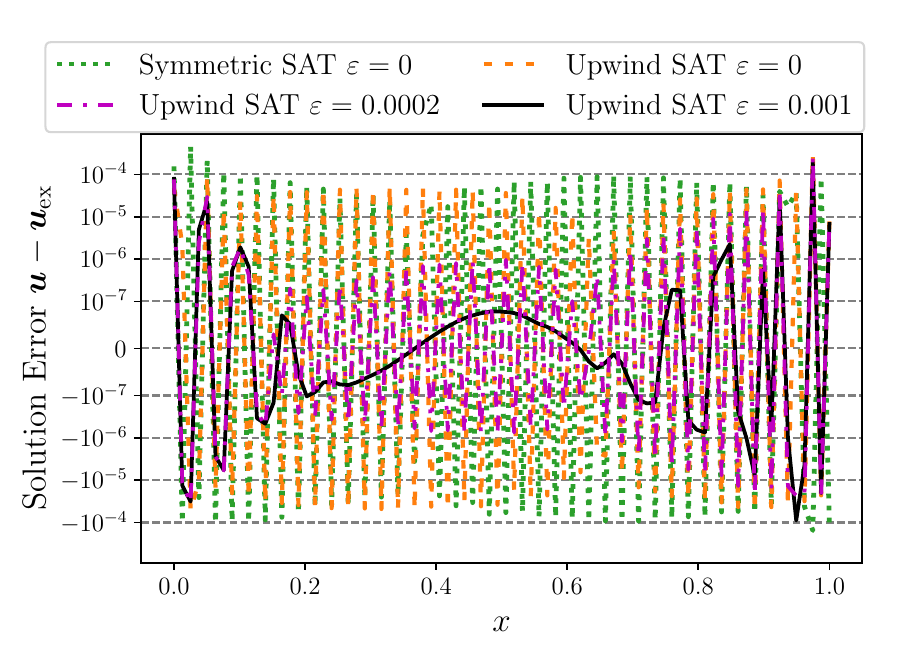}
        \caption{Sinusoid, 1 block, 80 nodes, $t=1$}
    \end{subfigure}
    \hfill
    \begin{subfigure}[t]{0.48\textwidth}
        \centering
        \includegraphics[width=\textwidth, trim={10 13 15 15}, clip]{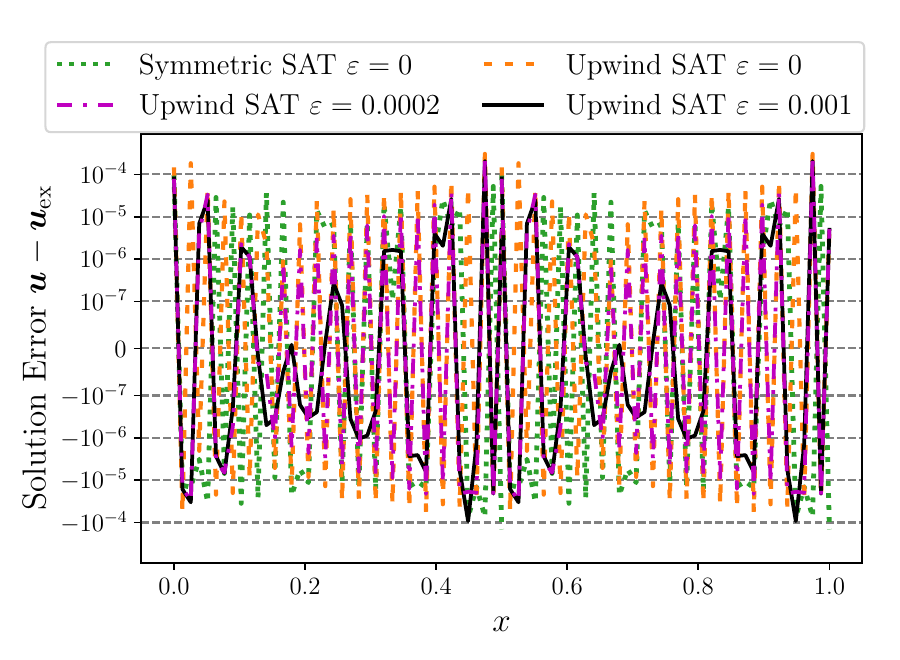}
        \caption{Sinusoid, 2 blocks, 40 nodes each, $t=1$}
    \end{subfigure}
    \caption{Solution error of the linear convection equation with either a Gaussian \eqnref{eq:gaussian} or Sinusoidal \eqnref{eq:sin4pi} initial condition for central SBP schemes using CSBP $p=4$ operators, either symmetric (\ie, no dissipation) or upwind SATs augmented with dissipation $\mat{A}_\mat{D} = -\varepsilon \HI \tilde{\mat{D}}_s^\T \mat{B} \tilde{\mat{D}}_s$ where $s=5$.}
    \label{fig:LCEsolerr}
\end{figure}

To further examine the effect of volume dissipation for FD operators, we plot the difference between numerical and exact solutions in \figref{fig:LCEsolerr} after one period. In addition to the Gaussian initial condition~\eqnref{eq:gaussian}, we also use a sinusoidal initial condition 
\begin{equation} \label{eq:sin4pi}
    \fnc{U}\left(x,0\right) = \sin{\left( 4 \pi x \right)}, \quad x \in [0,1].
\end{equation}
The discretization without any dissipation exhibits highly oscillatory errors introduced by the boundary stencils, which are reduced but not completely eliminated by the interface dissipation. Volume dissipation removes these high frequency modes from the solution. Increasing the dissipation coefficient beyond $\varepsilon = 3.125 \times 5^{-s}$ has a negligible impact on the errors observed in \figref{fig:LCEsolerr}. In examining the results of the sinusoidal initial condition in particular, it becomes clear that the volume dissipation is effective in removing high-order oscillations from the interior of the domain. Immediately adjacent to the block boundaries however, solution accuracy is not significantly improved. This is likely related to the presence of lower-order and possibly non-dissipative stencils at the boundaries of $\mat{A}_\mat{D}$, visible for example in the Taylor expansions \eqnref{eq:diss_2_b1} and \eqnref{eq:diss_2_b2}, and a direct result of the symmetric negative semi-definite structure of $\mat{A}_\mat{D}$. This explains why discretizations with volume dissipation often exceed expected error convergence rates when refining in an FD manner. Volume dissipation becomes more effective as the number of interior nodes increases, and since the interior also constitutes a larger portion of the domain, reducing its error has a greater impact on the overall global error. This also provides insight into why volume dissipation does not improve the solution accuracy of Mattsson operators. Since Mattsson operators are optimized to minimize solution error near block boundaries, the influence of the dissipation operator’s lower-order boundary stencils likely becomes more pronounced. Optimizing the dissipation operators by introducing unknowns near the boundaries of $\tilde{\mat{D}}_s$, as mentioned in \S \ref{sec:construction}, could improve their accuracy benefit.

\subsection{The Burgers Equation} \label{sec:burgers}

To verify the nonlinear stability properties and examine the local linear stability properties of the dissipation operators, we study the inviscid Burgers equation on a periodic domain $x \in [0,1]$. We employ the entropy-stable split-form SBP semi-discretization described in \cite{DelReyFernandez2014_review} and \cite{Gassner2022} with Rusanov entropy-dissipative interface SATs, \ie a numerical interface flux with dissipative part
$$f^\text{diss}(u_L,u_R) = - \max \left( \abs{u_L}, \abs{u_R} \right) \left( u_R - u_L \right). $$
We choose the initial condition
\begin{equation}
    \fnc{U}\left(x,0\right) = \sin{\left( 2 \pi x \right)} + \beta , 
\end{equation}
where setting $\beta = 1.5$ leads to a strictly positive solution, ensuring that the eigenvalues of the linearized semi-discretization Jacobian should have non-positive real parts in order to be consistent with the local linear stability properties of the continuous PDE \cite{Gassner2022}. We augment the baseline entropy-dissipative scheme with volume dissipation $\mat{A}_\mat{D} = -\varepsilon \HI \tilde{\mat{D}}^\T_s \mat{B} \mat{U} \tilde{\mat{D}}_s$, where $\mat{U} = \text{diag}\left(\vec{u}_h\right)$ is the scalar flux Jacobian (with a simple average at half-nodes for odd $s$ as explained in \S \ref{sec:variable_coeffient}), $s=p+1$ and $s=p$ for FD and SE operators, respectively, and the values of $\varepsilon$ selected in \S \ref{sec:LCE}. We run the simulation until the breaking time of $t=1/2\pi$, and compare the behaviour to the UFD operators of \cite{Mattsson2017} and USE operators of \cite{Glaubitz2024} using the flux-vector splitting presented in \cite{Ranocha2024}, as well as baseline entropy-conservative and entropy-dissipative discretizations relying solely on interface dissipation ($\varepsilon = 0$).

\begin{figure}[t] 
    \centering
    \begin{subfigure}[t]{0.32\textwidth}
        \centering
        \includegraphics[width=\textwidth, trim={11 13 10 10}, clip]{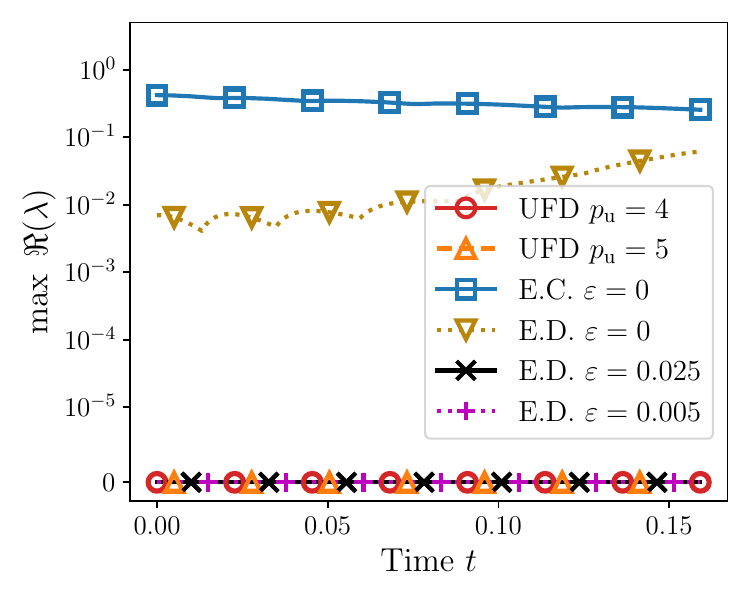}
        \caption{CSBP $p=2$, 40 nodes}
    \end{subfigure}
    \hfill
    \begin{subfigure}[t]{0.32\textwidth}
        \centering
        \includegraphics[width=\textwidth, trim={11 13 10 10}, clip]{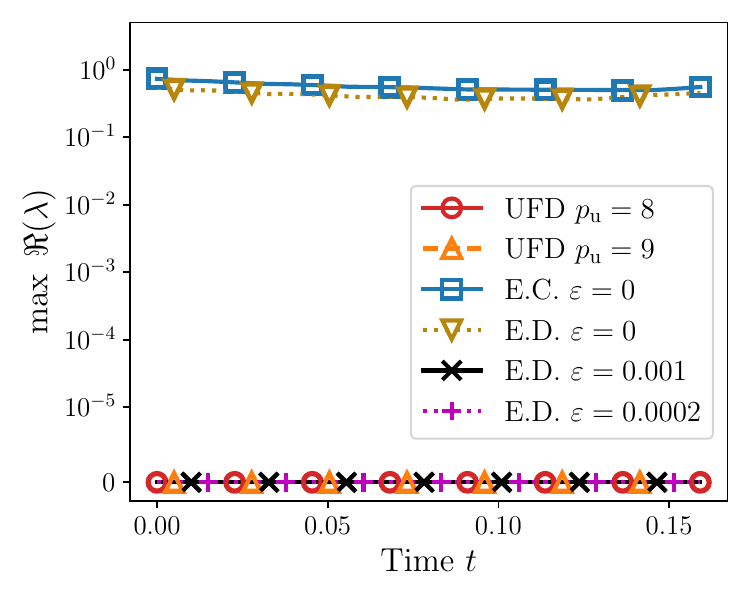}
        \caption{CSBP $p=4$, 40 nodes}
    \end{subfigure}
    \hfill
    \begin{subfigure}[t]{0.32\textwidth}
        \centering
        \includegraphics[width=\textwidth, trim={11 13 10 10}, clip]{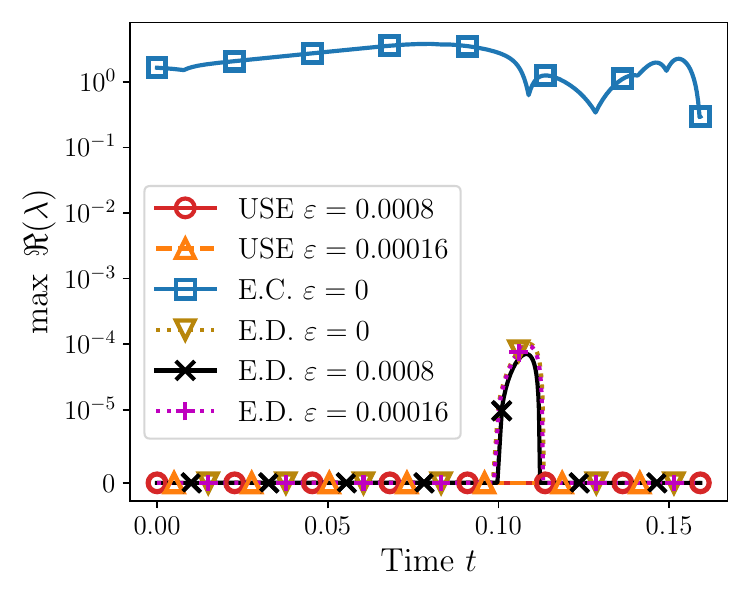}
        \caption{LGL $p=6$, 5 elements}
    \end{subfigure}
    \caption{Maximum real part of the eigenvalues of the semi-discretization Jacobian for the Burgers equation using UFD, USE, and entropy-stable split form SBP schemes with either entropy-conservative (E.C.) or entropy-dissipative (E.D.) Rusanov SATs. Values shown below $10^{-5}$ are zero to machine-precision.}
    \label{fig:BurgersMaxEig}
\end{figure}

To examine local linear stability, \figref{fig:BurgersMaxEig} plots the maximum real part of the eigenvalues of the semi-discretization Jacobian (linearization of the semi-discretization), computed via the complex step approximation, at each time step. Consistent with \cite{Ranocha2024, Glaubitz2024}, the upwind schemes are locally linearly stable with $\beta = 1.5$, \ie, they possess no eigenvalues with positive real parts. By contrast, and as discussed in \cite{Gassner2022}, both the entropy-conservative scheme and the entropy-dissipative scheme relying on interface dissipation alone are locally linearly \emph{unstable}. Reassuringly, adding a small amount of volume dissipation to FD schemes is sufficient to maintain local linear stability to machine precision for all time steps up until the breaking time. Similar results were obtained with varying degrees $p$, numbers of blocks, and numbers of interior nodes; however it should be noted that these results do not hold past the breaking time. Past $t=1/2 \pi$, a shock develops that often leads to the emergence of eigenvalues with positive real parts, likely due to the connection between local linear instabilities and under-resolved modes suggested in \cite{Gassner2022}. Although SE schemes with volume dissipation also usually exhibit local linear stability, \figref{fig:BurgersMaxEig} includes a counterexample where adding volume dissipation does not prevent the appearance of eigenvalues with positive real parts, even for smooth solutions. A possible explanation is discussed in Appendix~\ref{sec:appendix_spectral}.

\begin{figure}[t] 
    \centering
    \begin{subfigure}[t]{0.32\textwidth}
        \centering
        \includegraphics[width=\textwidth, trim={10 13 10 10}, clip]{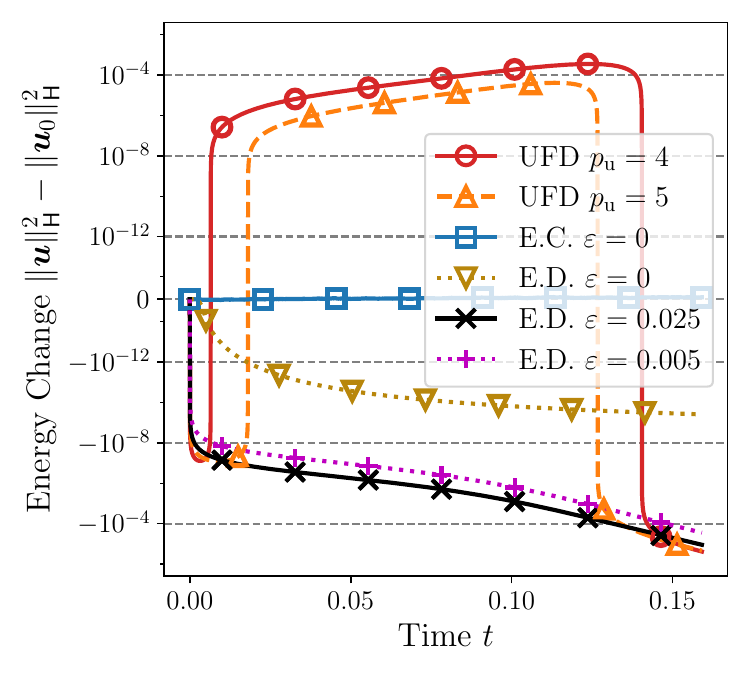}
        \caption{CSBP $p=2$, 40 nodes}
    \end{subfigure}
    \hfill
    \begin{subfigure}[t]{0.32\textwidth}
        \centering
        \includegraphics[width=\textwidth, trim={10 13 10 10}, clip]{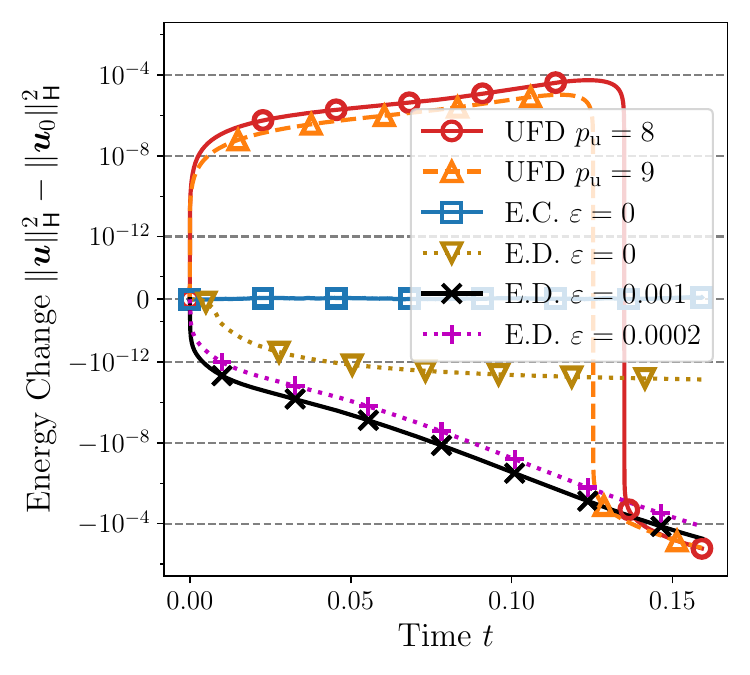}
        \caption{CSBP $p=4$, 40 nodes}
    \end{subfigure}
    \hfill
    \begin{subfigure}[t]{0.32\textwidth}
        \centering
        \includegraphics[width=\textwidth, trim={10 13 10 10}, clip]{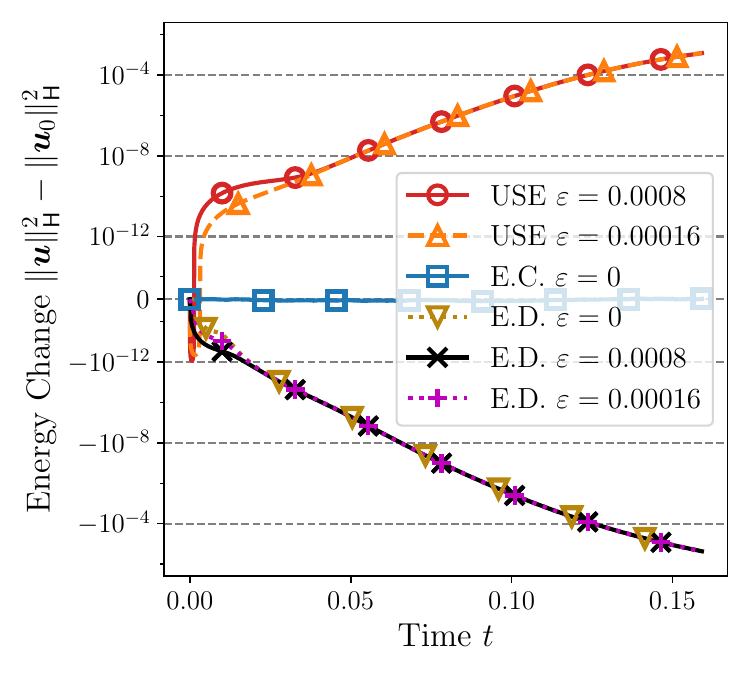}
        \caption{LGL $p=6$, 5 elements}
    \end{subfigure}
    \caption{Change in discrete energy $\norm{\vec{u}}_\Hnrm^2 = \vec{u}^\T \Hnrm \vec{u}$ for the Burgers equation using the UFD, USE, and entropy-stable split form SBP schemes with either entropy-conservative (E.C.) or entropy-dissipative (E.D.) Rusanov SATs.}
    \label{fig:BurgersEnergy}
\end{figure}

To verify the nonlinear stability properties, Figure \ref{fig:BurgersEnergy} plots the discrete energy in the $\Hnrm$-norm for the case $\beta=0$, where nonlinear instabilities are more easily observed. Neglecting time marching errors, the entropy-conservative discretization conserves energy, as expected, while the entropy-dissipative discretizations exhibit a monotonic decrease in energy. An increase in energy is observed for the upwind schemes, which stems from the lack of provable stability discussed in \S \ref{sec:relation_to_upwind}. Past the breaking time of $t=1/2 \pi$, the presence of the shock can cause the upwind discretizations to blow up, whereas the entropy-stable discretizations remain stable for all times. The shock and resulting spurious oscillations also significantly reduce the solution accuracy of FD schemes without volume dissipation, whereas those augmented with volume dissipation are able to more accurately represent the correct solution. All discretizations remain discretely conservative to machine precision.

\subsection{1D Euler Density-Wave} \label{sec:1deuler}

In \cite{Gassner2022}, a simple density-wave problem for the 1D Euler equations \eqnref{eq:1D_euler} was found to be problematic for high-order entropy-stable schemes. Using an initial condition where density is close to zero, \ie, 
\begin{align*}
    \vec{u} = \begin{bmatrix}
        \rho \\ \rho v \\ e
    \end{bmatrix} \ , \ \
    \rho = 1 + 0.98 \sin{\left( 2 \pi x \right)} \ , \ \
    v = 0.1 \ , \ \
    p = 20 \ , \ \
    \gamma = 1.4 
\end{align*}
on a domain $x \in [-1,1]$, the presence of unphysical eigenvalues with positive real parts was shown to quickly lead to spurious oscillations and positivity violations. To further investigate the ability of volume dissipation to suppress these oscillations, we employ the entropy-stable discretization of \cite{Fisher_2013, Crean2018} with the Chandrashekar entropy-conserving two-point flux \cite{Chandrashekar2013} and entropy-dissipative interface dissipation of \cite{Winters2017}. We augment this with matrix-matrix and scalar-matrix variants (introduced in \S \ref{sec:systems}) of the volume dissipation $\mat{A}_\mat{D} = -\varepsilon \HI \tilde{\mat{D}}_s^\T \mat{B} \mat{A} \tilde{\mat{D}}_s$ acting on entropy variables. For comparison, we include results for the UFD scheme of~\cite{Mattsson2017, Ranocha2024} with Steger-Warming flux splitting~\cite{Steger_flux} and Roe interface dissipation. We use the same problem formulation and initial condition described in Equation (50) of~\cite{Gassner2022} with a $p=4$ CSBP operator on grids of between 20 and 320 nodes.

\begin{figure}[t] 
    \centering
    \begin{subfigure}[t]{0.32\textwidth}
        \centering
        \includegraphics[height=3.5cm, trim={14 10 35 25}, clip]{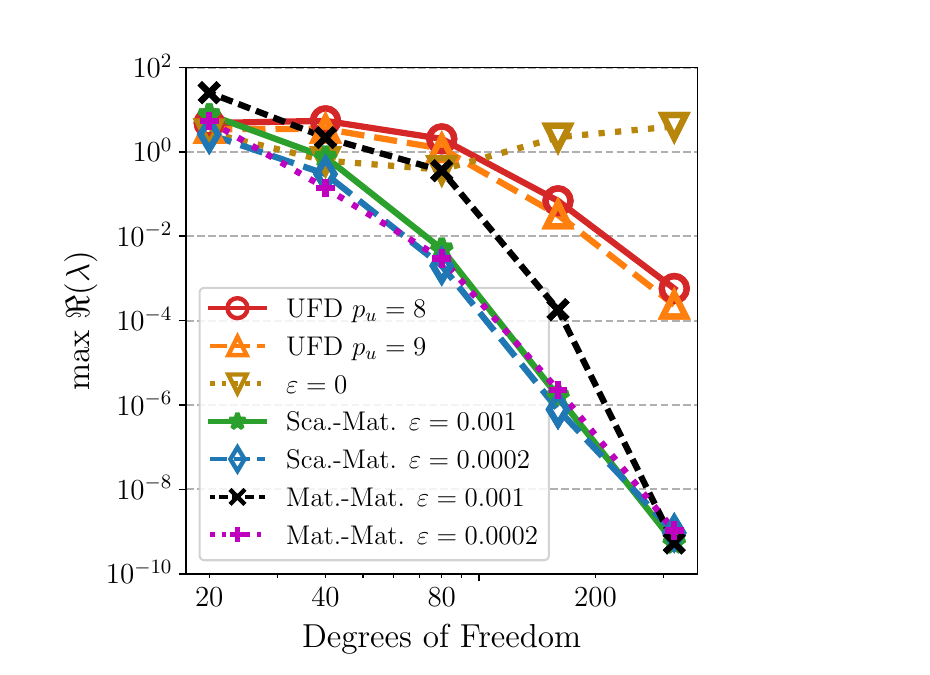}
    \end{subfigure}
    \hfill
    \begin{subfigure}[t]{0.32\textwidth}
        \centering
        \includegraphics[height=3.5cm, trim={6 10 35 25}, clip]{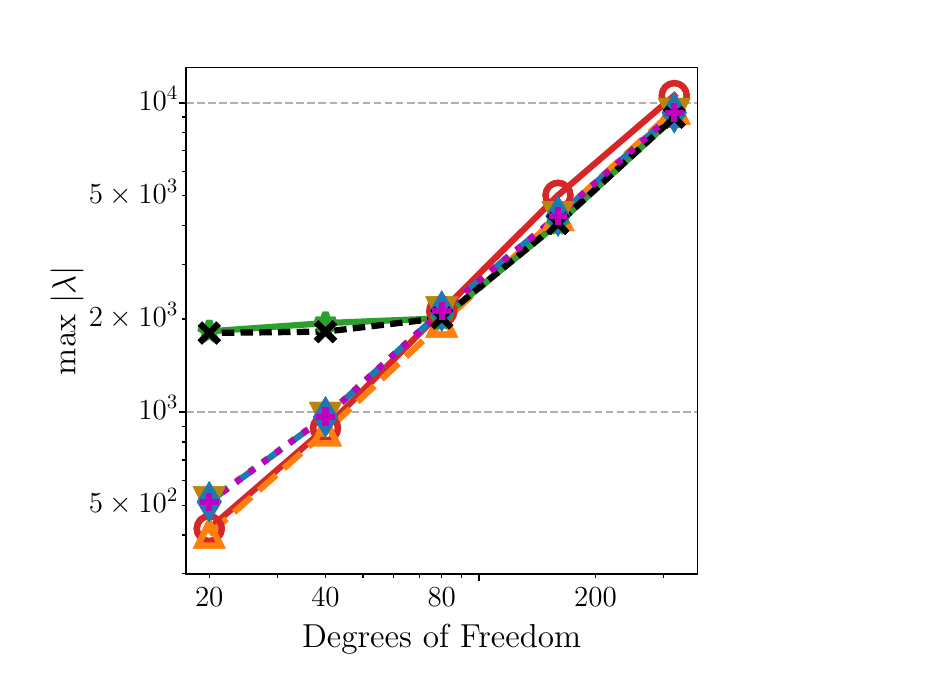}
    \end{subfigure}
    \hfill
    \begin{subfigure}[t]{0.32\textwidth}
        \centering
        \includegraphics[height=3.5cm, trim={27 10 35 25}, clip]{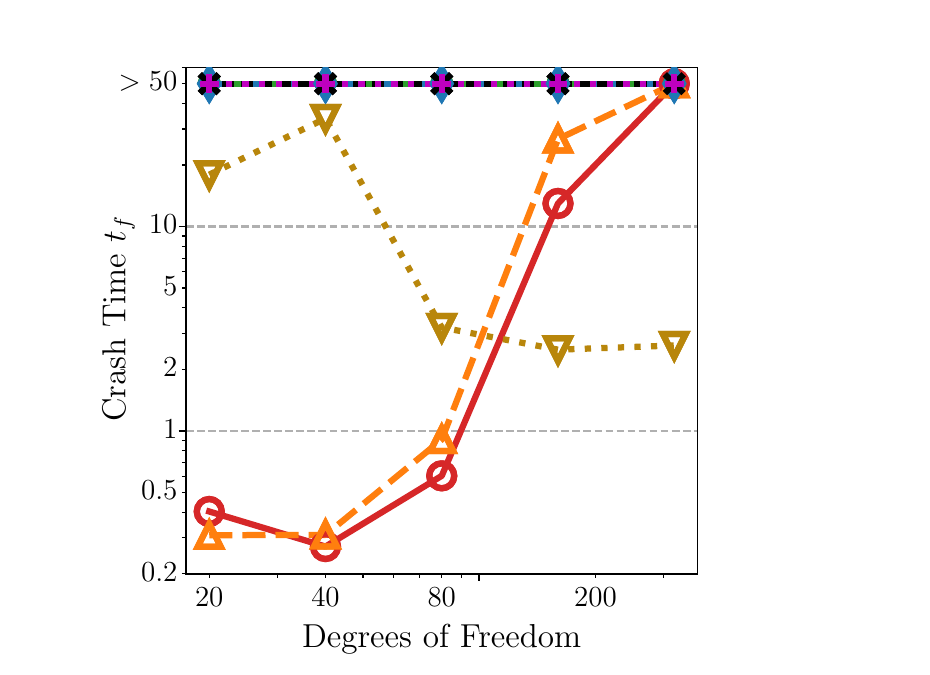}
    \end{subfigure}
    \caption{The maximum real part of the eigenvalues and spectral radius of the semi-discretization Jacobian, and crash times due to positivity errors for the 1D Euler density-wave problem using an UFD or entropy-dissipative scheme with a $p=4$ CSBP operator, entropy-dissipative SATs, and either matrix-matrix or scalar-matrix volume dissipation acting on entropy variables.}
    \label{fig:1deuler}
\end{figure}

\figref{fig:1deuler} shows the maximum real part of the eigenvalues and the spectral radius of the semi-discretization Jacobian linearized about the initial condition, along with the final crash time (up to $t_f=50$). On coarse grids, all discretizations---including those with volume dissipation---exhibit eigenvalues with large positive real parts. However, as noted in~\cite{Gassner2022}, grid refinement pushes these unphysical modes to higher frequencies.  Since volume dissipation removes high-frequency modes, and does so more effectively as the grid is refined, grid refinement sees the real parts of these eigenvalues converge to zero at a rate far surpassing the order of the discretization. Therefore, although entropy-dissipative schemes can initially exhibit unphysical linear stability behavior on coarse grids, volume dissipation ensures rapid convergence to the correct behavior under grid refinement, faster even than UFD schemes, without adversely impacting the spectral radius. Furthermore, while UFD and entropy-stable discretizations relying solely on interface dissipation often crash due to positivity violations, the addition of entropy-based volume dissipation reduces oscillations and maintains positivity of thermodynamic variables for long times. The smaller values of coefficient $\varepsilon$ are found to be at least as effective as---if not slightly more than--- the larger values in controlling local linear perturbation growth, while also maintaining a smaller spectral radius on coarse grids. Similar behavior is observed using different FD operators, entropy-conservative two-point fluxes, and interface SATs. In some cases, volume dissipation reduces the positive real parts of eigenvalues to machine precision even on relatively coarse grids. For SE operators, however, volume dissipation does not appear to be effective in accelerating convergence of the positive real parts of the eigenvalues to zero. We provide some discussion on this in Appendix \ref{sec:appendix_spectral}. Nevertheless, the dissipation acts to enhance robustness by maintaining positivity for long times.

\subsection{2D Euler Unsteady Isentropic Vortex} \label{sec:2dvortex}

To verify the accuracy of the dissipation operators in a nonlinear setting, we use the well-known two-dimensional isentropic vortex problem for the compressible Euler equations (see for example, \cite{spiegel2015}) with initial condition
\begin{gather*}
    \vec{u} = \begin{bmatrix}
        \rho \\ \rho \vec{v} \\ e
    \end{bmatrix} , \ \
    \vec{v} = M_\infty \begin{bmatrix}
        1 - \beta \frac{y}{R} e^{-r^2/2} \\
        \beta \frac{x}{R} e^{-r^2/2}
    \end{bmatrix} , \ \ 
    e = \frac{p}{\gamma - 1} + \frac{1}{2} \rho \norm{\vec{v}}^2 , \\
    \rho = \left( 1 - \frac{1}{2} \left( M_\infty \beta \right)^2 \left( \gamma - 1 \right) e^{-r^2} \right)^{1/\left( \gamma-1 \right)} , \ \ p = \frac{\rho^\gamma}{\gamma} , \ \
    r^2 = \left(\frac{x}{R}\right)^2 + \left(\frac{y}{R}\right)^2
\end{gather*}
where $\gamma = 1.4$, $M_\infty = 0.5$, $\beta = 0.2$, and $R=0.5$ on a periodic domain $(x,y) \in [-5,5]^2$. The solution is advanced in time for one period (until $t_f=20$). We employ two different schemes. The first is a central SBP scheme as described in \cite{hicken2008} that is not entropy-stable, augmented with Roe interface dissipation and either scalar or matrix variants of the volume dissipation \eqnref{eq:2d_diss} acting on conservative variables. The second is the entropy-stable discretization of \cite{Fisher_2013, Crean2018} using the Ranocha entropy-conserving two-point flux \cite{Ranocha2018} and entropy-dissipative interface dissipation of \cite{Winters2017}. To both schemes, we add one of four variants of the volume dissipation \eqnref{eq:2d_diss}: scalar or matrix dissipation acting on conservative variables, or scalar-matrix or matrix-matrix dissipation acting on entropy variables. In each direction we use 3 blocks of either $20^2$, $40^2$, $80^2$, and $160^2$ nodes such that the effect of block interfaces is included, ensuring that the error convergence rates are not portrayed in an overly favourable scenario. For comparison, we include comparable UFD and USE schemes using a modified variant of the Steger-Warming flux splitting due to Drikakis and Tsangaris~\cite{Drikakis} as described in~\cite{Ranocha2024}.

\begin{figure}[t] 
    \centering
        \begin{subfigure}[t]{0.32\textwidth}
        \centering
        \includegraphics[width=\textwidth, trim={5 10 5 5}, clip]{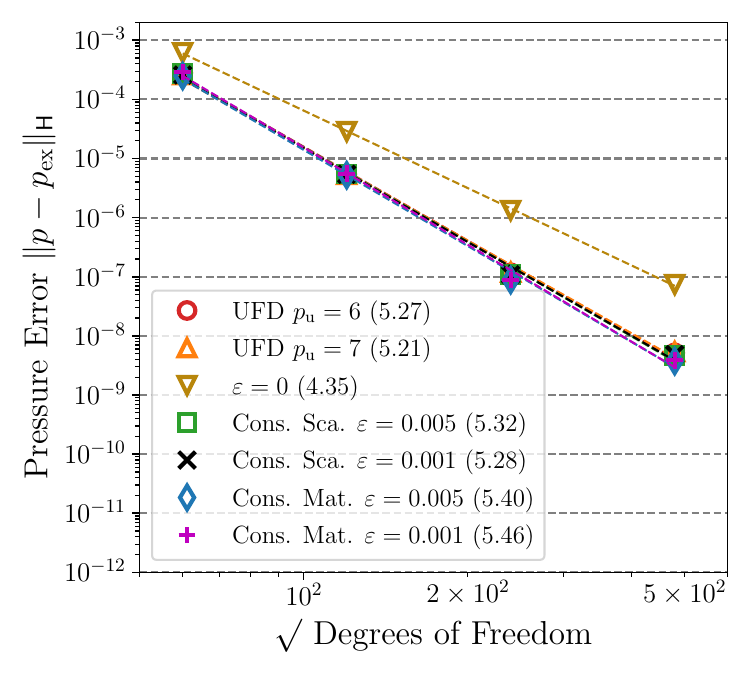}
        \caption{Central CSBP $p=3$ \\ \phantom{(a)} and UFD}
    \end{subfigure}
    \hfill
    \begin{subfigure}[t]{0.32\textwidth}
        \centering
        \includegraphics[width=\textwidth, trim={5 10 5 5}, clip]{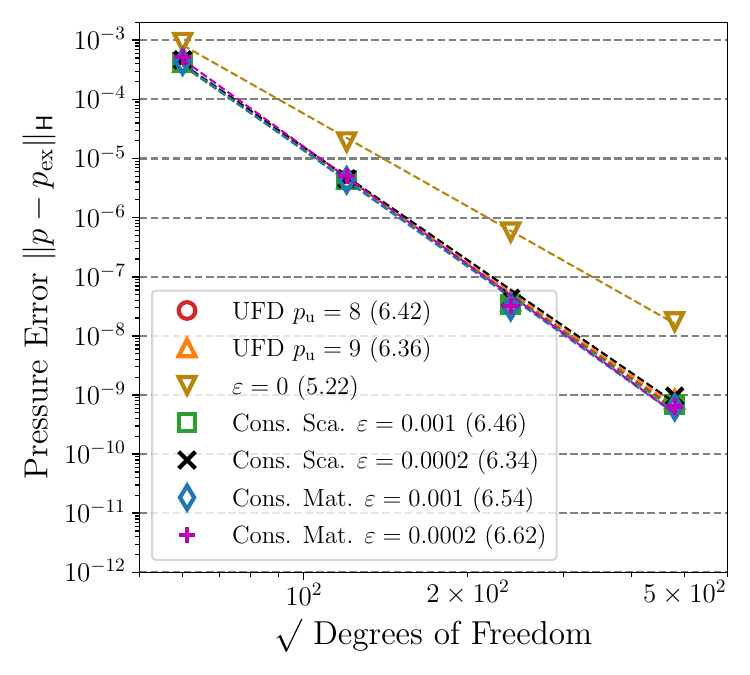}
        \caption{Central CSBP $p=4$ \\ \phantom{(b)} and UFD}
    \end{subfigure}
    \hfill
    \begin{subfigure}[t]{0.32\textwidth}
        \centering
        \includegraphics[width=\textwidth, trim={5 10 5 5}, clip]{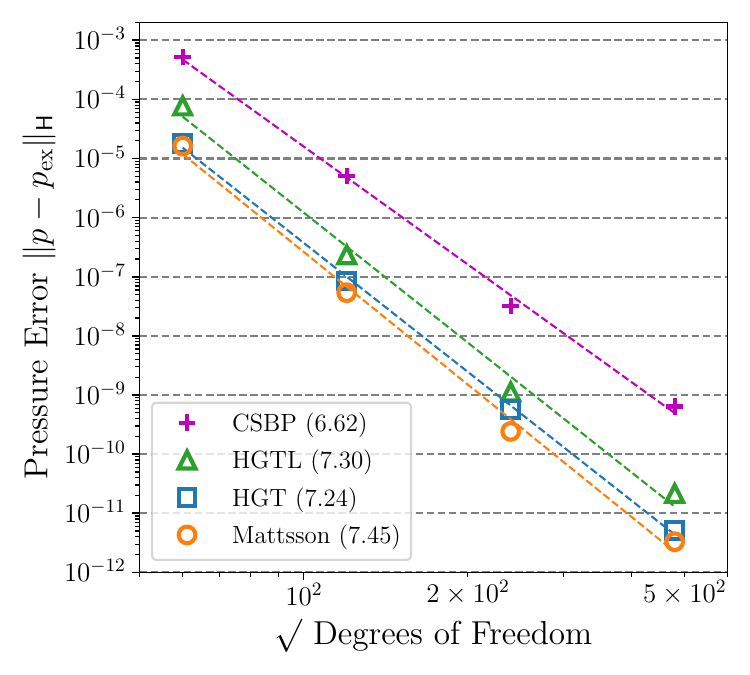}
        \caption{Central $p=4$ with \\ \phantom{(c)} Cons. Mat. $\varepsilon = 0.0002$}
    \end{subfigure}
    
    \vskip\baselineskip 

    \begin{subfigure}[t]{0.32\textwidth}
        \centering
        \includegraphics[width=\textwidth, trim={5 10 5 5}, clip]{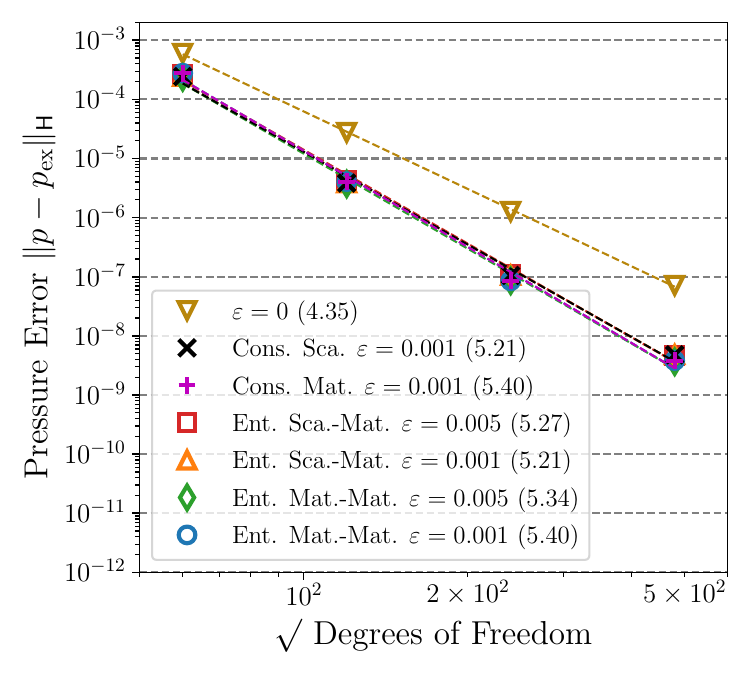}
        \caption{Entropy-Stable CSBP \\ \phantom{(d)} $p=3$}
    \end{subfigure}
    \hfill
    \begin{subfigure}[t]{0.32\textwidth}
        \centering
        \includegraphics[width=\textwidth, trim={5 10 5 5}, clip]{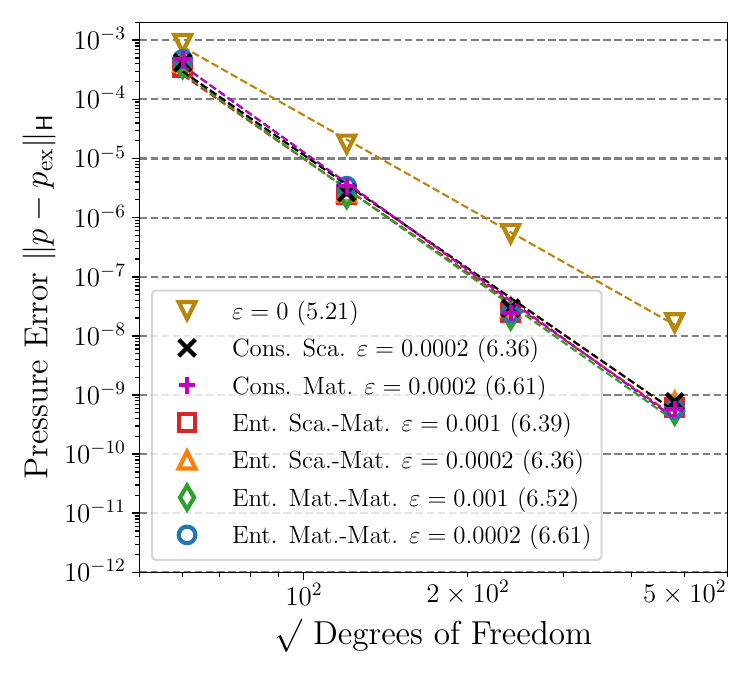}
        \caption{Entropy-Stable CSBP \\ \phantom{(e)} $p=4$}
    \end{subfigure}
    \hfill
    \begin{subfigure}[t]{0.32\textwidth}
        \centering
        \includegraphics[width=\textwidth, trim={5 10 5 5}, clip]{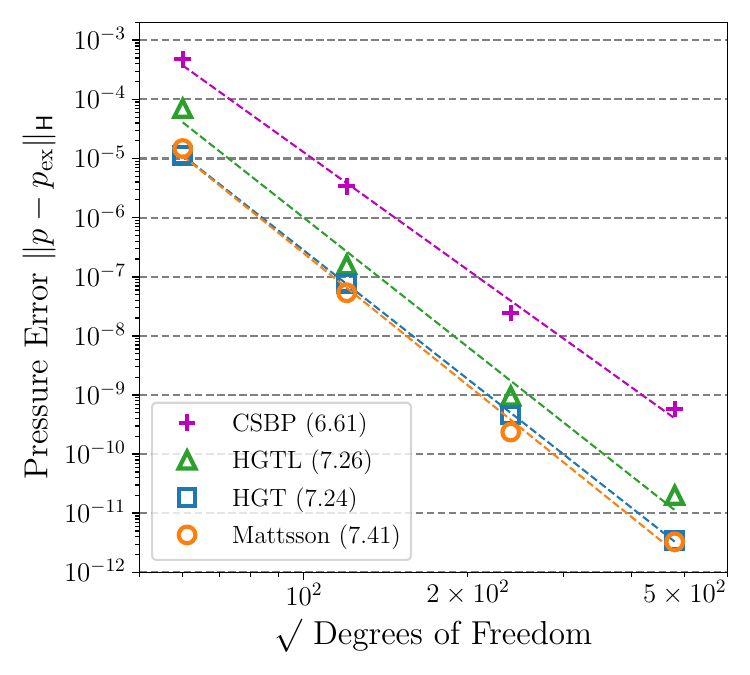}
        \caption{Entropy-Stable $p=4$ with \\ \phantom{(f} Ent. Mat.-Mat. $\varepsilon = 0.0002$}
    \end{subfigure}
    \caption{Pressure error after one period of the isentropic vortex problem for the 2D Euler equations using UFD~\cite{Mattsson2017}, central (\eg, \cite{hicken2008}), and entropy-stable SBP schemes of \cite{Fisher_2013, Crean2018}. The latter two are augmented with volume dissipation \eqnref{eq:2d_diss}, either as scalar or matrix formulations acting on conservative variables, or scalar-scalar or scalar-matrix formulations acting on entropy variables. Convergence rates are given in the legends.}
    \label{fig:Vortexpconv}
\end{figure}

The convergence of pressure errors is shown in Figure \ref{fig:Vortexpconv} for CSBP operators. The results are consistent with those observed for the linear convection equation: all discretizations with volume dissipation achieve convergence rates of~${\geq p+1.5}$, with the majority gaining one order of accuracy and exhibiting over an order of magnitude reduction in pressure error on finer grid levels, relative to discretizations without volume dissipation. Generally, there are negligible differences in accuracy between schemes augmented with different volume dissipations, be it scalar or matrix variants, central, entropy-stable, or even upwind discretizations. Additional figures in Appendix \ref{sec:additional_figures} demonstrate the same conclusions for HGTL, HGT, and Mattsson operators, except for a few cases where the scalar dissipation is seen to be overly-dissipative. Appendix \ref{sec:additional_figures} also includes results for LGL and LG operators demonstrating the expected SE convergence rates of~$p$. Similar results are observed for density and entropy errors. Comparisons in \figref{fig:Vortexpconv} once again demonstrate that using HGTL, HGT, or Mattsson operators can result in significant reductions in error as compared to using equispaced CSBP operators.

\subsection{2D Euler Kelvin-Helmholtz Instability} \label{sec:KelvinHelmholtz}

Finally, we investigate the robustness of the presented volume dissipation using the Kelvin-Helmholtz instability for the compressible Euler equations. We use the same initial condition as in \cite{Ranocha2024} and \cite{Glaubitz2024}, \ie
\begin{gather*}
    \rho = \frac{1}{2} + \frac{3}{4} \beta , \quad
    \vec{v} = \begin{bmatrix}
        \frac{1}{2} \beta - 1 \\
        \frac{1}{10} \sin\left(2 \pi x\right)
    \end{bmatrix} , \quad
    p = 1 ,  \\
    \beta\left(x,y\right) = \tanh\left(15y+7.5\right) - \tanh\left(15y-7.5\right),
\end{gather*}
on a periodic domain $(x,y) \in [-1,1]^2$. The solution is advanced in time until $t_f=15$ using the adaptive Dormand–Prince algorithm \cite{Hairer1993} with relative and absolute tolerances of $10^{-7}$. To verify that robustness benefits of including volume dissipation are independent of the choice of time-marching method and time step, we also performed simulations using the classical fourth-order Runge–Kutta method with a fixed time step set at $\mathrm{CFL}=1$. Although temporal accuracy suffers, the robustness benefits discussed below persist.

As before, we employ three different schemes: i) a central SBP scheme \cite{hicken2008} augmented with Roe interface dissipation and either scalar or matrix variants of the volume dissipation \eqnref{eq:2d_diss} acting on conservative variables, ii) the entropy-stable discretization of \cite{Fisher_2013, Crean2018} using the Ranocha entropy-conserving two-point flux \cite{Ranocha2018}, entropy-dissipative interface dissipation of \cite{Winters2017}, and either scalar-matrix or matrix-matrix dissipation acting on entropy variables, and iii) the UFD and USE schemes of \cite{Mattsson2017, Ranocha2024} and \cite{Glaubitz2024}, respectively, with Drikakis-Tsangaris flux splitting~\cite{Drikakis} and Roe interface dissipation. For FD discretizations, we use either one or three blocks in each direction, with between $20^2$ and $480^2$ nodes in each block. We match the total degrees of freedom across the two configurations; \eg one block of $60^2$ nodes matches $3^2$ blocks of $20^2$ nodes each. For SE discretizations, we use between 4 and 32 elements in each direction with a fixed $(p+1)^2$ nodes in each element. To verify that the results do not depend on the choice of matrix interface dissipation, we also include SE results with scalar (Rusanov) interface dissipation. The crash time is determined by either thermodynamic variables becoming non-positive, or the time step dropping below $10^{-10}$. A vanishingly small time step is a result of the semi-discretization Jacobian possessing extremely large eigenvalues, often with large positive real parts, which likewise occurs when thermodynamic variables approach zero.

\begin{figure}[t] 
    \centering
    \begin{subfigure}[t]{0.43\textwidth}
        \centering
        \includegraphics[height=5.3cm, trim={10 28 92 28}, clip]{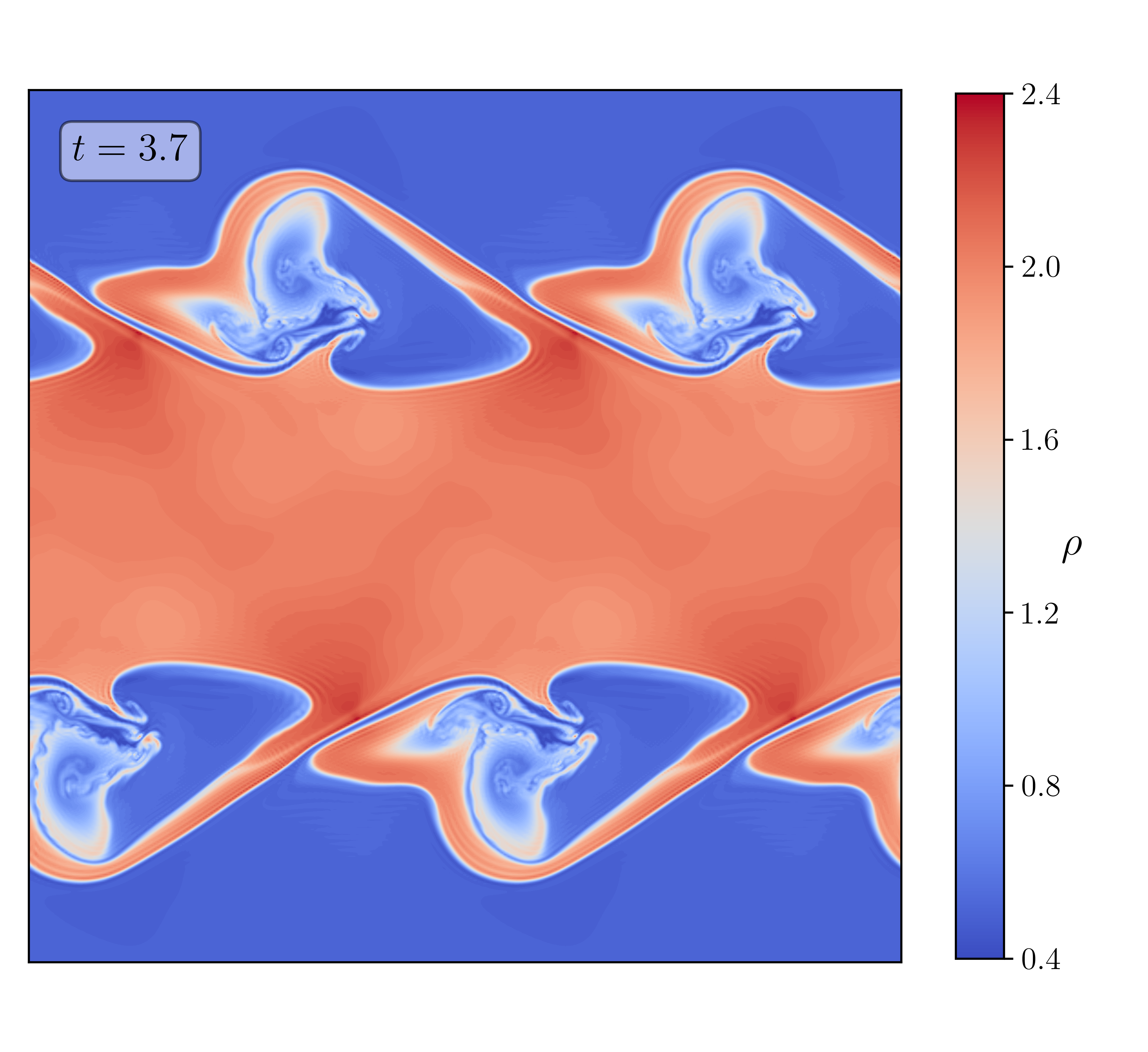}
    \end{subfigure}
    \hfill
    \begin{subfigure}[t]{0.56\textwidth}
        \centering
        \includegraphics[height=5.3cm, trim={10 28 25 28}, clip]{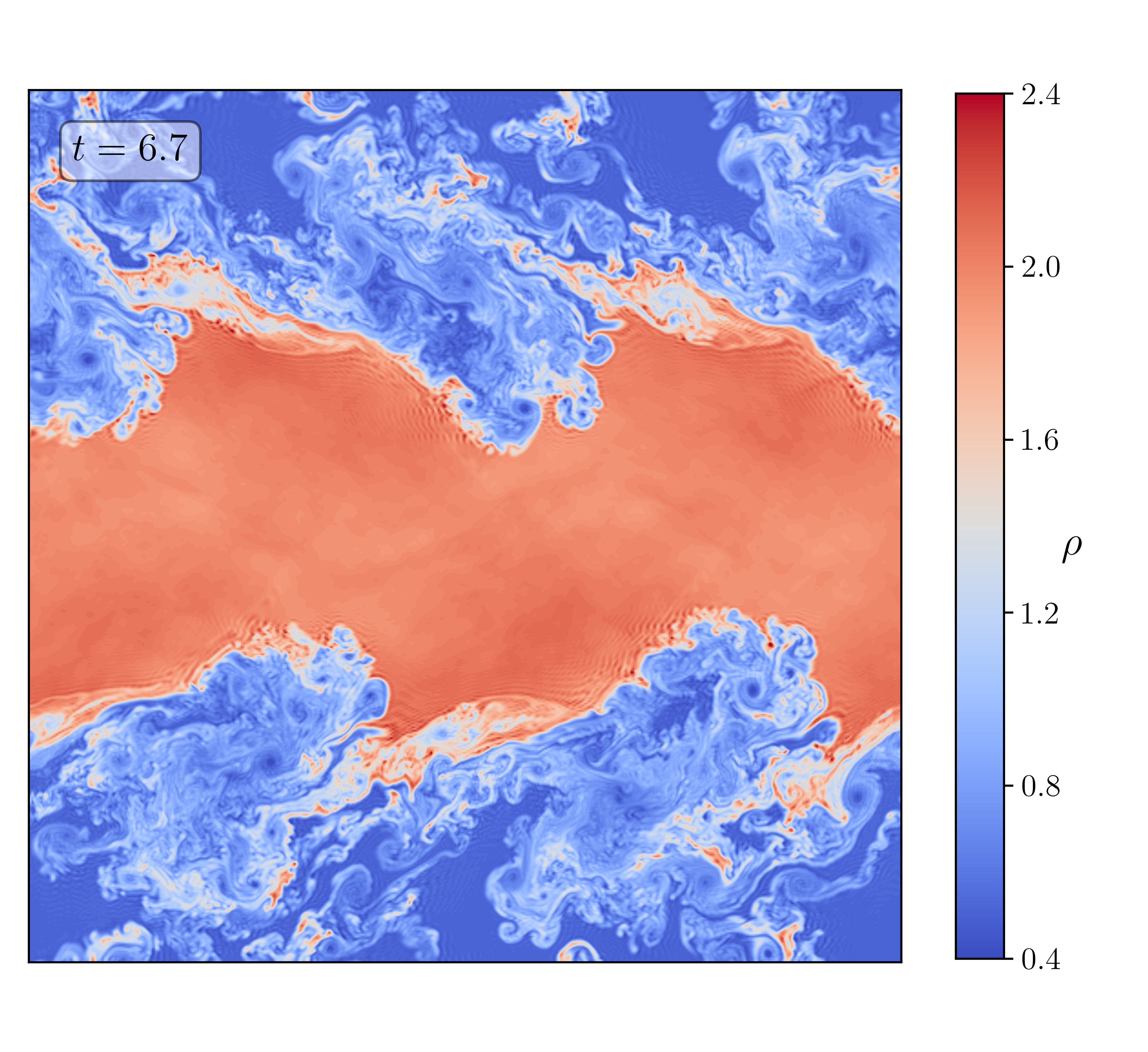}
    \end{subfigure}
    \caption{Density contours for the Kelvin-Helmholtz instability using the entropy-stable scheme of \cite{Fisher_2013, Crean2018} with a $p=4$ CSBP operator of $480^2$ nodes and matrix-matrix $s=p+1$ dissipation \eqnref{eq:2d_diss} acting on entropy variables with $\varepsilon = 0.0002$.}
    \label{fig:KHI}
\end{figure}

Figure~\ref{fig:KHI} shows density contours for the matrix-matrix low-dissipation case $\varepsilon = 0.625 \times 5^{-s}$ at $t=3.7$ and $t=6.7$, matching the times reported in~\cite{rueda_KHI_2022} for a $p=7$ entropy-stable discontinuous Galerkin scheme augmented with a positivity-preserving mechanism. Both the primary and secondary vortex structures observed in Figure~\ref{fig:KHI} are qualitatively similar to those in~\cite{rueda_KHI_2022}. Comparable results are also observed across different operators. However, fine structures appear more smoothed out at higher dissipation coefficients $\varepsilon = 3.125 \times 5^{-s}$ and with scalar-matrix dissipation, while spurious oscillations emerge at lower dissipation coefficients ($\varepsilon = 0.125 \times 5^{-s}$).

\begin{table}[!t]
\centering
\caption{Crash times (up to $t_f = 15$) for the Kelvin-Helmholtz instability using CSBP and UFD operators. Degrees of freedom (DOF) are given in terms of number of blocks times number of nodes per block.}
\renewcommand{\arraystretch}{1.2} 

\begin{subtable}{\textwidth}
\centering
\captionsetup{justification=centering}
\caption{UFD \cite{Mattsson2017} with Drikakis-Tsangaris flux splitting \cite{Drikakis}, as described in \cite{Ranocha2024}. \\
Recall that comparable central schemes use $p \sim \lfloor \frac{p_u}{2} \rfloor$, as detailed in Table \ref{tab:comparison}.}
\begin{tabular}{rcccccc}
\hline
DOF & $p_u=4$ & $p_u=5$ & $p_u=6$ & $p_u=7$ & $p_u=8$ & $p_u=9$ \\ \hline
$1 \times 30^2$ & 15.0 & 15.0 & 15.0 & 15.0 & 15.0 & 15.0
\\
$1 \times 60^2$ & 15.0 & 15.0 & 15.0 & 15.0 & 15.0 & {\color{BrickRed}14.97}
\\
$1 \times 120^2$ & 15.0 & 15.0 & 15.0 & {\color{BrickRed}6.02} & {\color{BrickRed}4.56} & {\color{BrickRed}4.52}
\\
$1 \times 240^2$ & 15.0 & 15.0 & {\color{BrickRed}8.24} & {\color{BrickRed}5.69} & {\color{BrickRed}4.69} & {\color{BrickRed}4.61}
\\
$1 \times 480^2$ & 15.0 & 15.0 & {\color{BrickRed}6.04} & {\color{BrickRed}4.62} & {\color{BrickRed}4.59} & {\color{BrickRed}4.59}
\\ \hline
$3^2 \times 20^2$ & 15.0 & {\color{BrickRed}5.37} &  {\color{BrickRed}4.24} &  {\color{BrickRed}4.08} &  {\color{BrickRed}5.10} &  {\color{BrickRed}9.78}
\\
$3^2 \times 40^2$ & {\color{BrickRed}5.26} & {\color{BrickRed}5.28} &  {\color{BrickRed}4.78} &  {\color{BrickRed}5.20} &  {\color{BrickRed}4.54} & {\color{BrickRed}3.36}
\\
$3^2 \times 80^2$ & 15.0 & {\color{BrickRed}3.91} & {\color{BrickRed}3.86} & {\color{BrickRed}3.85} & {\color{BrickRed}3.75} & {\color{BrickRed}3.68}
\\
$3^2 \times 160^2$ & 15.0 & {\color{BrickRed}3.87} & {\color{BrickRed}3.83} & {\color{BrickRed}3.81} & {\color{BrickRed}3.78} & {\color{BrickRed}3.72}
\\ \hline
\end{tabular}
\end{subtable}

\vskip\baselineskip 

\begin{subtable}[t]{0.48\textwidth}
        \centering
        \caption{Central CSBP, \\ \phantom{(a)} no volume dissipation ($\varepsilon=0$).}
        \begin{tabular}{rccc}
        \hline
DOF & $p=2$ & $p=3$ & $p=4$  \\ \hline
$1 \times 30^2$ & {\color{BrickRed}1.56} & {\color{BrickRed}1.50} & {\color{BrickRed}1.41} \\
$1 \times 60^2$ & {\color{BrickRed}1.26} & {\color{BrickRed}1.23} & {\color{BrickRed}1.23} \\
$1 \times 120^2$ & {\color{BrickRed}1.40} & {\color{BrickRed}1.43} & {\color{BrickRed}1.45} \\
$1 \times 240^2$ & {\color{BrickRed}1.89} & {\color{BrickRed}1.93} & {\color{BrickRed}1.94} \\
$1 \times 480^2$ & {\color{BrickRed}2.29} & {\color{BrickRed}2.83} & {\color{BrickRed}2.84} \\
\hline
$3^2 \times 20^2$ & {\color{BrickRed}1.19} & {\color{BrickRed}1.07} & {\color{BrickRed}1.16} \\
$3^2 \times 40^2$ & {\color{BrickRed}1.39} & {\color{BrickRed}1.42} & {\color{BrickRed}1.39} \\
$3^2 \times 80^2$ & {\color{BrickRed}1.88} & {\color{BrickRed}1.91} & {\color{BrickRed}1.91} \\
$3^2 \times 160^2$ & {\color{BrickRed}2.39} & {\color{BrickRed}2.82} & {\color{BrickRed}2.57} \\
\hline
        \end{tabular}
    \end{subtable}
    \hfill
    \begin{subtable}[t]{0.48\textwidth}
        \centering
        \caption{Entropy-stable CSBP,\\ \phantom{(b)} no volume dissipation ($\varepsilon=0$).}
        \begin{tabular}{rccc}
        \hline
DOF & $p=2$ & $p=3$ & $p=4$  \\ \hline
$1 \times 30^2$ & {\color{BrickRed}2.92} & {\color{BrickRed}3.50} & {\color{BrickRed}3.67} \\
$1 \times 60^2$ & {\color{BrickRed}3.43} & {\color{BrickRed}3.21} & {\color{BrickRed}3.04} \\
$1 \times 120^2$ & {\color{BrickRed}3.10} & {\color{BrickRed}3.18} & {\color{BrickRed}3.25} \\
$1 \times 240^2$ & {\color{BrickRed}3.61} & {\color{BrickRed}3.61} & {\color{BrickRed}3.63} \\
$1 \times 480^2$ & {\color{BrickRed}3.69} & {\color{BrickRed}3.66} & {\color{BrickRed}3.50} \\
\hline
$3^2 \times 20^2$ & {\color{BrickRed}2.92} & {\color{BrickRed}2.88} & {\color{BrickRed}2.31} \\
$3^2 \times 40^2$ & {\color{BrickRed}3.13} & {\color{BrickRed}2.95} & {\color{BrickRed}2.77} \\
$3^2 \times 80^2$ & {\color{BrickRed}3.66} & {\color{BrickRed}3.62} & {\color{BrickRed}3.56} \\
$3^2 \times 160^2$ & {\color{BrickRed}3.64} & {\color{BrickRed}3.67} & {\color{BrickRed}3.55} \\
\hline
        \end{tabular}
    \end{subtable}

\vskip\baselineskip 

\begin{subtable}[t]{0.48\textwidth}
        \centering
        \vspace{1\baselineskip}
        \caption{Central CSBP, matrix dissipation, \\ 
        \phantom{(d)} $s=p+1$, large $\varepsilon=3.125 \times 5^{-s}$.}
        \begin{tabular}{rccc}
        \hline
DOF & $p=2$ & $p=3$ & $p=4$  \\ \hline
$1 \times 30^2$ & 15.0 & 15.0 & {\color{BrickRed}7.58} \\
$1 \times 60^2$ & {\color{BrickRed}4.68} & {\color{BrickRed}3.75} & {\color{BrickRed}3.63} \\
$1 \times 120^2$ & {\color{BrickRed}4.50} & {\color{BrickRed}4.35} & {\color{BrickRed}4.16} \\
$1 \times 240^2$ & {\color{BrickRed}5.13} & {\color{BrickRed}3.61} & {\color{BrickRed}3.59} \\
$1 \times 480^2$ & {\color{BrickRed}4.57} & {\color{BrickRed}4.14} & {\color{BrickRed}3.53} \\
\hline
$3^2 \times 20^2$ & {\color{BrickRed}5.20} & {\color{BrickRed}1.35} & {\color{BrickRed}2.95} \\
$3^2 \times 40^2$ & {\color{BrickRed}4.48} & {\color{BrickRed}4.20} & {\color{BrickRed}3.15} \\
$3^2 \times 80^2$ & {\color{BrickRed}3.94} & {\color{BrickRed}3.58} & {\color{BrickRed}3.58} \\
$3^2 \times 160^2$ & {\color{BrickRed}3.87} & {\color{BrickRed}3.71} & {\color{BrickRed}3.30} \\
\hline
        \end{tabular}
    \end{subtable}
    \hfill
    \begin{subtable}[t]{0.48\textwidth}
        \centering
        \caption{Entropy-stable CSBP, matrix-matrix or \\ 
        \phantom{(e)} scalar-matrix dissipation, $s=p+1$, \\ \phantom{(e)} $\varepsilon=3.125 \times 5^{-s}$ or $\varepsilon=0.625 \times 5^{-s}$.}
        \begin{tabular}{rccc}
        \hline
DOF & $p=2$ & $p=3$ & $p=4$  \\ \hline
$1 \times 30^2$ & 15.0 & 15.0 & 15.0 \\
$1 \times 60^2$ & 15.0 & 15.0 & 15.0 \\
$1 \times 120^2$ & 15.0 & 15.0 & 15.0 \\
$1 \times 240^2$ & 15.0 & 15.0 & 15.0 \\
$1 \times 480^2$ & 15.0 & 15.0 & 15.0 \\
\hline
$3^2 \times 20^2$ & 15.0 & 15.0 & 15.0 \\
$3^2 \times 40^2$ & 15.0 & 15.0 & 15.0 \\
$3^2 \times 80^2$ & 15.0 & 15.0 & 15.0 \\
$3^2 \times 160^2$ & 15.0 & 15.0 & 15.0 \\
\hline
        \end{tabular}
    \end{subtable}

\vskip\baselineskip 
\vskip\baselineskip 

\label{tab:KelvinHelmholtz}
\end{table}

The final crash times for UFD and CSBP schemes are reported in Table~\ref{tab:KelvinHelmholtz}. Among all methods tested, the central discretization without volume dissipation is the least robust, consistently crashing between $t=1$ and $t=2$. Adding volume dissipation offers modest improvement, increasing crash times to between $t=3$ and $t=5$. Additional results in Appendix~\ref{sec:additional_figures} show that employing scalar dissipation and larger $\varepsilon$ yield further robustness improvements. A more substantial improvement, however, is achieved by switching to a UFD discretization, with a significantly higher proportion of runs reaching $t = 15$. The reason for this enhanced robustness warrants further investigation. Regardless, consistent with~\cite{Ranocha2024}, UFD robustness eventually degrades with increasing polynomial degree and grid refinement. 

The entropy-stable scheme without volume dissipation, while more robust than its central counterpart, also crashes between $t=2$ and $t=4$. Adding volume dissipation applied to entropy variables, however, results in schemes that consistently complete the simulation to $t=15$. All entropy-dissipative variants tested---both matrix-matrix and scalar-matrix variants, with small and large values of $\varepsilon$ (and even values at a lower one twenty-fifth of the large coefficient)---exhibited this high level of robustness, monotonically dissipating entropy and maintaining positivity of thermodynamic variables. This also holds true for HGTL and Mattsson operators; as detailed in Appendix~\ref{sec:additional_figures}, the entropy-stable scheme with entropy-based volume dissipation reaches $t = 15$ for all tested cases. However, HGT simulations do not exhibit this same robustness. Although the entropy-dissipative discretizations still exhibit improved robustness over central discretizations, they crash for increasing polynomial degree and grid refinement. Examining the solutions at crash times reveals that thermodynamic variables at solution nodes remain positive, well above zero. However, extrapolations of the solution to the block boundaries result in interface solutions with near-zero thermodynamic variables. Since these unphysical interface solutions are used in SAT calculations, the semi-discrete Jacobian becomes extremely stiff despite the solution maintaining smoothness and positive thermodynamic variables everywhere in the interior.

\begin{table}[!t]
\centering
\caption{Crash times (up to $t_f = 15$) for the Kelvin-Helmholtz instability using LGL operators. Degrees of freedom are calculated as ${\text{DOF} = K^2 (p+1)^2}$, where $K$ is the number of elements per direction. The dissipation coefficients are $\varepsilon_1 = [0.01,0.004,0.002,0.0008,0.0004,0.0002]$ for $p=[3\text{--}8]$, respectively, and $\varepsilon_{1/5}$ as one fifth of these values. Results with matrix and Rusanov scalar interface dissipation are shown left/right, respectively.
}
\renewcommand{\arraystretch}{1.2} 

\begin{subtable}{\textwidth}
\centering
\captionsetup{justification=centering}
\caption{USE-LGL with Drikakis–Tsangaris flux splitting \cite{Drikakis}. \\ Recall that $s = p$ in this work corresponds to $p_u = p - 1$ in \cite{Glaubitz2024}.}
\begin{tabular}{cccccccc}
\hline
$K$ & $\varepsilon$ & $p=3$ & $p=4$ & $p=5$ & $p=6$ & $p=7$ & $p=8$ \\ \hline
4 & $\varepsilon_1$ & 15.0/15.0 & {\color{BrickRed}2.63}/{\color{BrickRed}2.23} & {\color{BrickRed}1.71}/{\color{BrickRed}1.54} & {\color{BrickRed}2.60}/{\color{BrickRed}2.24} & {\color{BrickRed}2.11}/{\color{BrickRed}2.05} & {\color{BrickRed}1.86}/{\color{BrickRed}2.47}
\\
8 & $\varepsilon_1$ & 15.0/15.0 & 15.0/{\color{BrickRed}8.50} & {\color{BrickRed}2.49}/{\color{BrickRed}2.45} & {\color{BrickRed}2.25}/{\color{BrickRed}3.16} & {\color{BrickRed}2.25}/{\color{BrickRed}2.24} & {\color{BrickRed}1.87}/{\color{BrickRed}1.79}
\\
16 & $\varepsilon_1$ & {\color{BrickRed}4.69}/{\color{BrickRed}3.83} & {\color{BrickRed}3.94}/{\color{BrickRed}2.78} & {\color{BrickRed}2.45}/{\color{BrickRed}2.45} & {\color{BrickRed}2.52}/{\color{BrickRed}2.56} & {\color{BrickRed}2.19}/{\color{BrickRed}1.85} & {\color{BrickRed}2.23}/{\color{BrickRed}1.97}
\\
32 & $\varepsilon_1$ & {\color{BrickRed}3.18}/{\color{BrickRed}3.27} & {\color{BrickRed}3.23}/{\color{BrickRed}3.54} & {\color{BrickRed}2.34}/{\color{BrickRed}2.46} & {\color{BrickRed}2.37}/{\color{BrickRed}2.37} & {\color{BrickRed}2.41}/{\color{BrickRed}2.39} & {\color{BrickRed}2.17}/{\color{BrickRed}2.22}
\\ \hline
4 & $\varepsilon_{1/5}$ & 15.0/15.0 & {\color{BrickRed}2.93}/{\color{BrickRed}1.57} & {\color{BrickRed}3.63}/{\color{BrickRed}2.00} & {\color{BrickRed}4.23}/{\color{BrickRed}4.22} & {\color{BrickRed}1.91}/{\color{BrickRed}1.84} & {\color{BrickRed}2.17}/{\color{BrickRed}2.41}
\\
8 & $\varepsilon_{1/5}$ & {\color{BrickRed}1.82}/{\color{BrickRed}1.70} & {\color{BrickRed}3.48}/{\color{BrickRed}3.06} & {\color{BrickRed}4.22}/{\color{BrickRed}2.86} & {\color{BrickRed}1.54}/{\color{BrickRed}1.61} & {\color{BrickRed}2.88}/{\color{BrickRed}3.22} & {\color{BrickRed}3.69}/{\color{BrickRed}2.99}
\\
16 & $\varepsilon_{1/5}$ & {\color{BrickRed}4.31}/{\color{BrickRed}4.19} & {\color{BrickRed}3.46}/{\color{BrickRed}3.47} & {\color{BrickRed}4.05}/{\color{BrickRed}3.61} & {\color{BrickRed}3.45}/{\color{BrickRed}3.14} & {\color{BrickRed}3.58}/{\color{BrickRed}3.23} & {\color{BrickRed}3.14}/{\color{BrickRed}2.56}
\\
32 & $\varepsilon_{1/5}$ & {\color{BrickRed}3.42}/{\color{BrickRed}3.50} & {\color{BrickRed}3.58}/{\color{BrickRed}4.13} & {\color{BrickRed}3.62}/{\color{BrickRed}3.54} & {\color{BrickRed}3.36}/{\color{BrickRed}3.57} & {\color{BrickRed}3.49}/{\color{BrickRed}3.31} & {\color{BrickRed}3.20}/{\color{BrickRed}3.16}
\\ \hline
\end{tabular}
\end{subtable}

\vskip\baselineskip 

\begin{subtable}{\textwidth}
\centering
\captionsetup{justification=centering}
\caption{Entropy-stable (often called `DGSEM'), no volume dissipation ($\varepsilon=0$).}
\begin{tabular}{cccccccc}
\hline
$K$ & $p=3$ & $p=4$ & $p=5$ & $p=6$ & $p=7$ & $p=8$ \\ \hline
4 & 15.0/{\color{BrickRed}4.46} & 15.0/{\color{BrickRed}2.47} & 15.0/{\color{BrickRed}3.01} & 15.0/{\color{BrickRed}2.80} & 15.0/{\color{BrickRed}3.63} & 15.0/{\color{BrickRed}4.13}
\\
8 & 15.0/{\color{BrickRed}1.53} & 15.0/{\color{BrickRed}4.04} & {\color{BrickRed}5.79}/{\color{BrickRed}3.70} & {\color{BrickRed}4.20}/{\color{BrickRed}4.10} & 15.0/{\color{BrickRed}3.56} & {\color{BrickRed}4.25}/{\color{BrickRed}3.66}
\\
16  & 15.0/{\color{BrickRed}3.77} & 15.0/{\color{BrickRed}4.44} & {\color{BrickRed}5.79}/{\color{BrickRed}3.74} & 15.0/{\color{BrickRed}3.37} & {\color{BrickRed}5.01}/{\color{BrickRed}3.64} & {\color{BrickRed}4.95}/{\color{BrickRed}3.83}
\\
32 & {\color{BrickRed}3.52}/{\color{BrickRed}3.66} & 15.0/{\color{BrickRed}4.27} & {\color{BrickRed}4.80}/{\color{BrickRed}3.54} & 15.0/{\color{BrickRed}3.66} & {\color{BrickRed}5.95}/{\color{BrickRed}3.56} & {\color{BrickRed}4.44}/{\color{BrickRed}3.50}
\\ \hline
\end{tabular}
\end{subtable}

\vskip\baselineskip 

\begin{subtable}{\textwidth}
\centering
\captionsetup{justification=centering}
\caption{Entropy-stable (DGSEM), matrix-matrix or scalar-matrix dissipation, $s=p$, scalar Rusanov or matrix interface dissipation of \cite{Winters2017}.}
\begin{tabular}{cccccccc}
\hline
$K$ & $\varepsilon$ & $p=3$ & $p=4$ & $p=5$ & $p=6$ & $p=7$ & $p=8$ \\ \hline
4 & $\varepsilon_1$ & 15.0 & 15.0 & 15.0 & 15.0 & 15.0 & 15.0
\\
8 & $\varepsilon_1$ & 15.0 & 15.0 & 15.0 & 15.0 & 15.0 & 15.0
\\
16 & $\varepsilon_1$ & 15.0 & 15.0 & 15.0 & 15.0 & 15.0 & 15.0
\\
32 & $\varepsilon_1$ & 15.0 & 15.0 & 15.0 & 15.0 & 15.0 & 15.0
\\ \hline
4 & $\varepsilon_{1/5}$ & 15.0 & 15.0 & 15.0 & 15.0 & 15.0 & 15.0
\\
8 & $\varepsilon_{1/5}$ & 15.0 & 15.0 & 15.0 & 15.0 & 15.0 & 15.0
\\
16 & $\varepsilon_{1/5}$ & 15.0 & 15.0 & 15.0 & 15.0 & 15.0 & 15.0
\\
32 & $\varepsilon_{1/5}$ & 15.0 & 15.0 & 15.0 & 15.0 & 15.0 & 15.0
\\ \hline
\end{tabular}
\end{subtable}

\label{tab:KelvinHelmholtz2}
\end{table}

Similar results are observed for spectral-element discretizations. As reported in Table \ref{tab:KelvinHelmholtz2} and consistent with~\cite{Glaubitz2024}, the USE and the entropy-stable scheme without volume dissipation exhibit comparable robustness, with simulations frequently crashing between $t = 2$ and $t = 4$. Robustness improves substantially when the matrix-based interface dissipation of~\cite{Winters2017} is used for entropy-stable formulations, enabling a significant fraction of simulations to reach the final time of $t = 15$. As in the finite-difference case though, entropy-stable schemes augmented with entropy-based volume dissipation demonstrate exceptional robustness: for LGL operators, both matrix-matrix and scalar-matrix dissipation, and for both small and large values of $\varepsilon$, all simulations completed successfully, ensuring both long-time stability and positivity of thermodynamic variables. As detailed in Appendix~\ref{sec:additional_figures}, however, the LG operators are less robust. Although entropy-dissipative LG discretizations still exhibit improved robustness over both central and USE discretizations, they crash for increasing polynomial degree and number of elements. Similar to HGT operators, solutions at the final crash times possess positive thermodynamic variables at all solution nodes, but at intermediate interpolation points or at extrapolated interface nodes, solutions can contain become unphysical.  

In~\cite{chan_entropyprojection}, entropy-stable schemes employing an entropy projection demonstrated improved robustness over collocation schemes (such as those employed in this work). Although the mechanism behind the improved robustness was not definitively explained, they highlighted that as the flow approaches a vacuum state, the entropy approaches non-convexity. A similar explanation could explain the positivity results of \S\ref{sec:1deuler} and \S\ref{sec:KelvinHelmholtz}, in that higher-order derivatives of the entropy variables become large in regions of low density, leading to increased volume dissipation, and better positivity preservation. Our results therefore may offer additional insight into this robustness difference, though further work is needed to clarify the roles of entropy projection and artificial dissipation in the robustness of entropy-stable schemes.


\section{Conclusions}

We have shown how to construct volume dissipation operators for finite-difference and spectral-element methods with the generalized SBP property such that they are design-order accurate, dimensionally consistent, and preserve energy or entropy stability when used with energy- or entropy-stable schemes. The dissipation operators were proven to be stable, conservative, design-order accurate, and free-stream preserving. An extension to multiple dimensions was performed using a tensor-product formulation. Several new artificial dissipation operators were then constructed for use with HGTL, HGT, Mattsson, LGL, and LG operators. We have clarified the inclusion of the variable coefficient, noting that averaging the variable coefficient at half-nodes is required for odd-degree $\tilde{\mat{D}}_s$ operators to ensure that the dissipation operator is not directionally biased. Furthermore, we have shown that different choices for the boundary correction matrix $\mat{B}$ can affect both the accuracy and amount of dissipation applied at block/domain boundaries.

We also connected the constructed dissipation operators to other approaches for including volume dissipation in the literature. In contrast to upwind SBP operators implemented with flux-vector splitting, the present approach preserves the nonlinear stability and free-stream preservation properties of the underlying scheme while introducing an adjustable mechanism to control local linear instabilities and suppress spurious high-order modes. When applied to linear problems in a spectral-element context, we have shown that the present approach reproduces volume dissipation operators that are equivalent to other more involved approaches leveraging orthogonal polynomials. 

Numerical examples have verified the nonlinear stability, conservation, and accuracy properties of the presented dissipation operators using a suite of test problems involving the linear convection, Burgers, and Euler equations. The results suggest that volume dissipation can be beneficial even when it is not required for stabilization purposes. For finite-difference operators, it was shown that adding volume dissipation can significantly increase the accuracy of the resulting discretizations. For the Burgers equation, we have also shown that entropy-stable discretizations augmented with a small amount of volume dissipation are able to maintain local linear stability. In general, while adding volume dissipation does not guarantee local linear stability, it extends the range of problems for which entropy-stable discretizations exhibit this property, and accelerates convergence towards it under finite-difference grid refinement. Finally, we have shown that volume dissipation greatly enhances robustness and helps maintain the positivity of thermodynamic variables in under-resolved flows. We do not claim our volume dissipation to be a substitute for positivity-preserving mechanisms; rather, it broadens the class of problems that entropy-stable methods can solve without requiring such mechanisms, and may reduce reliance on them in more challenging scenarios.

Future work can involve the optimization of the operators presented in this work. This can be achieved by adding more degrees of freedom of the $\tilde{\mat{D}}_s$ matrices near the boundaries, or investigating optimal ways to average the variable coefficient $\mat{A}$ at half-nodes. Strategies to localize the volume dissipation by modifying the structure of $\mat{B}$, similar to~\cite{Kord2023}, could also be investigated. In terms of applications, future work could focus on how the presented dissipation influences both convergence to steady state and the accuracy of functional outputs. Application of the dissipation operators to the residual correction schemes of~\cite{Abgrall_correction} through the framework established in~\cite{Edoh_correction} may also be of interest. Finally, the combination of volume dissipation with appropriate shock capturing and positivity-preserving mechanisms should be investigated.

\begin{acknowledgements}
The authors would like to thank Professor Masayuki Yano for an insightful discussion on the accuracy of spectral-element volume dissipation.
\end{acknowledgements}

\section*{Statements and declarations}

\subsection*{Funding}
This research was supported by the Natural Sciences and Engineering Research Council of Canada, the Government of Ontario, and the University of Toronto.


\subsection*{Conflict of interest}
The authors declare that they have no conflict of interest.

\subsection*{Data Availability}
The supplementary material and scripts to reproduce all of the numerical results are accessible via the open-source repository \url{https://github.com/alexbercik/ESSBP/tree/main/Paper-StableVolumeDissipation}. 

\appendix

\section{Some Comments Regarding the Performance of Spectral-Element Volume Dissipation} \label{sec:appendix_spectral}

The question of whether or not volume dissipation should be incorporated in SE schemes is nuanced, given that interface dissipation is often sufficient in the absence of strong shocks \cite{chen_review}. With the recent introduction of USE SBP schemes in \cite{Glaubitz2024}, however, this question has received more attention.

We first discuss the order of accuracy in the context of the linear convection equation. We use the same SBP discretizations and Gaussian initial condition \eqnref{eq:gaussian} described in \S \ref{sec:LCE}, augmented with dissipation ${\mat{A}_\mat{D} = -\varepsilon \HI \tilde{\mat{D}}_s^\T \tilde{\mat{D}}_s}$ and $s=p$. Example results which verify the expected convergence rates of approximately $p$ are shown in \figref{fig:LCEconv_spec}. Consistent with \cite{Glaubitz2024}, we observe an improvement in solution error when compared to baseline discretizations of equal degree. However, this is potentially misleading, because by simply `turning off' the dissipation one can obtain even more accurate solutions. \figref{fig:LCEsolerr_LGL} compares the the difference between the numerical and exact solution after one period using the initial condition
\begin{equation} \label{eq:sinusoid2}
    \fnc{U}\left(x,0\right) = \sin{\left( 2 \pi x \right)} , \quad x \in [0,1].
\end{equation}
Once again, the solution accuracy of the discretizations with volume dissipation falls somewhere between the higher and equal degree baseline discretizations. Experiments performed with the Burgers equation likewise show that SE volume dissipation constructed with the current framework is not any more effective in improving solution accuracy than schemes relying on interface dissipation alone, even after the formation of shocks.

\begin{figure}[!t] 
    \centering
    \begin{subfigure}[t]{0.32\textwidth}
        \centering
        \includegraphics[width=\textwidth, trim={5 10 5 5}, clip]{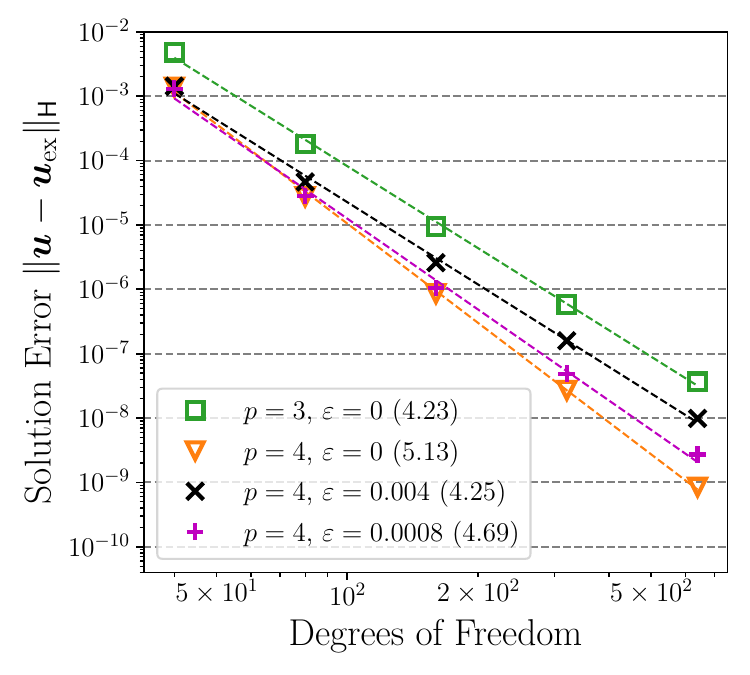}
        \caption{LGL $p=4$}
    \end{subfigure}
    \hfill
    \begin{subfigure}[t]{0.32\textwidth}
        \centering
        \includegraphics[width=\textwidth, trim={5 10 5 5}, clip]{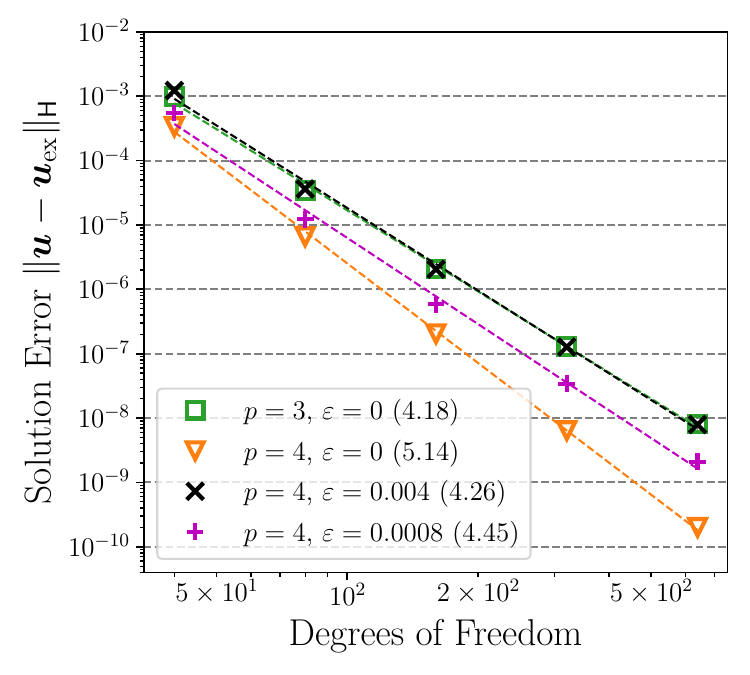}
        \caption{LG $p=4$}
    \end{subfigure}
    \hfill
    \begin{subfigure}[t]{0.32\textwidth}
        \centering
        \includegraphics[width=\textwidth, trim={5 10 5 5}, clip]{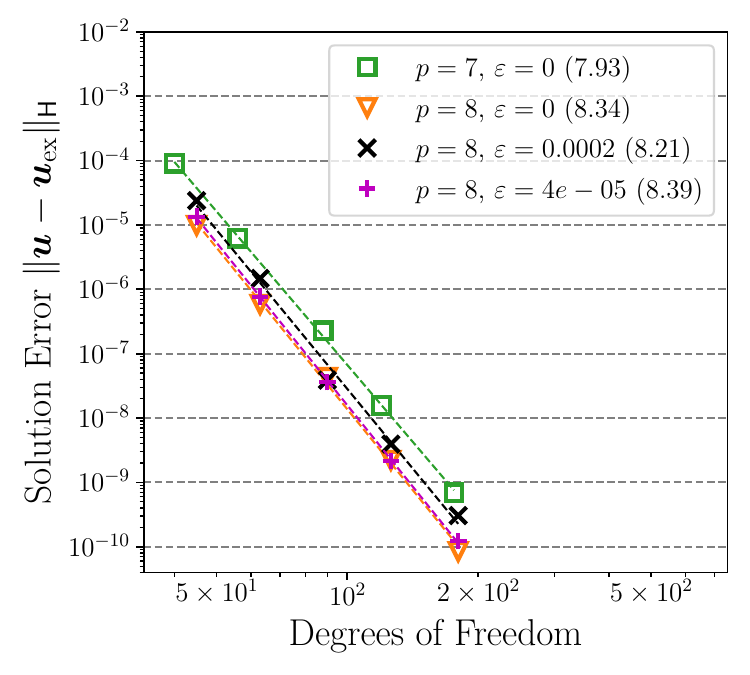}
        \caption{LGL $p=8$}
    \end{subfigure}
    \caption{Solution error convergence after one period of the linear convection equation for SE SBP schemes with upwind SATs, initial condition~\eqnref{eq:gaussian}, and dissipation $\mat{A}_\mat{D} = -\varepsilon \HI \tilde{\mat{D}}_s^\T \tilde{\mat{D}}_s$ with $s=p$. Convergence rates are given in the legends.}
    \label{fig:LCEconv_spec}

\bigskip

    \centering
    \begin{subfigure}[t]{0.48\textwidth}
        \centering
        \includegraphics[width=\textwidth, trim={10 13 10 15}, clip]{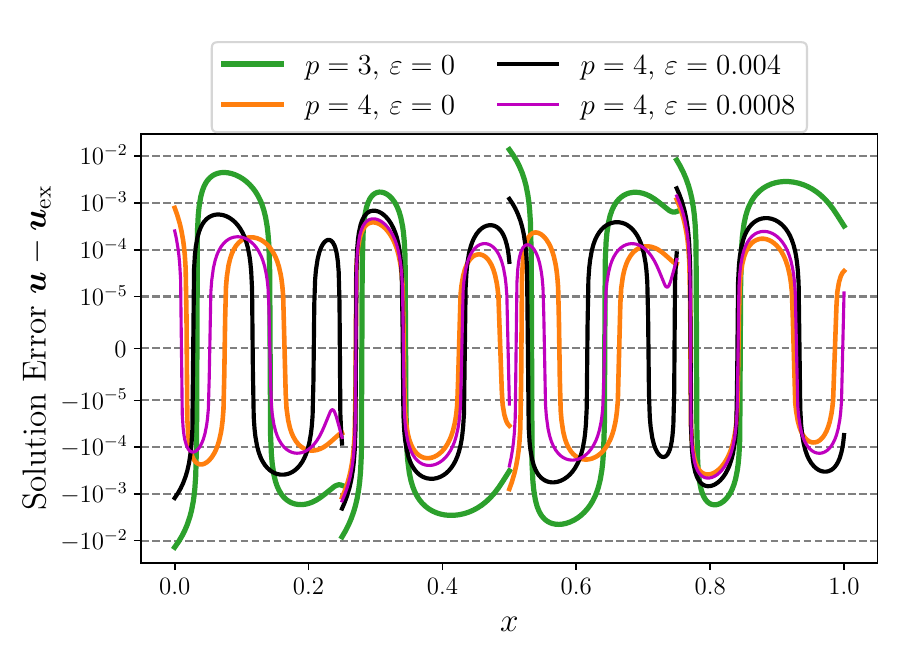}
        \caption{LGL $p=4$, 4 elements}
    \end{subfigure}
    \hfill
    \begin{subfigure}[t]{0.48\textwidth}
        \centering
        \includegraphics[width=\textwidth, trim={10 13 10 15}, clip]{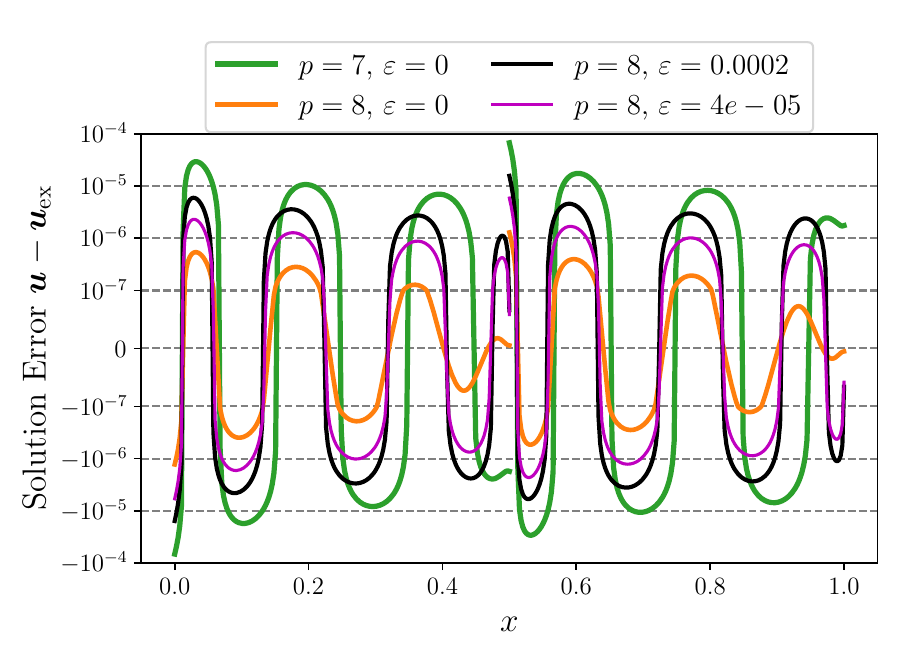}
        \caption{LGL $p=8$, 2 elements}
    \end{subfigure}
    \caption{Solution error after one period of the linear convection equation for SE SBP schemes with upwind SATs, initial condition~\eqnref{eq:sinusoid2}, and dissipation $\mat{A}_\mat{D} = -\varepsilon \HI \tilde{\mat{D}}_s^\T \tilde{\mat{D}}_s$ with $s=p$.}
    \label{fig:LCEsolerr_LGL}
\end{figure}

To understand why this is the case, it is useful to express $\mat{A}_\mat{D}$ as follows. The accuracy conditions \eqnref{eq:accuracy_conditions} coupled with a minimum number of $N=p+1$ nodes lead to $\tilde{\mat{D}}_s$ matrices with repeating rows, implying $\tilde{\mat{D}}_s = \vec{1} \vec{v}^\T$, where $\vec{v}$ contains the finite-difference stencil of $\tilde{\mat{D}}_s$. Subsequently,
\begin{equation}
    \mat{A}_\mat{D} = - \varepsilon \HI \tilde{\mat{D}}_s^\T \tilde{\mat{D}}_s = - \varepsilon N \HI \vec{v} \vec{v}^\T .
\end{equation}
$\mat{A}_\mat{D}$, with column space of dimension 1, always maps vectors to $\text{span}\left( \HI \vec{v} \right)$, which by using nullspace arguments, is orthogonal to degree $\leq p-1$ polynomials with respect to the discrete integration $\Hnrm$. Therefore, the only mechanism by which $\mat{A}_\mat{D}$ can decrease energy is by removing some amount of the highest-degree orthogonal polynomial $\HI \vec{v}$ from the solution $\vec{u}$. The current volume dissipation framework is therefore unlikely to lead to a more \emph{accurate} solution when applied to SE schemes. In some respect, the oscillations observed in \figref{fig:LCEsolerr_LGL} for the baseline discretizations are optimal in the $\Hnrm$-norm, because they are the result of finding the best possible polynomial approximation of $\fnc{U}$. Therefore, it is reasonable to conjecture that removing information from the highest-order mode will negatively impact the accuracy of the approximation. Furthermore, the absence of a repeating interior stencil likely precludes the accuracy benefits observed in~\S\ref{sec:LCE} in block interiors of FD operators.

Consider as another example the local linear stability of the 1D Euler density-wave problem discussed in~\S\ref{sec:1deuler}. Repeating the same experiment with LGL or LG operators shows that entropy-stable schemes augmented with volume dissipation do not have eigenvalues with smaller positive real parts than those relying on interface dissipation alone. Furthermore, the maxmimum real part does not reliably converge to zero as more elements are added. We hypothesize that because the dissipation operator only acts on the highest-order orthogonal mode, unphysical perturbations may not be properly damped if they contain energy in lower-order modes. To fully remove oscillations from lower-degree polynomial modes, one would have to further relax the accuracy conditions \eqnref{eq:accuracy_conditions} by setting $s < p$. In such a case, an alternative strategy would be necessary to retain high-order accuracy, such as applying the volume dissipation as a filter as done in \cite{Ranocha_dissipation}.

While we can conlude that the benefits of the volume dissipation presented in this work are less significant for SE schemes than they are for FD schemes, it is important to remember that SE volume dissipation can certainly still be beneficial for various reasons. These include accelerating convergence to steady-state, stabilizing SBP schemes that lack nonlinear stability, and aiding in positivity preservation. The latter two benefits were demonstrated in~\S\ref{sec:KelvinHelmholtz}, where in particular, entropy-stable SE schemes augmented with volume dissipation were found to be remarkably robust for the Kelvin-Helmholtz instability.

\FloatBarrier

\section{Additional Spectra, Error Convergence Figures, and Kelvin-Helmholtz Numerical Results} \label{sec:additional_figures}

This section includes results from additional numerical experiments conducted alongside those presented in the main article. These results provide a detailed comparison between the CSBP, HGTL, HGT, Mattsson, LGL, and LG operators. We first present additional spectra and error convergence plots for the linear convection equation grid convergence study described in \S \ref{sec:LCE}. Since the spectral-element schemes are equivalent to those of \cite{Glaubitz2024} in the context of the linear convection equation, we omit linear convection equation results for LGL and LG operators (though some discussion is included in Appendix \ref{sec:appendix_spectral}). Then, we include additional pressure error convergence plots for the 2D Euler equations isentropic vortex problem described in \S \ref{sec:2dvortex}, both for finite-difference and spectral-element operators. Finally, we provide more detailed tables showing final crash times of the Kelvin-Helmholtz instability described in~\S \ref{sec:KelvinHelmholtz}.

\begin{figure}[t] 
    \centering
    \begin{subfigure}[t]{0.32\textwidth}
        \centering
        \includegraphics[width=\textwidth, trim={5 10 5 5}, clip]{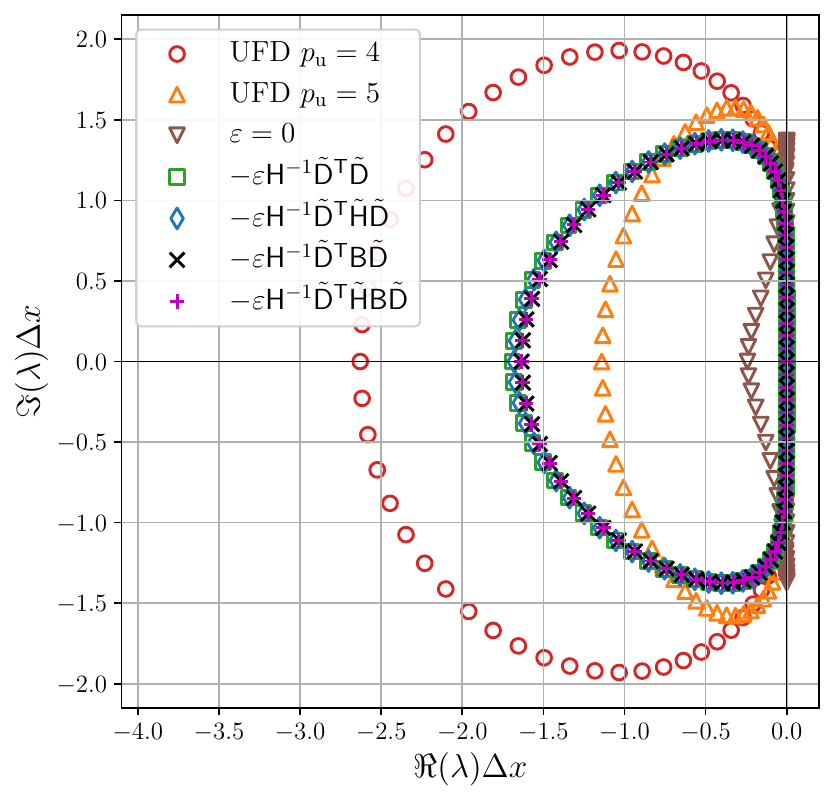}
        \caption{CSBP $p=2$ $\varepsilon = 0.025$}
    \end{subfigure}
    \hfill
    \begin{subfigure}[t]{0.32\textwidth}
        \centering
        \includegraphics[width=\textwidth, trim={5 10 5 5}, clip]{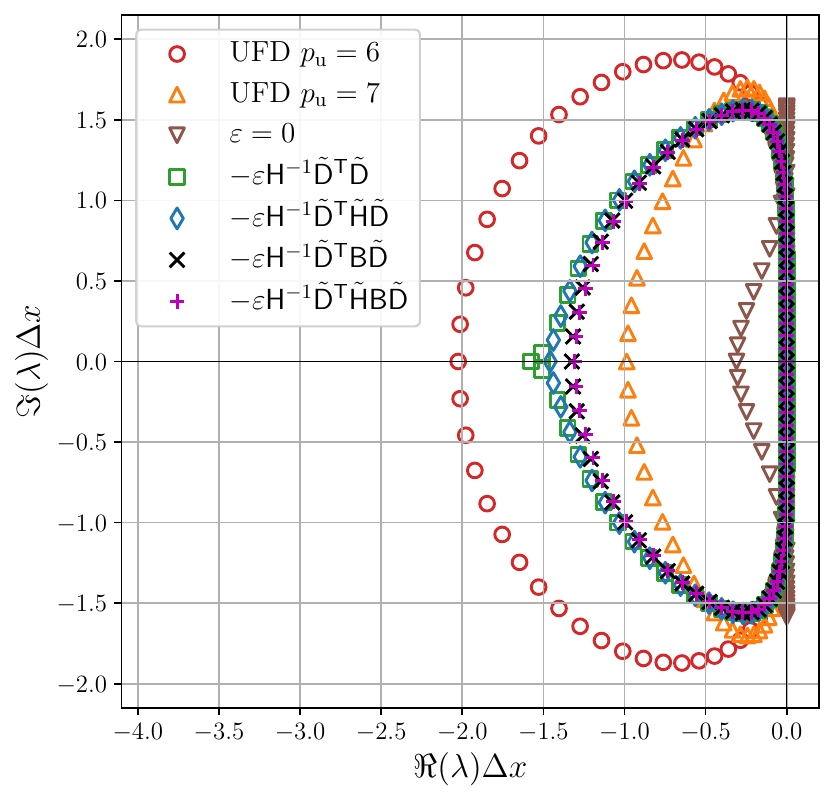}
        \caption{CSBP $p=3$ $\varepsilon = 0.005$}
    \end{subfigure}
    \hfill
    \begin{subfigure}[t]{0.32\textwidth}
        \centering
        \includegraphics[width=\textwidth, trim={5 10 5 5}, clip]{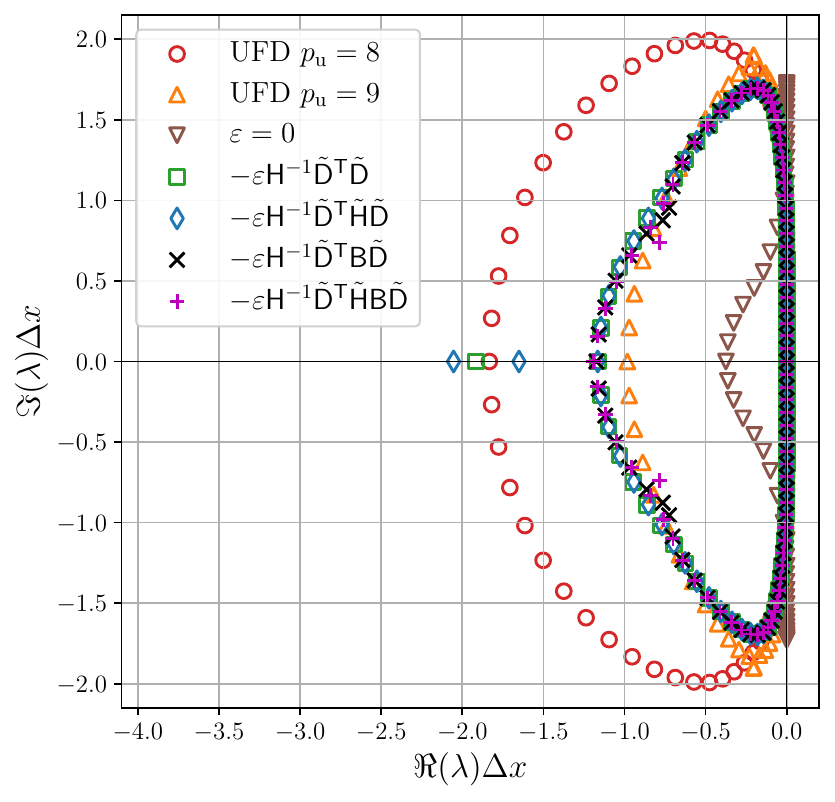}
        \caption{CSBP $p=4$ $\varepsilon = 0.001$}
    \end{subfigure}
    \vskip\baselineskip 
    \begin{subfigure}[t]{0.32\textwidth}
        \centering
        \includegraphics[width=\textwidth, trim={5 10 5 5}, clip]{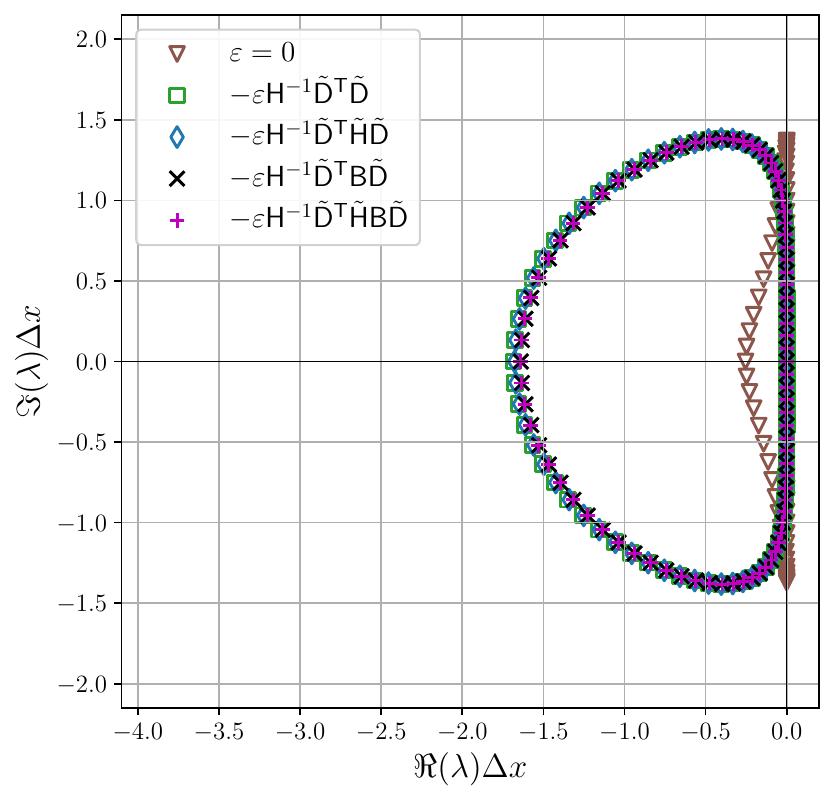}
        \caption{HGTL $p=2$ $\varepsilon = 0.025$}
    \end{subfigure}
    \hfill
    \begin{subfigure}[t]{0.32\textwidth}
        \centering
        \includegraphics[width=\textwidth, trim={5 10 5 5}, clip]{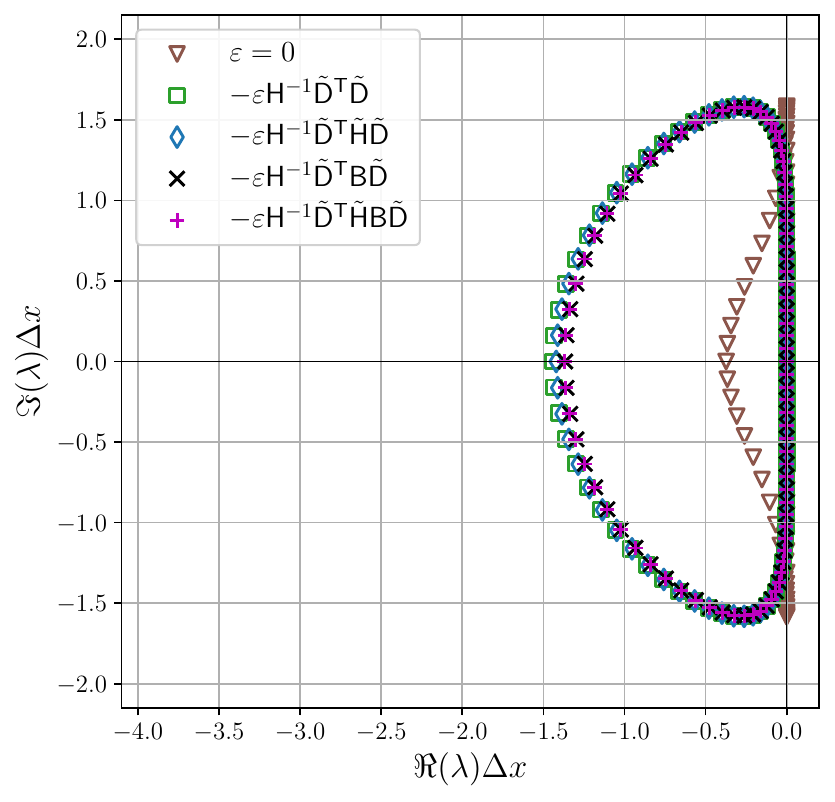}
        \caption{HGTL $p=3$ $\varepsilon = 0.005$}
    \end{subfigure}
    \hfill
    \begin{subfigure}[t]{0.32\textwidth}
        \centering
        \includegraphics[width=\textwidth, trim={5 10 5 5}, clip]{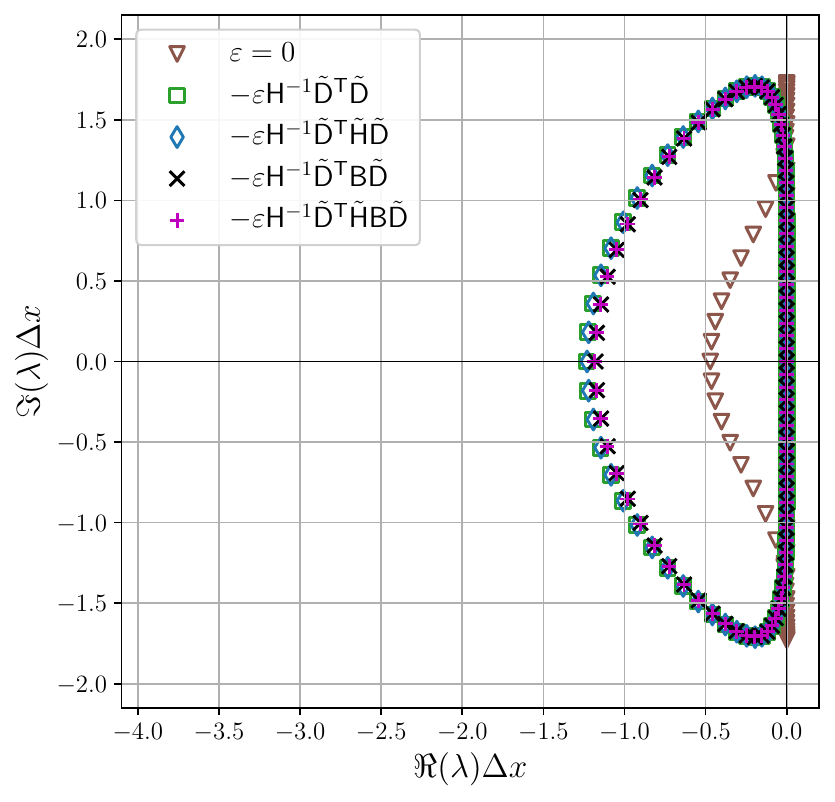}
        \caption{HGTL $p=4$ $\varepsilon = 0.001$}
    \end{subfigure}
    \vskip\baselineskip 
    \begin{subfigure}[t]{0.32\textwidth}
        \centering
        \includegraphics[width=\textwidth, trim={5 10 5 5}, clip]{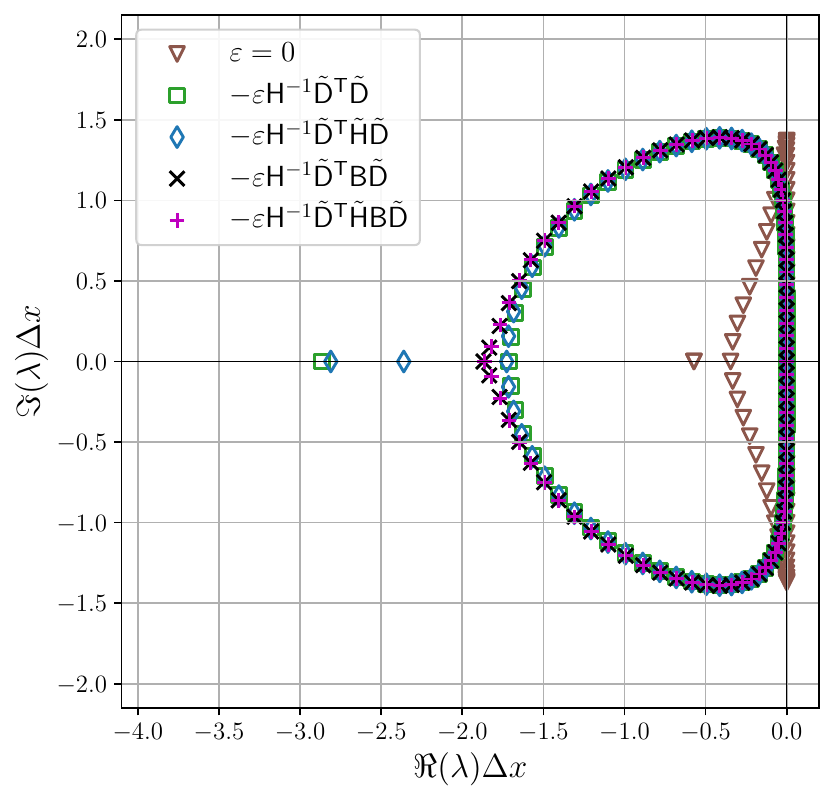}
        \caption{HGT $p=2$ $\varepsilon = 0.025$}
    \end{subfigure}
    \hfill
    \begin{subfigure}[t]{0.32\textwidth}
        \centering
        \includegraphics[width=\textwidth, trim={5 10 5 5}, clip]{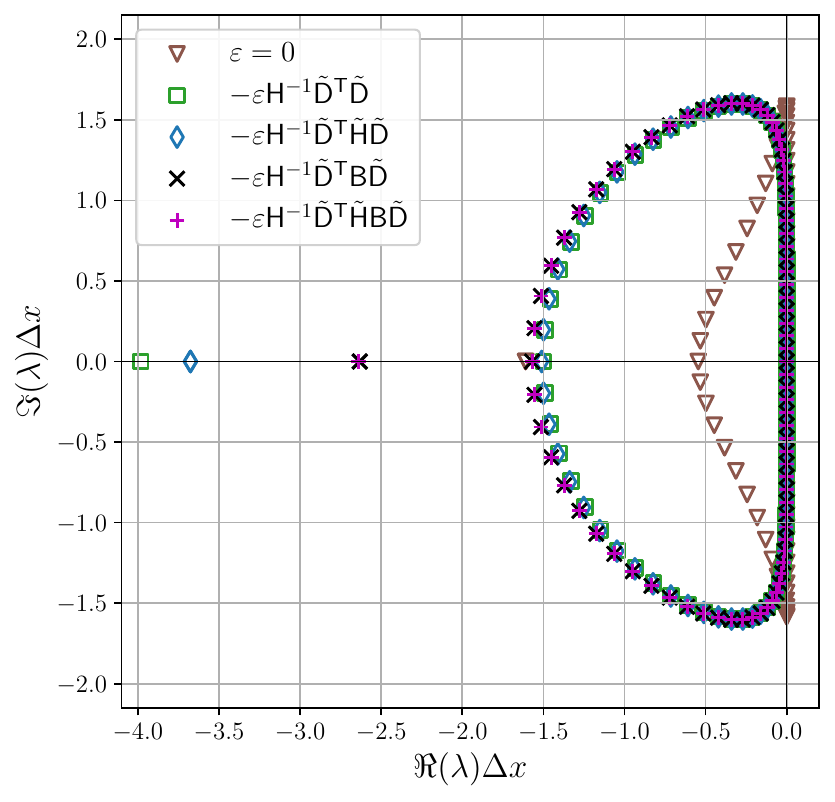}
        \caption{HGT $p=3$ $\varepsilon = 0.005$}
    \end{subfigure}
    \hfill
    \begin{subfigure}[t]{0.32\textwidth}
        \centering
        \includegraphics[width=\textwidth, trim={5 10 5 5}, clip]{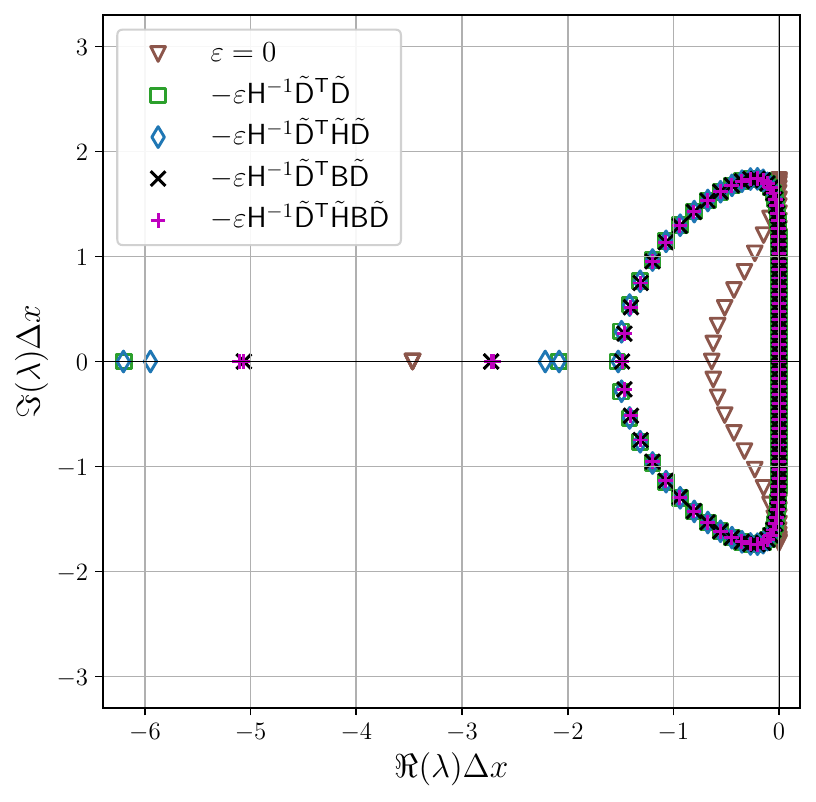}
        \caption{HGT $p=4$ $\varepsilon = 0.001$}
    \end{subfigure}
    \vskip\baselineskip 
    \begin{subfigure}[t]{0.32\textwidth}
        \centering
        \includegraphics[width=\textwidth, trim={5 10 5 5}, clip]{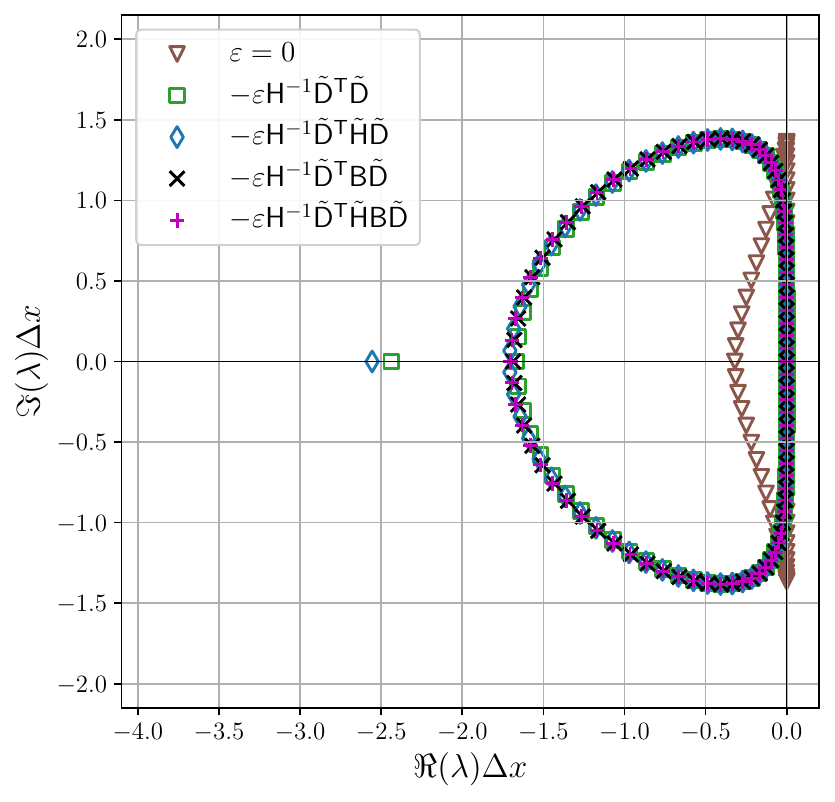}
        \caption{Matt. $p=2$ $\varepsilon = 0.025$}
    \end{subfigure}
    \hfill
    \begin{subfigure}[t]{0.32\textwidth}
        \centering
        \includegraphics[width=\textwidth, trim={5 10 5 5}, clip]{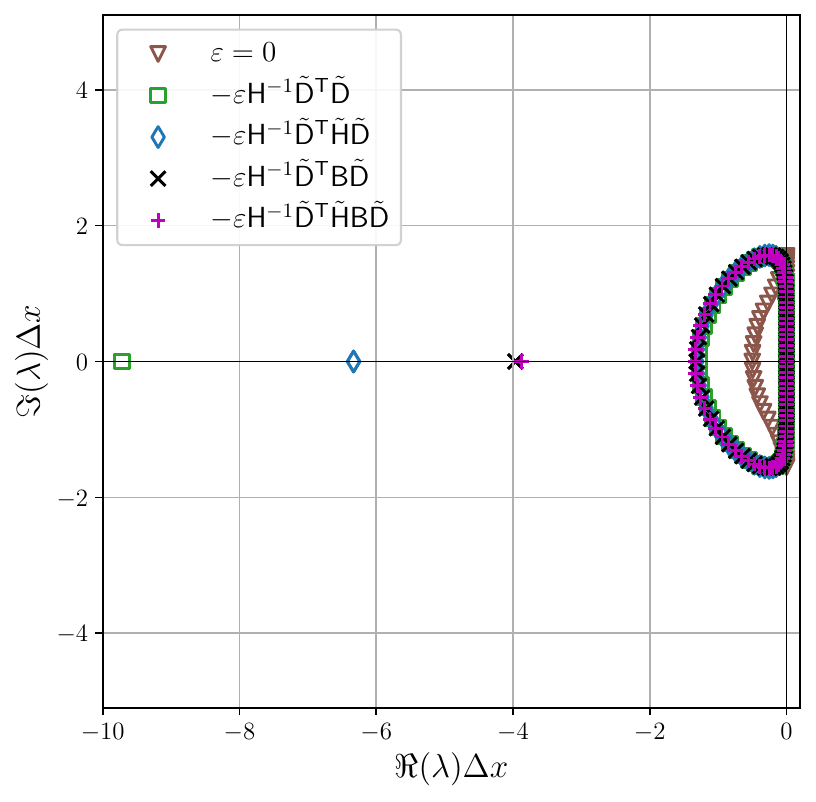}
        \caption{Matt. $p=3$ $\varepsilon = 0.005$}
    \end{subfigure}
    \hfill
    \begin{subfigure}[t]{0.32\textwidth}
        \centering
        \includegraphics[width=\textwidth, trim={5 10 5 5}, clip]{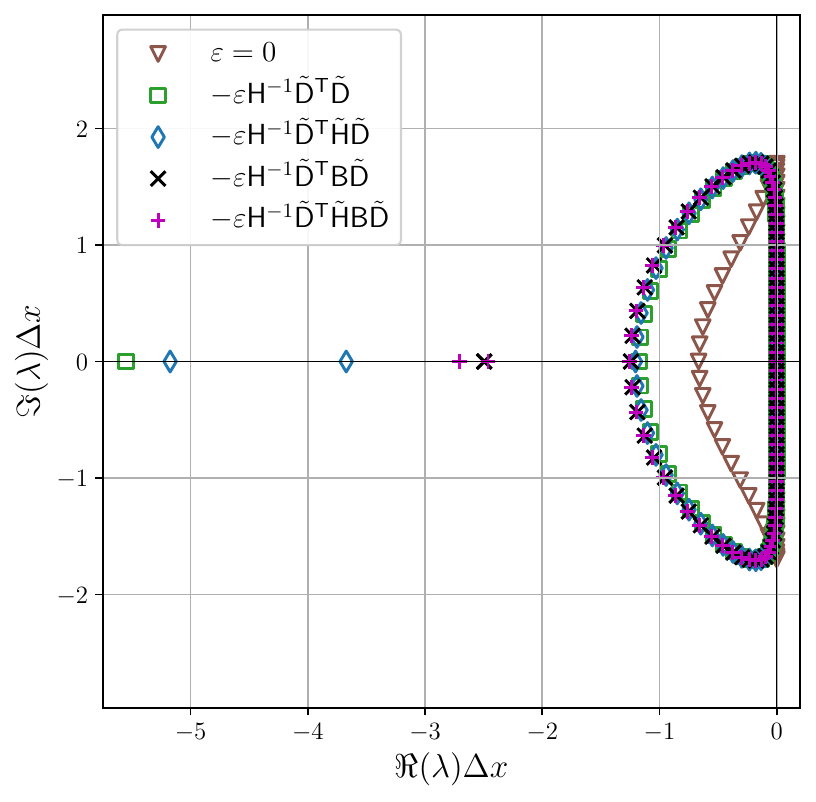}
        \caption{Matt. $p=4$ $\varepsilon = 0.001$}
    \end{subfigure}
    \caption{Eigenspectra for the Linear Convection Equation of central SBP semi-discretizations using CSBP, HGTL, HGT, or Mattsson operators on a grid of 80 nodes, upwind (Lax-Friedrichs) SATs and dissipation $s=p+1$ with a large coefficient value of $\varepsilon = 3.125 \times 5^{-s}$. Spectra without any volume dissipation ($\varepsilon = 0$) and of UFD discretizations of \cite{Mattsson2017} included for comparison.}
    \label{fig:LCEeigs_additional}
\end{figure}

\begin{figure}[t] 
    \centering
    \begin{subfigure}[t]{0.32\textwidth}
        \centering
        \includegraphics[width=\textwidth, trim={5 10 5 5}, clip]{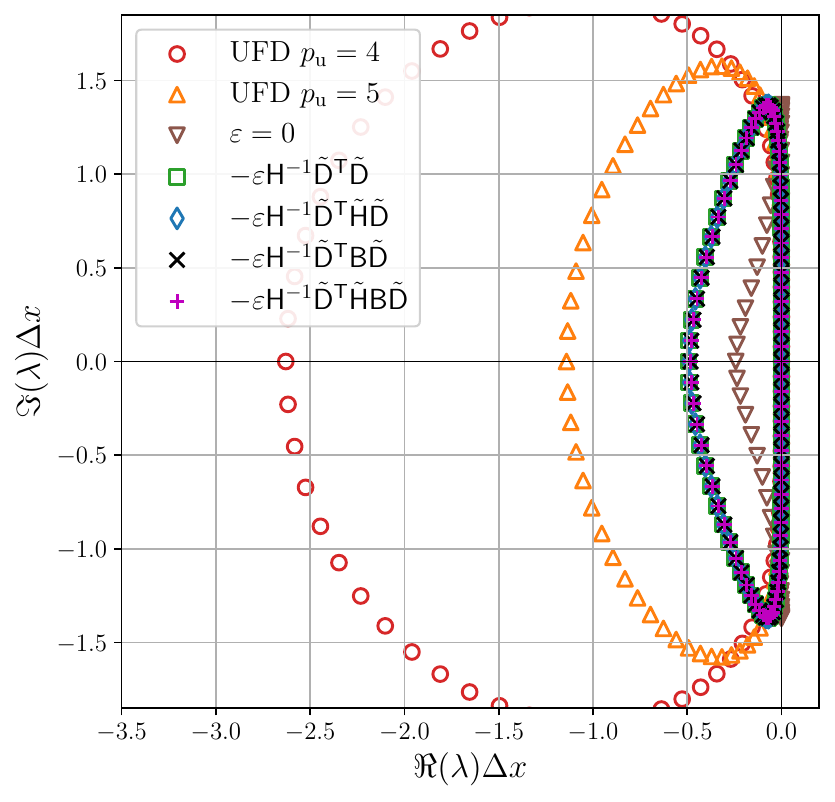}
        \caption{CSBP $p=2$ $\varepsilon = 0.005$}
    \end{subfigure}
    \hfill
    \begin{subfigure}[t]{0.32\textwidth}
        \centering
        \includegraphics[width=\textwidth, trim={5 10 5 5}, clip]{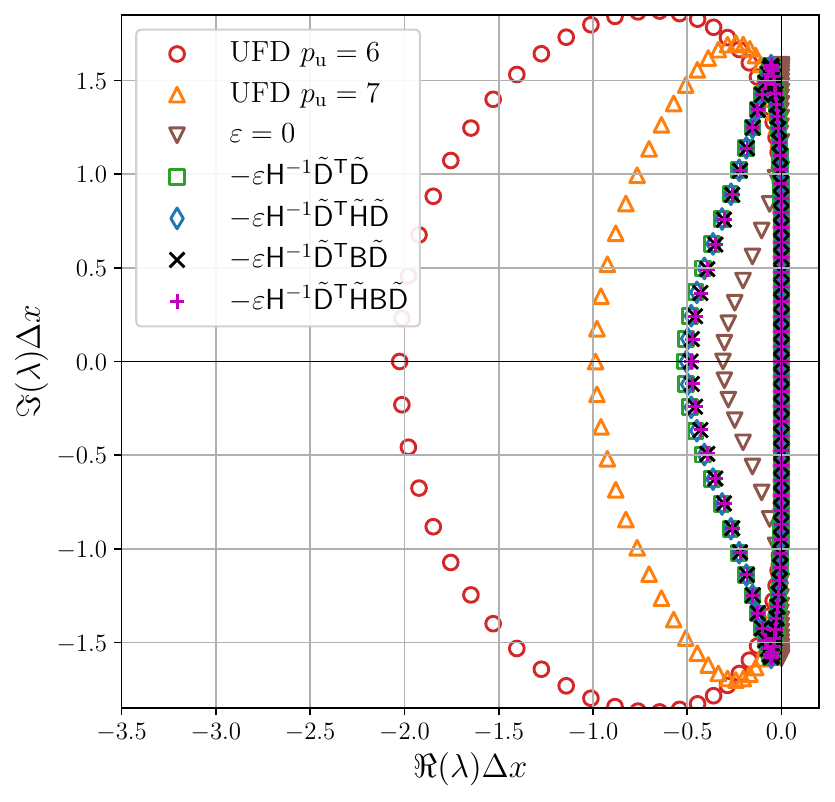}
        \caption{CSBP $p=3$ $\varepsilon = 0.001$}
    \end{subfigure}
    \hfill
    \begin{subfigure}[t]{0.32\textwidth}
        \centering
        \includegraphics[width=\textwidth, trim={5 10 5 5}, clip]{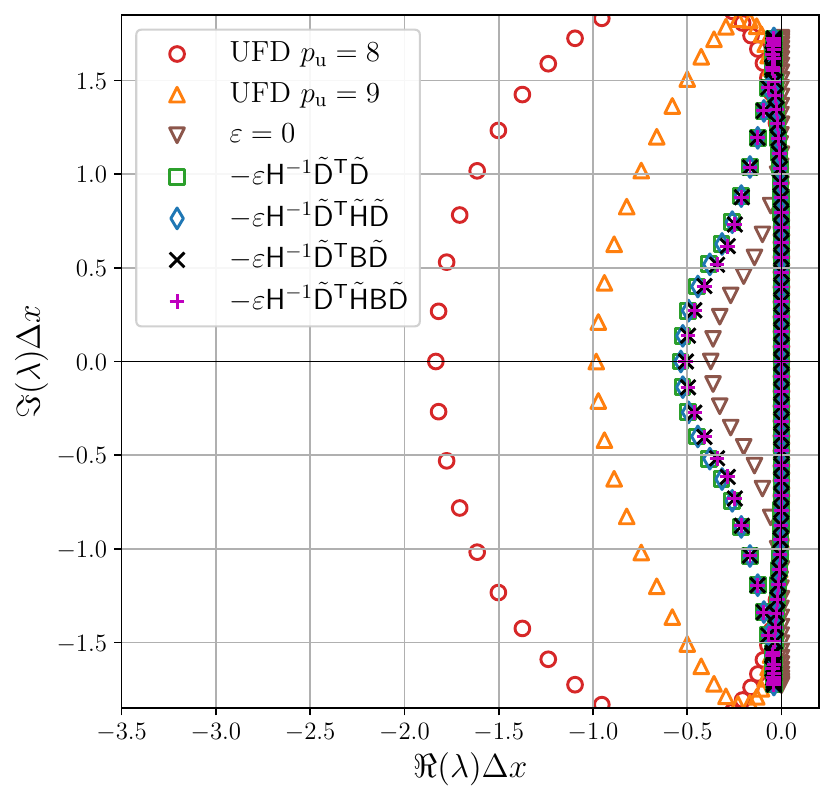}
        \caption{CSBP $p=4$ $\varepsilon = 0.0002$}
    \end{subfigure}
    \vskip\baselineskip 
    \begin{subfigure}[t]{0.32\textwidth}
        \centering
        \includegraphics[width=\textwidth, trim={5 10 5 5}, clip]{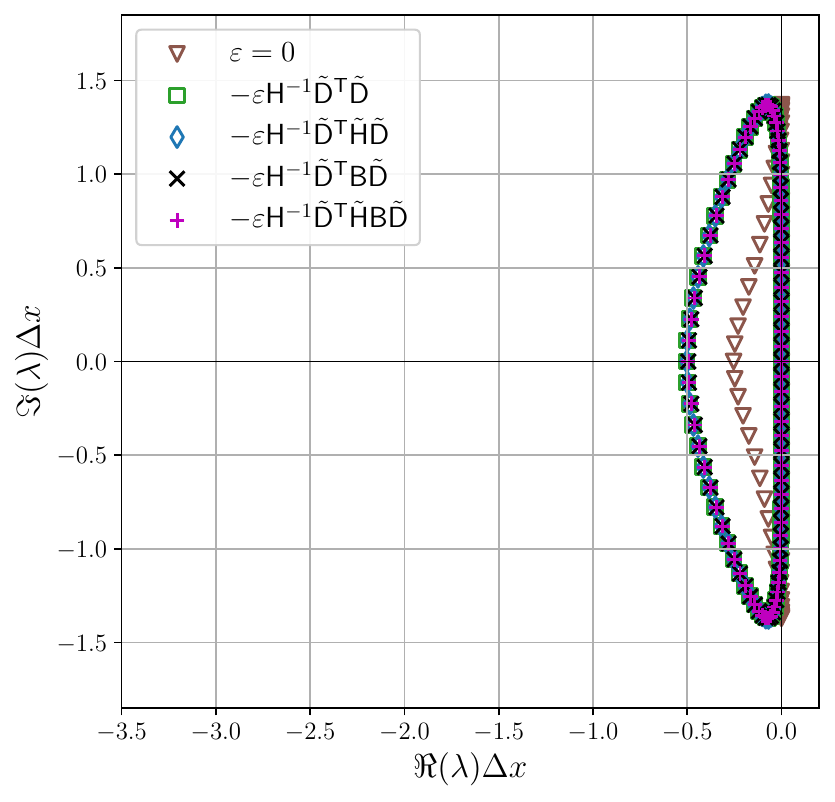}
        \caption{HGTL $p=2$ $\varepsilon = 0.005$}
    \end{subfigure}
    \hfill
    \begin{subfigure}[t]{0.32\textwidth}
        \centering
        \includegraphics[width=\textwidth, trim={5 10 5 5}, clip]{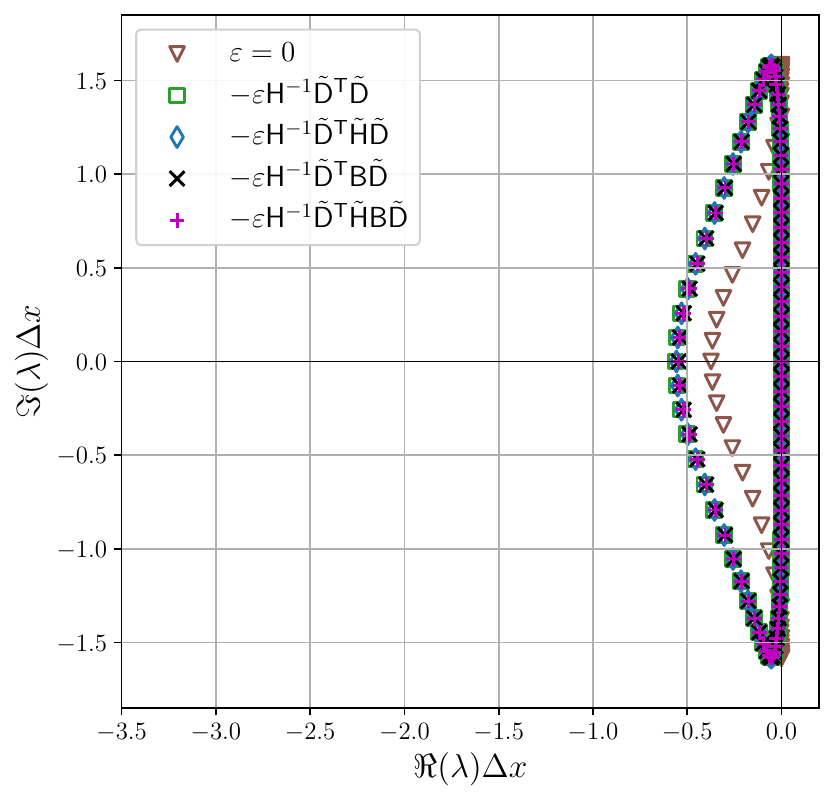}
        \caption{HGTL $p=3$ $\varepsilon = 0.001$}
    \end{subfigure}
    \hfill
    \begin{subfigure}[t]{0.32\textwidth}
        \centering
        \includegraphics[width=\textwidth, trim={5 10 5 5}, clip]{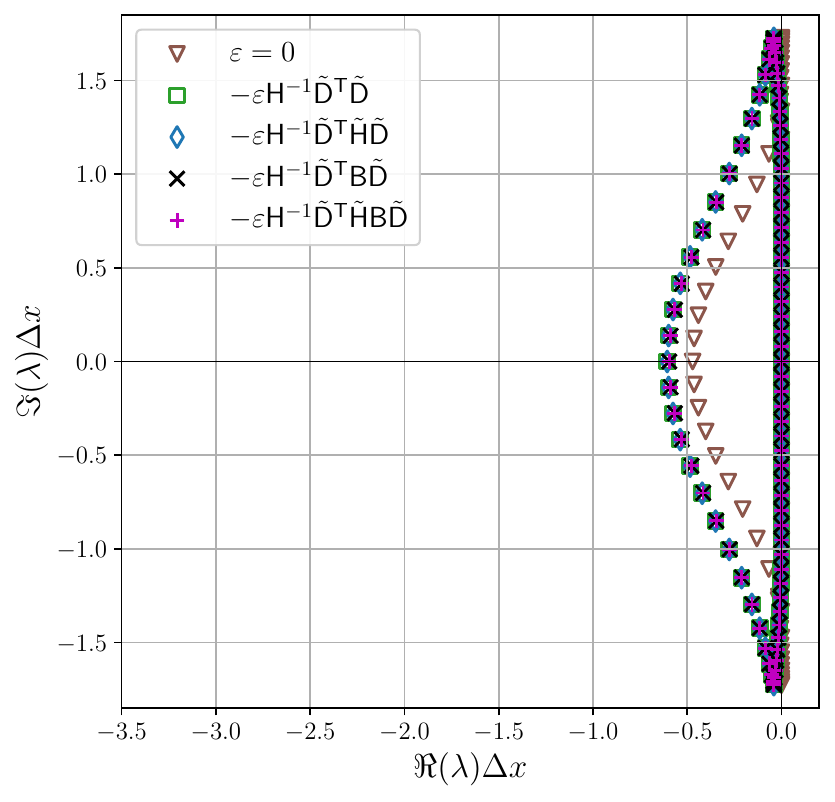}
        \caption{HGTL $p=4$ $\varepsilon = 0.0002$}
    \end{subfigure}
    \vskip\baselineskip 
    \begin{subfigure}[t]{0.32\textwidth}
        \centering
        \includegraphics[width=\textwidth, trim={5 10 5 5}, clip]{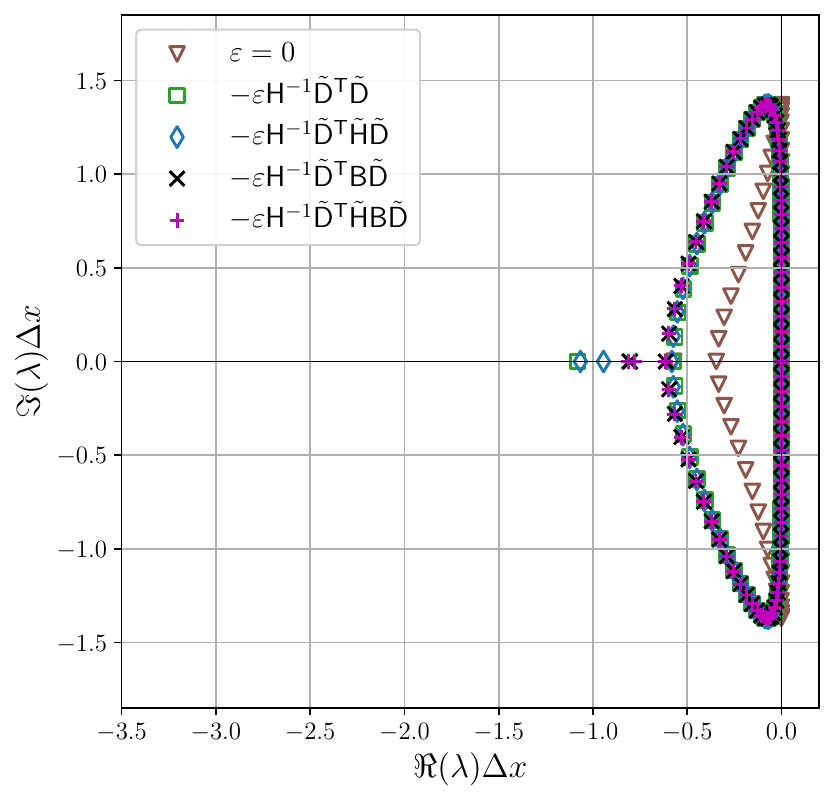}
        \caption{HGT $p=2$ $\varepsilon = 0.005$}
    \end{subfigure}
    \hfill
    \begin{subfigure}[t]{0.32\textwidth}
        \centering
        \includegraphics[width=\textwidth, trim={5 10 5 5}, clip]{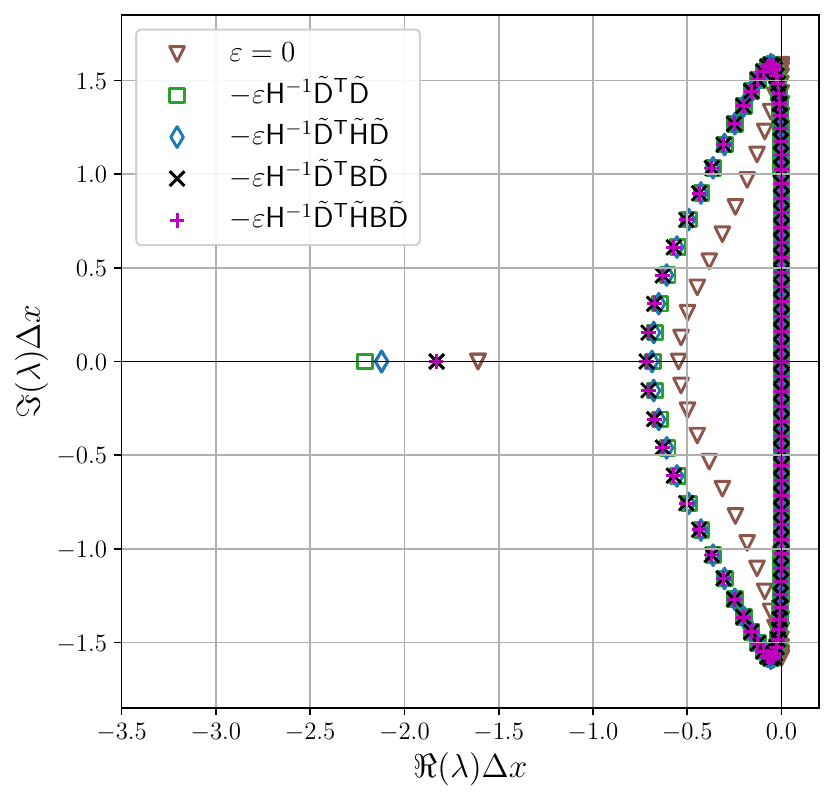}
        \caption{HGT $p=3$ $\varepsilon = 0.001$}
    \end{subfigure}
    \hfill
    \begin{subfigure}[t]{0.32\textwidth}
        \centering
        \includegraphics[width=\textwidth, trim={5 10 5 5}, clip]{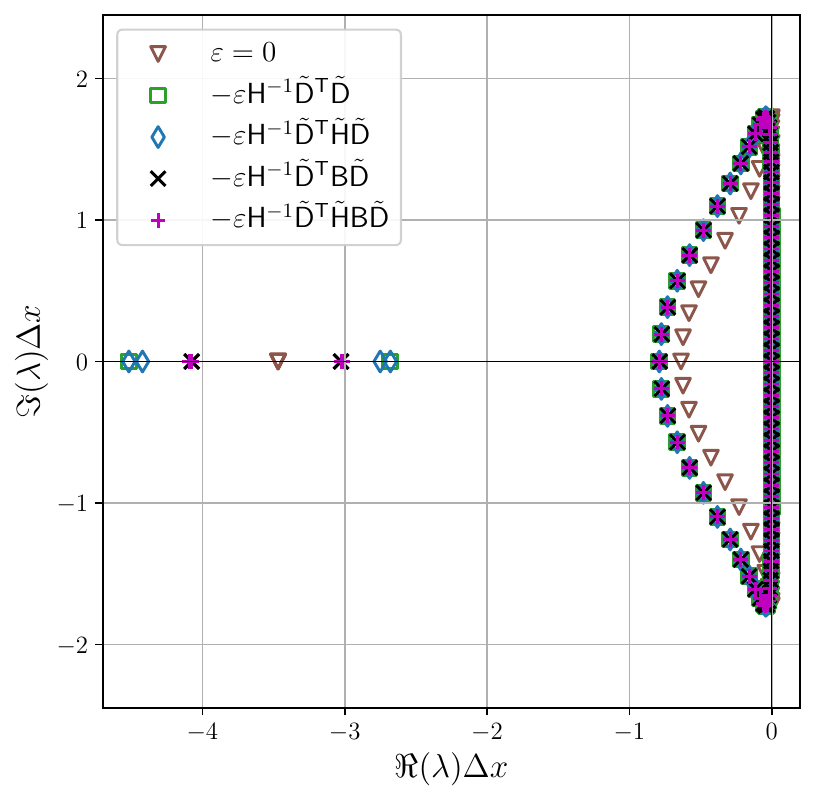}
        \caption{HGT $p=4$ $\varepsilon = 0.0002$}
    \end{subfigure}
    \vskip\baselineskip 
    \begin{subfigure}[t]{0.32\textwidth}
        \centering
        \includegraphics[width=\textwidth, trim={5 10 5 5}, clip]{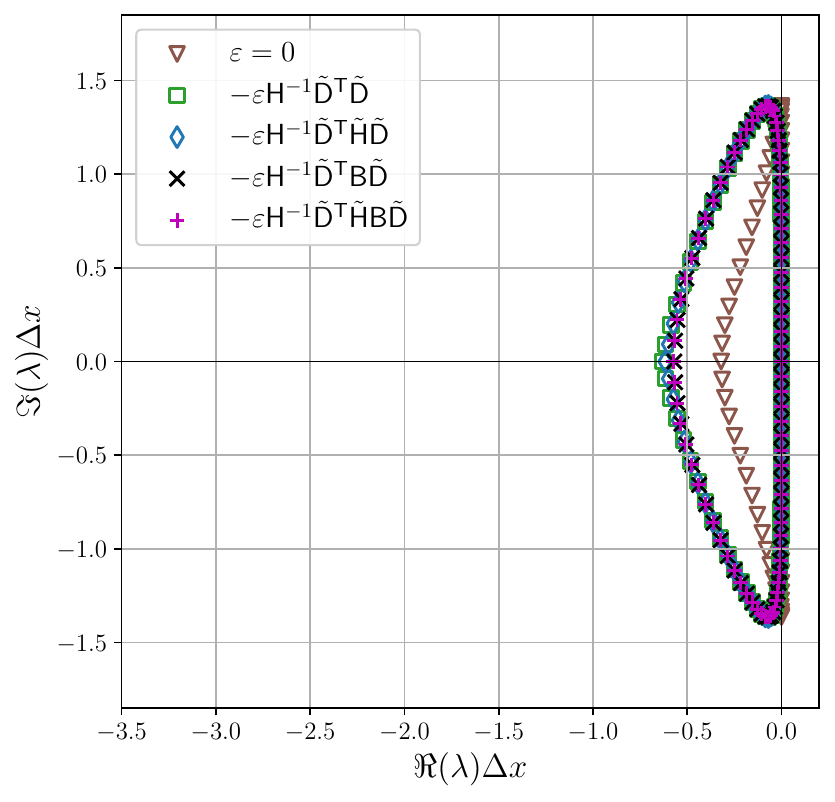}
        \caption{Matt. $p=2$ $\varepsilon = 0.005$}
    \end{subfigure}
    \hfill
    \begin{subfigure}[t]{0.32\textwidth}
        \centering
        \includegraphics[width=\textwidth, trim={5 10 5 5}, clip]{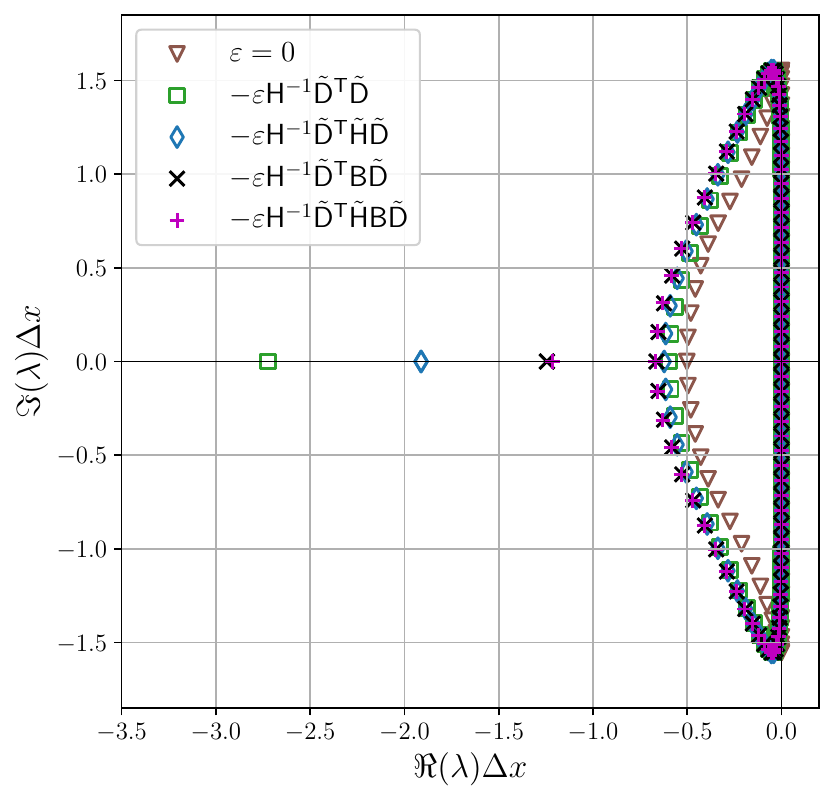}
        \caption{Matt. $p=3$ $\varepsilon = 0.001$}
    \end{subfigure}
    \hfill
    \begin{subfigure}[t]{0.32\textwidth}
        \centering
        \includegraphics[width=\textwidth, trim={5 10 5 5}, clip]{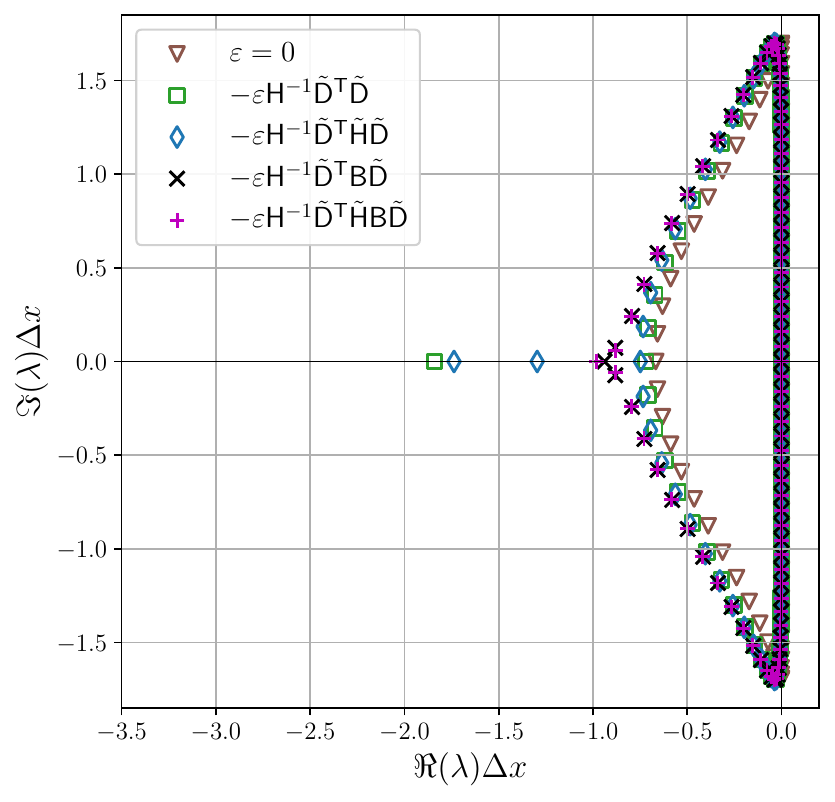}
        \caption{Matt. $p=4$ $\varepsilon = 0.0002$}
    \end{subfigure}
    \caption{Eigenspectra for the Linear Convection Equation of central SBP semi-discretizations using CSBP, HGTL, HGT, or Mattsson operators on a grid of 80 nodes, upwind (Lax-Friedrichs) SATs and dissipation $s=p+1$ with a small coefficient value of $\varepsilon = 0.625 \times 5^{-s}$. Spectra without any volume dissipation ($\varepsilon = 0$) and of UFD discretizations of \cite{Mattsson2017} are included for comparison.}
    \label{fig:LCEeigs_additional2}
\end{figure}

\begin{figure}[t] 
    \centering
    \begin{subfigure}[t]{0.32\textwidth}
        \centering
        \includegraphics[width=\textwidth, trim={10 10 10 10}, clip]{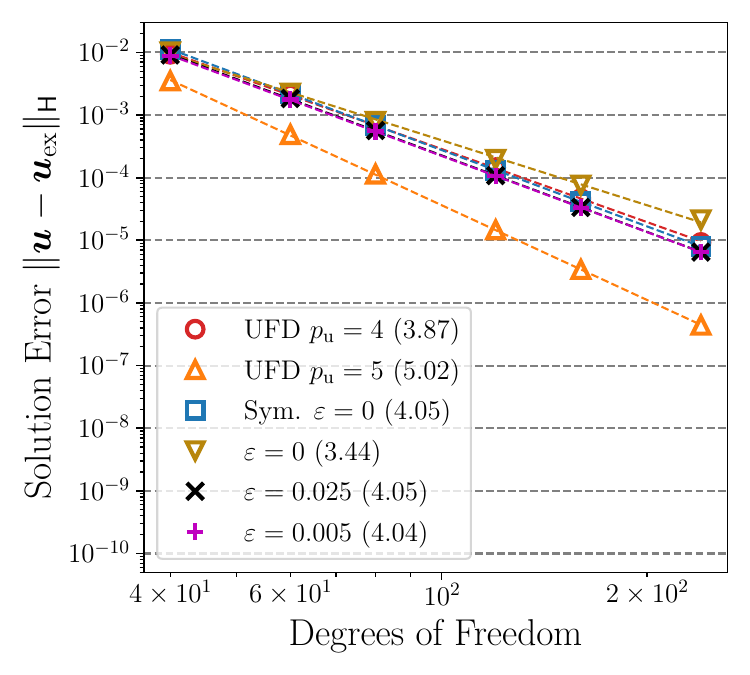}
        \caption{CSBP $p=2$}
    \end{subfigure}
    \hfill
    \begin{subfigure}[t]{0.32\textwidth}
        \centering
        \includegraphics[width=\textwidth, trim={10 10 10 10}, clip]{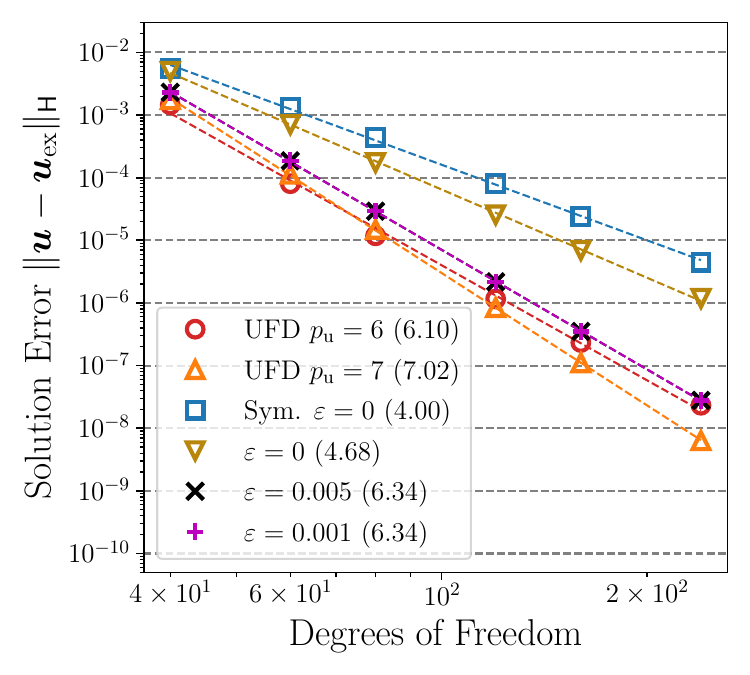}
        \caption{CSBP $p=3$}
    \end{subfigure}
    \hfill
    \begin{subfigure}[t]{0.32\textwidth}
        \centering
        \includegraphics[width=\textwidth, trim={10 10 10 10}, clip]{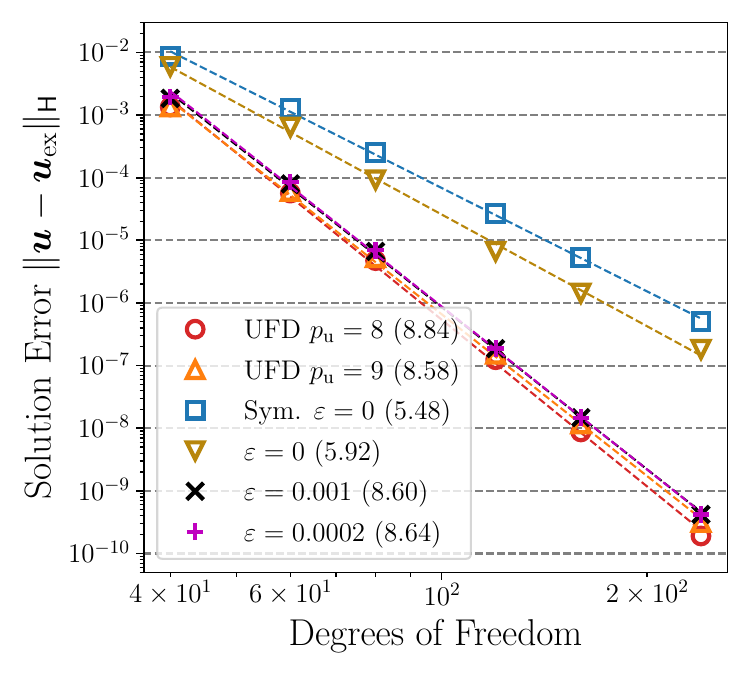}
        \caption{CSBP $p=4$}
    \end{subfigure}
    \vskip\baselineskip 
    \begin{subfigure}[t]{0.32\textwidth}
        \centering
        \includegraphics[width=\textwidth, trim={10 10 10 10}, clip]{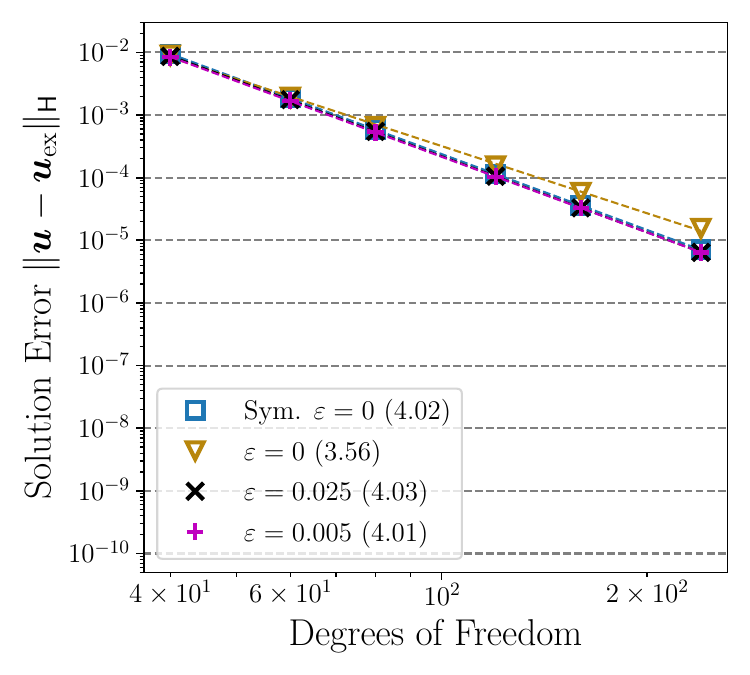}
        \caption{HGTL $p=2$}
    \end{subfigure}
    \hfill
    \begin{subfigure}[t]{0.32\textwidth}
        \centering
        \includegraphics[width=\textwidth, trim={10 10 10 10}, clip]{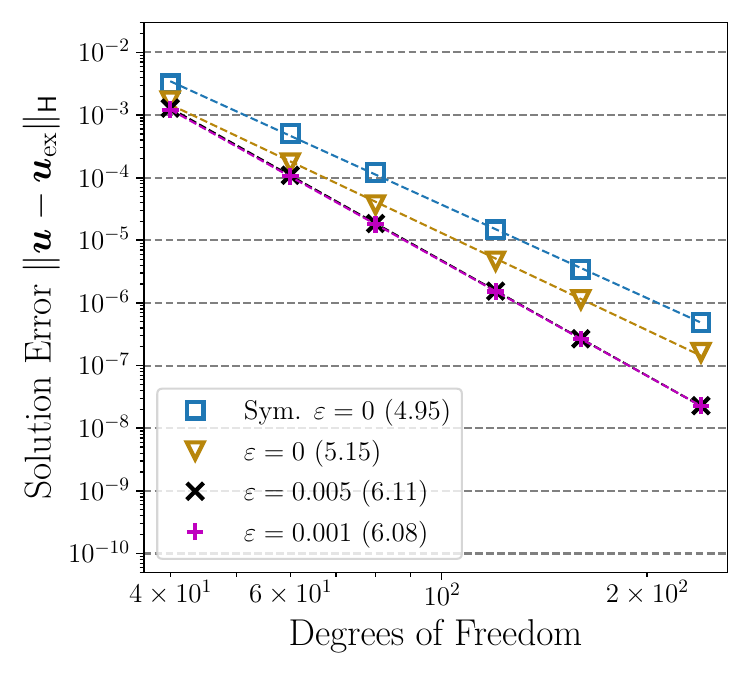}
        \caption{HGTL $p=3$}
    \end{subfigure}
    \hfill
    \begin{subfigure}[t]{0.32\textwidth}
        \centering
        \includegraphics[width=\textwidth, trim={10 10 10 10}, clip]{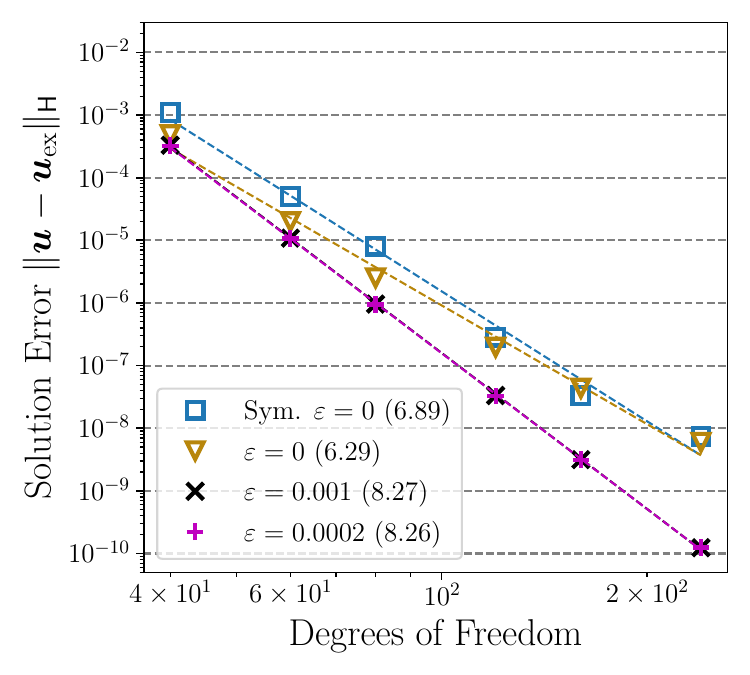}
        \caption{HGTL $p=4$}
    \end{subfigure}
    \vskip\baselineskip 
    \begin{subfigure}[t]{0.32\textwidth}
        \centering
        \includegraphics[width=\textwidth, trim={10 10 10 10}, clip]{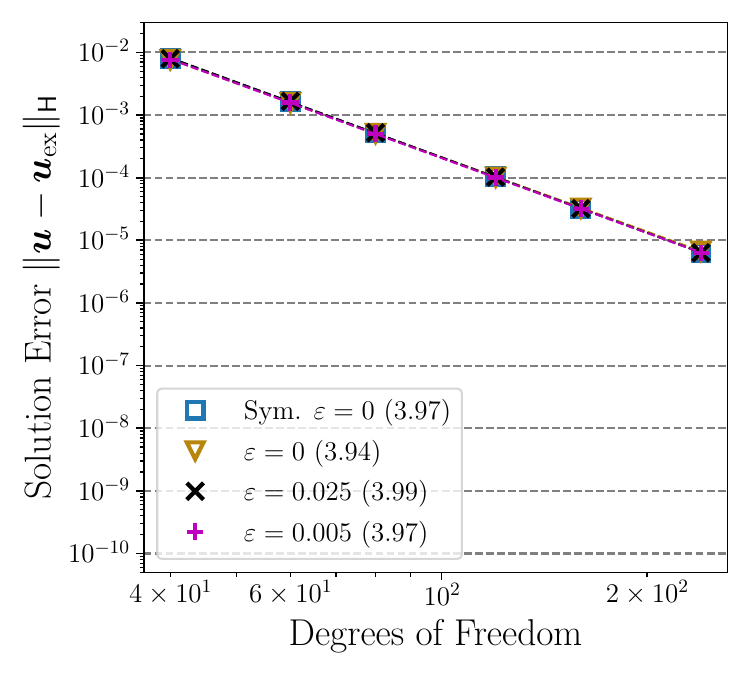}
        \caption{HGT $p=2$}
    \end{subfigure}
    \hfill
    \begin{subfigure}[t]{0.32\textwidth}
        \centering
        \includegraphics[width=\textwidth, trim={10 10 10 10}, clip]{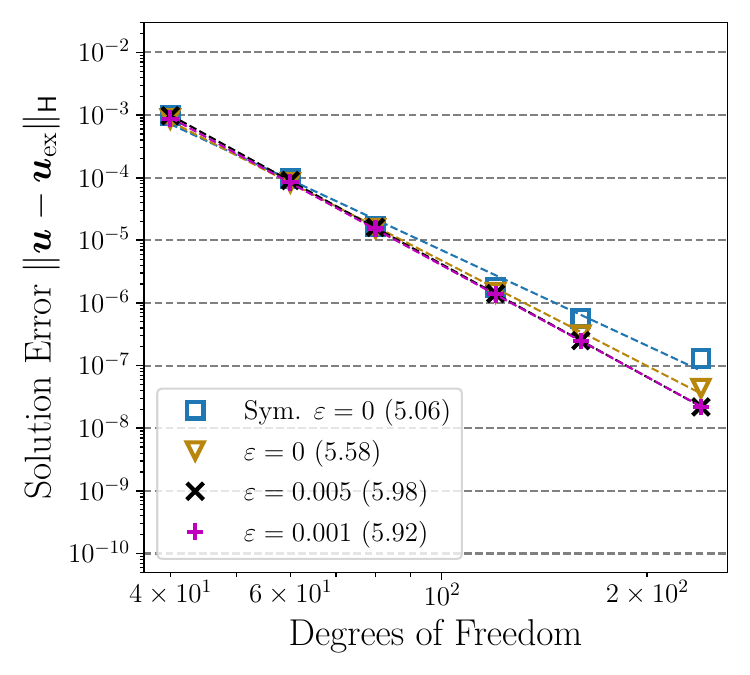}
        \caption{HGT $p=3$}
    \end{subfigure}
    \hfill
    \begin{subfigure}[t]{0.32\textwidth}
        \centering
        \includegraphics[width=\textwidth, trim={10 10 10 10}, clip]{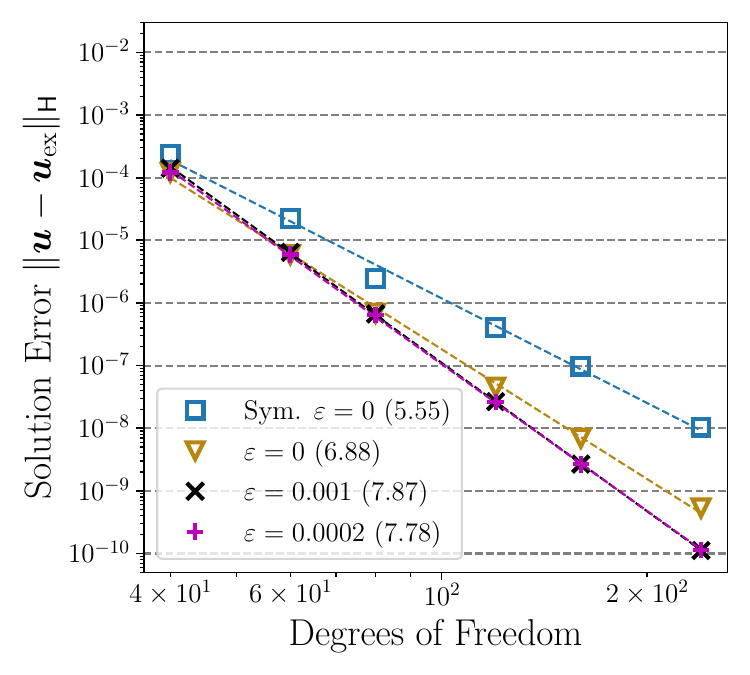}
        \caption{HGT $p=4$}
    \end{subfigure}
    \vskip\baselineskip 
    \begin{subfigure}[t]{0.32\textwidth}
        \centering
        \includegraphics[width=\textwidth, trim={10 10 10 10}, clip]{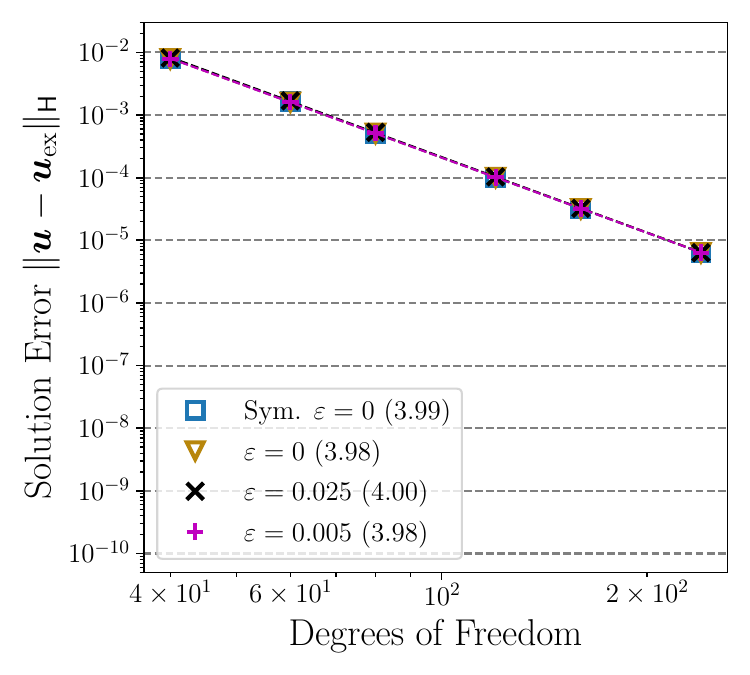}
        \caption{Mattsson $p=2$}
    \end{subfigure}
    \hfill
    \begin{subfigure}[t]{0.32\textwidth}
        \centering
        \includegraphics[width=\textwidth, trim={10 10 10 10}, clip]{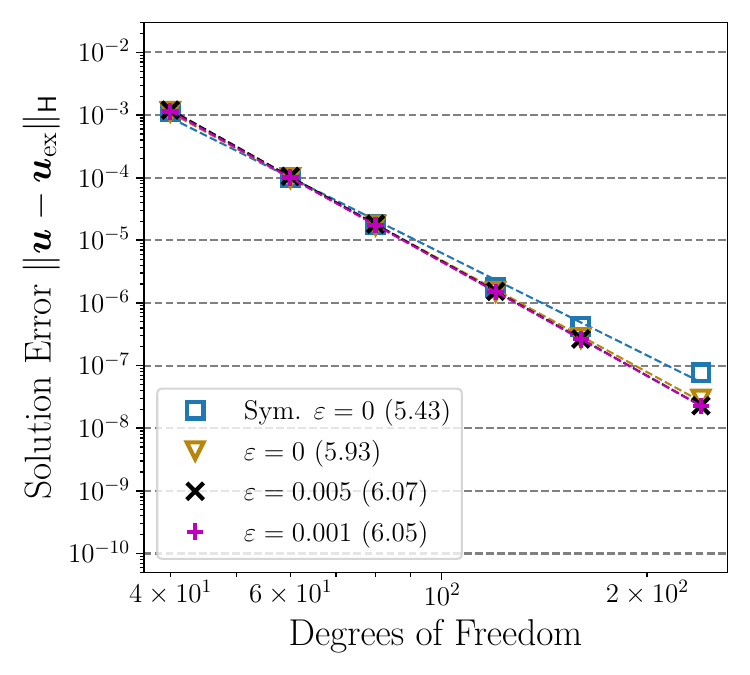}
        \caption{Mattsson $p=3$}
    \end{subfigure}
    \hfill
    \begin{subfigure}[t]{0.32\textwidth}
        \centering
        \includegraphics[width=\textwidth, trim={10 10 10 10}, clip]{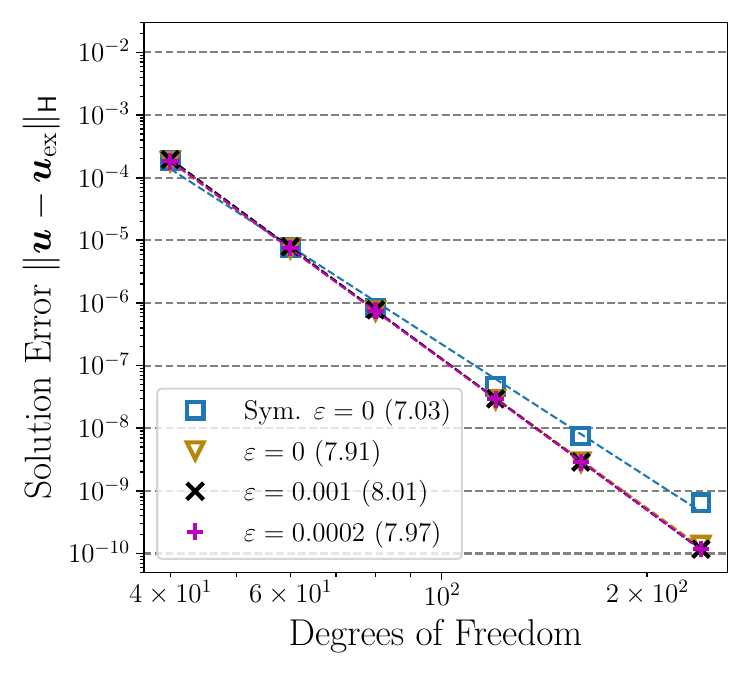}
        \caption{Mattsson $p=4$}
    \end{subfigure}
    \caption{Solution error convergence after one period ($t=1$) of the linear convection equation for UFD SBP semi-discretizations of \cite{Mattsson2017} and central SBP semi-discretizations with either symmetric (\ie, no dissipation) or upwind SATs augmented with dissipation $\mat{A}_\mat{D} = -\varepsilon \HI \tilde{\mat{D}}_s^\T \mat{B} \tilde{\mat{D}}_s$ with $s=p+1$ and varying~$\varepsilon$. A grid of 40, 60, 80, 120, 160, and 240 nodes is used. Solution error convergence rates are given in the legends.}
    \label{fig:LCEconv_additional_t1}
\end{figure}

\begin{figure}[t] 
    \centering
    \begin{subfigure}[t]{0.32\textwidth}
        \centering
        \includegraphics[width=\textwidth, trim={10 10 10 10}, clip]{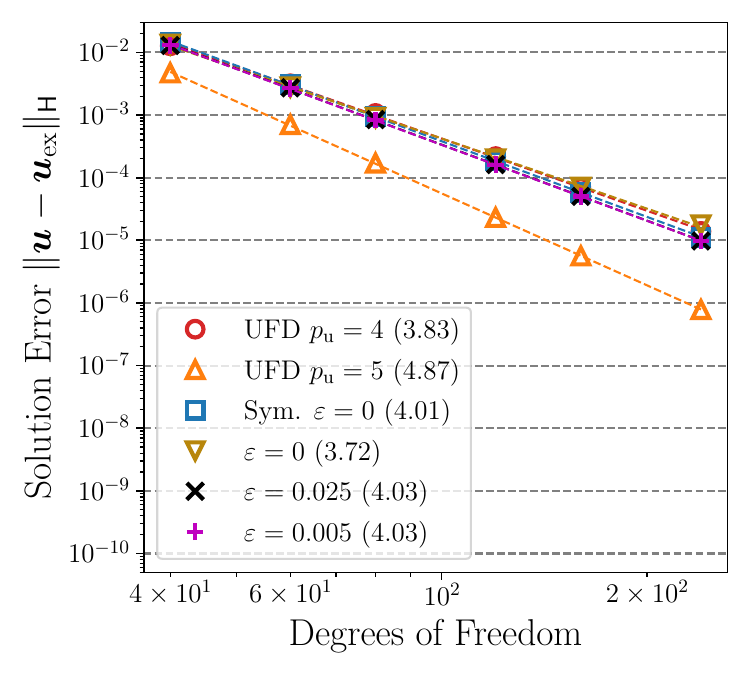}
        \caption{CSBP $p=2$}
    \end{subfigure}
    \hfill
    \begin{subfigure}[t]{0.32\textwidth}
        \centering
        \includegraphics[width=\textwidth, trim={10 10 10 10}, clip]{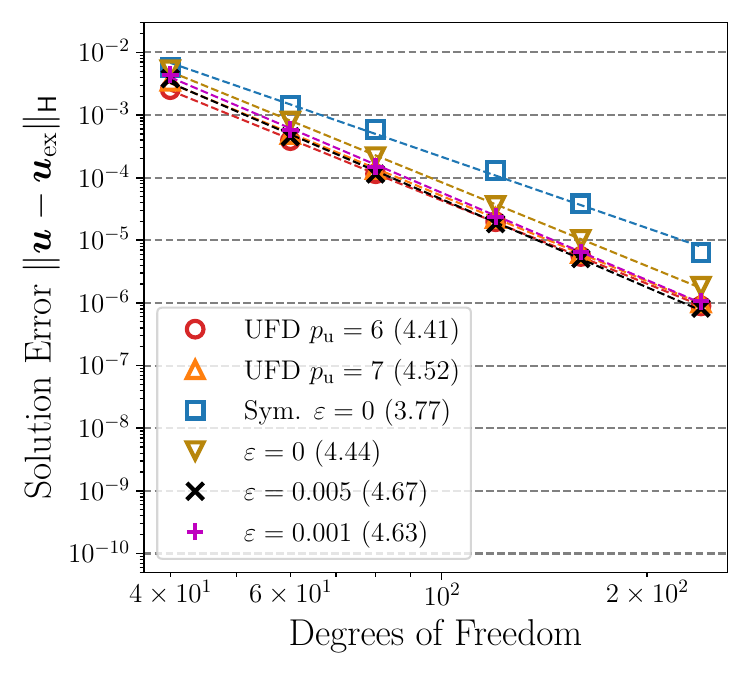}
        \caption{CSBP $p=3$}
    \end{subfigure}
    \hfill
    \begin{subfigure}[t]{0.32\textwidth}
        \centering
        \includegraphics[width=\textwidth, trim={10 10 10 10}, clip]{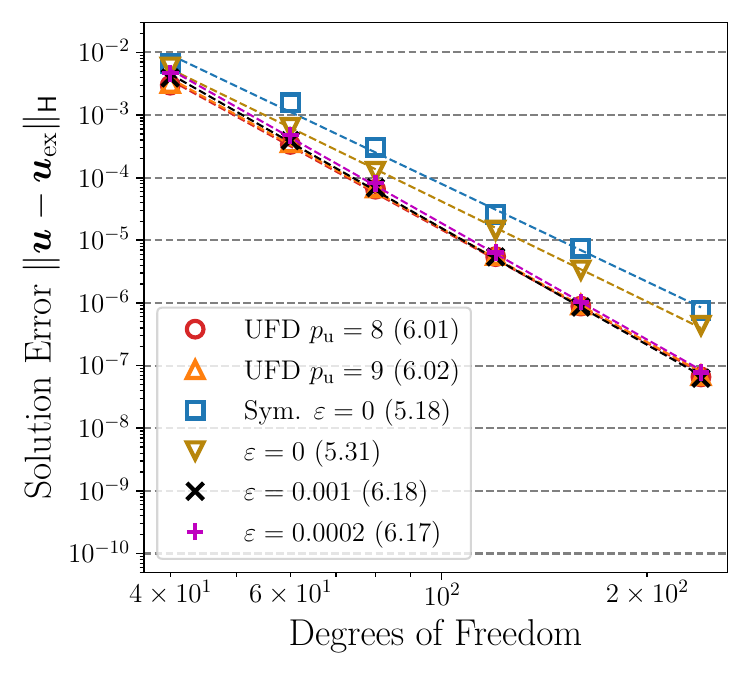}
        \caption{CSBP $p=4$}
    \end{subfigure}
    \vskip\baselineskip 
    \begin{subfigure}[t]{0.32\textwidth}
        \centering
        \includegraphics[width=\textwidth, trim={10 10 10 10}, clip]{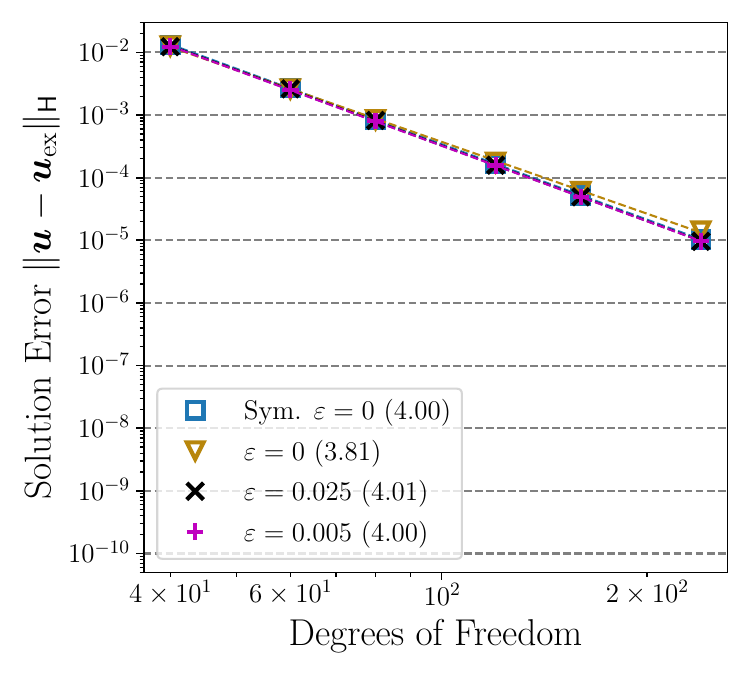}
        \caption{HGTL $p=2$}
    \end{subfigure}
    \hfill
    \begin{subfigure}[t]{0.32\textwidth}
        \centering
        \includegraphics[width=\textwidth, trim={10 10 10 10}, clip]{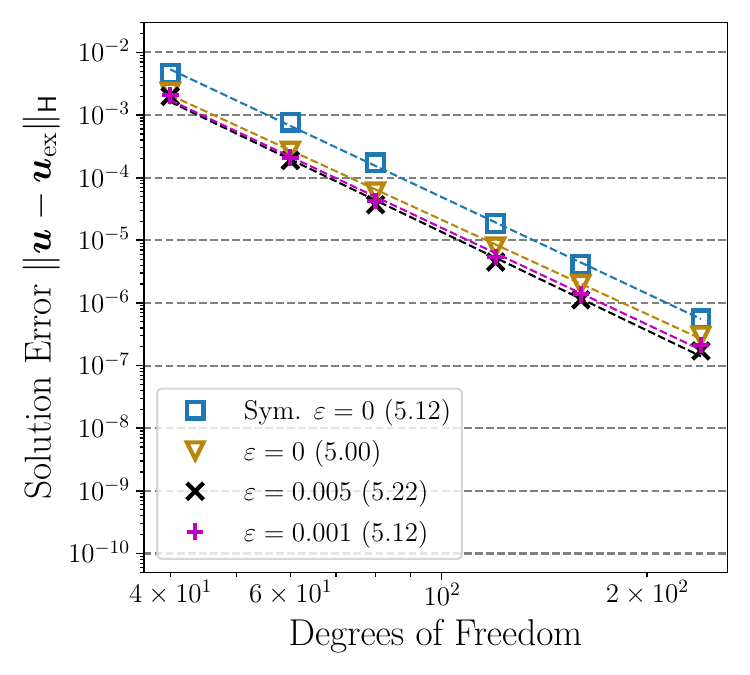}
        \caption{HGTL $p=3$}
    \end{subfigure}
    \hfill
    \begin{subfigure}[t]{0.32\textwidth}
        \centering
        \includegraphics[width=\textwidth, trim={10 10 10 10}, clip]{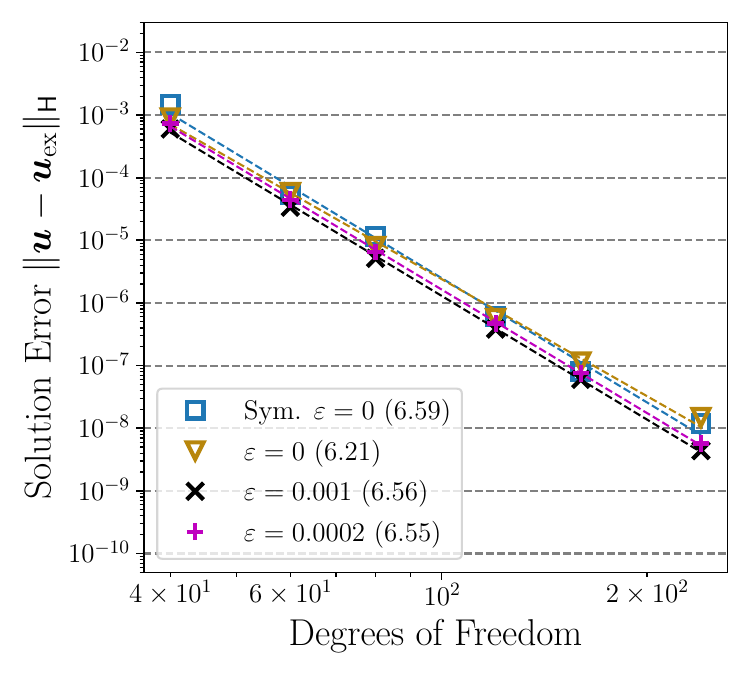}
        \caption{HGTL $p=4$}
    \end{subfigure}
    \vskip\baselineskip 
    \begin{subfigure}[t]{0.32\textwidth}
        \centering
        \includegraphics[width=\textwidth, trim={10 10 10 10}, clip]{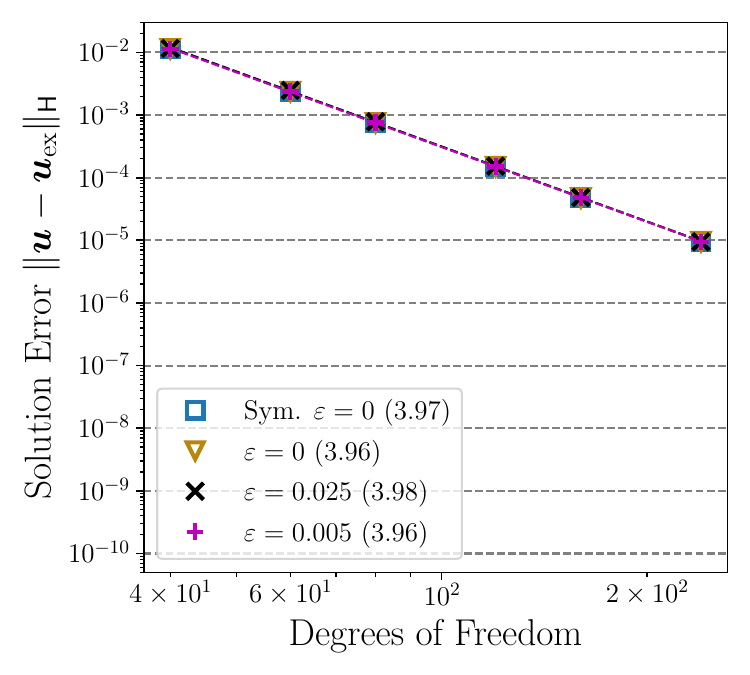}
        \caption{HGT $p=2$}
    \end{subfigure}
    \hfill
    \begin{subfigure}[t]{0.32\textwidth}
        \centering
        \includegraphics[width=\textwidth, trim={10 10 10 10}, clip]{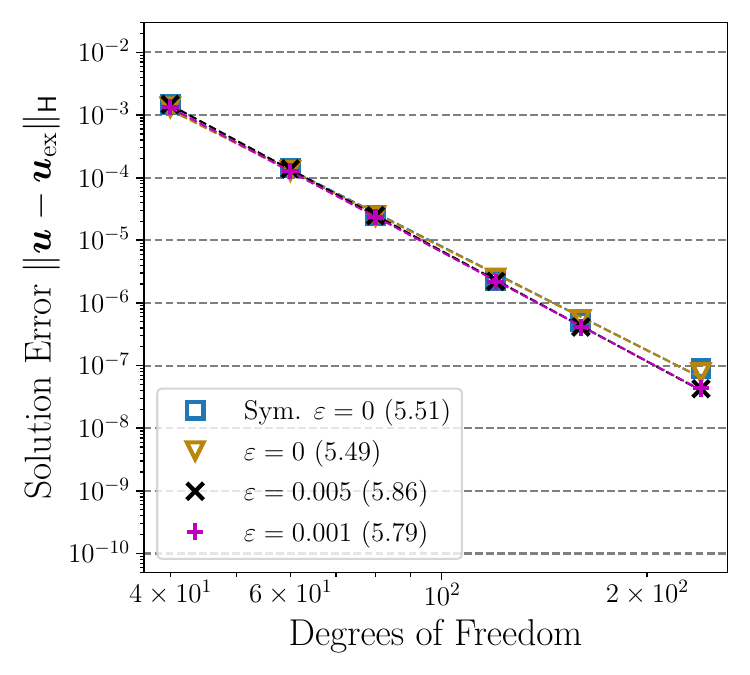}
        \caption{HGT $p=3$}
    \end{subfigure}
    \hfill
    \begin{subfigure}[t]{0.32\textwidth}
        \centering
        \includegraphics[width=\textwidth, trim={10 10 10 10}, clip]{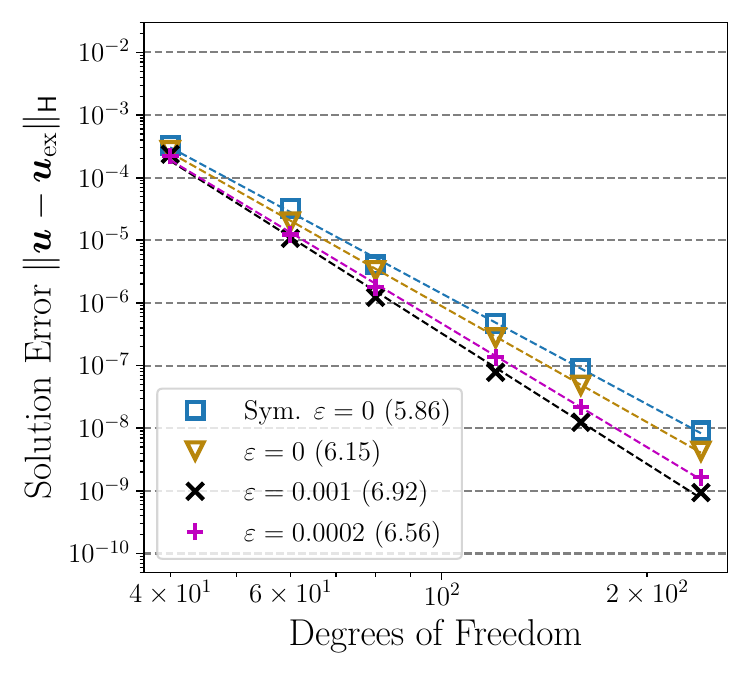}
        \caption{HGT $p=4$}
    \end{subfigure}
    \vskip\baselineskip 
    \begin{subfigure}[t]{0.32\textwidth}
        \centering
        \includegraphics[width=\textwidth, trim={10 10 10 10}, clip]{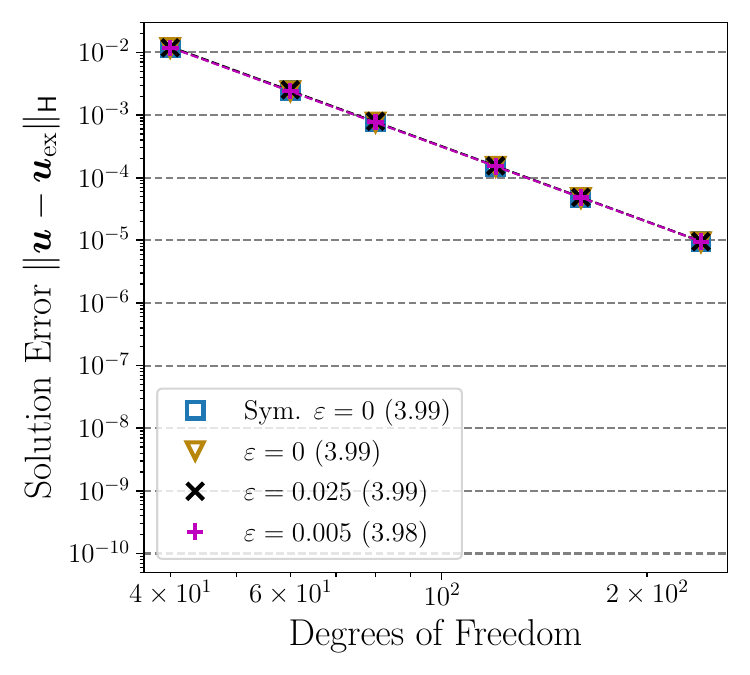}
        \caption{Mattsson $p=2$}
    \end{subfigure}
    \hfill
    \begin{subfigure}[t]{0.32\textwidth}
        \centering
        \includegraphics[width=\textwidth, trim={10 10 10 10}, clip]{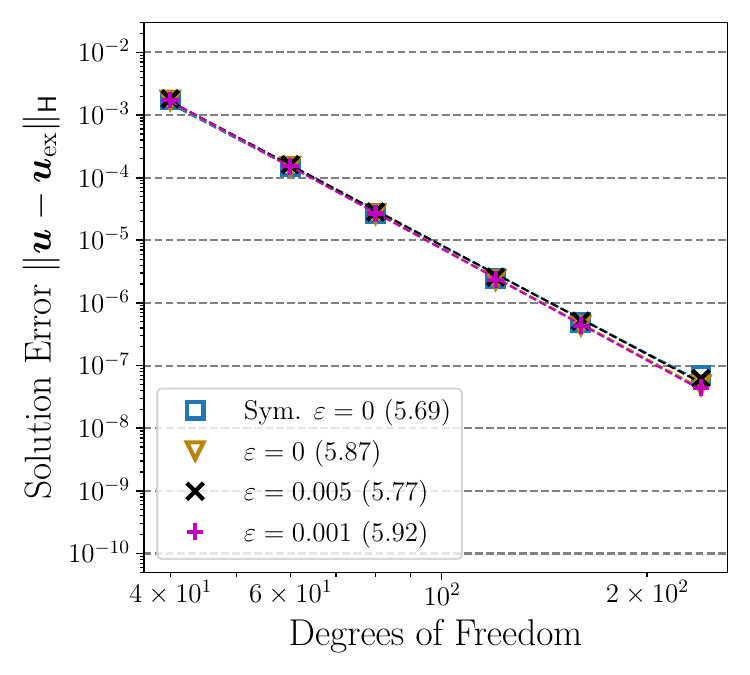}
        \caption{Mattsson $p=3$}
    \end{subfigure}
    \hfill
    \begin{subfigure}[t]{0.32\textwidth}
        \centering
        \includegraphics[width=\textwidth, trim={10 10 10 10}, clip]{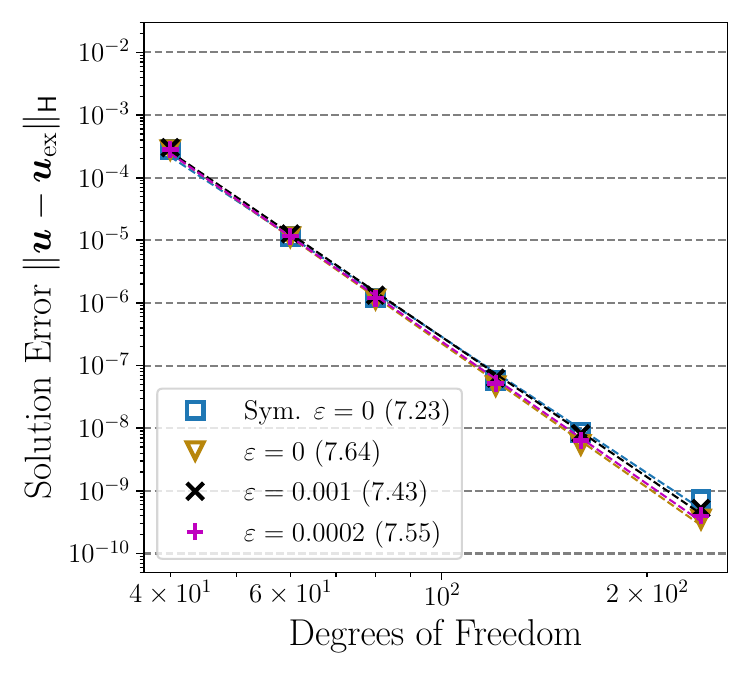}
        \caption{Mattsson $p=4$}
    \end{subfigure}
    \caption{Solution error convergence after one and a half periods ($t=1.5$) of the linear convection equation for UFD SBP semi-discretizations of \cite{Mattsson2017} and central SBP semi-discretizations with either symmetric (\ie, no dissipation) or upwind SATs augmented with dissipation $\mat{A}_\mat{D} = -\varepsilon \HI \tilde{\mat{D}}_s^\T \mat{B} \tilde{\mat{D}}_s$ with $s=p+1$ and varying~$\varepsilon$. A grid of 40, 60, 80, 120, 160, and 240 nodes is used. Solution error convergence rates are given in the legends.}
    \label{fig:LCEconv_additional_t15}
\end{figure}

\begin{figure}[t] 
    \centering
    \begin{subfigure}[t]{0.32\textwidth}
        \centering
        \includegraphics[width=\textwidth, trim={10 10 10 10}, clip]{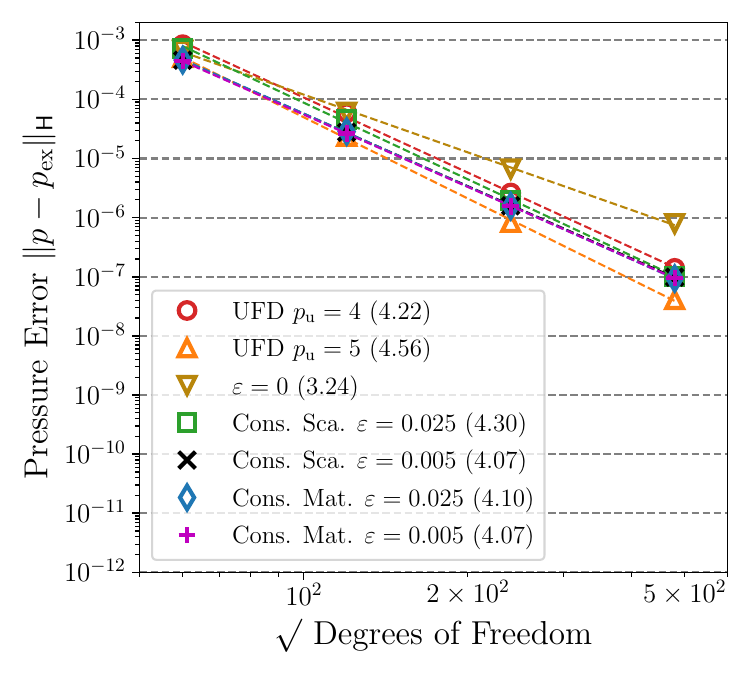}
        \caption{CSBP $p=2$}
    \end{subfigure}
    \hfill
    \begin{subfigure}[t]{0.32\textwidth}
        \centering
        \includegraphics[width=\textwidth, trim={10 10 10 10}, clip]{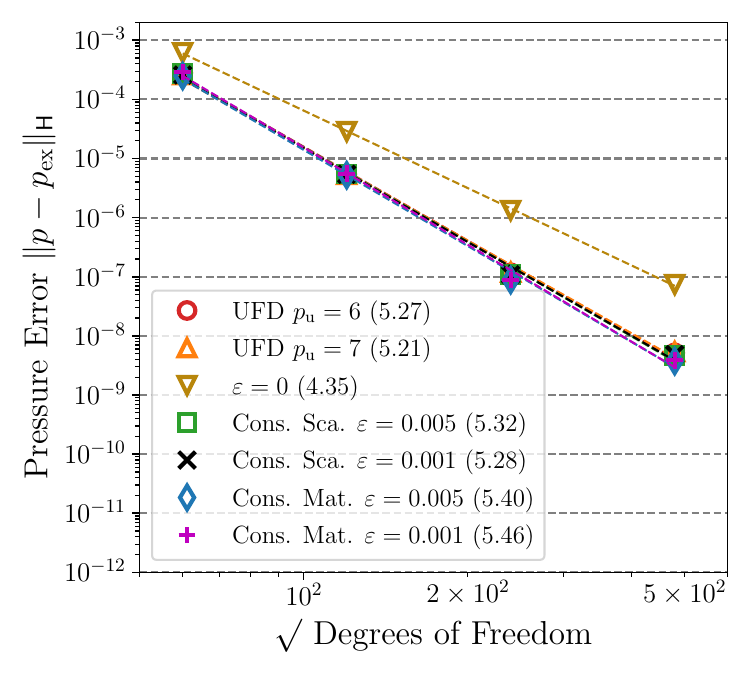}
        \caption{CSBP $p=3$}
    \end{subfigure}
    \hfill
    \begin{subfigure}[t]{0.32\textwidth}
        \centering
        \includegraphics[width=\textwidth, trim={10 10 10 10}, clip]{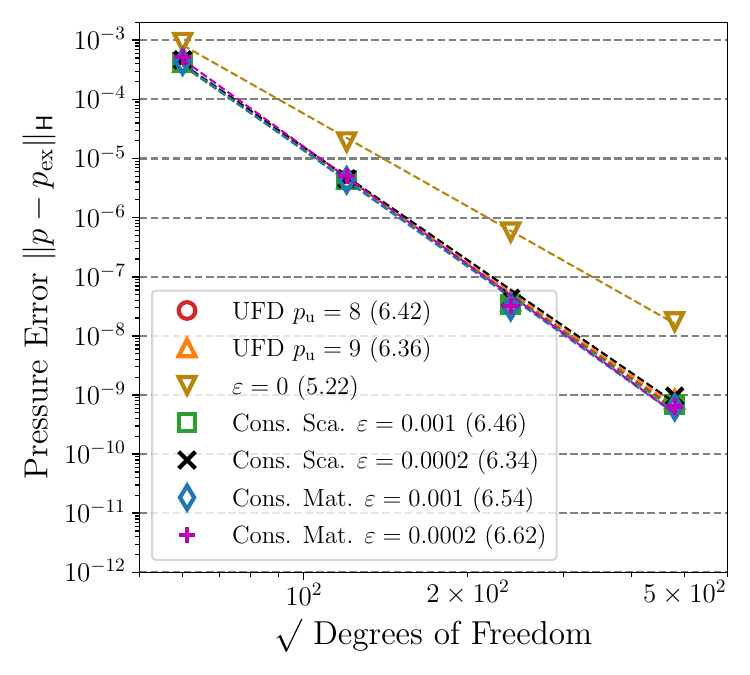}
        \caption{CSBP $p=4$}
    \end{subfigure}
    \vskip\baselineskip 
    \begin{subfigure}[t]{0.32\textwidth}
        \centering
        \includegraphics[width=\textwidth, trim={10 10 10 10}, clip]{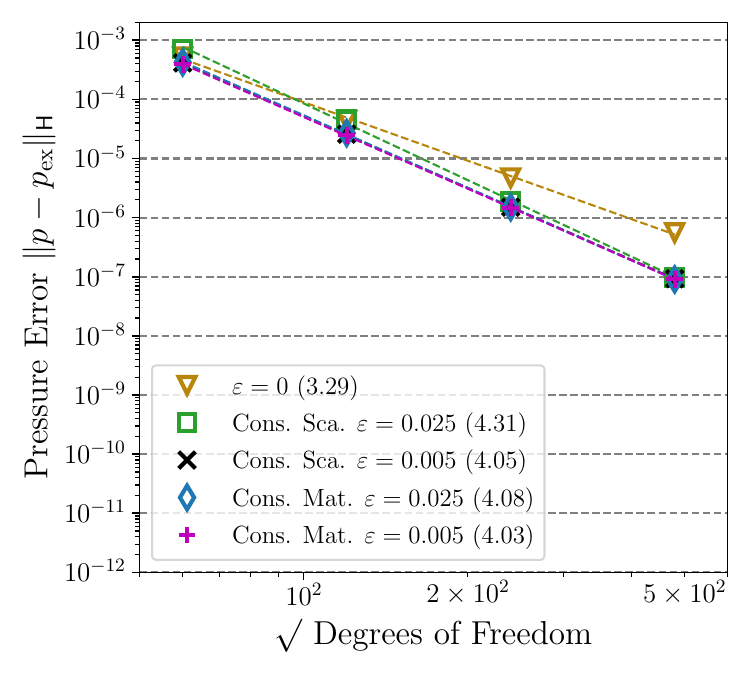}
        \caption{HGTL $p=2$}
    \end{subfigure}
    \hfill
    \begin{subfigure}[t]{0.32\textwidth}
        \centering
        \includegraphics[width=\textwidth, trim={10 10 10 10}, clip]{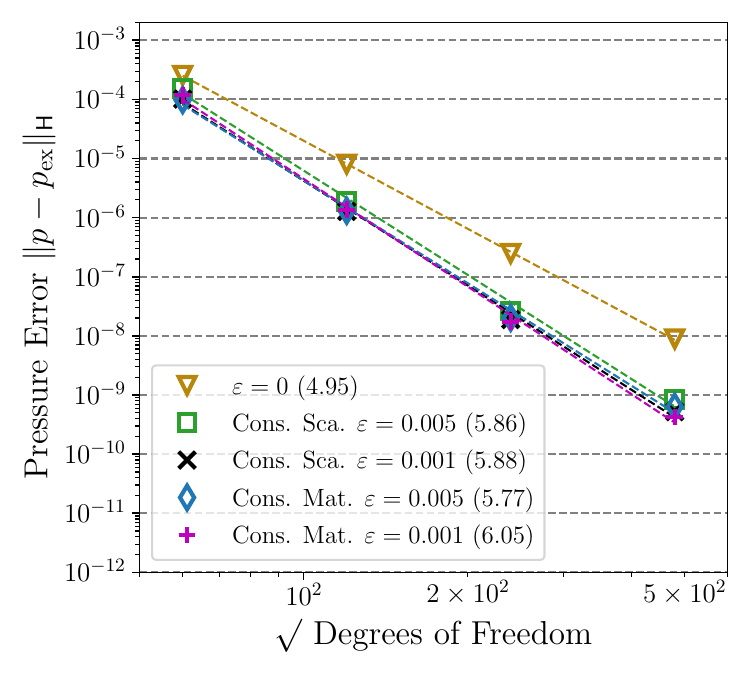}
        \caption{HGTL $p=3$}
    \end{subfigure}
    \hfill
    \begin{subfigure}[t]{0.32\textwidth}
        \centering
        \includegraphics[width=\textwidth, trim={10 10 10 10}, clip]{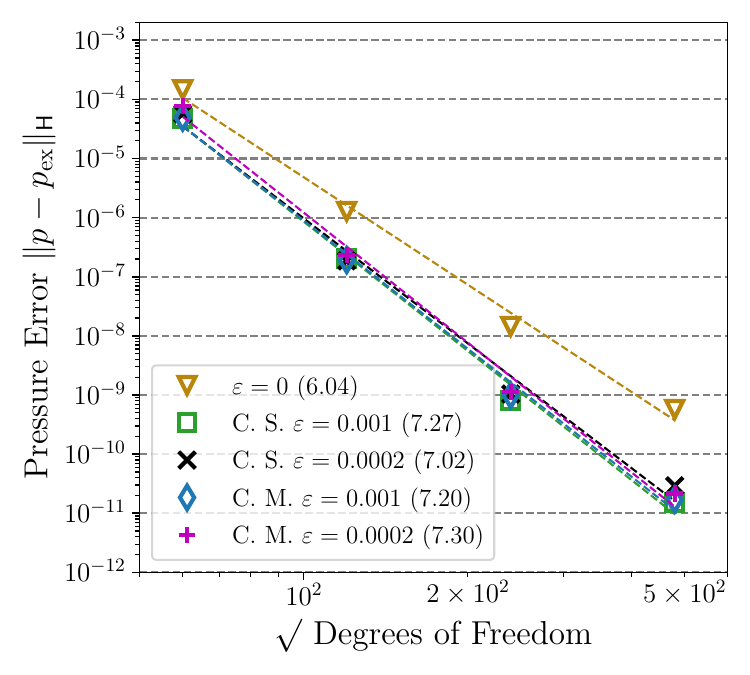}
        \caption{HGTL $p=4$}
    \end{subfigure}
    \vskip\baselineskip 
    \begin{subfigure}[t]{0.32\textwidth}
        \centering
        \includegraphics[width=\textwidth, trim={10 10 10 10}, clip]{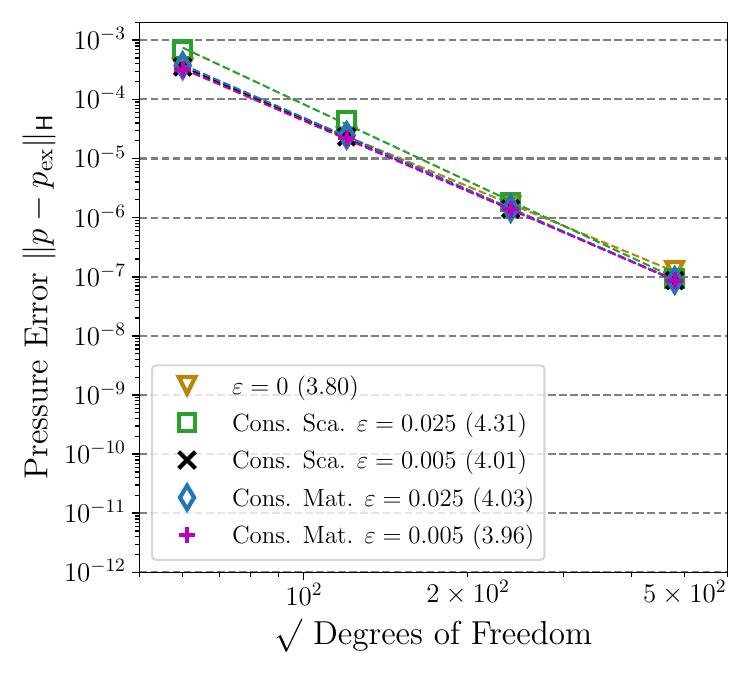}
        \caption{HGT $p=2$}
    \end{subfigure}
    \hfill
    \begin{subfigure}[t]{0.32\textwidth}
        \centering
        \includegraphics[width=\textwidth, trim={10 10 10 10}, clip]{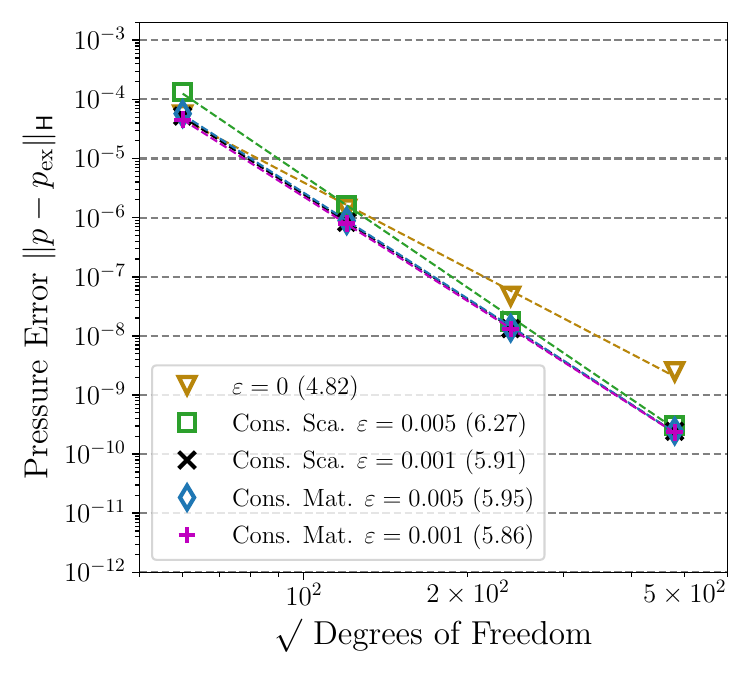}
        \caption{HGT $p=3$}
    \end{subfigure}
    \hfill
    \begin{subfigure}[t]{0.32\textwidth}
        \centering
        \includegraphics[width=\textwidth, trim={10 10 10 10}, clip]{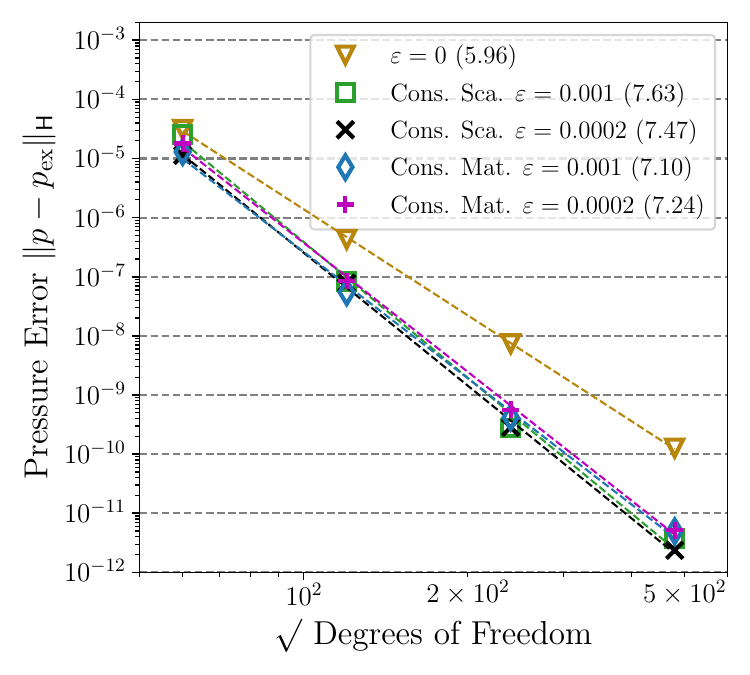}
        \caption{HGT $p=4$}
    \end{subfigure}
    \vskip\baselineskip 
    \begin{subfigure}[t]{0.32\textwidth}
        \centering
        \includegraphics[width=\textwidth, trim={10 10 10 10}, clip]{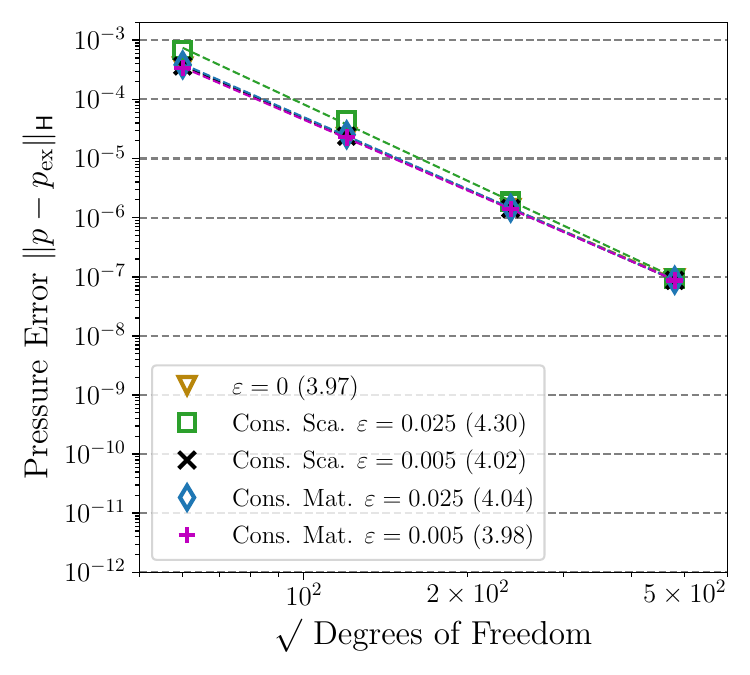}
        \caption{Mattsson $p=2$}
    \end{subfigure}
    \hfill
    \begin{subfigure}[t]{0.32\textwidth}
        \centering
        \includegraphics[width=\textwidth, trim={10 10 10 10}, clip]{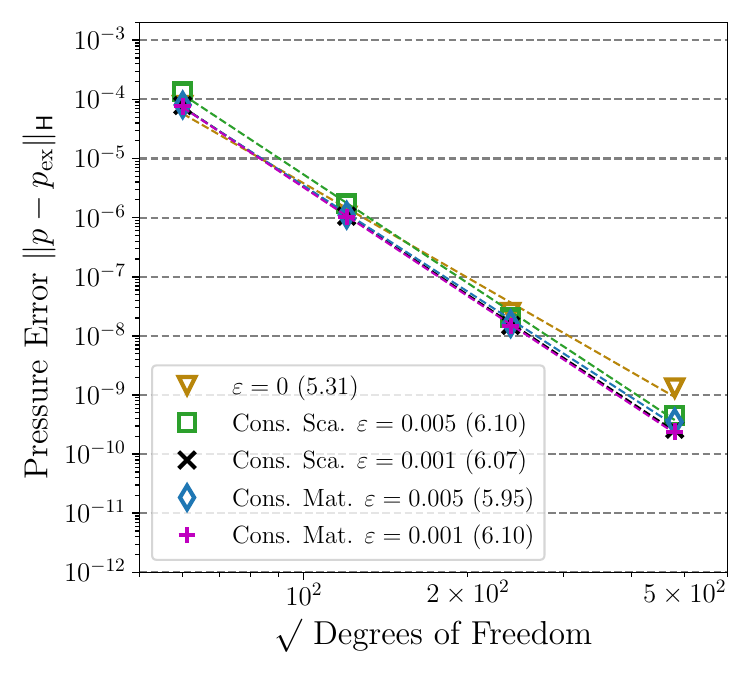}
        \caption{Mattsson $p=3$}
    \end{subfigure}
    \hfill
    \begin{subfigure}[t]{0.32\textwidth}
        \centering
        \includegraphics[width=\textwidth, trim={10 10 10 10}, clip]{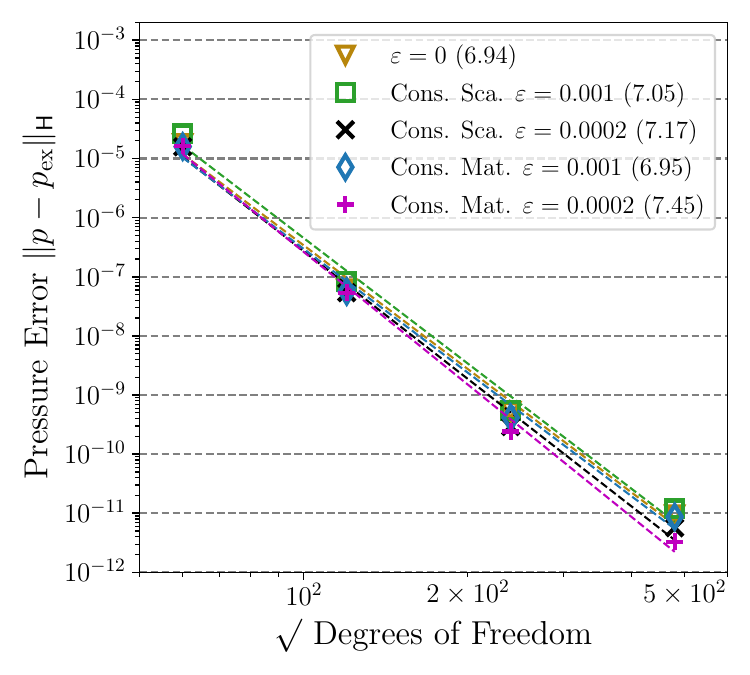}
        \caption{Mattsson $p=4$}
    \end{subfigure}
    \caption{Pressure error after one period of the isentropic vortex problem for the 2D Euler equations using UFD SBP discretizations of \cite{Mattsson2017} and SBP central discretizations (\eg, \cite{hicken2008}). The latter is augmented with volume dissipation \eqnref{eq:2d_diss}, either as scalar or matrix formulations acting on conservative variables with $s=p+1$. Three blocks in each direction of $20^2$, $40^2$, $80^2$, and $160^2$ nodes are used. Convergence rates are given in the legends.}
    \label{fig:Vortexpconv_central_additional}
\end{figure}

\begin{figure}[t] 
    \centering
    \begin{subfigure}[t]{0.32\textwidth}
        \centering
        \includegraphics[width=\textwidth, trim={10 10 10 10}, clip]{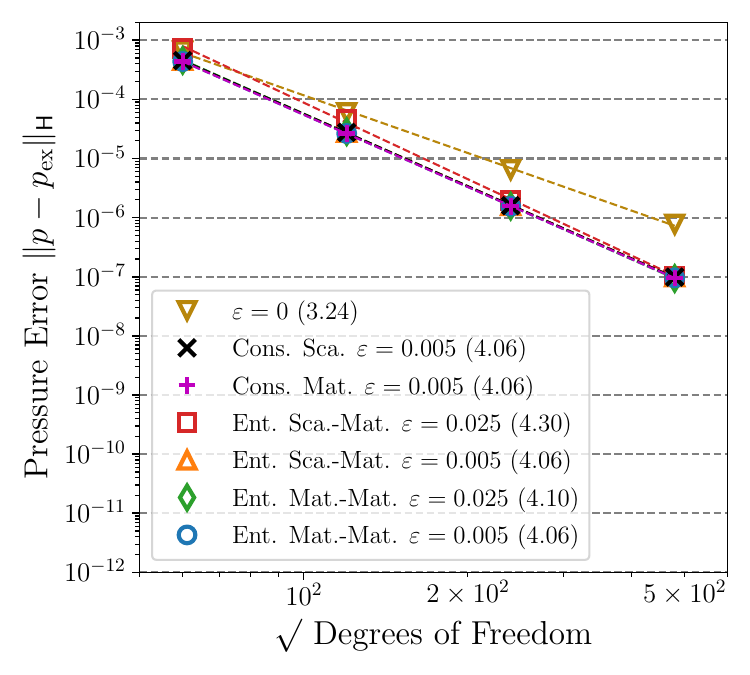}
        \caption{CSBP $p=2$}
    \end{subfigure}
    \hfill
    \begin{subfigure}[t]{0.32\textwidth}
        \centering
        \includegraphics[width=\textwidth, trim={10 10 10 10}, clip]{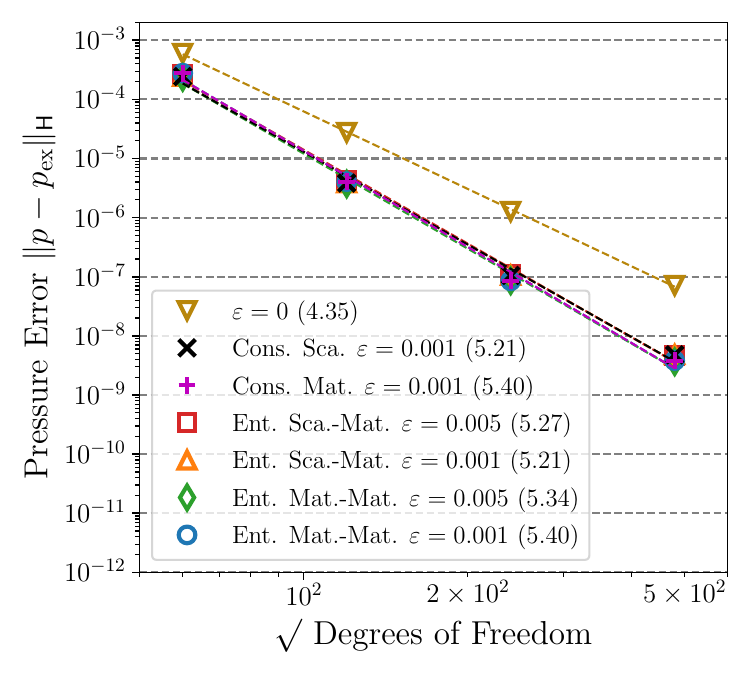}
        \caption{CSBP $p=3$}
    \end{subfigure}
    \hfill
    \begin{subfigure}[t]{0.32\textwidth}
        \centering
        \includegraphics[width=\textwidth, trim={10 10 10 10}, clip]{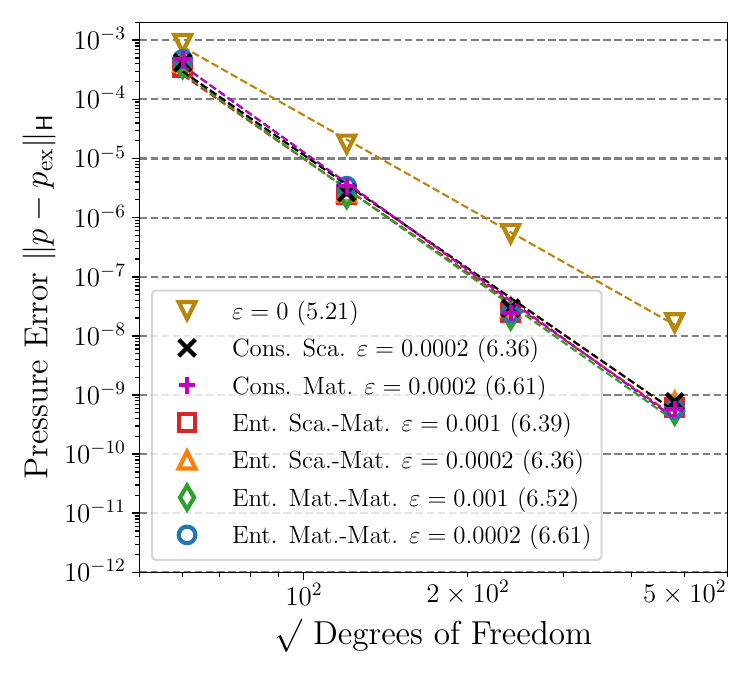}
        \caption{CSBP $p=4$}
    \end{subfigure}
    \vskip\baselineskip 
    \begin{subfigure}[t]{0.32\textwidth}
        \centering
        \includegraphics[width=\textwidth, trim={10 10 10 10}, clip]{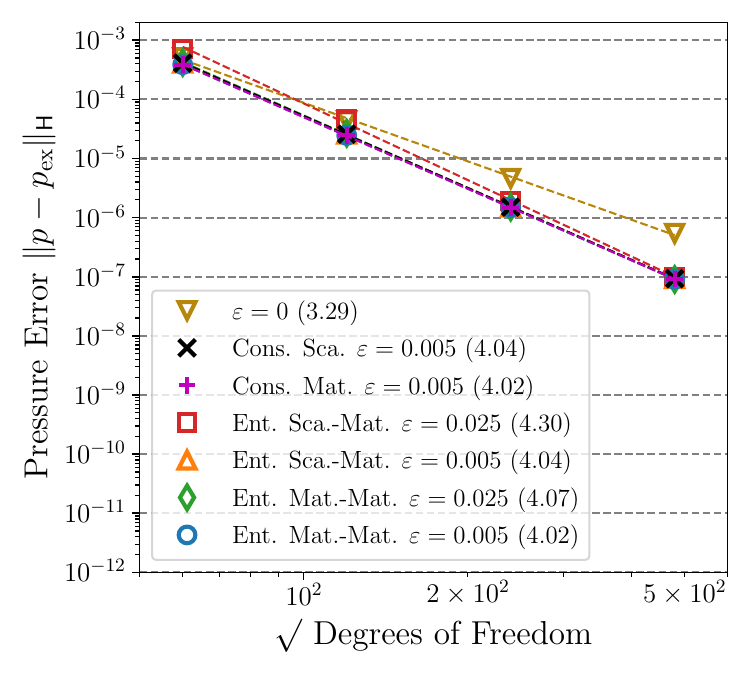}
        \caption{HGTL $p=2$}
    \end{subfigure}
    \hfill
    \begin{subfigure}[t]{0.32\textwidth}
        \centering
        \includegraphics[width=\textwidth, trim={10 10 10 10}, clip]{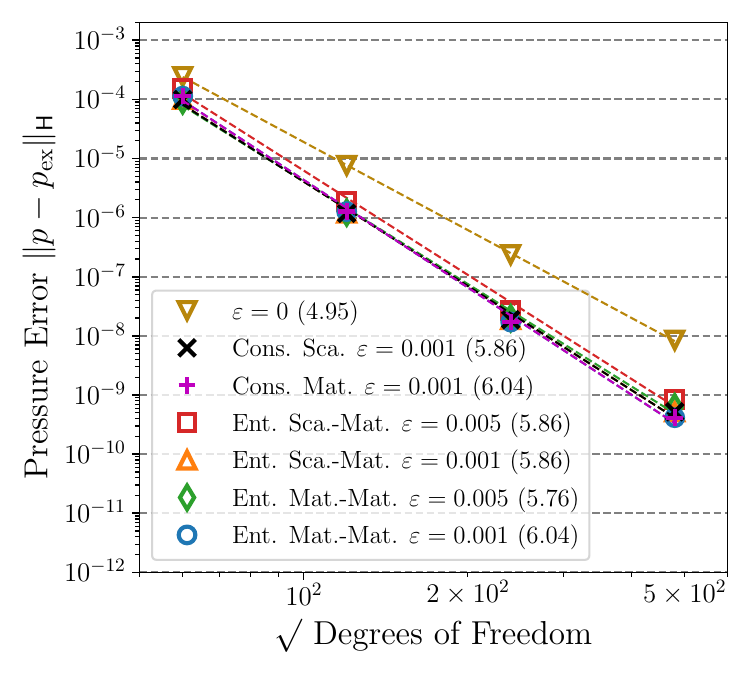}
        \caption{HGTL $p=3$}
    \end{subfigure}
    \hfill
    \begin{subfigure}[t]{0.32\textwidth}
        \centering
        \includegraphics[width=\textwidth, trim={10 10 10 10}, clip]{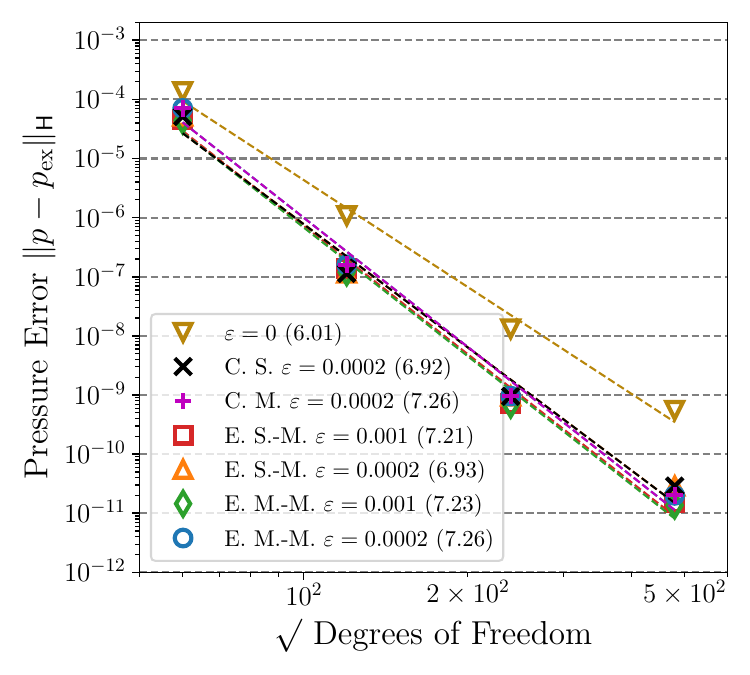}
        \caption{HGTL $p=4$}
    \end{subfigure}
    \vskip\baselineskip 
    \begin{subfigure}[t]{0.32\textwidth}
        \centering
        \includegraphics[width=\textwidth, trim={10 10 10 10}, clip]{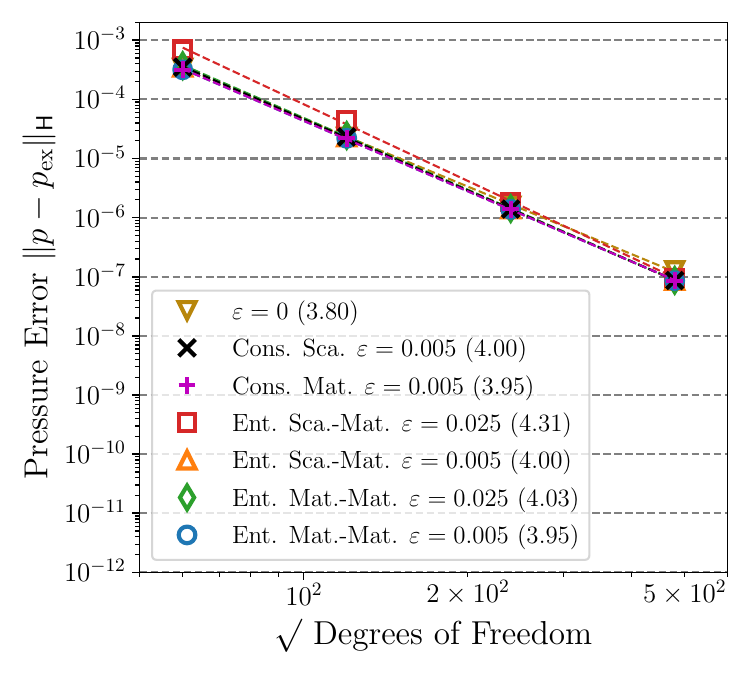}
        \caption{HGT $p=2$}
    \end{subfigure}
    \hfill
    \begin{subfigure}[t]{0.32\textwidth}
        \centering
        \includegraphics[width=\textwidth, trim={10 10 10 10}, clip]{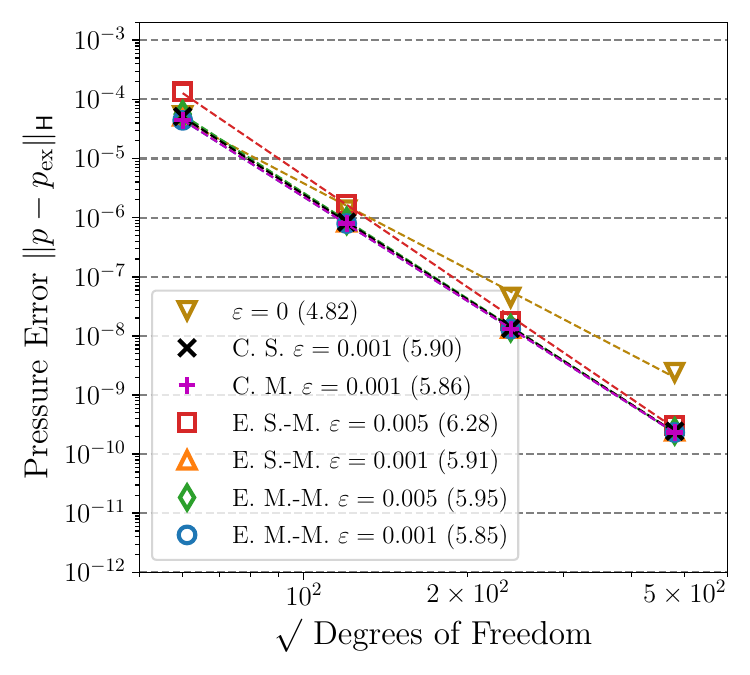}
        \caption{HGT $p=3$}
    \end{subfigure}
    \hfill
    \begin{subfigure}[t]{0.32\textwidth}
        \centering
        \includegraphics[width=\textwidth, trim={10 10 10 10}, clip]{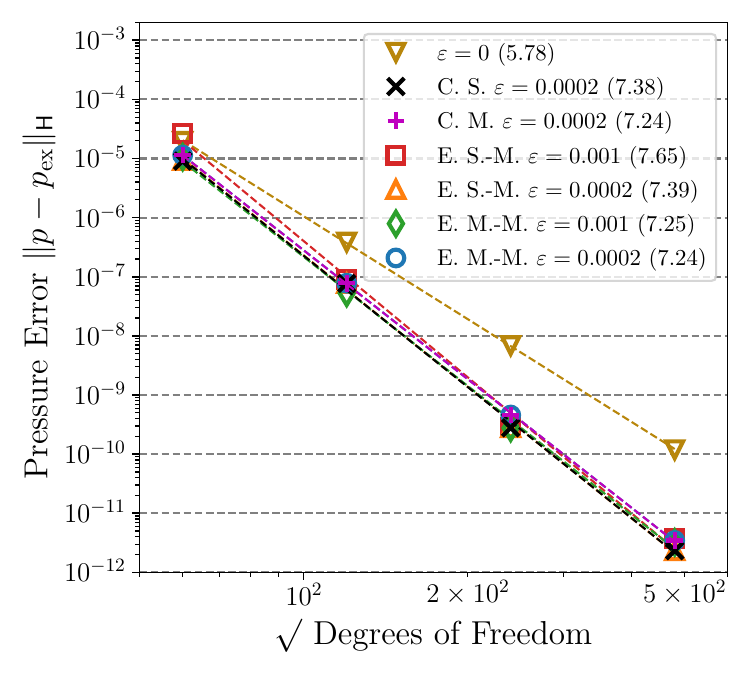}
        \caption{HGT $p=4$}
    \end{subfigure}
    \vskip\baselineskip 
    \begin{subfigure}[t]{0.32\textwidth}
        \centering
        \includegraphics[width=\textwidth, trim={10 10 10 10}, clip]{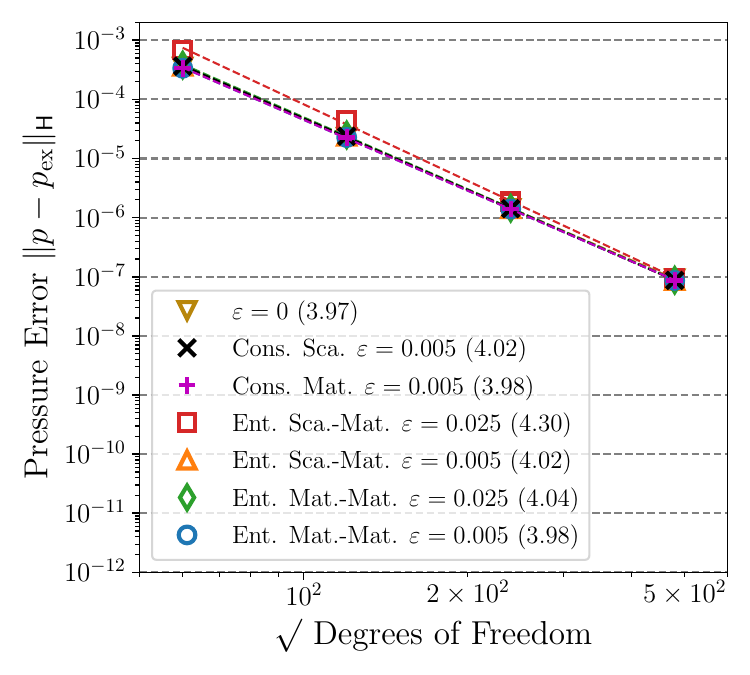}
        \caption{Mattsson $p=2$}
    \end{subfigure}
    \hfill
    \begin{subfigure}[t]{0.32\textwidth}
        \centering
        \includegraphics[width=\textwidth, trim={10 10 10 10}, clip]{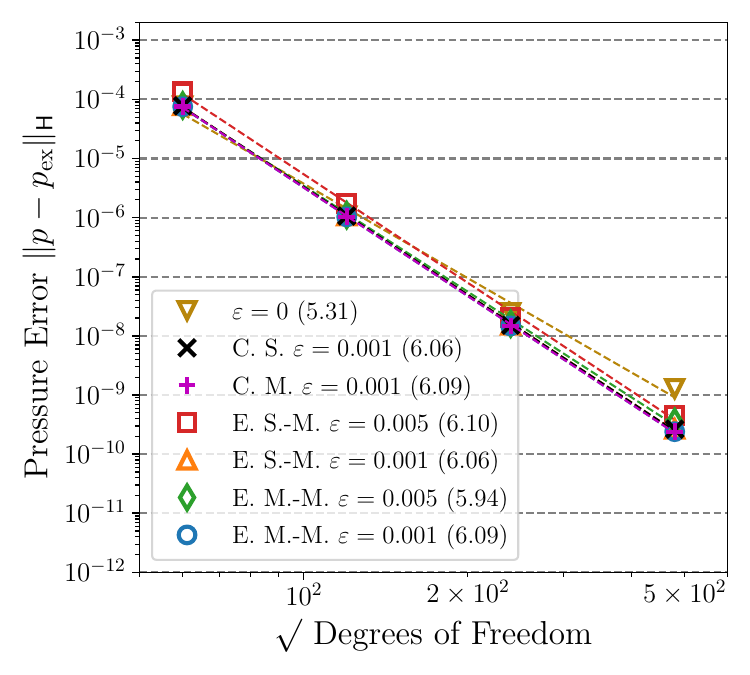}
        \caption{Mattsson $p=3$}
    \end{subfigure}
    \hfill
    \begin{subfigure}[t]{0.32\textwidth}
        \centering
        \includegraphics[width=\textwidth, trim={10 10 10 10}, clip]{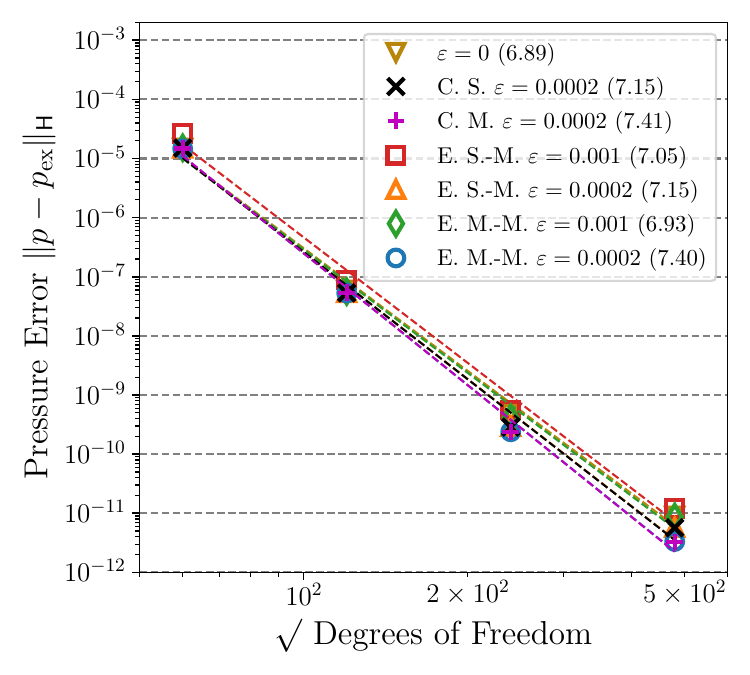}
        \caption{Mattsson $p=4$}
    \end{subfigure}
    \caption{Pressure error after one period of the isentropic vortex problem for the 2D Euler equations using the entropy-stable discretizations of \cite{Fisher_2013, Crean2018} with volume dissipation \eqnref{eq:2d_diss}, either as scalar or matrix formulations acting on conservative variables, or scalar-scalar or scalar-matrix formulations acting on entropy variables with $s=p+1$. Three blocks in each direction of $20^2$, $40^2$, $80^2$, and $160^2$ nodes are used. Convergence rates are given in the legends.}
    \label{fig:Vortexpconv_had_additional}
\end{figure}

\begin{figure}[t] 
    \centering
    \begin{subfigure}[t]{0.32\textwidth}
        \centering
        \includegraphics[width=\textwidth, trim={10 10 10 10}, clip]{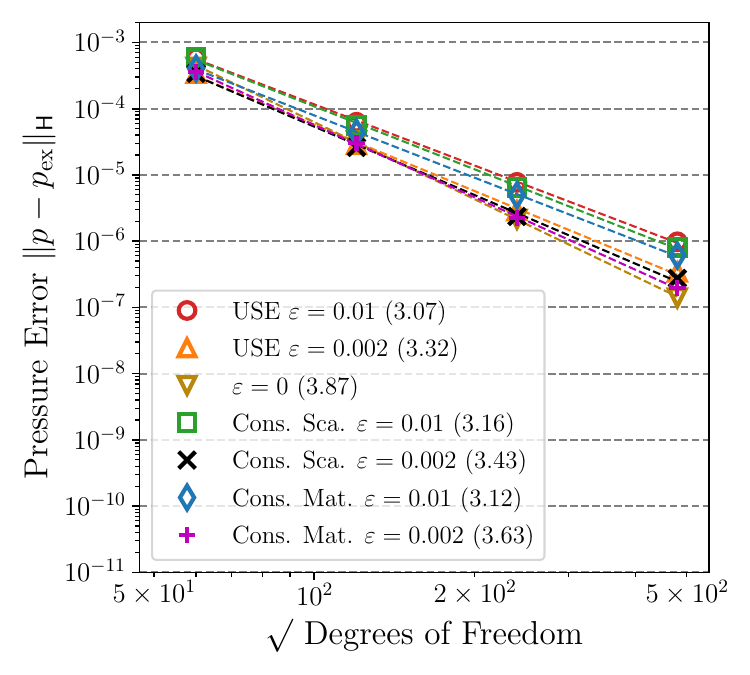}
        \caption{LGL $p=3$}
    \end{subfigure}
    \hfill
    \begin{subfigure}[t]{0.32\textwidth}
        \centering
        \includegraphics[width=\textwidth, trim={10 10 10 10}, clip]{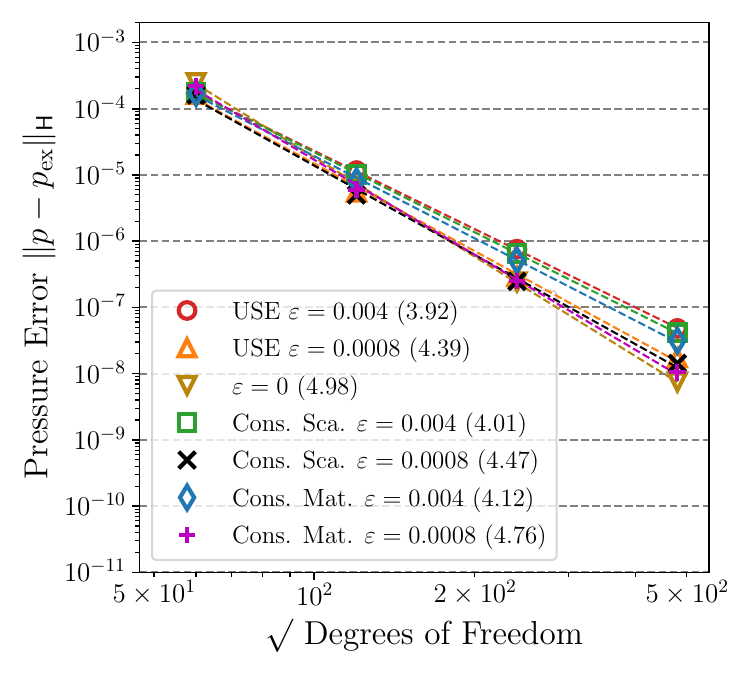}
        \caption{LGL $p=4$}
    \end{subfigure}
    \hfill
    \begin{subfigure}[t]{0.32\textwidth}
        \centering
        \includegraphics[width=\textwidth, trim={10 10 10 10}, clip]{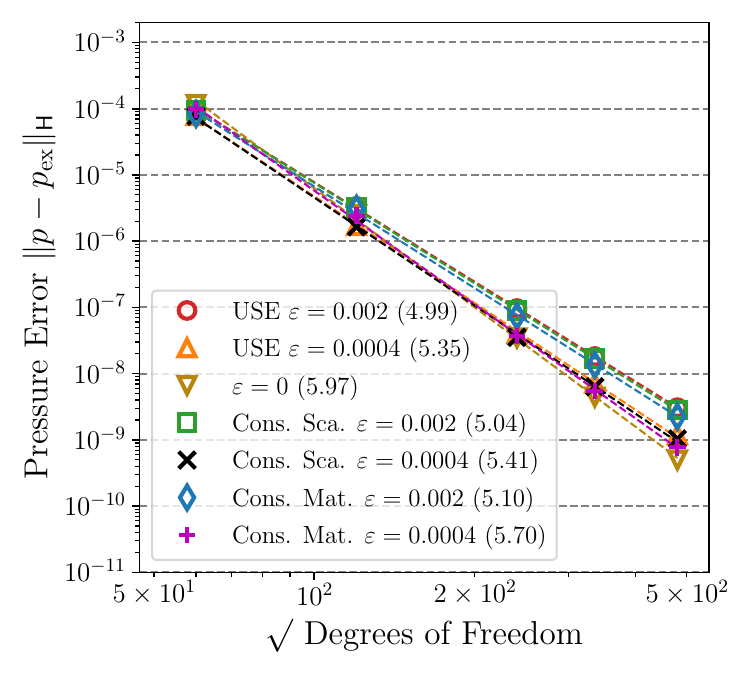}
        \caption{LGL $p=5$}
    \end{subfigure}
    \vskip\baselineskip 
    \begin{subfigure}[t]{0.32\textwidth}
        \centering
        \includegraphics[width=\textwidth, trim={10 10 10 10}, clip]{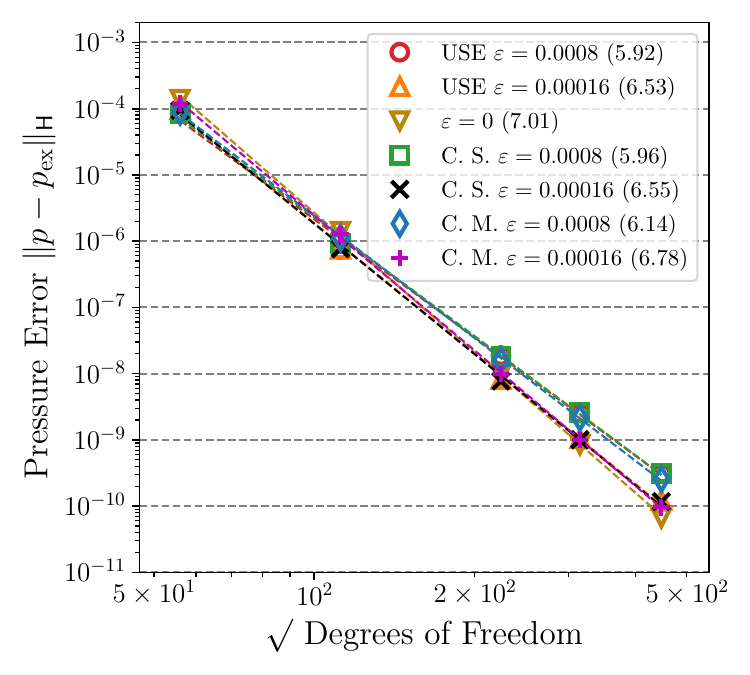}
        \caption{LGL $p=6$}
    \end{subfigure}
    \hfill
    \begin{subfigure}[t]{0.32\textwidth}
        \centering
        \includegraphics[width=\textwidth, trim={10 10 10 10}, clip]{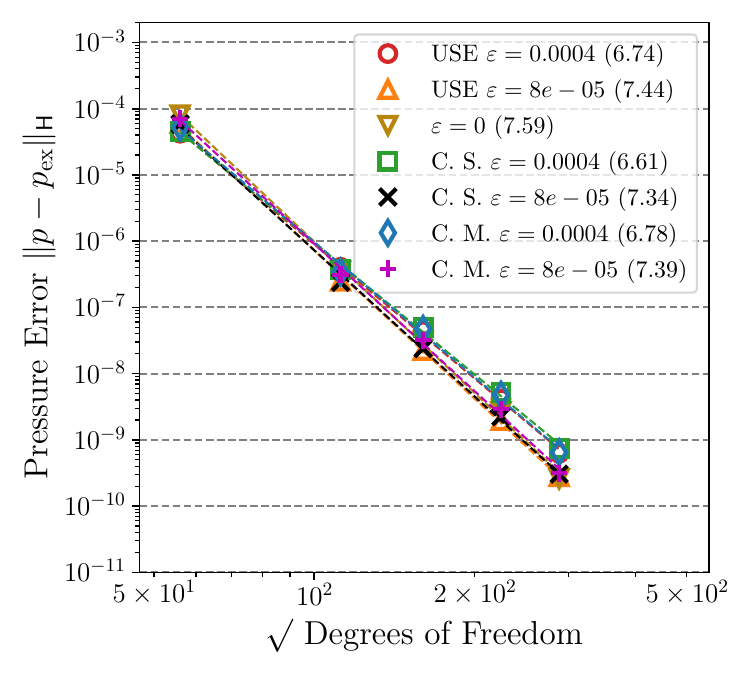}
        \caption{LGL $p=7$}
    \end{subfigure}
    \hfill
    \begin{subfigure}[t]{0.32\textwidth}
        \centering
        \includegraphics[width=\textwidth, trim={10 10 10 10}, clip]{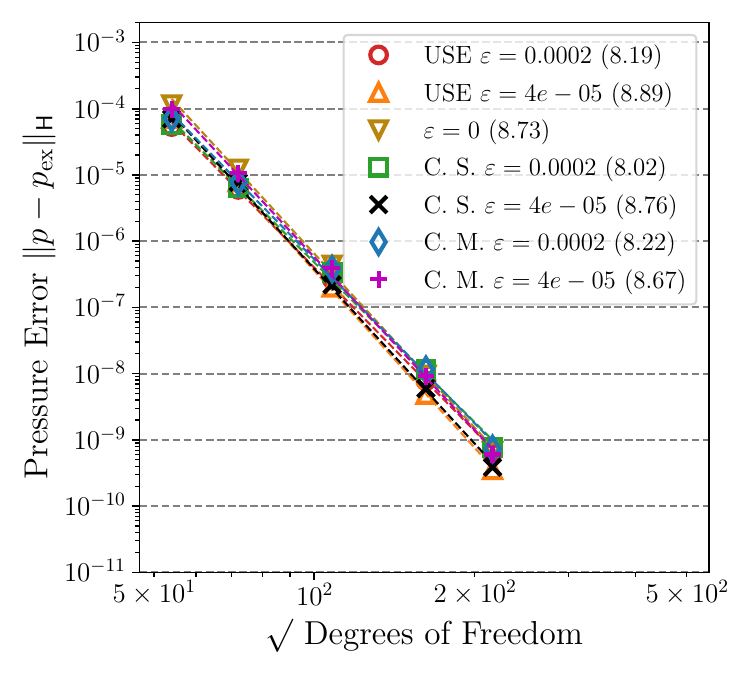}
        \caption{LGL $p=8$}
    \end{subfigure}
    \vskip\baselineskip 
    \begin{subfigure}[t]{0.32\textwidth}
        \centering
        \includegraphics[width=\textwidth, trim={10 10 10 10}, clip]{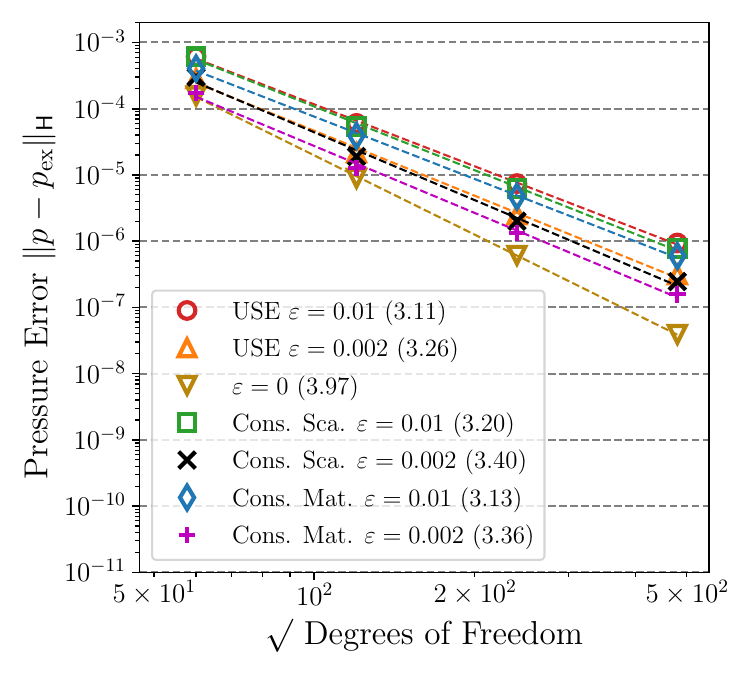}
        \caption{LG $p=3$}
    \end{subfigure}
    \hfill
    \begin{subfigure}[t]{0.32\textwidth}
        \centering
        \includegraphics[width=\textwidth, trim={10 10 10 10}, clip]{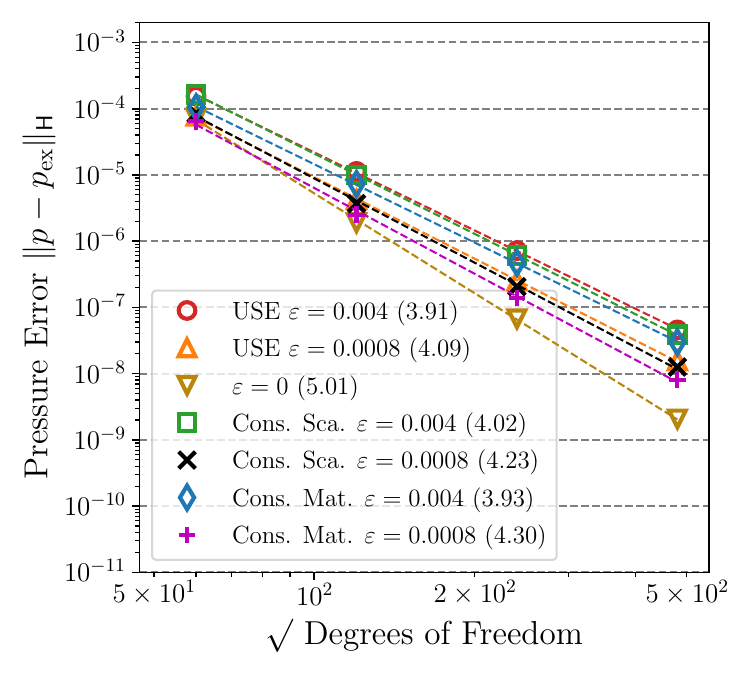}
        \caption{LG $p=4$}
    \end{subfigure}
    \hfill
    \begin{subfigure}[t]{0.32\textwidth}
        \centering
        \includegraphics[width=\textwidth, trim={10 10 10 10}, clip]{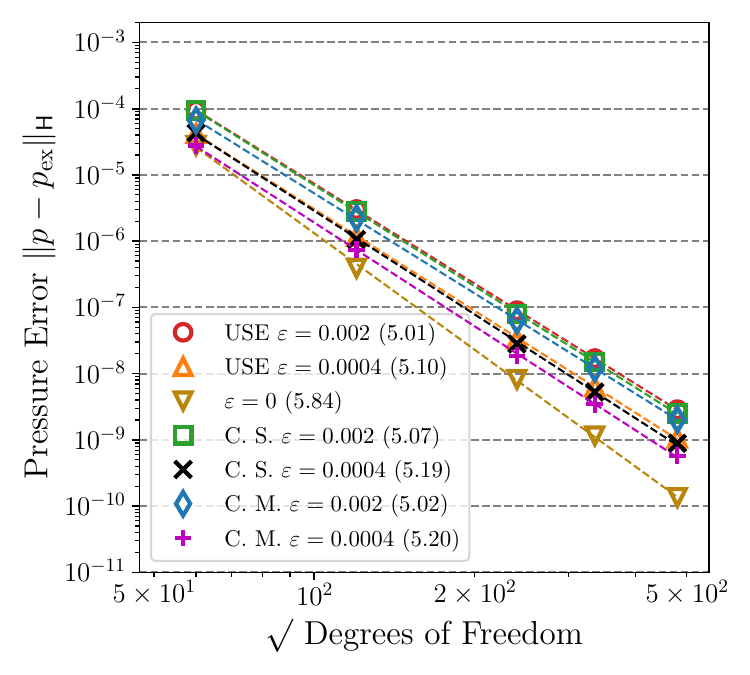}
        \caption{LG $p=5$}
    \end{subfigure}
    \vskip\baselineskip 
    \begin{subfigure}[t]{0.32\textwidth}
        \centering
        \includegraphics[width=\textwidth, trim={10 10 10 10}, clip]{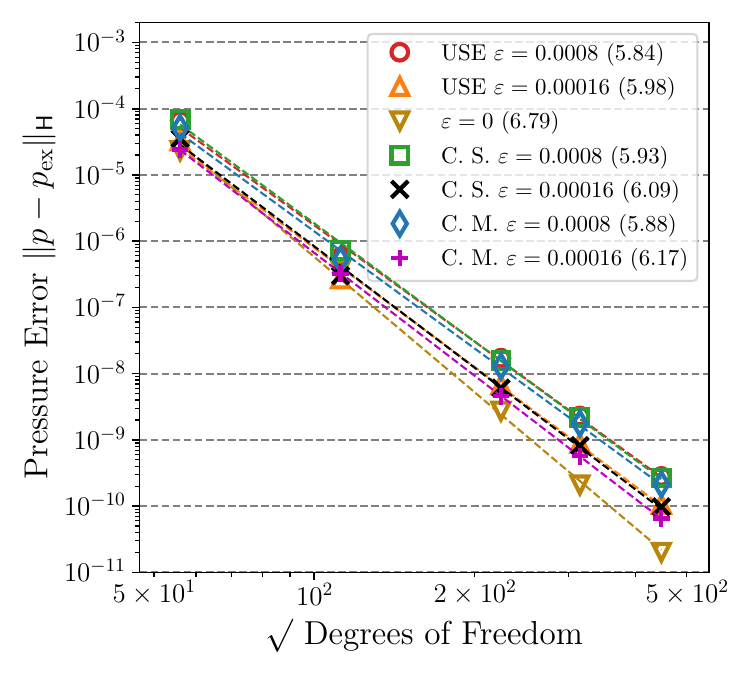}
        \caption{LG $p=6$}
    \end{subfigure}
    \hfill
    \begin{subfigure}[t]{0.32\textwidth}
        \centering
        \includegraphics[width=\textwidth, trim={10 10 10 10}, clip]{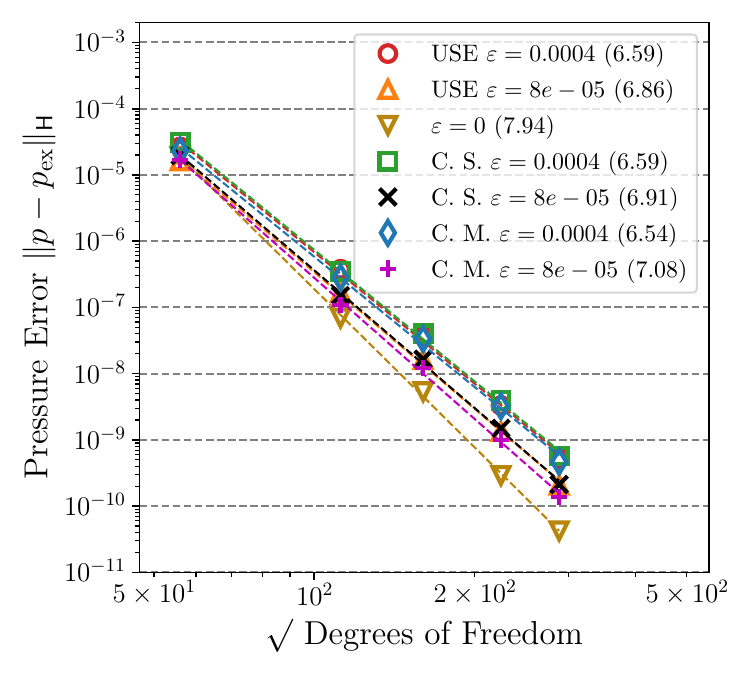}
        \caption{LG $p=7$}
    \end{subfigure}
    \hfill
    \begin{subfigure}[t]{0.32\textwidth}
        \centering
        \includegraphics[width=\textwidth, trim={10 10 10 10}, clip]{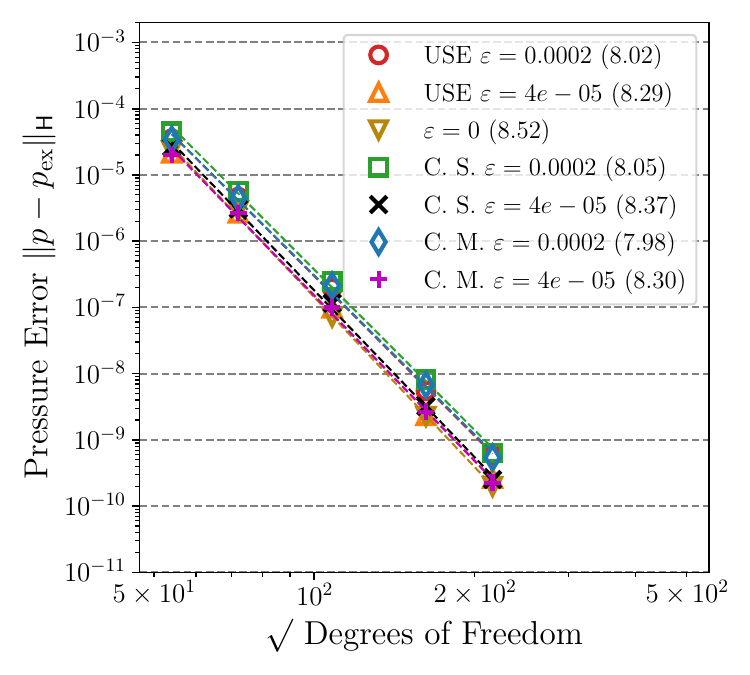}
        \caption{LG $p=8$}
    \end{subfigure}
    \caption{Pressure error after one period of the isentropic vortex problem for the 2D Euler equations using USE SBP discretizations of \cite{Glaubitz2024} and SE SBP central discretizations. The latter is augmented with volume dissipation \eqnref{eq:2d_diss}, either as scalar or matrix formulations acting on conservative variables with $s=p$. Convergence rates are given in the legends.}
    \label{fig:Vortexpconv_central_additional_SE}
\end{figure}

\begin{figure}[t] 
    \centering
    \begin{subfigure}[t]{0.32\textwidth}
        \centering
        \includegraphics[width=\textwidth, trim={10 10 10 10}, clip]{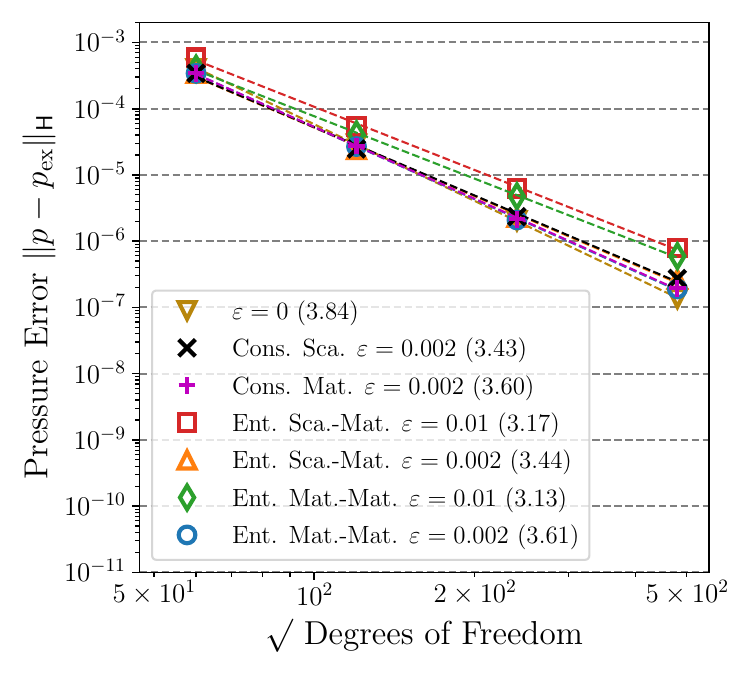}
        \caption{LGL $p=3$}
    \end{subfigure}
    \hfill
    \begin{subfigure}[t]{0.32\textwidth}
        \centering
        \includegraphics[width=\textwidth, trim={10 10 10 10}, clip]{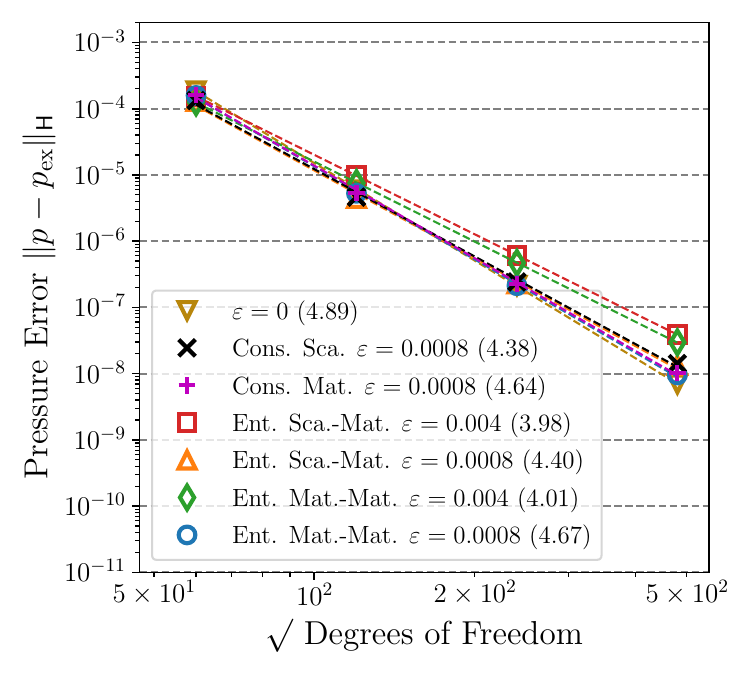}
        \caption{LGL $p=4$}
    \end{subfigure}
    \hfill
    \begin{subfigure}[t]{0.32\textwidth}
        \centering
        \includegraphics[width=\textwidth, trim={10 10 10 10}, clip]{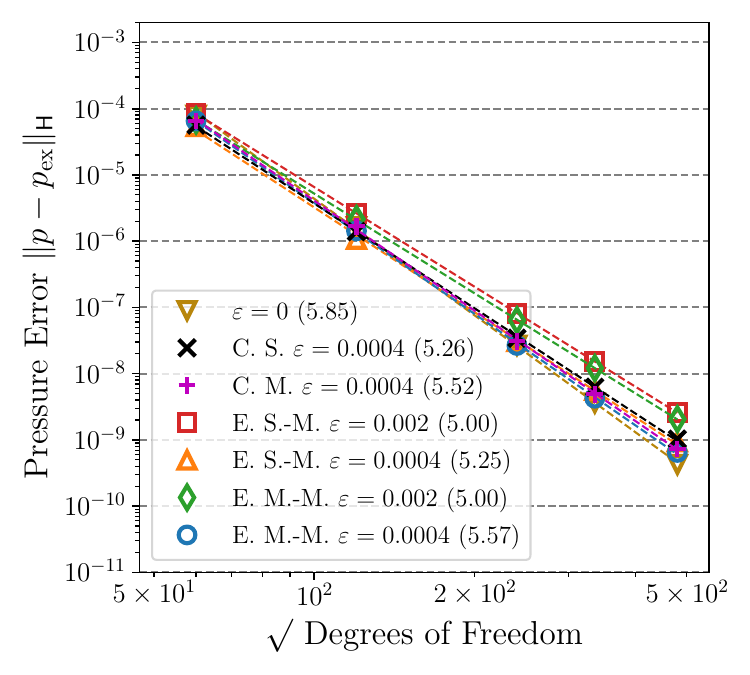}
        \caption{LGL $p=5$}
    \end{subfigure}
    \vskip\baselineskip 
    \begin{subfigure}[t]{0.32\textwidth}
        \centering
        \includegraphics[width=\textwidth, trim={10 10 10 10}, clip]{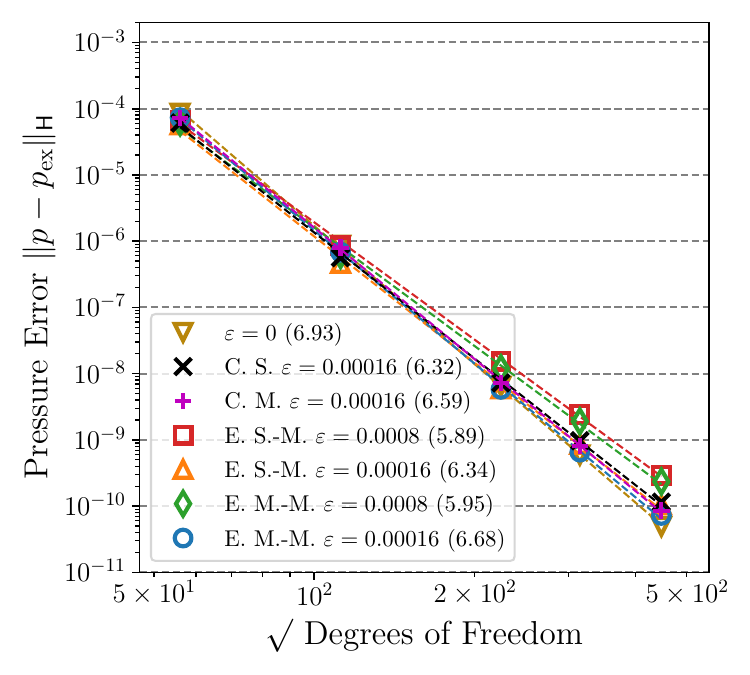}
        \caption{LGL $p=6$}
    \end{subfigure}
    \hfill
    \begin{subfigure}[t]{0.32\textwidth}
        \centering
        \includegraphics[width=\textwidth, trim={10 10 10 10}, clip]{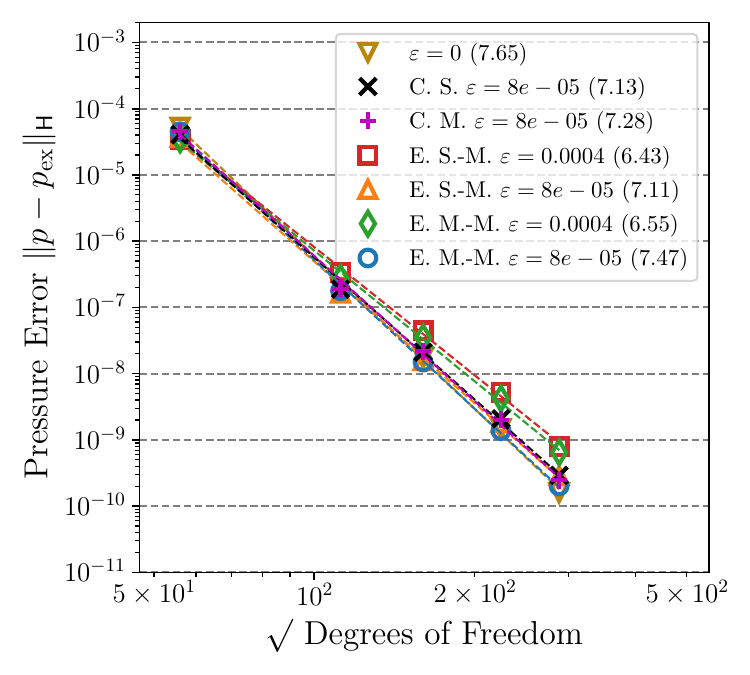}
        \caption{LGL $p=7$}
    \end{subfigure}
    \hfill
    \begin{subfigure}[t]{0.32\textwidth}
        \centering
        \includegraphics[width=\textwidth, trim={10 10 10 10}, clip]{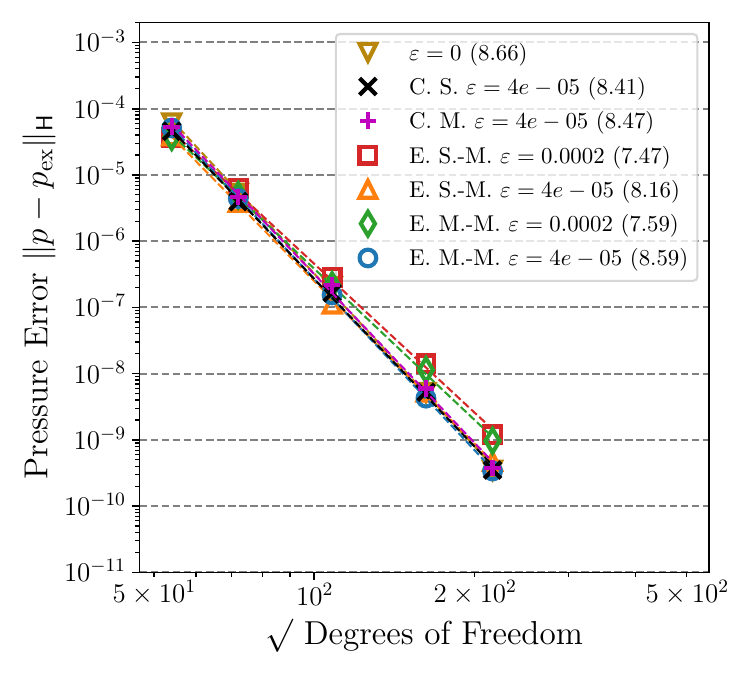}
        \caption{LGL $p=8$}
    \end{subfigure}
    \vskip\baselineskip 
    \begin{subfigure}[t]{0.32\textwidth}
        \centering
        \includegraphics[width=\textwidth, trim={10 10 10 10}, clip]{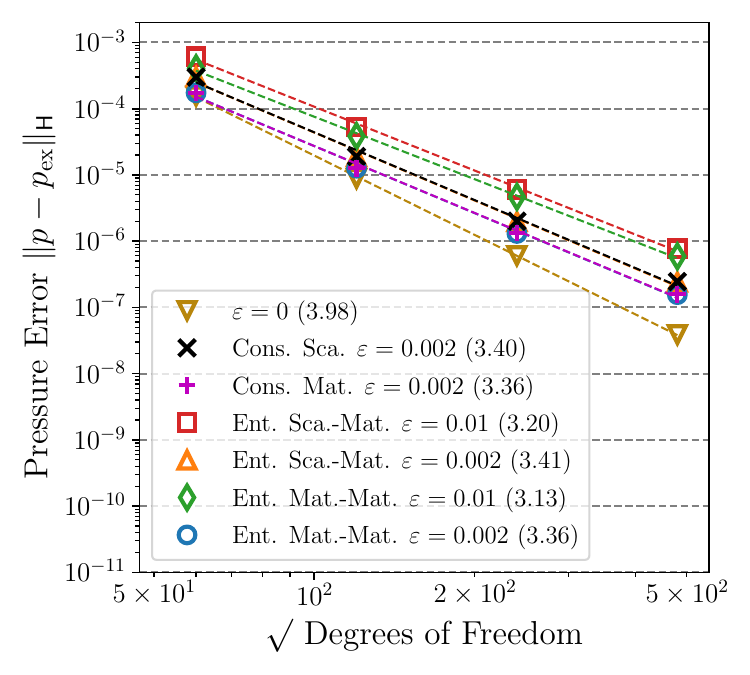}
        \caption{LG $p=3$}
    \end{subfigure}
    \hfill
    \begin{subfigure}[t]{0.32\textwidth}
        \centering
        \includegraphics[width=\textwidth, trim={10 10 10 10}, clip]{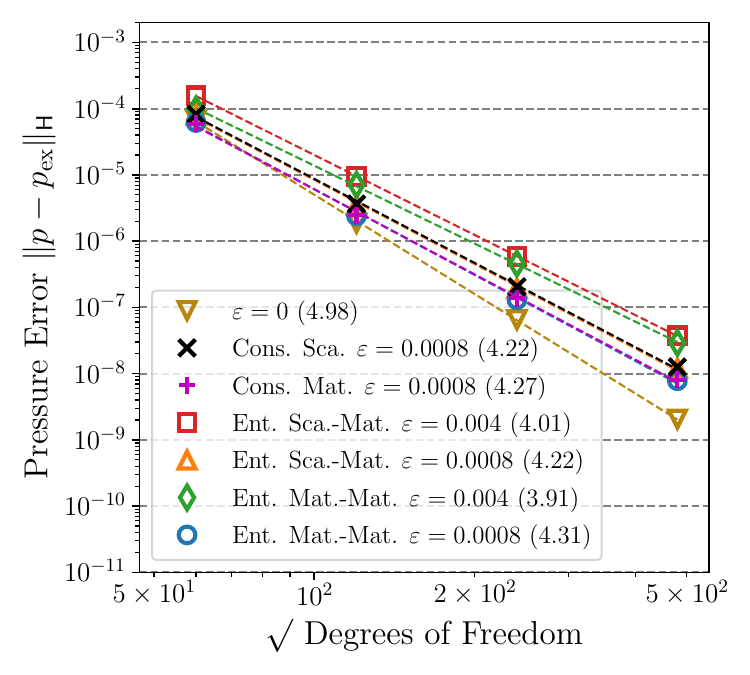}
        \caption{LG $p=4$}
    \end{subfigure}
    \hfill
    \begin{subfigure}[t]{0.32\textwidth}
        \centering
        \includegraphics[width=\textwidth, trim={10 10 10 10}, clip]{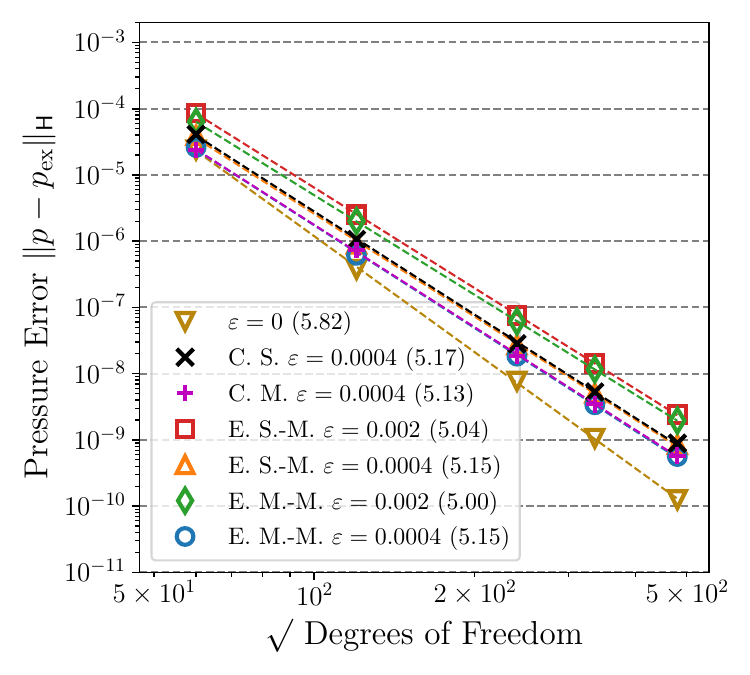}
        \caption{LG $p=5$}
    \end{subfigure}
    \vskip\baselineskip 
    \begin{subfigure}[t]{0.32\textwidth}
        \centering
        \includegraphics[width=\textwidth, trim={10 10 10 10}, clip]{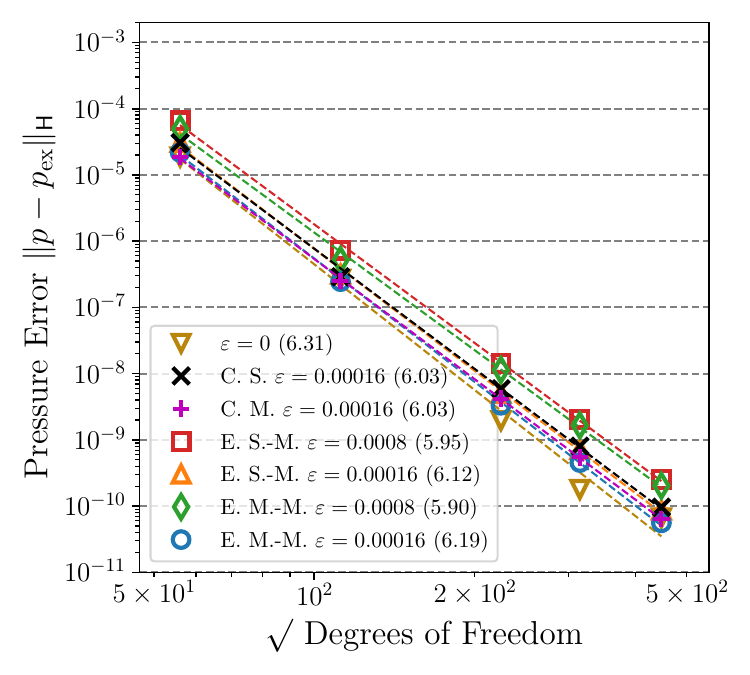}
        \caption{LG $p=6$}
    \end{subfigure}
    \hfill
    \begin{subfigure}[t]{0.32\textwidth}
        \centering
        \includegraphics[width=\textwidth, trim={10 10 10 10}, clip]{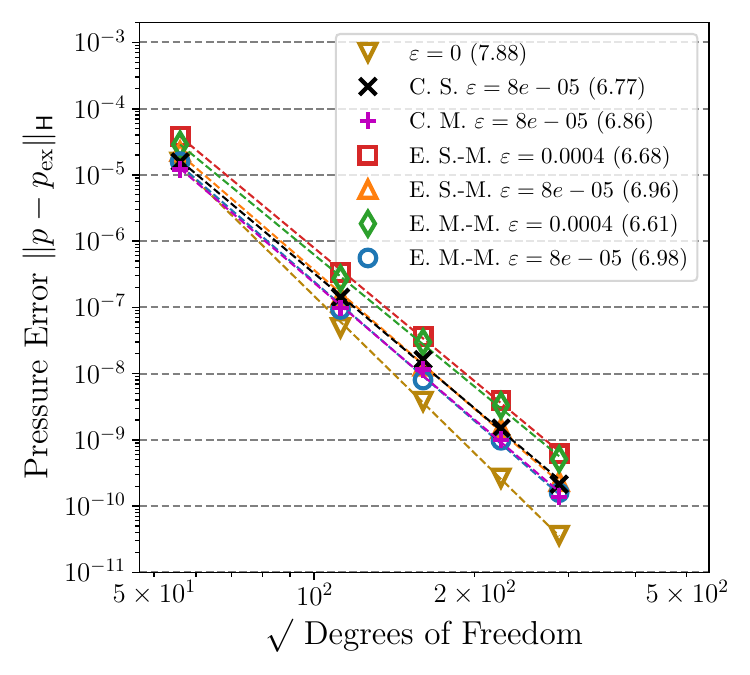}
        \caption{LG $p=7$}
    \end{subfigure}
    \hfill
    \begin{subfigure}[t]{0.32\textwidth}
        \centering
        \includegraphics[width=\textwidth, trim={10 10 10 10}, clip]{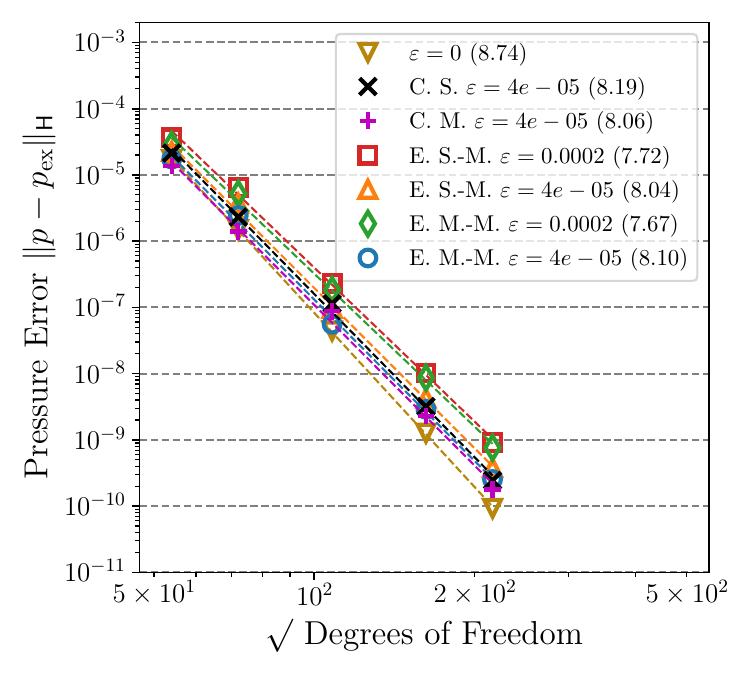}
        \caption{LG $p=8$}
    \end{subfigure}
    \caption{Pressure error after one period of the isentropic vortex problem for the 2D Euler equations using the SE entropy-stable discretizations of \cite{Fisher_2013, Crean2018} with volume dissipation \eqnref{eq:2d_diss}, either as scalar or matrix formulations acting on conservative variables, or scalar-scalar or scalar-matrix formulations acting on entropy variables with $s=p$. Convergence rates are given in the legends.}
    \label{fig:Vortexpconv_had_additional_SE}
\end{figure}

\FloatBarrier

\begin{table}[t]
\centering
\caption{Crash times (up to $t_f = 15$) for the Kelvin-Helmholtz instability described in~\S \ref{sec:KelvinHelmholtz} using CSBP operators with $s=p+1$ and a single block.}
\renewcommand{\arraystretch}{1.2} 

\begin{subtable}[t]{0.48\textwidth}
        \centering
        \caption{Central, $\varepsilon=0$.}
        \begin{tabular}{rccc}
        \hline
DOF & $p=2$ & $p=3$ & $p=4$  \\ \hline
$30^2$ & {\color{BrickRed}1.56} & {\color{BrickRed}1.50} & {\color{BrickRed}1.41} \\
$60^2$ & {\color{BrickRed}1.26} & {\color{BrickRed}1.23} & {\color{BrickRed}1.23} \\
$120^2$ & {\color{BrickRed}1.40} & {\color{BrickRed}1.43} & {\color{BrickRed}1.45} \\
$240^2$ & {\color{BrickRed}1.89} & {\color{BrickRed}1.93} & {\color{BrickRed}1.94} \\
$480^2$ & {\color{BrickRed}2.29} & {\color{BrickRed}2.83} & {\color{BrickRed}2.84} \\
\hline
        \end{tabular}
    \end{subtable}
    \hfill
    \begin{subtable}[t]{0.48\textwidth}
        \centering
        \caption{Entropy-stable, $\varepsilon=0$.}
        \begin{tabular}{rccc}
        \hline
DOF & $p=2$ & $p=3$ & $p=4$  \\ \hline
$30^2$ & {\color{BrickRed}2.92} & {\color{BrickRed}3.50} & {\color{BrickRed}3.67} \\
$60^2$ & {\color{BrickRed}3.43} & {\color{BrickRed}3.21} & {\color{BrickRed}3.04} \\
$120^2$ & {\color{BrickRed}3.10} & {\color{BrickRed}3.18} & {\color{BrickRed}3.25} \\
$240^2$ & {\color{BrickRed}3.61} & {\color{BrickRed}3.61} & {\color{BrickRed}3.63} \\
$480^2$ & {\color{BrickRed}3.69} & {\color{BrickRed}3.66} & {\color{BrickRed}3.50} \\
\hline
        \end{tabular}
    \end{subtable}

\vskip\baselineskip 

\begin{subtable}[t]{0.48\textwidth}
        \centering
        \caption{Central, matrix \\
        \phantom{(c)} $\varepsilon=3.125 \times 5^{-s}$.}
        \begin{tabular}{rccc}
        \hline
DOF & $p=2$ & $p=3$ & $p=4$  \\ \hline
$30^2$ & 15.0 & 15.0 & {\color{BrickRed}7.58} \\
$60^2$ & {\color{BrickRed}4.68} & {\color{BrickRed}3.75} & {\color{BrickRed}3.63} \\
$120^2$ & {\color{BrickRed}4.50} & {\color{BrickRed}4.35} & {\color{BrickRed}4.16} \\
$240^2$ & {\color{BrickRed}5.13} & {\color{BrickRed}3.61} & {\color{BrickRed}3.59} \\
$480^2$ & {\color{BrickRed}4.57} & {\color{BrickRed}4.14} & {\color{BrickRed}3.53} \\
\hline
        \end{tabular}
    \end{subtable}
    \hfill
    \begin{subtable}[t]{0.48\textwidth}
        \centering
        \captionsetup{justification=centering}
        \caption{Entropy-stable, matrix-matrix \\
        \phantom{(d)} $\varepsilon=3.125 \times 5^{-s}$.}
        \begin{tabular}{rccc}
        \hline
DOF & $p=2$ & $p=3$ & $p=4$  \\ \hline
$30^2$ & 15.0 & 15.0 & 15.0 \\
$60^2$ & 15.0 & 15.0 & 15.0 \\
$120^2$ & 15.0 & 15.0 & 15.0 \\
$240^2$ & 15.0 & 15.0 & 15.0 \\
$480^2$ & 15.0 & 15.0 & 15.0 \\
\hline
        \end{tabular}
    \end{subtable}

\vskip\baselineskip 

\begin{subtable}[t]{0.48\textwidth}
        \centering
        \caption{Central, matrix \\
        \phantom{(c)} $\varepsilon=0.625 \times 5^{-s}$.}
        \begin{tabular}{rccc}
        \hline
DOF & $p=2$ & $p=3$ & $p=4$  \\ \hline
$30^2$ & {\color{BrickRed}4.19} & {\color{BrickRed}2.58} & {\color{BrickRed}1.70} \\
$60^2$ & {\color{BrickRed}1.61} & {\color{BrickRed}1.51} & {\color{BrickRed}1.35} \\
$120^2$ & {\color{BrickRed}3.73} & {\color{BrickRed}3.69} & {\color{BrickRed}3.23} \\
$240^2$ & {\color{BrickRed}3.57} & {\color{BrickRed}3.53} & {\color{BrickRed}3.54} \\
$480^2$ & {\color{BrickRed}3.54} & {\color{BrickRed}3.37} & {\color{BrickRed}3.49} \\
\hline
        \end{tabular}
    \end{subtable}
    \hfill
    \begin{subtable}[t]{0.48\textwidth}
        \centering
        \captionsetup{justification=centering}
        \caption{Entropy-stable, matrix-matrix \\
        \phantom{(d)} $\varepsilon=0.625 \times 5^{-s}$.}
        \begin{tabular}{rccc}
        \hline
DOF & $p=2$ & $p=3$ & $p=4$  \\ \hline
$30^2$ & 15.0 & 15.0 & 15.0 \\
$60^2$ & 15.0 & 15.0 & 15.0 \\
$120^2$ & 15.0 & 15.0 & 15.0 \\
$240^2$ & 15.0 & 15.0 & 15.0 \\
$480^2$ & 15.0 & 15.0 & 15.0 \\
\hline
        \end{tabular}
    \end{subtable}

\vskip\baselineskip 

\begin{subtable}[t]{0.48\textwidth}
        \centering
        \caption{Central, scalar \\
        \phantom{(c)} $\varepsilon=3.125 \times 5^{-s}$.}
        \begin{tabular}{rccc}
        \hline
DOF & $p=2$ & $p=3$ & $p=4$  \\ \hline
$30^2$ & 15.0 & 15.0 & 15.0 \\
$60^2$ & 15.0 & {\color{BrickRed}4.18} & {\color{BrickRed}4.54} \\
$120^2$ & 15.0 & {\color{BrickRed}4.57} & {\color{BrickRed}4.52} \\
$240^2$ & 15.0 & {\color{BrickRed}5.95} & {\color{BrickRed}4.52} \\
$480^2$ & {\color{BrickRed}4.84} & {\color{BrickRed}4.22} & {\color{BrickRed}4.51} \\
\hline
        \end{tabular}
    \end{subtable}
    \hfill
    \begin{subtable}[t]{0.48\textwidth}
        \centering
        \captionsetup{justification=centering}
        \caption{Entropy-stable, scalar-matrix \\
        \phantom{(d)} $\varepsilon=3.125 \times 5^{-s}$.}
        \begin{tabular}{rccc}
        \hline
DOF & $p=2$ & $p=3$ & $p=4$  \\ \hline
$30^2$ & 15.0 & 15.0 & 15.0 \\
$60^2$ & 15.0 & 15.0 & 15.0 \\
$120^2$ & 15.0 & 15.0 & 15.0 \\
$240^2$ & 15.0 & 15.0 & 15.0 \\
$480^2$ & 15.0 & 15.0 & 15.0 \\
\hline
        \end{tabular}
    \end{subtable}

\vskip\baselineskip 

\begin{subtable}[t]{0.48\textwidth}
        \centering
        \caption{Central, scalar \\
        \phantom{(c)} $\varepsilon=0.625 \times 5^{-s}$.}
        \begin{tabular}{rccc}
        \hline
DOF & $p=2$ & $p=3$ & $p=4$  \\ \hline
$30^2$ & 15.0 & 15.0 & {\color{BrickRed}13.48} \\
$60^2$ & 15.0 & {\color{BrickRed}4.27} & {\color{BrickRed}4.42} \\
$120^2$ & {\color{BrickRed}4.54} & {\color{BrickRed}4.51} & {\color{BrickRed}4.45} \\
$240^2$ & {\color{BrickRed}4.59} & {\color{BrickRed}4.56} & {\color{BrickRed}3.80} \\
$480^2$ & {\color{BrickRed}4.52} & {\color{BrickRed}4.46} & {\color{BrickRed}3.72} \\
\hline
        \end{tabular}
    \end{subtable}
    \hfill
    \begin{subtable}[t]{0.48\textwidth}
        \centering
        \captionsetup{justification=centering}
        \caption{Entropy-stable, scalar-matrix \\
        \phantom{(d)} $\varepsilon=0.625 \times 5^{-s}$.}
        \begin{tabular}{rccc}
        \hline
DOF & $p=2$ & $p=3$ & $p=4$  \\ \hline
$30^2$ & 15.0 & 15.0 & 15.0 \\
$60^2$ & 15.0 & 15.0 & 15.0 \\
$120^2$ & 15.0 & 15.0 & 15.0 \\
$240^2$ & 15.0 & 15.0 & 15.0 \\
$480^2$ & 15.0 & 15.0 & 15.0 \\
\hline
        \end{tabular}
    \end{subtable}

\label{tab:KelvinHelmholtz_csbp}
\end{table}


\begin{table}[t]
\centering
\caption{Crash times (up to $t_f = 15$) for the Kelvin-Helmholtz instability described in~\S \ref{sec:KelvinHelmholtz} using HGTL operators with $s=p+1$ and a single block.}
\renewcommand{\arraystretch}{1.2} 

\begin{subtable}[t]{0.48\textwidth}
        \centering
        \caption{Central, $\varepsilon=0$.}
        \begin{tabular}{rccc}
        \hline
DOF & $p=2$ & $p=3$ & $p=4$  \\ \hline
$30^2$ & {\color{BrickRed}1.56} & {\color{BrickRed}1.50} & {\color{BrickRed}1.50} \\
$60^2$ & {\color{BrickRed}1.26} & {\color{BrickRed}1.23} & {\color{BrickRed}1.23} \\
$120^2$ & {\color{BrickRed}1.40} & {\color{BrickRed}1.43} & {\color{BrickRed}1.45} \\
$240^2$ & {\color{BrickRed}1.89} & {\color{BrickRed}1.93} & {\color{BrickRed}1.94} \\
$480^2$ & {\color{BrickRed}2.29} & {\color{BrickRed}2.83} & {\color{BrickRed}2.84} \\
\hline
        \end{tabular}
    \end{subtable}
    \hfill
    \begin{subtable}[t]{0.48\textwidth}
        \centering
        \caption{Entropy-stable, $\varepsilon=0$.}
        \begin{tabular}{rccc}
        \hline
DOF & $p=2$ & $p=3$ & $p=4$  \\ \hline
$30^2$ & {\color{BrickRed}2.88} & {\color{BrickRed}3.80} & {\color{BrickRed}3.65} \\
$60^2$ & {\color{BrickRed}3.44} & {\color{BrickRed}3.31} & {\color{BrickRed}3.11} \\
$120^2$ & {\color{BrickRed}3.10} & {\color{BrickRed}3.18} & {\color{BrickRed}3.25} \\
$240^2$ & {\color{BrickRed}3.74} & {\color{BrickRed}3.61} & {\color{BrickRed}3.63} \\
$480^2$ & {\color{BrickRed}3.69} & {\color{BrickRed}3.65} & {\color{BrickRed}3.66} \\
\hline
        \end{tabular}
    \end{subtable}

\vskip\baselineskip 

\begin{subtable}[t]{0.48\textwidth}
        \centering
        \caption{Central, matrix \\
        \phantom{(c)} $\varepsilon=3.125 \times 5^{-s}$.}
        \begin{tabular}{rccc}
        \hline
DOF & $p=2$ & $p=3$ & $p=4$  \\ \hline
$30^2$ & 15.0 & 15.0 & {\color{BrickRed}14.09} \\
$60^2$ & {\color{BrickRed}4.68} & {\color{BrickRed}4.31} & {\color{BrickRed}4.28} \\
$120^2$ & {\color{BrickRed}4.53} & {\color{BrickRed}4.32} & {\color{BrickRed}3.79} \\
$240^2$ & {\color{BrickRed}4.84} & {\color{BrickRed}3.61} & {\color{BrickRed}3.55} \\
$480^2$ & {\color{BrickRed}3.77} & {\color{BrickRed}4.14} & {\color{BrickRed}3.52} \\
\hline
        \end{tabular}
    \end{subtable}
    \hfill
    \begin{subtable}[t]{0.48\textwidth}
        \centering
        \captionsetup{justification=centering}
        \caption{Entropy-stable, matrix-matrix \\
        \phantom{(d)} $\varepsilon=3.125 \times 5^{-s}$.}
        \begin{tabular}{rccc}
        \hline
DOF & $p=2$ & $p=3$ & $p=4$  \\ \hline
$30^2$ & 15.0 & 15.0 & 15.0 \\
$60^2$ & 15.0 & 15.0 & 15.0 \\
$120^2$ & 15.0 & 15.0 & 15.0 \\
$240^2$ & 15.0 & 15.0 & 15.0 \\
$480^2$ & 15.0 & 15.0 & 15.0 \\
\hline
        \end{tabular}
    \end{subtable}

\vskip\baselineskip 

\begin{subtable}[t]{0.48\textwidth}
        \centering
        \caption{Central, matrix \\
        \phantom{(c)} $\varepsilon=0.625 \times 5^{-s}$.}
        \begin{tabular}{rccc}
        \hline
DOF & $p=2$ & $p=3$ & $p=4$  \\ \hline
$30^2$ & {\color{BrickRed}4.33} & {\color{BrickRed}2.27} & {\color{BrickRed}2.44} \\
$60^2$ & {\color{BrickRed}1.56} & {\color{BrickRed}1.51} & {\color{BrickRed}1.41} \\
$120^2$ & {\color{BrickRed}3.71} & {\color{BrickRed}3.69} & {\color{BrickRed}2.24} \\
$240^2$ & {\color{BrickRed}3.53} & {\color{BrickRed}3.54} & {\color{BrickRed}3.53} \\
$480^2$ & {\color{BrickRed}3.37} & {\color{BrickRed}3.37} & {\color{BrickRed}3.40} \\
\hline
        \end{tabular}
    \end{subtable}
    \hfill
    \begin{subtable}[t]{0.48\textwidth}
        \centering
        \captionsetup{justification=centering}
        \caption{Entropy-stable, matrix-matrix \\
        \phantom{(d)} $\varepsilon=0.625 \times 5^{-s}$.}
        \begin{tabular}{rccc}
        \hline
DOF & $p=2$ & $p=3$ & $p=4$  \\ \hline
$30^2$ & 15.0 & 15.0 & 15.0 \\
$60^2$ & 15.0 & 15.0 & 15.0 \\
$120^2$ & 15.0 & 15.0 & 15.0 \\
$240^2$ & 15.0 & 15.0 & 15.0 \\
$480^2$ & 15.0 & 15.0 & 15.0 \\
\hline
        \end{tabular}
    \end{subtable}

\vskip\baselineskip 

\begin{subtable}[t]{0.48\textwidth}
        \centering
        \caption{Central, scalar \\
        \phantom{(c)} $\varepsilon=3.125 \times 5^{-s}$.}
        \begin{tabular}{rccc}
        \hline
DOF & $p=2$ & $p=3$ & $p=4$  \\ \hline
$30^2$ & 15.0 & 15.0 & 15.0 \\
$60^2$ & 15.0 & 15.0 & {\color{BrickRed}4.38} \\
$120^2$ & 15.0 & {\color{BrickRed}4.61} & {\color{BrickRed}4.31} \\
$240^2$ & 15.0 & 15.0 & {\color{BrickRed}4.54} \\
$480^2$ & {\color{BrickRed}4.86} & {\color{BrickRed}5.71} & {\color{BrickRed}4.55} \\
\hline
        \end{tabular}
    \end{subtable}
    \hfill
    \begin{subtable}[t]{0.48\textwidth}
        \centering
        \captionsetup{justification=centering}
        \caption{Entropy-stable, scalar-matrix \\
        \phantom{(d)} $\varepsilon=3.125 \times 5^{-s}$.}
        \begin{tabular}{rccc}
        \hline
DOF & $p=2$ & $p=3$ & $p=4$  \\ \hline
$30^2$ & 15.0 & 15.0 & 15.0 \\
$60^2$ & 15.0 & 15.0 & 15.0 \\
$120^2$ & 15.0 & 15.0 & 15.0 \\
$240^2$ & 15.0 & 15.0 & 15.0 \\
$480^2$ & 15.0 & 15.0 & 15.0 \\
\hline
        \end{tabular}
    \end{subtable}

\vskip\baselineskip 

\begin{subtable}[t]{0.48\textwidth}
        \centering
        \caption{Central, scalar \\
        \phantom{(c)} $\varepsilon=0.625 \times 5^{-s}$.}
        \begin{tabular}{rccc}
        \hline
DOF & $p=2$ & $p=3$ & $p=4$  \\ \hline
$30^2$ & 15.0 & 15.0 & {\color{BrickRed}6.99} \\
$60^2$ & 15.0 & {\color{BrickRed}4.59} & {\color{BrickRed}4.39} \\
$120^2$ & {\color{BrickRed}4.53} & {\color{BrickRed}4.50} & {\color{BrickRed}4.45} \\
$240^2$ & {\color{BrickRed}6.17} & {\color{BrickRed}4.55} & {\color{BrickRed}3.67} \\
$480^2$ & {\color{BrickRed}4.24} & {\color{BrickRed}4.45} & {\color{BrickRed}3.66} \\
\hline
        \end{tabular}
    \end{subtable}
    \hfill
    \begin{subtable}[t]{0.48\textwidth}
        \centering
        \captionsetup{justification=centering}
        \caption{Entropy-stable, scalar-matrix \\
        \phantom{(d)} $\varepsilon=0.625 \times 5^{-s}$.}
        \begin{tabular}{rccc}
        \hline
DOF & $p=2$ & $p=3$ & $p=4$  \\ \hline
$30^2$ & 15.0 & 15.0 & 15.0 \\
$60^2$ & 15.0 & 15.0 & 15.0 \\
$120^2$ & 15.0 & 15.0 & 15.0 \\
$240^2$ & 15.0 & 15.0 & 15.0 \\
$480^2$ & 15.0 & 15.0 & 15.0 \\
\hline
        \end{tabular}
    \end{subtable}

\label{tab:KelvinHelmholtz_hgtl}
\end{table}


\begin{table}[t]
\centering
\caption{Crash times (up to $t_f = 15$) for the Kelvin-Helmholtz instability described in~\S \ref{sec:KelvinHelmholtz} using HGT operators with $s=p+1$ and a single block.}
\renewcommand{\arraystretch}{1.2} 

\begin{subtable}[t]{0.48\textwidth}
        \centering
        \caption{Central, $\varepsilon=0$.}
        \begin{tabular}{rccc}
        \hline
DOF & $p=2$ & $p=3$ & $p=4$  \\ \hline
$30^2$ & {\color{BrickRed}1.57} & {\color{BrickRed}1.55} & {\color{BrickRed}1.43} \\
$60^2$ & {\color{BrickRed}1.26} & {\color{BrickRed}1.23} & {\color{BrickRed}1.23} \\
$120^2$ & {\color{BrickRed}1.40} & {\color{BrickRed}1.43} & {\color{BrickRed}1.45} \\
$240^2$ & {\color{BrickRed}1.89} & {\color{BrickRed}1.93} & {\color{BrickRed}1.94} \\
$480^2$ & {\color{BrickRed}2.29} & {\color{BrickRed}2.83} & {\color{BrickRed}2.84} \\
\hline
        \end{tabular}
    \end{subtable}
    \hfill
    \begin{subtable}[t]{0.48\textwidth}
        \centering
        \caption{Entropy-stable, $\varepsilon=0$.}
        \begin{tabular}{rccc}
        \hline
DOF & $p=2$ & $p=3$ & $p=4$  \\ \hline
$30^2$ & {\color{BrickRed}2.83} & {\color{BrickRed}2.95} & {\color{BrickRed}2.45} \\
$60^2$ & {\color{BrickRed}3.36} & {\color{BrickRed}2.42} & {\color{BrickRed}2.49} \\
$120^2$ & {\color{BrickRed}3.23} & {\color{BrickRed}3.18} & {\color{BrickRed}3.25} \\
$240^2$ & {\color{BrickRed}3.74} & {\color{BrickRed}3.42} & {\color{BrickRed}3.41} \\
$480^2$ & {\color{BrickRed}3.70} & {\color{BrickRed}3.53} & {\color{BrickRed}3.52} \\
\hline
        \end{tabular}
    \end{subtable}

\vskip\baselineskip 

\begin{subtable}[t]{0.48\textwidth}
        \centering
        \caption{Central, matrix \\
        \phantom{(c)} $\varepsilon=3.125 \times 5^{-s}$.}
        \begin{tabular}{rccc}
        \hline
DOF & $p=2$ & $p=3$ & $p=4$  \\ \hline
$30^2$ & 15.0 & 15.0 & {\color{BrickRed}2.32} \\
$60^2$ & {\color{BrickRed}4.03} & {\color{BrickRed}2.06} & {\color{BrickRed}1.74} \\
$120^2$ & {\color{BrickRed}4.51} & {\color{BrickRed}4.34} & {\color{BrickRed}3.10} \\
$240^2$ & {\color{BrickRed}4.84} & {\color{BrickRed}3.61} & {\color{BrickRed}3.05} \\
$480^2$ & {\color{BrickRed}3.77} & {\color{BrickRed}4.14} & {\color{BrickRed}3.11} \\
\hline
        \end{tabular}
    \end{subtable}
    \hfill
    \begin{subtable}[t]{0.48\textwidth}
        \centering
        \captionsetup{justification=centering}
        \caption{Entropy-stable, matrix-matrix \\
        \phantom{(d)} $\varepsilon=3.125 \times 5^{-s}$.}
        \begin{tabular}{rccc}
        \hline
DOF & $p=2$ & $p=3$ & $p=4$  \\ \hline
$30^2$ & 15.0 & 15.0 & {\color{BrickRed}2.51} \\
$60^2$ & 15.0 & 15.0 & {\color{BrickRed}4.60} \\
$120^2$ & 15.0 & 15.0 & {\color{BrickRed}4.36} \\
$240^2$ & 15.0 & 15.0 & {\color{BrickRed}5.78} \\
$480^2$ & 15.0 & {\color{BrickRed}5.45} & {\color{BrickRed}5.35} \\
\hline
        \end{tabular}
    \end{subtable}

\vskip\baselineskip 

\begin{subtable}[t]{0.48\textwidth}
        \centering
        \caption{Central, matrix \\
        \phantom{(c)} $\varepsilon=0.625 \times 5^{-s}$.}
        \begin{tabular}{rccc}
        \hline
DOF & $p=2$ & $p=3$ & $p=4$  \\ \hline
$30^2$ & {\color{BrickRed}4.34} & {\color{BrickRed}2.85} & {\color{BrickRed}1.56} \\
$60^2$ & {\color{BrickRed}1.56} & {\color{BrickRed}1.51} & {\color{BrickRed}1.41} \\
$120^2$ & {\color{BrickRed}3.71} & {\color{BrickRed}2.14} & {\color{BrickRed}1.96} \\
$240^2$ & {\color{BrickRed}3.53} & {\color{BrickRed}3.53} & {\color{BrickRed}3.01} \\
$480^2$ & {\color{BrickRed}3.37} & {\color{BrickRed}3.37} & {\color{BrickRed}3.12} \\
\hline
        \end{tabular}
    \end{subtable}
    \hfill
    \begin{subtable}[t]{0.48\textwidth}
        \centering
        \captionsetup{justification=centering}
        \caption{Entropy-stable, matrix-matrix \\
        \phantom{(d)} $\varepsilon=0.625 \times 5^{-s}$.}
        \begin{tabular}{rccc}
        \hline
DOF & $p=2$ & $p=3$ & $p=4$  \\ \hline
$30^2$ & 15.0 & 15.0 & {\color{BrickRed}2.30} \\
$60^2$ & 15.0 & {\color{BrickRed}4.26} & {\color{BrickRed}3.05} \\
$120^2$ & 15.0 & {\color{BrickRed}4.35} & {\color{BrickRed}3.97} \\
$240^2$ & 15.0 & {\color{BrickRed}8.21} & {\color{BrickRed}4.25} \\
$480^2$ & 15.0 & {\color{BrickRed}5.05} & {\color{BrickRed}3.89} \\
\hline
        \end{tabular}
    \end{subtable}

\vskip\baselineskip 

\begin{subtable}[t]{0.48\textwidth}
        \centering
        \caption{Central, scalar \\
        \phantom{(c)} $\varepsilon=3.125 \times 5^{-s}$.}
        \begin{tabular}{rccc}
        \hline
DOF & $p=2$ & $p=3$ & $p=4$  \\ \hline
$30^2$ & 15.0 & 15.0 & {\color{BrickRed}3.66} \\
$60^2$ & 15.0 & 15.0 & {\color{BrickRed}3.99} \\
$120^2$ & 15.0 & 15.0 & {\color{BrickRed}5.57} \\
$240^2$ & 15.0 & 15.0 & {\color{BrickRed}4.55} \\
$480^2$ & {\color{BrickRed}4.86} & {\color{BrickRed}5.96} & {\color{BrickRed}4.55} \\
\hline
        \end{tabular}
    \end{subtable}
    \hfill
    \begin{subtable}[t]{0.48\textwidth}
        \centering
        \captionsetup{justification=centering}
        \caption{Entropy-stable, scalar-matrix \\
        \phantom{(d)} $\varepsilon=3.125 \times 5^{-s}$.}
        \begin{tabular}{rccc}
        \hline
DOF & $p=2$ & $p=3$ & $p=4$  \\ \hline
$30^2$ & 15.0 & 15.0 & 15.0 \\
$60^2$ & 15.0 & 15.0 & 15.0 \\
$120^2$ & 15.0 & 15.0 & {\color{BrickRed}4.62} \\
$240^2$ & 15.0 & 15.0 & 15.0 \\
$480^2$ & 15.0 & 15.0 & 15.0 \\
\hline
        \end{tabular}
    \end{subtable}

\vskip\baselineskip 

\begin{subtable}[t]{0.48\textwidth}
        \centering
        \caption{Central, scalar \\
        \phantom{(c)} $\varepsilon=0.625 \times 5^{-s}$.}
        \begin{tabular}{rccc}
        \hline
DOF & $p=2$ & $p=3$ & $p=4$  \\ \hline
$30^2$ & 15.0 & {\color{BrickRed}4.75} & {\color{BrickRed}1.93} \\
$60^2$ & 15.0 & {\color{BrickRed}4.01} & {\color{BrickRed}1.90} \\
$120^2$ & 15.0 & {\color{BrickRed}5.99} & {\color{BrickRed}3.05} \\
$240^2$ & {\color{BrickRed}6.00} & {\color{BrickRed}4.56} & {\color{BrickRed}3.08} \\
$480^2$ & {\color{BrickRed}4.24} & {\color{BrickRed}4.46} & {\color{BrickRed}3.58} \\
\hline
        \end{tabular}
    \end{subtable}
    \hfill
    \begin{subtable}[t]{0.48\textwidth}
        \centering
        \captionsetup{justification=centering}
        \caption{Entropy-stable, scalar-matrix \\
        \phantom{(d)} $\varepsilon=0.625 \times 5^{-s}$.}
        \begin{tabular}{rccc}
        \hline
DOF & $p=2$ & $p=3$ & $p=4$  \\ \hline
$30^2$ & 15.0 & 15.0 & {\color{BrickRed}2.68} \\
$60^2$ & 15.0 & 15.0 & {\color{BrickRed}3.97} \\
$120^2$ & 15.0 & 15.0 & {\color{BrickRed}4.26} \\
$240^2$ & 15.0 & 15.0 & {\color{BrickRed}4.25} \\
$480^2$ & 15.0 & 15.0 & {\color{BrickRed}5.50} \\
\hline
        \end{tabular}
    \end{subtable}

\label{tab:KelvinHelmholtz_hgt}
\end{table}


\begin{table}[t]
\centering
\caption{Crash times (up to $t_f = 15$) for the Kelvin-Helmholtz instability described in~\S \ref{sec:KelvinHelmholtz} using Mattsson operators with $s=p+1$ and a single block.}
\renewcommand{\arraystretch}{1.2} 

\begin{subtable}[t]{0.48\textwidth}
        \centering
        \caption{Central, $\varepsilon=0$.}
        \begin{tabular}{rccc}
        \hline
DOF & $p=2$ & $p=3$ & $p=4$  \\ \hline
$30^2$ & {\color{BrickRed}1.75} & {\color{BrickRed}1.50} & {\color{BrickRed}1.47} \\
$60^2$ & {\color{BrickRed}1.24} & {\color{BrickRed}1.17} & {\color{BrickRed}1.17} \\
$120^2$ & {\color{BrickRed}1.40} & {\color{BrickRed}1.43} & {\color{BrickRed}1.45} \\
$240^2$ & {\color{BrickRed}1.88} & {\color{BrickRed}1.92} & {\color{BrickRed}1.93} \\
$480^2$ & {\color{BrickRed}2.29} & {\color{BrickRed}2.83} & {\color{BrickRed}2.73} \\
\hline
        \end{tabular}
    \end{subtable}
    \hfill
    \begin{subtable}[t]{0.48\textwidth}
        \centering
        \caption{Entropy-stable, $\varepsilon=0$.}
        \begin{tabular}{rccc}
        \hline
DOF & $p=2$ & $p=3$ & $p=4$  \\ \hline
$30^2$ & {\color{BrickRed}4.26} & {\color{BrickRed}3.43} & {\color{BrickRed}3.31} \\
$60^2$ & {\color{BrickRed}3.31} & {\color{BrickRed}3.15} & {\color{BrickRed}2.96} \\
$120^2$ & {\color{BrickRed}3.23} & {\color{BrickRed}3.21} & {\color{BrickRed}3.28} \\
$240^2$ & {\color{BrickRed}3.53} & {\color{BrickRed}3.69} & {\color{BrickRed}3.72} \\
$480^2$ & {\color{BrickRed}3.68} & {\color{BrickRed}3.74} & {\color{BrickRed}3.68} \\
\hline
        \end{tabular}
    \end{subtable}

\vskip\baselineskip 

\begin{subtable}[t]{0.48\textwidth}
        \centering
        \caption{Central, matrix \\
        \phantom{(c)} $\varepsilon=3.125 \times 5^{-s}$.}
        \begin{tabular}{rccc}
        \hline
DOF & $p=2$ & $p=3$ & $p=4$  \\ \hline
$30^2$ & 15.0 & {\color{BrickRed}2.06} & {\color{BrickRed}2.15} \\
$60^2$ & {\color{BrickRed}4.57} & {\color{BrickRed}4.13} & {\color{BrickRed}3.91} \\
$120^2$ & {\color{BrickRed}5.13} & {\color{BrickRed}4.20} & {\color{BrickRed}3.73} \\
$240^2$ & {\color{BrickRed}4.84} & {\color{BrickRed}3.59} & {\color{BrickRed}3.57} \\
$480^2$ & {\color{BrickRed}3.76} & {\color{BrickRed}3.71} & {\color{BrickRed}3.51} \\
\hline
        \end{tabular}
    \end{subtable}
    \hfill
    \begin{subtable}[t]{0.48\textwidth}
        \centering
        \captionsetup{justification=centering}
        \caption{Entropy-stable, matrix-matrix \\
        \phantom{(d)} $\varepsilon=3.125 \times 5^{-s}$.}
        \begin{tabular}{rccc}
        \hline
DOF & $p=2$ & $p=3$ & $p=4$  \\ \hline
$30^2$ & 15.0 & 15.0 & 15.0 \\
$60^2$ & 15.0 & 15.0 & 15.0 \\
$120^2$ & 15.0 & 15.0 & 15.0 \\
$240^2$ & 15.0 & 15.0 & 15.0 \\
$480^2$ & 15.0 & 15.0 & 15.0 \\
\hline
        \end{tabular}
    \end{subtable}

\vskip\baselineskip 

\begin{subtable}[t]{0.48\textwidth}
        \centering
        \caption{Central, matrix \\
        \phantom{(c)} $\varepsilon=0.625 \times 5^{-s}$.}
        \begin{tabular}{rccc}
        \hline
DOF & $p=2$ & $p=3$ & $p=4$  \\ \hline
$30^2$ & {\color{BrickRed}2.76} & {\color{BrickRed}1.63} & {\color{BrickRed}1.54} \\
$60^2$ & {\color{BrickRed}1.55} & {\color{BrickRed}1.33} & {\color{BrickRed}1.30} \\
$120^2$ & {\color{BrickRed}3.68} & {\color{BrickRed}3.64} & {\color{BrickRed}2.15} \\
$240^2$ & {\color{BrickRed}3.53} & {\color{BrickRed}3.53} & {\color{BrickRed}3.56} \\
$480^2$ & {\color{BrickRed}3.37} & {\color{BrickRed}3.38} & {\color{BrickRed}3.50} \\
\hline
        \end{tabular}
    \end{subtable}
    \hfill
    \begin{subtable}[t]{0.48\textwidth}
        \centering
        \captionsetup{justification=centering}
        \caption{Entropy-stable, matrix-matrix \\
        \phantom{(d)} $\varepsilon=0.625 \times 5^{-s}$.}
        \begin{tabular}{rccc}
        \hline
DOF & $p=2$ & $p=3$ & $p=4$  \\ \hline
$30^2$ & 15.0 & 15.0 & 15.0 \\
$60^2$ & 15.0 & 15.0 & 15.0 \\
$120^2$ & 15.0 & 15.0 & 15.0 \\
$240^2$ & 15.0 & 15.0 & 15.0 \\
$480^2$ & 15.0 & 15.0 & 15.0 \\
\hline
        \end{tabular}
    \end{subtable}

\vskip\baselineskip 

\begin{subtable}[t]{0.48\textwidth}
        \centering
        \caption{Central, scalar \\
        \phantom{(c)} $\varepsilon=3.125 \times 5^{-s}$.}
        \begin{tabular}{rccc}
        \hline
DOF & $p=2$ & $p=3$ & $p=4$  \\ \hline
$30^2$ & 15.0 & 15.0 & 15.0 \\
$60^2$ & 15.0 & 15.0 & {\color{BrickRed}4.42} \\
$120^2$ & 15.0 & 15.0 & 15.0 \\
$240^2$ & 150.0 & 15.0 & {\color{BrickRed}4.54} \\
$480^2$ & {\color{BrickRed}4.86} & {\color{BrickRed}5.93} & {\color{BrickRed}4.03} \\
\hline
        \end{tabular}
    \end{subtable}
    \hfill
    \begin{subtable}[t]{0.48\textwidth}
        \centering
        \captionsetup{justification=centering}
        \caption{Entropy-stable, scalar-matrix \\
        \phantom{(d)} $\varepsilon=3.125 \times 5^{-s}$.}
        \begin{tabular}{rccc}
        \hline
DOF & $p=2$ & $p=3$ & $p=4$  \\ \hline
$30^2$ & 15.0 & 15.0 & 15.0 \\
$60^2$ & 15.0 & 15.0 & 15.0 \\
$120^2$ & 15.0 & 15.0 & 15.0 \\
$240^2$ & 15.0 & 15.0 & 15.0 \\
$480^2$ & 15.0 & 15.0 & 15.0 \\
\hline
        \end{tabular}
    \end{subtable}

\vskip\baselineskip 

\begin{subtable}[t]{0.48\textwidth}
        \centering
        \caption{Central, scalar \\
        \phantom{(c)} $\varepsilon=0.625 \times 5^{-s}$.}
        \begin{tabular}{rccc}
        \hline
DOF & $p=2$ & $p=3$ & $p=4$  \\ \hline
$30^2$ & 15.0 & 15.0 & {\color{BrickRed}5.56} \\
$60^2$ & 15.0 & 15.0 & {\color{BrickRed}8.17} \\
$120^2$ & 15.0 & {\color{BrickRed}6.33} & {\color{BrickRed}4.22} \\
$240^2$ & {\color{BrickRed}5.99} & {\color{BrickRed}4.55} & {\color{BrickRed}3.69} \\
$480^2$ & {\color{BrickRed}4.25} & {\color{BrickRed}4.03} & {\color{BrickRed}3.65} \\
\hline
        \end{tabular}
    \end{subtable}
    \hfill
    \begin{subtable}[t]{0.48\textwidth}
        \centering
        \captionsetup{justification=centering}
        \caption{Entropy-stable, scalar-matrix \\
        \phantom{(d)} $\varepsilon=0.625 \times 5^{-s}$.}
        \begin{tabular}{rccc}
        \hline
DOF & $p=2$ & $p=3$ & $p=4$  \\ \hline
$30^2$ & 15.0 & 15.0 & 15.0 \\
$60^2$ & 15.0 & 15.0 & 15.0 \\
$120^2$ & 15.0 & 15.0 & 15.0 \\
$240^2$ & 15.0 & 15.0 & 15.0 \\
$480^2$ & 15.0 & 15.0 & 15.0 \\
\hline
        \end{tabular}
    \end{subtable}

\label{tab:KelvinHelmholtz_matt}
\end{table}


\begin{table}[!t]
\centering
\caption{Crash times (up to $t_f = 15$) for the Kelvin-Helmholtz instability described in~\S \ref{sec:KelvinHelmholtz} using LGL operators with $s=p$ and $K^2$ elements. The dissipation coefficients are $\varepsilon_1 = [0.01,0.004,0.002,0.0008,0.0004,0.0002]$ for $p=[3\text{--}8]$, respectively, and $\varepsilon_{1/5}$ as one fifth of these values. Results with matrix and scalar (for central schemes), or matrix-matrix and scalar-matrix (for entropy-stable schemes) volume dissipation are shown left/right, respectively.
}
\renewcommand{\arraystretch}{1.2} 

\begin{subtable}{\textwidth}
\centering
\captionsetup{justification=centering}
\caption{USE-LGL with Drikakis–Tsangaris flux splitting \cite{Drikakis}.}
\begin{tabular}{cccccccc}
\hline
$K$ & $\varepsilon$ & $p=3$ & $p=4$ & $p=5$ & $p=6$ & $p=7$ & $p=8$ \\ \hline
4 & $\varepsilon_1$ & 15.0 & {\color{BrickRed}2.63} & {\color{BrickRed}1.71} & {\color{BrickRed}2.60} & {\color{BrickRed}2.11} & {\color{BrickRed}1.86}
\\
8 & $\varepsilon_1$ & 15.0 & 15.0 & {\color{BrickRed}2.49} & {\color{BrickRed}2.25}& {\color{BrickRed}2.25} & {\color{BrickRed}1.87}
\\
16 & $\varepsilon_1$ & {\color{BrickRed}4.69} & {\color{BrickRed}3.94} & {\color{BrickRed}2.45} & {\color{BrickRed}2.52} & {\color{BrickRed}2.19} & {\color{BrickRed}2.23}
\\
32 & $\varepsilon_1$ & {\color{BrickRed}3.18} & {\color{BrickRed}3.23} & {\color{BrickRed}2.34} & {\color{BrickRed}2.37} & {\color{BrickRed}2.41} & {\color{BrickRed}2.17}
\\ \hline
4 & $\varepsilon_{1/5}$ & 15.0 & {\color{BrickRed}2.93} & {\color{BrickRed}3.63} & {\color{BrickRed}4.23} & {\color{BrickRed}1.91} & {\color{BrickRed}2.17}
\\
8 & $\varepsilon_{1/5}$ & {\color{BrickRed}1.82} & {\color{BrickRed}3.48} & {\color{BrickRed}4.22} & {\color{BrickRed}1.54} & {\color{BrickRed}2.88} & {\color{BrickRed}3.69}
\\
16 & $\varepsilon_{1/5}$ & {\color{BrickRed}4.31} & {\color{BrickRed}3.46} & {\color{BrickRed}4.05} & {\color{BrickRed}3.45} & {\color{BrickRed}3.58} & {\color{BrickRed}3.14}
\\
32 & $\varepsilon_{1/5}$ & {\color{BrickRed}3.42} & {\color{BrickRed}3.58} & {\color{BrickRed}3.62} & {\color{BrickRed}3.36} & {\color{BrickRed}3.49} & {\color{BrickRed}3.20}
\\ \hline
\end{tabular}
\end{subtable}

\vskip\baselineskip 

\begin{subtable}{\textwidth}
\centering
\captionsetup{justification=centering}
\caption{Central with matrix/scalar conservative dissipation.}
\begin{tabular}{cccccccc}
\hline
$K$ & $\varepsilon$ & $p=3$ & $p=4$ & $p=5$ & $p=6$ & $p=7$ & $p=8$ \\ \hline
4 & 0 & {\color{BrickRed}1.72} & {\color{BrickRed}1.46} & {\color{BrickRed}2.81} & {\color{BrickRed}2.57} & {\color{BrickRed}1.57} & {\color{BrickRed}1.47}
\\
8 & 0 & {\color{BrickRed}2.89} & {\color{BrickRed}2.94} & {\color{BrickRed}1.27} & {\color{BrickRed}1.65} & {\color{BrickRed}1.65} & {\color{BrickRed}1.26}
\\
16 & 0 & {\color{BrickRed}1.59} & {\color{BrickRed}1.50} & {\color{BrickRed}2.07} & {\color{BrickRed}1.69} & {\color{BrickRed}2.97} & {\color{BrickRed}1.71}
\\
32 & 0 & {\color{BrickRed}1.64} & {\color{BrickRed}3.49} & {\color{BrickRed}2.90} & {\color{BrickRed}3.19} & {\color{BrickRed}2.86} & {\color{BrickRed}3.22}
\\ \hline
4 & $\varepsilon_1$ & {\color{BrickRed}13.58}/{\color{BrickRed}6.37} & {\color{BrickRed}2.86}/{\color{BrickRed}2.83} & {\color{BrickRed}2.87}/{\color{BrickRed}3.28} & {\color{BrickRed}2.98}/{\color{BrickRed}2.75} & {\color{BrickRed}1.81}/{\color{BrickRed}1.64} & {\color{BrickRed}2.08}/{\color{BrickRed}1.88}
\\
8 & $\varepsilon_1$ & {\color{BrickRed}3.20}/{\color{BrickRed}1.50} & {\color{BrickRed}2.98}/{\color{BrickRed}2.11} & {\color{BrickRed}1.36}/{\color{BrickRed}2.07} & {\color{BrickRed}2.39}/{\color{BrickRed}1.71} & {\color{BrickRed}1.65}/{\color{BrickRed}1.70} & {\color{BrickRed}1.48}/{\color{BrickRed}2.07}
\\
16 & $\varepsilon_1$ & {\color{BrickRed}1.64}/{\color{BrickRed}2.24} & {\color{BrickRed}2.36}/{\color{BrickRed}3.47} & {\color{BrickRed}2.09}/{\color{BrickRed}3.38} & {\color{BrickRed}1.93}/{\color{BrickRed}2.98} & {\color{BrickRed}3.32}/{\color{BrickRed}3.53} & {\color{BrickRed}2.95}/{\color{BrickRed}2.33}
\\
32 & $\varepsilon_1$ & {\color{BrickRed}1.90}/{\color{BrickRed}2.06} & {\color{BrickRed}3.65}/{\color{BrickRed}3.17} & {\color{BrickRed}3.55}/{\color{BrickRed}3.01} & {\color{BrickRed}3.56}/{\color{BrickRed}3.03} & {\color{BrickRed}3.47}/{\color{BrickRed}3.08} & {\color{BrickRed}3.23}/{\color{BrickRed}3.13}
\\ \hline
4 & $\varepsilon_{1/5}$ & {\color{BrickRed}1.85}/15.0 & {\color{BrickRed}1.52}/{\color{BrickRed}2.96} & {\color{BrickRed}2.74}/15.0 & {\color{BrickRed}2.59}/{\color{BrickRed}3.44} & {\color{BrickRed}1.63}/{\color{BrickRed}2.14} & {\color{BrickRed}1.48}/{\color{BrickRed}1.80}
\\
8 & $\varepsilon_{1/5}$ & {\color{BrickRed}2.92}/{\color{BrickRed}1.93} & {\color{BrickRed}2.95}/{\color{BrickRed}4.98} & {\color{BrickRed}1.27}/{\color{BrickRed}3.51} & {\color{BrickRed}1.76}/{\color{BrickRed}1.72} & {\color{BrickRed}1.64}/{\color{BrickRed}2.87} & {\color{BrickRed}1.28}/{\color{BrickRed}2.33}
\\
16 & $\varepsilon_{1/5}$ & {\color{BrickRed}1.64}/{\color{BrickRed}2.37} & {\color{BrickRed}1.54}/{\color{BrickRed}4.08} & {\color{BrickRed}2.07}/{\color{BrickRed}3.37} & {\color{BrickRed}1.72}/{\color{BrickRed}3.02} & {\color{BrickRed}3.01}/{\color{BrickRed}3.31} & {\color{BrickRed}2.94}/{\color{BrickRed}2.99}
\\
32 & $\varepsilon_{1/5}$ & {\color{BrickRed}1.74}/{\color{BrickRed}3.20} & {\color{BrickRed}3.60}/{\color{BrickRed}3.24} & {\color{BrickRed}2.91}/{\color{BrickRed}3.57} & {\color{BrickRed}3.35}/{\color{BrickRed}3.08} & {\color{BrickRed}3.18}/{\color{BrickRed}3.14} & {\color{BrickRed}3.22}/{\color{BrickRed}3.17}
\\ \hline
\end{tabular}
\end{subtable}

\vskip\baselineskip 

\begin{subtable}{\textwidth}
\centering
\captionsetup{justification=centering}
\caption{Entropy-stable with matrix-matrix/scalar-matrix entropy dissipation.}
\begin{tabular}{cccccccc}
\hline
$K$ & $\varepsilon$ & $p=3$ & $p=4$ & $p=5$ & $p=6$ & $p=7$ & $p=8$ \\ \hline
4 & 0 & 15.0 & 15.0 & 15.0 & 15.0 & 15.0 & 15.0
\\
8 & 0 & 15.0 & 15.0 & {\color{BrickRed}5.79} & {\color{BrickRed}4.20} & 15.0 & {\color{BrickRed}4.25}
\\
16 & 0 & 15.0 & 15.0 & {\color{BrickRed}5.79} & 15.0 & {\color{BrickRed}5.01} & {\color{BrickRed}4.95}
\\
32 & 0 & {\color{BrickRed}3.52} & 15.0 & {\color{BrickRed}4.80} & 15.0 & {\color{BrickRed}5.95} & {\color{BrickRed}4.44}
\\ \hline
4 & $\varepsilon_1$ & 15.0/15.0 & 15.0/15.0 & 15.0/15.0 & 15.0/15.0 & 15.0/15.0 & 15.0/15.0
\\
8 & $\varepsilon_1$ & 15.0/15.0 & 15.0/15.0 & 15.0/15.0 & 15.0/15.0 & 15.0/15.0 & 15.0/15.0
\\
16 & $\varepsilon_1$ & 15.0/15.0 & 15.0/15.0 & 15.0/15.0 & 15.0/15.0 & 15.0/15.0 & 15.0/15.0
\\
32 & $\varepsilon_1$ & 15.0/15.0 & 15.0/15.0 & 15.0/15.0 & 15.0/15.0 & 15.0/15.0 & 15.0/15.0
\\ \hline
4 & $\varepsilon_{1/5}$ & 15.0/15.0 & 15.0/15.0 & 15.0/15.0 & 15.0/15.0 & 15.0/15.0 & 15.0/15.0
\\
8 & $\varepsilon_{1/5}$ & 15.0/15.0 & 15.0/15.0 & 15.0/15.0 & 15.0/15.0 & 15.0/15.0 & 15.0/15.0
\\
16 & $\varepsilon_{1/5}$ & 15.0/15.0 & 15.0/15.0 & 15.0/15.0 & 15.0/15.0 & 15.0/15.0 & 15.0/15.0
\\
32 & $\varepsilon_{1/5}$ & 15.0/15.0 & 15.0/15.0 & 15.0/15.0 & 15.0/15.0 & 15.0/15.0 & 15.0/15.0
\\ \hline
\end{tabular}
\end{subtable}

\label{tab:KelvinHelmholtz_lgl}
\end{table}


\begin{table}[!t]
\centering
\caption{Crash times (up to $t_f = 15$) for the Kelvin-Helmholtz instability described in~\S \ref{sec:KelvinHelmholtz} using LG operators with $s=p$ and $K^2$ elements. The dissipation coefficients are $\varepsilon_1 = [0.01,0.004,0.002,0.0008,0.0004,0.0002]$ for $p=[3\text{--}8]$, respectively, and $\varepsilon_{1/5}$ as one fifth of these values. Results with matrix and scalar (for central schemes), or matrix-matrix and scalar-matrix (for entropy-stable schemes) volume dissipation are shown left/right, respectively.
}
\renewcommand{\arraystretch}{1.2} 

\begin{subtable}{\textwidth}
\centering
\captionsetup{justification=centering}
\caption{USE-LG with Drikakis–Tsangaris flux splitting \cite{Drikakis}.}
\begin{tabular}{cccccccc}
\hline
$K$ & $\varepsilon$ & $p=3$ & $p=4$ & $p=5$ & $p=6$ & $p=7$ & $p=8$ \\ \hline
4 & $\varepsilon_1$ & 15.0 & 15.0 & {\color{BrickRed}1.82} & {\color{BrickRed}2.09} & {\color{BrickRed}1.91} & {\color{BrickRed}1.92}
\\
8 & $\varepsilon_1$ & 15.0 & {\color{BrickRed}2.15} & {\color{BrickRed}2.81} & {\color{BrickRed}2.33}& {\color{BrickRed}1.74} & {\color{BrickRed}2.80}
\\
16 & $\varepsilon_1$ & 15.0 & {\color{BrickRed}3.56} & {\color{BrickRed}3.21} & {\color{BrickRed}3.08} & {\color{BrickRed}2.55} & {\color{BrickRed}2.95}
\\
32 & $\varepsilon_1$ & {\color{BrickRed}3.61} & {\color{BrickRed}3.63} & {\color{BrickRed}3.14} & {\color{BrickRed}2.60} & {\color{BrickRed}2.47} & {\color{BrickRed}2.31}
\\ \hline
4 & $\varepsilon_{1/5}$ & 15.0 & {\color{BrickRed}2.19} & {\color{BrickRed}1.90} & {\color{BrickRed}1.66} & {\color{BrickRed}2.72} & {\color{BrickRed}2.22}
\\
8 & $\varepsilon_{1/5}$ & 15.0 & {\color{BrickRed}2.94} & {\color{BrickRed}3.95} & {\color{BrickRed}3.77} & {\color{BrickRed}1.92} & {\color{BrickRed}2.89}
\\
16 & $\varepsilon_{1/5}$ & {\color{BrickRed}13.25} & {\color{BrickRed}3.82} & {\color{BrickRed}3.18} & {\color{BrickRed}3.08} & {\color{BrickRed}3.48} & {\color{BrickRed}3.09}
\\
32 & $\varepsilon_{1/5}$ & {\color{BrickRed}4.64} & {\color{BrickRed}3.97} & {\color{BrickRed}3.72} & {\color{BrickRed}3.20} & {\color{BrickRed}3.20} & {\color{BrickRed}3.27}
\\ \hline
\end{tabular}
\end{subtable}

\vskip\baselineskip 

\begin{subtable}{\textwidth}
\centering
\captionsetup{justification=centering}
\caption{Central with matrix/scalar conservative dissipation.}
\begin{tabular}{cccccccc}
\hline
$K$ & $\varepsilon$ & $p=3$ & $p=4$ & $p=5$ & $p=6$ & $p=7$ & $p=8$ \\ \hline
4 & 0 & {\color{BrickRed}2.23} & {\color{BrickRed}2.43} & {\color{BrickRed}1.01} & {\color{BrickRed}1.13} & {\color{BrickRed}2.04} & {\color{BrickRed}1.65}
\\
8 & 0 & {\color{BrickRed}2.58} & {\color{BrickRed}1.09} & {\color{BrickRed}1.64} & {\color{BrickRed}1.70} & {\color{BrickRed}1.54} & {\color{BrickRed}2.07}
\\
16 & 0 & {\color{BrickRed}2.13} & {\color{BrickRed}1.97} & {\color{BrickRed}2.14} & {\color{BrickRed}2.32} & {\color{BrickRed}2.44} & {\color{BrickRed}2.90}
\\
32 & 0 & {\color{BrickRed}2.05} & {\color{BrickRed}3.17} & {\color{BrickRed}2.99} & {\color{BrickRed}3.11} & {\color{BrickRed}3.15} & {\color{BrickRed}3.19}
\\ \hline
4 & $\varepsilon_1$ & 15.0/15.0 & {\color{BrickRed}2.06}/{\color{BrickRed}1.76} & {\color{BrickRed}0.96}/{\color{BrickRed}1.06} & {\color{BrickRed}1.03}/{\color{BrickRed}1.23} & {\color{BrickRed}2.13}/{\color{BrickRed}1.68} & {\color{BrickRed}2.42}/{\color{BrickRed}2.44}
\\
8 & $\varepsilon_1$ & {\color{BrickRed}2.15}/{\color{BrickRed}2.03} & {\color{BrickRed}1.10}/{\color{BrickRed}2.39} & {\color{BrickRed}3.96}/{\color{BrickRed}2.86} & {\color{BrickRed}3.08}/{\color{BrickRed}1.81} & {\color{BrickRed}1.64}/{\color{BrickRed}1.76} & {\color{BrickRed}2.11}/{\color{BrickRed}2.11}
\\
16 & $\varepsilon_1$ & {\color{BrickRed}3.89}/{\color{BrickRed}3.17} & {\color{BrickRed}3.85}/{\color{BrickRed}2.60} & {\color{BrickRed}3.07}/{\color{BrickRed}3.39} & {\color{BrickRed}3.57}/{\color{BrickRed}3.54} & {\color{BrickRed}3.24}/{\color{BrickRed}3.21} & {\color{BrickRed}3.35}/{\color{BrickRed}3.58}
\\
32 & $\varepsilon_1$ & {\color{BrickRed}4.23}/{\color{BrickRed}2.01} & {\color{BrickRed}3.93}/{\color{BrickRed}3.99} & {\color{BrickRed}3.53}/{\color{BrickRed}3.57} & {\color{BrickRed}3.50}/{\color{BrickRed}3.05} & {\color{BrickRed}3.23}/{\color{BrickRed}3.11} & {\color{BrickRed}3.21}/{\color{BrickRed}3.15}
\\ \hline
4 & $\varepsilon_{1/5}$ & {\color{BrickRed}2.49}/15.0 & {\color{BrickRed}2.16}/{\color{BrickRed}2.21} & {\color{BrickRed}0.98}/{\color{BrickRed}1.08} & {\color{BrickRed}1.08}/{\color{BrickRed}1.46} & {\color{BrickRed}2.10}/{\color{BrickRed}2.08} & {\color{BrickRed}1.85}/{\color{BrickRed}2.66}
\\
8 & $\varepsilon_{1/5}$ & {\color{BrickRed}2.55}/{\color{BrickRed}2.44} & {\color{BrickRed}1.12}/{\color{BrickRed}1.38} & {\color{BrickRed}1.65}/{\color{BrickRed}2.75} & {\color{BrickRed}1.71}/{\color{BrickRed}1.93} & {\color{BrickRed}1.62}/{\color{BrickRed}1.73} & {\color{BrickRed}2.10}/{\color{BrickRed}3.06}
\\
16 & $\varepsilon_{1/5}$ & {\color{BrickRed}3.66}/{\color{BrickRed}3.33} & {\color{BrickRed}1.92}/{\color{BrickRed}2.53} & {\color{BrickRed}2.14}/{\color{BrickRed}3.92} & {\color{BrickRed}2.30}/{\color{BrickRed}3.67} & {\color{BrickRed}3.10}/{\color{BrickRed}3.65} & {\color{BrickRed}2.98}/{\color{BrickRed}3.38}
\\
32 & $\varepsilon_{1/5}$ & {\color{BrickRed}2.04}/{\color{BrickRed}4.36} & {\color{BrickRed}3.61}/{\color{BrickRed}3.72} & {\color{BrickRed}3.19}/{\color{BrickRed}3.56} & {\color{BrickRed}3.11}/{\color{BrickRed}3.49} & {\color{BrickRed}3.14}/{\color{BrickRed}3.18} & {\color{BrickRed}3.19}/{\color{BrickRed}3.15}
\\ \hline
\end{tabular}
\end{subtable}

\vskip\baselineskip 

\begin{subtable}{\textwidth}
\centering
\captionsetup{justification=centering}
\caption{Entropy-stable with matrix-matrix/scalar-matrix entropy dissipation.}
\begin{tabular}{cccccccc}
\hline
$K$ & $\varepsilon$ & $p=3$ & $p=4$ & $p=5$ & $p=6$ & $p=7$ & $p=8$ \\ \hline
4 & 0 & {\color{BrickRed}8.84} & {\color{BrickRed}2.40} & {\color{BrickRed}2.16} & {\color{BrickRed}3.03} & {\color{BrickRed}2.76} & {\color{BrickRed}1.86}
\\
8 & 0 & {\color{BrickRed}2.38} & {\color{BrickRed}2.71} & {\color{BrickRed}4.03} & {\color{BrickRed}2.00} & {\color{BrickRed}2.32} & {\color{BrickRed}2.09}
\\
16 & 0 & {\color{BrickRed}2.85} & {\color{BrickRed}3.61} & {\color{BrickRed}3.48} & {\color{BrickRed}3.08} & {\color{BrickRed}3.20} & {\color{BrickRed}3.03}
\\
32 & 0 & {\color{BrickRed}3.52} & {\color{BrickRed}3.57} & {\color{BrickRed}3.14} & {\color{BrickRed}3.12} & {\color{BrickRed}3.15} & {\color{BrickRed}3.23}
\\ \hline
4 & $\varepsilon_1$ & 15.0/15.0 & 15.0/15.0 & 15.0/15.0 & 15.0/15.0 & 15.0/15.0 & {\color{BrickRed}2.77}/{\color{BrickRed}7.62}
\\
8 & $\varepsilon_1$ & 15.0/15.0 & 15.0/15.0 & 15.0/15.0 & {\color{BrickRed}4.49}/15.0 & 15.0/{\color{BrickRed}10.97} & {\color{BrickRed}2.56}/{\color{BrickRed}3.64}
\\
16 & $\varepsilon_1$ & 15.0/15.0 & 15.0/15.0 & 15.0/{\color{BrickRed}10.70} & {\color{BrickRed}4.62}/{\color{BrickRed}3.80} & {\color{BrickRed}4.91}/{\color{BrickRed}4.19} & {\color{BrickRed}3.84}/{\color{BrickRed}3.92}
\\
32 & $\varepsilon_1$ & 15.0/15.0 & 15.0/{\color{BrickRed}5.41} & 15.0/{\color{BrickRed}4.78} & {\color{BrickRed}4.88}/{\color{BrickRed}4.69} & {\color{BrickRed}4.67}/{\color{BrickRed}4.52} & {\color{BrickRed}3.74}/{\color{BrickRed}3.63}
\\ \hline
4 & $\varepsilon_{1/5}$ & 15.0/15.0 & 15.0/15.0 & {\color{BrickRed}2.42}/15.0 & {\color{BrickRed}2.70}/15.0 & 15.0/15.0 & {\color{BrickRed}3.90}/15.0
\\
8 & $\varepsilon_{1/5}$ & 15.0/15.0 & 15.0/15.0 & {\color{BrickRed}4.22}/15.0 & {\color{BrickRed}4.21}/{\color{BrickRed}4.02} & {\color{BrickRed}3.69}/{\color{BrickRed}4.07} & {\color{BrickRed}3.87}/{\color{BrickRed}3.99}
\\
16 & $\varepsilon_{1/5}$ & 15.0/15.0 & 15.0/15.0 & {\color{BrickRed}4.09}/{\color{BrickRed}4.14} & {\color{BrickRed}4.36}/15.0 & {\color{BrickRed}5.89}/{\color{BrickRed}4.15} & {\color{BrickRed}4.20}/{\color{BrickRed}4.11}
\\
32 & $\varepsilon_{1/5}$ & 15.0/15.0 & {\color{BrickRed}3.91}/15.0 & {\color{BrickRed}5.08}/15.0 & {\color{BrickRed}4.73}/{\color{BrickRed}5.18} & {\color{BrickRed}4.08}/{\color{BrickRed}5.84} & {\color{BrickRed}4.89}/{\color{BrickRed}5.04}
\\ \hline
\end{tabular}
\end{subtable}

\label{tab:KelvinHelmholtz_lg}
\end{table}

\FloatBarrier

\bibliographystyle{spmpsci}      
\bibliography{references.bib}

%
%

\endgroup
\end{document}